\documentclass[12pt]{amsart}
\usepackage{amssymb,amsmath,amsthm,amscd}
\usepackage[english]{babel}
\usepackage{graphicx}
\usepackage[mathscr]{eucal}
\evensidemargin=\oddsidemargin

\def\qed{\vbox{\hrule
  \hbox{\vrule\hbox to 5pt{\vbox to 8pt{\vfil}\hfil}\vrule}\hrule}}

\newcommand{\SO}{\operatorname{SO}}
\newcommand{\SU}{\operatorname{SU}}
\newcommand{\Un}{\operatorname{U}}
\newcommand{\Sp}{\operatorname{Sp}}
\newcommand{\Spin}{\operatorname{Spin}}

\newcommand{\ti}{\tilde}
\newcommand{\beg}{\begin{eqnarray*}}
\newcommand{\begn}{\begin{eqnarray}}
\newcommand{\en}{\end{eqnarray*}}
\newcommand{\enn}{\end{eqnarray}}
\newcommand{\eps}{\epsilon}
\newcommand{\join}{\ast}
\newcommand{\emp}{\emptyset}
\newcommand{\rk}{{\rm rk} \,}

\newcommand{\sca}{\langle \, \cdot \,  ,\cdot \,\rangle}
\newcommand{\dsp}{\displaystyle}

\newcommand{\Ric}{\mbox{\rm Ric}}

\renewcommand{\sc}{\mbox{\rm sc}}
\newcommand{\ric}{\mbox{\rm ric}}
\newcommand{\tr}{\mbox{\rm tr\,}}
\newcommand{\ande}{\mbox{\rm and}}

\newcommand{\bs}{\backslash}
\newcommand\Ad{\operatorname{Ad}}
\newcommand\Id{\operatorname{Id}}
\newcommand\ad{\operatorname{ad}}
\newcommand{\so}{\mbox{${\mathfrak s \mathfrak o}$}}
\newcommand{\su}{\mbox{${\mathfrak s \mathfrak u}$}}
\newcommand{\sy}{\mbox{${\mathfrak s \mathfrak p}$}}
\newcommand{\un}{\mbox{${\mathfrak u}$}}

\newcommand{\af}{\mbox{${\mathfrak a}$}}

\newcommand{\cf}{\mbox{${\mathfrak c}$}}

\newcommand{\g}{\mbox{${\mathfrak g}$}}
\newcommand{\h}{\mbox{${\mathfrak h}$}}
\newcommand{\kf}{\mbox{${\mathfrak k}$}}

\newcommand{\lf}{\mbox{${\mathfrak l}$}}
\newcommand{\m}{\mbox{${\mathfrak m}$}}
\newcommand{\n}{\mbox{${\mathfrak n}$}}
\newcommand{\p}{\mbox{${\mathfrak p}$}}

\newcommand{\rf}{\mbox{${\mathfrak r}$}}

\newcommand{\tf}{\mbox{${\mathfrak t}$}}
\newcommand{\z}{\mbox{${\mathfrak z}$}}
\newcommand{\U}{\mbox{${\mathbb U}$}}
\newcommand{\CC}{\mbox{${\mathbb C}$}}
\newcommand{\HH}{\mbox{${\mathbb H}$}}

\newcommand{\NN}{\mbox{${\mathbb N}$}}
\newcommand{\PP}{\mbox{${\mathbb P}$}}
\newcommand{\UU}{\mbox{${\mathbb U}$}}
\newcommand{\Q}{\mbox{${\mathbb Q}$}}
\newcommand{\RR}{\mbox{${\mathbb R}$}}
\newcommand{\T}{\mbox{${\mathbb T}$}}
\newcommand{\ZZ}{\mbox{${\mathbb Z}$}}
\newcommand{\WW}{\mbox{${\mathbb W}$}}
\newcommand{\N}{ \mathbf{N} }

\newcommand{\B}{\mbox{${\mathbb B}$}}
\newcommand{\Bt}{\mbox{${\mathbb B}_t$}}
\newcommand{\Bth}{\mbox{$\hat {\mathbb B}_t$}}
\newcommand{\X}{\mbox{${\mathbb X}$}}
\newcommand{\XXs}{{\mathbb X}_{s}}
\newcommand{\XXas}{{\mathbb X}_{as}}
\newcommand{\Br}{\operatorname{B}}
 \newcommand{\WWs}{\WW_s}
 \newcommand{\WWst}{\WW_{s,t}}
   
\newcommand{\co}{\textsf{c}}	
 \newcommand{\MG}{\mathcal{M}^G}
\newcommand{\MGo}{\mathcal{M}^G_1}
\newcommand{\MGoL}{(\MGo)^L}
\newcommand{\On}{\operatorname{O}}
\newcommand{\Sph}{\operatorname{S}}
\newcommand{\Sphb}{\mbox{${\mathbb S}$}}
\newcommand{\D}{\operatorname{D}}
\newcommand{\Ps}{\operatorname{P_{alg}}}
\newcommand{\Psn}{\operatorname{P_s}}
\newcommand{\Psnm}{\operatorname{P_{min}}}
\newcommand{\Psnass}{\operatorname{P_{as}}}
\newcommand{\Psnmg}{\operatorname{P_{\langle min \rangle}}}

\newcommand{\Pst}{\operatorname{P_t}}
\newcommand{\SymmH}{{\rm Sym}_{\m}^{H}}
\newcommand{\SymgH}{{\rm Sym}_{\g}^{H}}
\newcommand{\Si}{{\operatorname{S}}}
\newcommand{\WA}{\operatorname{W}}
\newcommand{\WSi}{\WA^{\Si}}

\newcommand{\BA}{\operatorname{B}}
\newcommand{\BSi}{\BA^{\Si}}

\newcommand{\TA}{\operatorname{T}}
\newcommand{\TSi}{\TA^{\Si}}

\newcommand{\WSis}{\operatorname{\WSi_s}}

\newcommand{\USi}{\operatorname{U}^{\Si}}
\newcommand{\Usre}{\USi_{\epsilon_0}}

\newcommand{\Xsre}{\BSi_{\epsilon_0}}
\newcommand{\XXsre}{\B_{\epsilon_0}}

\newcommand{\DeltaT}{\Delta^{A}} %changed T to A 
\newcommand{\Deltaf}{\Delta^{\rm min}}

\newcommand{\XSGH}{\operatorname{X}^{\Si}_{G/H}}
\newcommand{\XXGH}{{\X}_{G/H}}
\newcommand{\XXGHass}{{\X}_{G/H}^{as}}
\newcommand{\XGH}{\operatorname{X}_{G/H}}
\newcommand{\XGHass}{\operatorname{X}_{G/H}^{as}}
\newcommand{\XGHmin}{\operatorname{X}_{G/H}^{min}}
\newcommand{\DGH}{\Delta_{G/H}}

\newcommand{\Flts}{\mathcal{F}_{t,s}}
\newcommand{\Fln}{\mathcal{F}_s}
\newcommand{\Flnas}{\mathcal{F}_{as}}
\newcommand{\Flt}{\mathcal{F}_t}
\newcommand{\Fl}{\mathcal{F}}
\newcommand{\FF}{\mathscr F}
\newcommand{\Sub}{\mathcal{K}}
\newcommand{\Submin}{\mathcal{K}_{min}}
\newcommand{\Submg}{\mathcal{K}_{\langle min\rangle}}
\newcommand{\Subass}{\mathcal{K}_{as}}
\newcommand{\Cl}{\mathcal{C}}
\newcommand{\Clh}{\Cl_{\rm hom}}
\newcommand{\Nl}{\mathcal{N}}
\newcommand{\blds}[1]{\mbox{\scriptsize $\kf$}}
\newcommand{\bldss}[1]{\mbox{\scriptsize ${#1}$}}
\newcommand{\bldstik}[1]{\mbox{\scriptsize $\ti\kf$}}
\newcommand{\bldsg}[1]{\mbox{\scriptsize $\g$}}
\newcommand{\bldst}[1]{\mbox{\scriptsize $\tf$}}
\newcommand{\vk}{v^{\blds \kf}}
\newcommand{\Ak}{A^{\blds \kf}}
\newcommand{\At}{A^{\bldst \tf}}
\newtheoremstyle{fancy}{5pt}{5pt}{\itshape}{-7pt}{\textsc\bgroup}{.\egroup}{8pt}{}
\newtheoremstyle{fancy2}{5pt}{5pt}{}{-7pt}{\itshape}{.}{8pt}{ }

\theoremstyle{fancy}
\newtheorem{theorem}[equation]{Theorem}
\newtheorem{lemma}[equation]{Lemma}
\newtheorem{proposition}[equation]{Proposition}
\newtheorem{corollary}[equation]{Corollary}
\newtheorem{main}{Theorem}

\newtheorem{problem}{Problem}

\theoremstyle{fancy2}
\newtheorem{definition}[equation]{Definition}

\newtheorem{example}[equation]{Example}
\newtheorem{remark}[equation]{Remark}
\makeatletter
\def\numberwithin#1#2{\@ifundefined{c@#1}{\@nocnterrr}{%
  \@ifundefined{c@#2}{\@nocnterr}{%
  \@addtoreset{#1}{#2}%
  \toks@\expandafter\expandafter\expandafter{\csname the#1\endcsname}%
  \expandafter\xdef\csname the#1\endcsname
    {\expandafter\noexpand\csname the#2\endcsname
     .\the\toks@}}}}
\makeatother
\numberwithin{equation}{section}

\begin{document}

\title{Homogeneous Einstein Metrics and Butterflies}
\thanks{``Funded by the Deutsche Forschungsgemeinschaft (DFG, German Research Foundation) under Germany's Excellence Strategy EXC 2044 –390685587, Mathematics Münster: Dynamics--Geometry--Structure, and the Collaborative Research Centre CRC 1442, Geometry: Deformations and Rigidity''}

\author{Christoph B\"ohm}
\address{Mathematisches Institut, Universit\"{a}t M\"unster, Einsteinstr. 62,
48149 M\"unster, Germany}
\email{cboehm@math.uni-muenster.de}
\author{Megan M. Kerr}
\address{Department of Mathematics, Wellesley College,
         106 Central St., Wellesley, MA 02481, USA}
\email{mkerr@wellesley.edu}

\thanks{Statements and Declarations: The authors have no financial or proprietary interests in any material discussed in this article. 
Data sharing is not applicable to this article as no datasets were generated or analyzed during this work.}

\begin{abstract} {M.~M.~Graev associated in \cite{Gr} to a compact homogeneous space $G/H$ a nerve $\XGH$, 
whose non-contractibility implies the existence of a $G$-invariant Einstein metric on $G/H$. The nerve $\XGH$
is a compact, semi-algebraic set, defined purely Lie theoretically by intermediate subgroups. In this paper
we present a detailed description of the work of Graev and the curvature estimates of \cite{Bo}.}
 \end{abstract}

 \maketitle

\bigskip
 \begin{center}
\dedicatory{{\em This paper is dedicated to the memory of M.~M.~Graev.} } 
\end{center}
\bigskip

A Riemannian manifold $(M^n,g)$ is called Einstein if it has constant Ricci tensor, that is if $\ric(g) =
\lambda \cdot g$, $\lambda \in \RR$.
Although Einstein metrics on compact manifolds can be
characterized as the  critical points of the Hilbert
action \cite{H.D},  general existence and
non-existence results are hard to obtain and there is no clear conjecture. 
Still there exist many examples of Einstein manifolds
which have been constructed using bundle, symmetry and holonomy assumptions:
see  \cite{Bes}, \cite{L-W}, \cite{Wang}, \cite{J.D} and references therein.
Among many others, we mention here only Sasakian-Einstein metrics \cite{B-G},
Einstein metrics on  spheres \cite{Bo98}, \cite{BGK}, \cite{FH}, \cite{NW} and Ricci-flat manifolds with holonomy
$\operatorname{G}_2$ and ${\rm Spin}(7)$ \cite{Joy1}, \cite{Joy2}. Another outstanding achievement is the classification of compact K\"ahler-Einstein manifolds: see \cite{Aub}, \cite{YST}, \cite{CDS1}, \cite{CDS2}, \cite{CDS3}, \cite{Ti15}.

The main results in this paper, Theorem  \ref{theoA} and Theorem \ref{theoB}, both
 are general existence results for Einstein metrics on compact homogeneous spaces $M^n$
 (with positive Einstein constant $\lambda$).
 Theorem \ref{theoB} was crucial for the classification in \cite{B-K}:   
 Every compact, simply connected 
homogeneous space of dimension $n\leq 11$ admits a homogeneous Einstein metric, a result
which is optimal due to \cite{WZ2}.  For $n=4, 5, 6, 7$, see \cite{JG1}, \cite{A-D-F}, \cite{N-R}, \cite{N.Y.04}.  In higher dimensions though, the classification of compact homogeneous Einstein manifolds is wide open.
Concerning further results on compact homogeneous Einstein metrics we refer to 
 \cite{WZ2}, \cite{K-S}, \cite{WZ3}, \cite{A.A}, \cite{BWZ}, \cite{CS}, \cite{N.Y.19},
  \cite{L.J.21}, \cite{L-W21} and further work of M. M. Graev   \cite{Gr1} and \cite{Gr3}.

 %%% [2023] added reference two more of Graev's papers \cite{Gr1, Gr3}

Finally, let us mention that by  \cite{BL2} a homogeneous Einstein manifold $(M^n,g)$ with negative Einstein constant must be  diffeomorphic to $\RR^n$, confirming
in the affirmative the Alekseevskii conjecture.
Moreover, homogeneous Einstein metrics on Euclidean spaces are  
by  \cite{BL1} isometric to Einstein solvmanifolds.
 As a consequence the deep structure theory developed by Heber 
 \cite{H.J} and Lauret  \cite{L.J.} applies and the classification
 of Einstein solvmanifolds can be reduced to the classification 
 of nilsolitions: see \cite{Lau01}.

A compact homogeneous space $(M^n,g)$ is a compact Riemannian manifold 
on which a  compact Lie group $G$ acts transitively by isometries. Thus, it 
has a presentation $M^n=G/H$, where $H$ is a compact subgroup of  $G$. Recall, that by the  theorem of Bochner, the Einstein constant $\lambda$ of a $G$-invariant Einstein metric on $G/H$ 
is non-negative. It is zero if and only if the metric is flat \cite{A-K}, which happens if and only if $G/H$ is a torus. If the Einstein constant is positive,  $G/H$ has  finite
fundamental group, by the theorem of Myers.

The first main result of this paper is

\begin{main}[\cite{Gr}]\label{theoA}
Let $G/H$ be a compact homogeneous space with $G,H$ connected. If 
the nerve $\XGH$ is non-contractible, then $G/H$
admits a $G$-invariant Einstein metric.
\end{main}

 The nerve $\XGH$ of a compact homogeneous space $G/H$ is defined as follows: to each
intermediate subalgebra $\kf$ of $\g=T_e G$ with $\h=T_e H < \kf <\g$ we associate a (self-adjoint) projection map 
$P \in {\rm End}(\g)$ with $P^2=P$ and $\ker (P)=\kf$. 
The intermediate subalgebra $\kf$ is called {\em non-toral}
if $\kf$ is not an abelian extension of $\h$. 
Now for each flag $\varphi = (\kf_1<\cdots < \kf_r)$ of non-toral  subalgebras,
 we define the simplex $\Delta_\varphi^P$ as the convex hull of
the corresponding projections in ${\rm End}(\g)$. 
The union of all such simplices is the nerve $\XGH$, a compact, 
semi-algebraic set: see  Definition  \ref{def:XGH} and Lemma \ref{lem:XGHXXGH}.

Notice that the nerve $\XGH$ of $G/H$ might be the empty set, 
which  is non-contractible, by definition. This is the case considered in \cite{WZ2} where 
Wang and Ziller observed for the very first time that global properties of the scalar curvature 
function of homogeneous unit volume metrics can be derived from purely Lie theoretic properties of $G/H$. 

Geometrically, we think of the nerve as the space of nested, non-toral foliations of $G/H$.
If a subalgebra $\kf$ is the Lie algebra of a compact subgroup $K$ of $G$ then $K$ non-toral means
geometrically that the fibre $K/H$ in the  fibration $K/H\to G/H \to G/K$ admits a $K$-invariant metric 
with positive scalar curvature. (If instead $\kf$ is an abelian extension of $\h$ then $K/H$ is a torus and all homogeneous metrics on $K/H$ are flat.) 
When $K$ is non-toral, by shrinking $K/H$ and expanding $G/K$ 
(to keep volume one), one obtains a one-parameter family of metrics in the space $\MGo$ 
of $G$-invariant, volume one metrics on $G/H$ with arbitrarily large scalar curvature.

If $\XGH$ is not connected, then non-contractibility follows, of course. Therefore,
the graph theorem of \cite{BWZ} is an immediate consequence of Theorem \ref{theoA}. 
Notice, however, that since we do not consider toral subalgebras, the
nerve $\XGH$ can be disconnected while the graph of $G/H$ is connected. An easy
example is $G/H=(\SU(2)\times \SU(2))/S^1_{k,l}$, where $S^1_{k,l}$ is embedded diagonally into
the maximal torus $T^2=S^1\times S^1$ of $G$ with slope $(k,l)$ for generic $k,l \in \ZZ$.
Whereas $\XGH$ consists of two singletons, in the graph, each of  these singletons
is connected by an edge to the node of the toral subalgebra $\tf=T_eT^2$. 

Assuming that  $G$ and $H$ are connected is not necessary; we discuss this in  Section \ref{sec:disc}.
Moreover, notice that if $M^n=G/H$ is simply connected then $M^n=G_0/H_0$, where $G_0$ denotes
the connected component of $G$ containing the identity.

 The second main theorem of this paper is

\begin{main}[\cite{Bo}]\label{theoB}
Let $G/H$ be a compact homogeneous space with finite fundamental group and $G, H$ connected.
If the simplicial complex  $\Delta_{G/H}$ is not contractible,
then $G/H$ admits a $G$-invariant Einstein metric.
\end{main}

The simplicial complex $\DGH$ is defined as follows:
Let $A$ be a  maximal torus of a compact complement of $H$ in $N_G(H)$. Then 
the simplicial complex $\Delta_{G/AH}$ is homeomorphic to the (finite) flag
complex of (connected) intermediate subgroups $K$  with $AH<K<G$.
Notice that $\Delta_{G/AH}$ is independent of the choice of $A$, see Remark \ref{rem:T},
and that $A=\{e\}$ is allowed; this is the case when $\dim N_G(H)=\dim H$. We set 
$\DGH:=\Delta_{G/AH}$: see Lemma \ref{lem:DGH} and Definition \ref{def:DGH}.

Note that if  $\dim N_G(H)=\dim H$ then $\XGH=\DGH$.
In general however,  the nerve $\XGH$ and the simplicial complex $\Delta_{G/H}$ will be different,
for instance if there exist infinitely many non-toral subalgebras 
$\kf$ with $\h <\kf <\g$.  
Moreover, in Examples \ref{exa:Graev} and \ref{rem:familiesofsolutions}, we describe families of compact homogeneous spaces  with non-contractible nerve $\XGH$ but contractible simplicial complex $\DGH$. Whether Theorem \ref{theoB} follows
from Theorem \ref{theoA} is still open.

We mention that Theorem \ref{theoB} is particularly useful for classification results: 
see \cite{B-K}; infinitely many different homogeneous spaces can be handled at once. One of the easiest examples
is the family of Aloff-Wallach spaces $G/H=\SU(3)/S^1_{k,l}$ where $S^1_{k,l}$ is embedded 
into a maximal torus $T^2$ of $\SU(3)$ with slope $(k,l,-(k+l))$. For all $(k,l) \in \ZZ^2$, the complex $\Delta_{\SU(3)/T^2}$ consists of
three singletons. Thus by Theorem \ref{theoB}, the Aloff-Wallach spaces admit 
homogeneous Einstein metrics: see \cite{Wa}.
In Section \ref{sec:examples} we give further examples of homogeneous spaces with non-contractible $\DGH$,
and in Section \ref{sec:simcom} we give a formula for the homotopy type of $\DGH$ when $G/H$ is a product of homogeneous spaces.

Theorem \ref{theoA} and Theorem \ref{theoB} detect when $G$-invariant Einstein metrics must exist for
global reasons. That is, the non-contractibility of $\XGH$ implies that the high energy 
superlevel set $\{\sc \geq \sc_+\}$ of the scalar curvature (function)
$ \sc :\MGo \to \RR$, 
for a large positive constant $\sc_+$, is non-contractible as well. 
Then a generalized mountain pass lemma implies
the existence of a critical point, using that Palais-Smale condition (C) holds by \cite{BWZ}: 
see Section \ref{sec:proofs}. We also would like to mention that we believe that 
high energy superlevel sets and the nerve $\XGH$ are homotopy equivalent.

In Example \ref{ex:SO(n+k)/SO(n)} we show that $\ti H_{q(k)}(\XGH)\neq 0$ for
$G/H=\SO(n+k)/\SO(n)$, $k\geq 2$, $n\geq 3$ and $q(k)=\tfrac 12 k(k+1)-2$.
Since we use variational methods,
the existence of an Einstein metric $g \in \MGo$ whose augmented coindex 
(of $\sc$) is bounded  below by $q(k)+1$ (cf.~Definition \ref{def:coindex}) follows from Lemma \ref{lemcoindex}. Let us mention here the work of Lauret and Will and
Lauret and Lauret
on the stability of  (standard) homogeneous Einstein metrics in \cite{L.J.21},  \cite{L-W21}, \cite{L.E.L.J23}, and that isolated maxima can never  be obtained
by such global variational methods unless $\XGH = \emptyset$: see also  \cite{L-W22} and \cite{SSW}.

The property of having a noncontractible nerve or simplex is sufficient but not necessary. There exist many examples of compact
homogeneous spaces $G/H$ admitting $G$-invariant Einstein metrics, with a contractible nerve 
and a contractible simplex. For instance, suppose that $H<K<G$, that $\dim \MGo=1$ and that $K$ is the only (connected) 
non-toral intermediate subgroup: see \cite{WZ2}, \cite{DK}, \cite[Thm A.1]{He} and Section \ref{sec:eq}.
Then $\XGH$ is a singleton, thus contractible. 
For such spaces, the scalar curvature behaves like a polynomial
of degree three; it may have two, one, or zero critical points (Einstein metrics): see \cite[9.72]{Bes}.  
For a survey of further  
methods constructing homogeneous Einstein metrics 
 we refer to \cite{Wang}.

We turn now to the classification of compact homogeneous Einstein manifolds and
important open problems. 
 In \cite{B-K}, we introduced for a simply-connected,
compact homogeneous space $M^n$, the term \emph{canonical presentation}
$M^n=G_{can}/H_{can}$. Here $G_{can}$ is semisimple and $H_{can}$ is a
compact subgroup such that the projections of the simple factors of $H_{can}$ to the
simple factors of $G_{can}$ are never onto. 
 Note that a fixed homogeneous space $M^n$ may have infinitely many different
canonical presentations, e.g. $M^5=S^2 \times S^3=(\SU(2)\times SU(2))/S^1_{k,l}$.

Note furthermore, that any $G$-invariant metric
on a simply-connected, compact homogeneous 
manifold $M^n=G/H$ is $G_{can }$-invariant for 
a canonical presentation $M^n=G_{can}/H_{can}$,  where $G_{can}$ is uniquely determined by $G$; in fact, $G_{can}$ is a subgroup of $G$.
 It follows that the
set of $G$-invariant metrics on $M^n$ can be viewed
 a subset of the set of $G_{can}$-invariant metrics on $M^n$. As a consequence,
for showing non-existence of homogeneous Einstein metrics on a fixed compact homogeneous
space $M^n$, it is sufficent to show that there exist no $G_{can}$-invariant
Einstein metrics on $M^n$ for all canonical presentations $M^n=G_{can}/H_{can}$.

\begin{problem}\label{problem-1}
For any fixed $n_0 \in \NN$,
show that the classification of compact, simply-connected homogeneous spaces
$M^n=G_{can}/H_{can}$ is a finite problem, provided
that $n \leq n_0$.
\end{problem}

What is meant is whether (or not)  $G/H=G_{can}/H_{can}$ admits at least one $G$-invariant Einstein metric or not. The classification of all $G$-invariant Einstein
metrics on a given space $G/H$ is in general 
out of reach, already on $S^3 \times S^3$: see \cite{BCHL} and Problem 8.

Even a partial solution to Problem \ref{problem-1} would be a major
breakthrough. 
Note that for a fixed dimension $n$ there exist
only finitely many semisimple (simply-connected) Lie groups $G$
which can appear in a canonical presentation $M^n=G/H$. Furthermore,
up to conjugation each such semisimple Lie group $G$ has only finitely many
semisimple (compact) subgroups $H_{ss}$: see \cite{BWZ}. Suppose now
that $\dim N_G(H_{ss})> \dim H_{ss}$ and as above,  let $A$ denote a maximal torus
of a compact complement of $H_{ss}$ in $N_G(H_{ss})$. 
Suppose that $a:=\dim A \geq 2$ and let $A_q$ be a compact subtorus of $A$
of dimension $a_q:=\dim A_q$ such that  $1 \leq a_q < a$ for $q:=(q_1,...,q_{a_q}) \in \Q^{a\cdot a_q}$. 
There are infinitely many such $A_q$. Define $H_q:=A_qH_{ss}$.

If the existence problem for $G$-invariant Einstein metrics on $G/H_q$
and $G/H_{\tilde q}$ were equivalent for any such $q,\tilde q$, Problem \ref{problem-1} could be answered in the affirmative.
Note however,  that  this is not  always true. One must
 first consider  \emph{generic} torus bundles and then non-generic ones,
 for which the space of homogeneous metrics is strictly larger; the latter then
 might admit a homogeneous Einstein metric even though generic ones do not: see
 Remark \ref{rem:familiesofsolutions}.
Finally note that for homogeneous space $G/H_{ss}$ for which Theorem \ref{theoB}
applies, the above is true! One obtains homogeneous Einstein metrics
on all such torus bundles over $G/H_{ss}$ provided that the simplicial
complex of $G/H_{ss}$ is non-contractible.

Even if Problem \ref{problem-1} turns out to be true, it does not at all answer the question
whether or not the existence of a $G$-homogeneous Einstein metric
on $G/H$, given as a canonical presentation, is a strong or a weak condition. The result \cite{B-K} covers only low dimensions, and in high dimensions the answer might be completely different.  To approach this problem let us denote by $\Clh$
the set of all homogeneous spaces $G/H$ where
we assume of course that $G,H$ are connected and that  $G/H$ is simply-connected
(and one may even consider only canonical presentations).

\begin{problem}\label{prob:0}
Identify classes $\Cl_E \subset \Clh$, described purely Lie theoretically, 
 for which the existence problem of $G$-invariant Einstein metrics
 can be ``solved''.
\end{problem}

 By ``solved'' we mean that one knows for such a class $\Cl_E$
 whether the existence of a $G$-invariant Einstein metric
 is a strong condition or not. It should be clear that interesting classes should not be 
 too small.
 Theorem \ref{theoA} and Theorem \ref{theoB} provide such  classes.
Further classes are isotropy irreducible homogeneous spaces \cite{WZ4},
homogeneous space for which the isotropy representations has only
two summands, see \cite{DK},\cite{He}, and the so called generalized Wallach spaces,
see \cite{N.Y.16}, for which the isotropy representation has three summands.

We describe now two more classes, $\Nl_<$ and $\Nl_>$. We would like to 
 mention that it is shown in Theorem B and Theorem C in \cite{Bo05} that when
certain (natural) curvature assumptions hold for all metrics $g \in \MGo$, 
guaranteeing non-existence of $G$-invariant Einstein metrics on $G/H$,
then $G/H$ necessarily has a very special subgroup structure. 
This observation guides the definitions of 
the classes $\Nl_<$ and $\Nl_>$.

The class $\Nl_<$  is defined as follows (cf. Example \ref{exa:<}):
We take $G$ simple, such that there exists an intermediate subgroup $K$ such that 
(i) $\dim G/K>1$ and $\g \ominus \kf$ is $\Ad(H)$-irreducible, (ii)
$K/H=K_1/H_1 \times \cdots \times K_r/H_r$ is a product of isotropy irreducible spaces 
with $\dim K_i/H_i>1$ for all $i=1,...,r$, and (iii) $G/H \neq \Spin(8)/G_2$.  It should
be possible to classify the elements  $G/H \in \Nl_<$. 
If for $G/H \in \Nl_<$ we also assume 
 $\dim \MGo=1$, then the existence problem of $G$-invariant Einstein metrics
 can be reduced to an  algebraic invariant, described in \cite{WZ2}:
it is a discriminant of a quadratic equation: see Section \ref{sec:eq}. 
In this case  actual classification results could be achieved: see e.g. \cite{DK},\cite{He}.
For the general case, when $\dim \MGo>1$, see equation (3.8) in \cite{Bo05}.

The next class we define is  $\Nl_>$  (cf. Example \ref{exa:>}):
Again, we take $G$ simple, such that there exists an intermediated subgroup $K$ such that 
(i) $K/H$ is isotropy irreducible with $\dim G/K>1$, 
(ii) all $G$-invariant metrics on $G/H$ are submersion metrics with respect to 
$K/H \to G/H \to G/K$, (iii) $\dim N_G(K)=\dim K$,  
and (iv) for any maximal subgroup $L$ of $G$ with $H<L$ we have $K \leq L$.
We would like to highlight 
that $\dim \MGo$ is not uniformly bounded for  $G/H \in \Nl_>$.
For the class  $\Nl_>$  we outline a starting point to this approach in  Section \ref{sec:eq}.

In order to find further classes $ \Cl_E$ 
solving the following problems might be helpful.

\begin{problem}\label{prob:1} %%%%% revised for precision 2023
Classify compact homogeneous Einstein manifolds $M^n=G/H$,
for $G/H$ a canonical presentation,
 in low dimensions $n\geq 13$.  
\end{problem}

In previous work  \cite{B-K} we solved this problem up to dimension $n=12$.
It is clear that Theorem \ref{theoA} and Theorem \ref{theoB} will be useful, but not sufficient. It will be necessary to use the existence of homogeneous Einstein metrics 
which are not guaranteed for global reasons,
 as well as non-existence results: see \cite{Bo05}. 
Let us mention  here  that in \cite{B-K} it is not shown that all 
simply-connected, homogeneous spaces
$G/H$ with $\dim G/H\leq 11$ admit a $G$-invariant Einstein metric: only 
canonical presentations are covered.

\begin{problem}\label{prob:1.5} %%%%% revised for precision 2023
Classify compact homogeneous  Einstein manifolds $G/H$ with $\rk G=\rk H$.
\end{problem}

We note that the family of compact homogeneous spaces with $\rk G=\rk H$ are precisely those for which the Euler characteristic $\chi(G/H)$ is positive.   
For all such homogeneous spaces, Rau\ss e \cite{Rau} was able to compute the 
simplical complex $\DGH$:
see Section \ref{sec:examples}:
if $G$ is one of  the simple classical groups ($H<G$ with $\rk G=\rk H$),
his result is particularly nice: $\DGH$ is non-contractible exactly 
when the simple factors of $H$ all have the same type
(see Proposition \ref{prop:DGHcontr}). While there are many examples in this family for which $\DGH$ is
contractible, remarkably, we don't know a single such  space $G/H$ not admitting a $G$-invariant Einstein metric. 
 By contrast, if $G$ is an exceptional Lie group with $H < G$ and $\rk G=\rk H$, 
non-existence results are known: see \cite{B-K}.

\begin{problem}\label{prob:2}
Classify those compact homogeneous spaces $G/H$ with $\dim N_G(H)=\dim H$, for which the simplicial complex $\Delta_{G/H}$ is non-contractible.
\end{problem}

In each dimension, there exist only finitely many
spaces $G/H$ with $\dim N_G(H)=\dim H$,
see Remark \ref{rem:finiteproblem}.  When $\rk G=\rk H$,  Rau\ss e \cite{Rau} was able to use 
inductive methods to show non-contractibility of $\DGH$ in many cases.

\begin{problem}\label{prob:4}
Describe compact homogeneous spaces $G/H$ such that the cuplength of $\XGH$
is arbitrarily large. 
\end{problem}

A description of the set of spaces for which $\XGH$ has many non-vanishing homology groups would also 
be quite interesting.

\begin{problem}\label{prob:3}
Describe topological invariants of $G/H$, e.g. certain (new) characteristic classes, which can be related
to $\XGH$, either as implications of topological properties of $\XGH$
 or consequences for $\XGH$ of  
 topological invariants of $G/H$.
\end{problem}

It might be most natural to start with the equal rank case. For instance, when $\rk G=\rk H$, does the non-contractibility of $\XGH$ imply certain topological properties of $G/H$?

\begin{problem} Prove the Finiteness Conjecture of \cite{BWZ}:
 If $G/H$ is a compact homogeneous space whose
isotropy representation consists of pairwise inequivalent irreducible summands,  then 
the algebraic Einstein equations have only ﬁnitely many real solutions.
\end{problem}

Note that all homogeneous spaces $G/H$ of equal rank  satisfy this assumption.

\begin{problem} \label{prob:5}
Provide an ``algebraic'' proof  that the scalar curvature $\sc :\MGo\to \RR$  
satisfies Palais-Smale
condition (C) in the superlevel $\{ \sc \geq \epsilon\}$, for some $\epsilon >0$. 
\end{problem}

The proof given in \cite{BWZ}
is based on the deep theory on spaces with     
Ricci curvature bounded from below, developed by
Cheeger and Colding \cite{CC1}, \cite{CC2}.

\begin{problem}\label{prob:6}
Show ``algebraically'' that a (locally) homogeneous Ricci flat space is flat.
\end{problem}

Originally this goes back to \cite{A-K} in the globally homogeneous (non-compact) case;
 there is an easy proof  using the splitting theorem. The locally homogeneous case, due to Spiro \cite{Spi}, 
uses the globally homogeneous case: see also \cite{Bo15}. Of course the corresponding statement
for (non-compact) cohomogeneity one manifolds does not hold.

A Riemannian manifold $(M^n,g)$ is called a cohomogeneity one manifold
if it admits an isometric action by a compact Lie group $G$
whose principal orbits are hypersurfaces in $M^n$.

\begin{problem}\label{prob:7}
Does there exist a compact cohomogeneity one manifold $(M^n,g)$ such that the high energy levels
of the scalar curvature functional are disconnected? 
\end{problem}

It is a pleasure to thank Jorge Lauret, Wolfgang Ziller and Linus Kramer for helpful conversations and suggestions.

\tableofcontents

%%%%%%%%%%%%%%%%%%%%%%%%%%%%%%%%%%%%%%%%%%%%%%%%%%%%%%%%%%%%%%%%%%%%%%%%%%%%%%%%%%%%%%%%%%%%

\section{On the proof of Theorem \ref{theoA} and Theorem \ref{theoB}}\label{sec:proofs}

A Riemannian manifold $(M^n,g)$ is called Einstein if the Ricci curvature satisfies 
$$
 \ric(g) (X,Y) = \lambda \cdot  g(X,Y)
$$ 
for some  $\lambda \in \RR$, and for all $p \in M^n$ and all $X,Y \in T_pM$.
 In what follows, we  assume that $M^n$ is compact. Let $\sc(g)$ denote the
scalar curvature of $g$. Then Einstein metrics are the critical
points of the Hilbert action (the total scalar curvature functional):
$$
  T(g)= \int_{M^n} \sc(g)\,  \mu_g
$$
on the space $\mathcal{M}_1$ of Riemannian metrics of volume one \cite{Ber},~\cite{H.D}.

A Riemannian manifold $(M^n,g)$ is defined to be $G$-homogeneous if a Lie group $G$
acts transitively and isometrically on $(M^n,g)$; that is, for any $p$ and $q \in M^n$, there exists some
$\varphi \in G$, an isometry  of $(M^n,g)$, with $\varphi(p) = q$.  We write $H_p = \{\varphi \in G \mid
\varphi(p) = p\}$ for the isotropy subgroup corresponding to $p$.  Via the map
$\varphi \mapsto \varphi(p)$ we identify $G/H$ with $M^n$, where $H:=H_p$.
  
We restrict the Hilbert action
to the finite-dimensional space of $G$-invariant metrics 
of volume one, 
denoted $\MGo$, then $T(g)=\sc(g)$, since the scalar curvature of a homogeneous metric is constant.  
Critical points of $T|_{\MGo}$ are precisely the $G$-invariant
Einstein metrics of volume one \cite[p.121]{Bes}. 

We fix a base point $Q$ in $\MGo$, induced by a biinvariant metric on $G$ (also denoted by $Q$), 
and denote by $\m$ the $Q$-orthogonal complement to $\h$ in $\g$. Then $G$-invariant metrics on $G/H$ are in one-to-one correspondence
with $\Ad(H)$-invariant scalar products on $\m=T_{eH}G/H$. Using $Q$ again, we consider these scalar products
as $\Ad(H)$-equivariant, positive-definite endomorphisms $P$ of $\m$. That is, for all $X,Y\in \m$, we have
$g(X,Y)=Q(P_g \cdot X,Y)$: see Section \ref{sec:ginvm} for more details.

The tangent space of $\MGo$, $T_{\Id} \MGo$ ($\Id$ corresponds to the basepoint $Q$), consists of traceless $\Ad(H)$-equivariant endomorphisms of $\m$.
Endow $\MGo$ with the natural $L^2$-metric. With the help of the unit sphere $\Sph$  in $T_{\Id} \MGo$, for $v \in \Sph$, we can  parametrize a unit speed geodesic $\gamma_v(t)$ emanating from the point $\Id$  in the complete manifold $(\MGo,L^2)$.

We will show that high energy levels of the function $\sc:\MGo\to \RR$ carry topology provided
that the nerve $\XGH$ is not contractible.
Using a generalized
mountain pass lemma, we know there exists a Palais-Smale-sequence with uniform, positive lower energy bounds.
By \cite{BWZ}, the existence of a critical point  can then be shown.

 We use the Graev homeomorphism ${\rm Gr}:\Sph \to \Sphb$ 
to identify the unit sphere $\Sph$  in $T_{\Id} \MGo$ with $\Sphb$ in $\SymmH$, the set of all positive semi-definite endomorphisms with nontrivial kernel and trace 1: see Section \ref{sec:grmap}.
 It is one of the great observations of Graev
that we can work in $\Sphb$, where the definitions of several important sets are much simpler: 
see also Section \ref{sec:kfdisc} and Section \ref{sec:defWT}. We will use both models $\Sph$ and $\Sphb$,  indicating the domain via the notation, e.g.  $\XGH \subset \Sph$ and $\XXGH \subset \Sphb$.

We turn to the first major step in the proof of Theorem \ref{theoA}.
 In Sections \ref{sec:homotopy1} and \ref{sec:homotopy2} 
 we define a certain compact subset $\XXsre \subset \Sphb$ and (via the Graev map) $\Xsre$ of $\Sph$
 and we show that $\Xsre$ and $\XGH$ are homotopy equivalent: see Theorem \ref{thm:homotopy2} and Theorem \ref{thm:homotopy3}. 
Note that  there is a natural
realization $\XSGH$ of the nerve $\XGH$ as a subset $\XXGH$ of $\XXsre$ and that
 Graev introduced the very important {\em butterflies}: see Definition \ref{def:butterfly}.  
These sets are associated to flags $\varphi=(\kf_1 <  \cdots < \kf_r)$ of intermediate subalgebras and they
have important intersection properties: see Lemma \ref{lem:butphipsi}.  

We turn to the second major step  in the proof of Theorem \ref{theoA}.
In Section \ref{secscalab}  (Theorem \ref{theo:scalestu}) we show that
 on the complement of $\Xsre$ in $\Sph$  the scalar curvature is uniformly bounded above.   
That is, on this subset (of all initial directions $v\in \Sph \bs \Xsre$), 
for all large $t\geq \bar t$, the energy $\sc(\gamma_v(t))$ is uniformly bounded above. That means that 
in this non-compact region  of $\MGo$ 
there are no points ($G$-invariant metrics) of high energy (scalar curvature). The corresponding 
curvature estimates are carried out in Section \ref{secscalab} and the key result, 
 Lemma \ref{lem:torvcan}, uses the  \L ojasiewicz inequality stated in Proposition \ref{propinequ}.
More precisely, such uniform curvature estimates
can only be shown for initial directions $v$ outside an open neighborhood of $\Xsre$; one does 
need that $\bar t$ is independent of $v$. Here, because $\Xsre$
is compact and semi-algebraic, we use that
this open neighborhood can be chosen to be homotopy equivalent to $\Xsre$.
                                
With the homotopy equivalence and the scalar curvature estimates 
established, proving the existence of a critical point is now routine, 
see Theorem \ref{theomain},
provided that the scalar curvature functional satisfies the Palais-Smale condition (C).
 By Theorem A in \cite{BWZ}  this is the case. 
 We mention that in order 
  to apply variational methods we also prove positive lower scalar curvature 
estimates for a cycle in $\MGo$ which is the geodesic cone over $\XSGH$ with basepoint $Q$: 
see Section \ref{sec:scalestbe}.

The proof of Theorem \ref{theoB}
goes back to \cite{Bo}: This is carried out in Section \ref{sec:tricks}, where	 we use  that a certain subset of
$\MGo$ is invariant under the volume normalized Ricci flow.

%%%%%%%%%%%%%%%%%%%%%%%%%%%%%%%%%%%%%%%%%%%%%%%%%%%%%%%%%%%%%%%%%%%%%%%
\section{Combinatorial framework and important sets}\label{combinatorics}
We start by describing the space of $G$-invariant metrics on $G/H$. We define a homeomorphism between the unit tangent sphere in $T_{\Id} \MGo$ and a subspace of the space of symmetric, $\Ad(H)$-invariant positive semi-definite endomorphisms of $\m$. There we identify several useful structures, including convex disks, butterflies of flags, and the nerve $\XXGH$.

%%%%%%%%%%%%%%%%%%%%%%%%%%%%%%%%%%%%%%%%%%%%%%%%%%%%%%%%%%%%%%%%%%%%%%%
\subsection{{\it The space of $G$-invariant metrics}}\label{sec:ginvm}

In this section we will introduce the space  $\MGo$ of $G$-invariant metrics of volume one on 
a compact homogeneous space $G/H$. We will show that $\MGo$, endowed with the $L^2$-metric,
is a non-compact symmetric space.

\medskip

Let $G/H$ be a connected, almost effective, $n$-dimensional
homogeneous space, where $G$ and $H$ are compact, {\it connected} Lie groups.
We refer to Section \ref{sec:disc} for the most general case.
By \cite[Chapter 0, Theorem 5.1]{B.G}, we
may assume $G\subset \SO(6N)$ for some $N \in \NN$.
%%%%% 2023 
Choosing  such a biinvariant background metric $Q$ has the advantage that
a Cartan subalgebra of the $Q$-orthogonal complement $\m_0$ of $\h$ in $\n(\h)$
is compact: see Lemma \ref{lemAH} and Lemma \ref{lemnor}.

The negative of the Killing
form on $\SO(6N)$ induces a biinvariant metric $Q$ on $G$.
The metric $Q$ induces a normal homogeneous metric on $G/H$, again
denoted by $Q$, which we may assume to have volume one after
rescaling. Let us fix $Q$ once and for all.

Let $\g :=T_eG$ and $\h:= T_e H$  denote the Lie algebras of $G$ and $H$, respectively, and let
$$
   \g=\h \oplus \m
$$
be the ${\rm Ad}(H)$-invariant decomposition of
$\g$ with $Q(\h,\m)=0$. Note that the biinvariant metric $Q$ gives rise to
an $\Ad(G)$-invariant
scalar product on $\g$, again denoted by $Q$. Thus, the restriction of
$Q$ to $\m$, still denoted by $Q$, is $\Ad(H)$-invariant, and consequently for all
\mbox{$h\in H$} we have
${\rm \Ad}(h)\vert_{\m}  \in \On(\m,Q)$. %added 2023

As is well-known (see \cite[7.23]{Bes}),
the space $\MG$ of $G$-invariant metrics on $G/H$ can be
identified with the set of $\Ad(H)$-invariant scalar products on $\m$.
Furthermore, for every $g \in \MG$ we have 
$$
g(\,\cdot \, ,\, \cdot \,)=Q(P_g \, \cdot \, ,\, \cdot\,)\,,
$$
where $P_g$ 
is an ${\rm Ad}(H)$-equivariant, $Q$-self-adjoint, positive definite
endomorphism of $\m$. Notice that $P_Q =\Id:=\Id_{\m}$.

 In Section \ref{sec:var} it will be convenient to endow $\MG$ with a complete metric;
we introduce the  $L^2$-metric $\sca$ on $\MG$, given by
\begn
  \langle V,W \rangle_{P_g} =
  \tr\, \big( P_g^{-1}\cdot V \cdot \,P_g^{-1} \cdot W \big)
	\label{ltosca}
\enn 
for $P_g\in \MG$ and for $Q$-self-adjoint tangent vectors
$V,W \in T_{P_g}\MG$: see \cite[3.1]{H.J}. Notice for any $P_g$, 
$T_{P_g}\MG$ is exactly
the space of $\Ad(H)$-equivariant, $Q$-self-adjoint endomorphisms of $\m$,
which we denote by $\SymmH$.

\begin{lemma}\label{lem:MGL2}
Let $G/H$ be a compact homogeneous space. Then
$(\MG,L^2)$ is a non-compact symmetric space.
\end{lemma}

\begin{proof}
The space $P(n)$ of symmetric, positive definite, real $(n\times n)$-matrices
endowed with the Riemannian metric 
$\langle V,W \rangle _P=\tr \big( P^{-1}\cdot V\cdot P^{-1} \cdot W  \big)$, for 
$P\in P(n)$ and $V,W \in T_P P(n)$,
is a symmetric space with nonpositive sectional curvatures: see e.g. \cite{E.J}: 
 The group ${\rm GL}(n)$ acts transitively and isometrically
on $P(n)$ by $(A,P) \mapsto A \cdot P \cdot A^t$,
$A \in {\rm GL}(n)$, $P \in P(n)$. By choosing a $Q$-orthonormal basis of $\m$,
we view $\MG$ as the set of symmetric, positive definite,
 ${\rm Ad}(H)$-equivariant $(n\times n)$-matrices, and we know it is
 the fixed point set of the isometric action of $H$ on $P(n)$ given by
\beg
  (h,P)\mapsto {\rm Ad}(h)\vert_{\m} \cdot P \cdot
               {\rm Ad}(h)^t\vert_{\m} \,.
\en
In particular, $\MG$ is a totally geodesic subspace of
$P(n)$. As such, $\MG$ itself is a symmetric space with
nonpositive sectional curvatures.
\end{proof}

Since $P_1(n):=\{P \in P(n)\mid \det P=1\}$ is a totally geodesic subspace of $(P(n),L^2)$,
from the above lemma we obtain

\begin{corollary}\label{cor:MGosymm}
Let $G/H$ be a compact homogeneous space with $\dim \MG \geq 2$. Let
$$
   \MGo :=\{ P_g \in \MG \mid \det P_g =1\}
$$
denote the space of $G$-invariant metrics on $G/H$ of volume one.
Then $(\MGo,L^2)$ is a non-compact symmetric space.
\end{corollary}

The tangent space of $\MGo$ at
the point $\Id$, corresponding to $Q$, is
$$
   T_{\Id} \MGo := \{  v \in \SymmH\mid \tr v=0\}\,.
$$
On this vector space the $L^2$-metric induces a norm given by 
$$
  \Vert v\Vert^2= \tr (v^2)\,.
$$

\begin{lemma}\label{lem:expdiff}
Let $G/H$ be a compact homogeneous space with $\dim \MG \geq 2$. Then 
for any 
$$
   v \in \Si := \{ v\in T_{\Id} \MGo \mid \Vert v \Vert =1\}\,,
$$
the curve $\gamma_v(t)=\exp(t\cdot v)$, $t \in \RR$, is  a unit speed geodesic in $(\MGo,L^2)$.
\end{lemma}

\begin{proof}
See \cite[p.~226]{H.S}. 
\end{proof}

\begin{remark}\label{rem:isotirr}
In case $\dim \, \MG=1$ the space
$\MGo$ is $\{Q\}$, the single point, and $Q$ is Einstein, by Schur's Lemma  \cite[7.44]{Bes}. This happens if and only if $G/H$
is isotropy irre\-du\-cible.
\end{remark}

\begin{remark}\label{rem:specialQwhy}
The special choice of the biinvariant metric $Q$ is convenient but not necessary.
The key point is that in this special case, the $Q$-orthogonal complement
of $\h$ in the Lie algebra $\n(\h)$ of the normalizer of $H$ in $G$ is 
the Lie algebra of a compact subgroup of $G$: see Lemma \ref{lemnor}. 
This is not true for a general $Q$. 
\end{remark}

\begin{remark}\label{rem:L2}
For a compact  manifold $M^n$, the $L^2$-metric on the space of Riemannian metrics 
$\mathcal{M}$ is defined
by
$$
  \langle h,k\rangle_g := \int_{M^n} (h,k)_g \,\mu_g\,.
$$
Here $k,h$ are symmetric bilinear forms on $TM^n$, $g$ is a Riemannian metric on $M^n$
and at any point $p \in M^n$ we
have $(h,k)_g=\sum_{i,j=1}^n h(e_i,e_j)\cdot k(e_i,e_j)$, where $(e_1,...,e_n)$ is 
an orthonormal basis of $(T_pM^n,g_p)$.

 If $g$ is a $G$-invariant metric of volume one on a compact
homogeneous space $M^n=G/H$ and $h,k$ are $G$-invariant as well, then $\langle h,k\rangle_g = (h,k)_g$,
where we now view $h,k$ as $\Ad(H)$-invariant, symmetric bilinear forms on $T_{eH}G/H = \m$.
Notice that all these homogeneous metrics have the same volume element $\bar \mu$. As a consequence,
elements of the tangent space to the space $\MGo$ of $G$-invariant metrics of volume one in $g$ are given by
$h$ with $\tr_g h=\sum_{i=1}^n h(e_i,e_i)=0$.
Now for some $H_g$ which is $g$-self-adjoint with
$\tr H_g=0$, we may write $h(X,Y)=g(H_g \cdot X,Y)$ for $X,Y \in \m$.  By writing $g(X,Y)=Q(P_g\cdot X,Y)$ we see that we can identify $h$ with $P_gH_g$.
Since $H_g$ is $g$-self-adjoint and traceless there exists some $v \in T_{\Id}\MGo$ with
$H_g = \sqrt{P_g^{-1}}\cdot v \cdot \sqrt{P_g}$, where $ \sqrt{P_g}\in \MGo$ denotes the square root of $P_g$.
This shows $P_gH_g=\sqrt{P_g}\cdot v \cdot \sqrt{P_g}\,,$ as it should.

The $L^2$-gradient of the total scalar curvature functional  restricted to the space of Riemannian
metrics with a fixed volume element $\bar \mu$ is $-\ric_0(g)=-(\ric(g)-\tfrac{1}{n}\cdot \sc(g)\cdot g)$
(see \cite[4.22]{Bes}).
Denoting by $\Ric(g)$ the Ricci endomorphism, defined by $\ric(g)(X,Y)=g(\Ric(g)\cdot X,Y)$, we see that
$\ric_0(g)$ corresponds to $\Ric_0(g)$, the traceless part of $\Ric(g)$. Moreover, the corresponding
tangent vector in $T_{P_g}\MGo$ is $P_g\Ric_0(g)$. As a consequence of \eqref{ltosca} we obtain
\begn\label{eqn:L2grad}
  \Vert (\nabla_{L^2}\sc)_g \Vert^2_{g}= \tr (\Ric_0(g))^2\,.
\enn
Since the scalar curvature of a homogeneous metric is constant, 
the $L^2$-gradient flow of $\sc:\MGo \to \RR$ 
and the volume-normalized Ricci flow on $\MGo$ agree (up to a factor of $2$).
\end{remark}

%%%%%%%%%%%%%%%%%%%%%%%%%%%%%%%%%%%%%%%%%%%%%%%%%%%%%%%%%%%%%%%%%%%%%%%%%%

\subsection{{\it The Graev-map}}\label{sec:grmap}

Let $G/H$ be a compact homogeneous space with $\dim \MGo \geq 1$. The 
unit sphere $\Si$ in $T_{\Id} \MGo$ parametrizes the set of
unit speed geodesics $\gamma_v(t)=\exp(t\cdot v)$ in $(\MGo,L^2)$ emanating from $\Id$: see
Lemma \ref{lem:expdiff}.
A first key observation of M. Graev is  that we can avoid 
many of the technicalities in
\cite{Bo}  concerning
semi-algebraic subsets of $\Si$ by introducing the following {\it Graev map}.

\begin{lemma}[Graev map]\label{lem:Graevmap}
Let 
\beg
    \Sphb &:=&
		 \{  A \in \SymmH \mid A \geq 0,\,\, \ker(A) \neq \{0\},\,\, \tr A=1 \}\,,
\en
where $\{A \geq 0\}$ denotes the set the positive semi-definite endomorphisms in $\SymmH$. 
Then
$$
 {\rm Gr}: \Si \to \Sphb\,\,;\,\,\,
 v \mapsto \tfrac{1}{n}\left(  \Id_n -\tfrac{1}{\lambda(v)}\cdot v\right)
$$
is a homeomorphism,
where $\lambda(v)$ denotes the smallest eigenvalue of $v$.
\end{lemma}

\begin{proof}
This  follows from the fact that the inverse function
$$
   {\rm Gr}^{-1}(A) =\tfrac{A-\frac{1}{n}\cdot\Id_n}{\sqrt{\Vert
A\Vert^2-\frac{1}{n}}}
$$
is a homeomorphism. 
\end{proof}

The map ${\rm Gr}^{-1}:\Sphb \to \Si$ has the following nice property:
it maps convex Euclidean sets in $\Sphb \subset \{\tr = 1\}$ first to convex Euclidean sets in $\{\tr = 0\}$ 
and then, by radial projection, onto spherical convex sets in $\Si$.

\begin{remark}\label{rem:Sphbmodel2}
By extending a positive semi-definite endormorphism $A$ on $\m$ to
a positive semi-definite endormorphism $\hat A$ on $\g$ with $\h \subset \ker \hat A$, we can view 
 $\Sphb$ as a subset of $\SymgH$, the set of $\Ad(H)$-equivariant, $Q$-self-adjoint endomorphisms of $\g$.
That is, we have
\beg
    \Sphb
		 &=&
			  \{ A \in \SymgH \mid A \geq 0,\,\, \ker(A) \supsetneq \h,\,\, \tr A=1 \}\,.
\en
This set is compact and semi-algebraic. See Section \ref{sec:semialgebraic} for a discussion of semi-algebraic sets.
\end{remark}

\begin{remark}\label{rem:ad-inv}
For any $A \in \SymmH$ 
 and its extension $\hat A:\g \to \g$ in $\SymgH$, we note that $\Ad(H)$-equi\-va\-riance of $A$ is equivalent to $\Ad(H)$-equivariance of $\hat A$. Furthermore, $\Ad(H)$-equivariance of $A$ is also equivalent to $[A, \ad(\h)|_{\m}]=0$ and analogously, on all of $\g$, we know $[\hat A,\ad(\h)]=0$: see Lemma \ref{lem:ad-invar}.
\end{remark}

%%%%%%%%%%%%%%%%%%%%%%%%%%%%%%%%%%%%%%%%%%%%%%%%%%%%%%%%%
\subsection{{\it The $\kf$-disks  $\D(\kf)$}}\label{sec:kfdisc}

In this section we associate to each intermediate subalgebra $\kf$ of $\g$ with
$\h < \kf <\g$ a non-empty convex set $\D(\kf)\subset \Sphb$. This  goes back to Graev \cite{Gr}.
Let us mention that the corresponding sets considered in \cite{Bo}
are subsets of $\D(\kf)$ but more difficult to deal with since they are only star-shaped.

\begin{definition}\label{defin-K}
We denote by $\Sub$ the {\em set of all intermediate subalgebras} $\kf$ such that $\h < \kf < \g$.
\end{definition}
For two subalgebras $\kf_1,\kf_2$ in $\Sub$ we write $\kf_1 \leq \kf_2$ if 
$\kf_1 \subseteq \kf_2$, and $\kf_1 <  \kf_2$ if $\kf_1 \subsetneq \kf_2$.

\begin{definition}[$\kf$-disk]\label{def:kfdisc} 
For any $\kf$ in $\Sub$, we define the {\em $\kf$-disk} in $\SymgH$,
$$
   \D(\kf):=\{A \in \SymgH \mid A\geq 0,\,\, \kf \subset \ker(A),\,\,\tr
A=1,\,\,[A,\ad(\kf)]=0\}
$$
and 
$$
\Ak:=\tfrac{1}{\dim \g-\dim \kf}\cdot (\Id_{\g} -\Id_{\kf})\,.
$$
We set
$$
   \D(\g):= \emptyset\,.
$$
\end{definition}

It is obvious that $\D(\kf)$ is a convex subset in a Euclidean space
and that $\Ak \in \D(\kf)$.
As a consequence $\D(\kf) = \overline{\D(\kf)}$ is a compact, convex set.
Topologically, $\D(\kf)$ is a compact disk.
Notice moreover, that if $\{\Ak\} \neq \D(\kf)$, then
$$
\Ak \in {\rm int}(\D(\kf)) = \{A \in \D(\kf) \mid \kf = \ker(A)\}.
$$
As convex sets of Euclidean spaces, $\kf$-disks have well defined dimensions and 
boundaries.

\begin{lemma}\label{lem:disksemialg}
For any  subalgebra $\kf$ in $\Sub$, the  disk $\D(\kf)$ is a
semi-algebraic set.
\end{lemma}

\begin{proof}
The equation $\tr A=1$ is algebraic in $A$.  Moreover, the set $\{A\geq 0\}$
of $Q$-self-adjoint, nonnegative operators in $\SymgH$
is a semi-algebraic set, as well. Finally, 
 the equations $A(X)=0$ and  $[A,\ad(X)]=0$  are algebraic for all  $X \in \kf$.
This shows the claim. %see Section \ref{sec:semialgebraic}.
\end{proof}

\begin{lemma}\label{lem:disk-inclusion}
Given any intermediate subalgebras  $\kf_1 < \kf_2$ in $\Sub$, we have $\D(\kf_2) \subsetneq \D(\kf_1)$.
\end{lemma} 

\begin{proof} Let $A\in \D(\kf_2) =\{A \in \SymgH \mid A\geq 0,\,\, \kf_2 \subset \ker(A),\,\,\tr
A=1,\,\,[A,\ad(\kf_2)]=0\}$. We see $\kf_1  < \kf_2 \subset \ker(A)$, and that
$[A,\ad(\kf_2)]=0$ implies $[A,\ad(\kf_1)]=0$. 
Of course we have $A \geq 0$ and $\tr(A)=1$. Thus $A\in \D(\kf_1)$.  

To see the containment is strict, we note that $A^{\blds\kf_1}$, for example, is an element of 
$\D(\kf_1)$. Since $\ker(A^{\blds\kf_1})=\kf_1  < \kf_2$, we see $A^{\blds\kf_1} \not\in \D(\kf_2)$.  
\end{proof}

\begin{definition}\label{def:generatedsubalg}
Let $\kf_1,\kf_2$ be two subalgebras in $\Sub$. We denote by 
$$
\kf^* := \langle \kf_1,\kf_2\rangle
$$
the smallest  subalgebra of $\g$ containing $\kf_1$ and $\kf_2$. 
\end{definition}

Notice that $\kf^*$ is well-defined and that $\kf^*=\g$ is possible.

\begin{lemma}\label{lem:disk-intersection}
Let $\kf_1,\kf_2$ be two subalgebras in $\Sub$, and $\kf^* = \langle
\kf_1,\kf_2\rangle$. Then
$$ 
     \D(\kf_1) \cap \D(\kf_2) = \D(\kf^*)\,.
$$ 
\end{lemma}

\begin{proof} 
If $\kf^* \subsetneq \g$, we know by Lemma \ref{lem:disk-inclusion} that 
$\D(\kf^*) \subseteq   \D(\kf_1) \cap \D(\kf_2)$.   
If $\kf^* = \langle \kf_1,\kf_2\rangle=\g$, this is trivial since by definition $\D(\g)=\emptyset$.

To show that $\D(\kf_1) \cap \D(\kf_2) \subseteq \D(\kf^*)$, let $A \in \D(\kf_1)  \cap   \D(\kf_2)$. 
By hypothesis,
$$
  [A,\ad(\kf_1)]= [A,\ad(\kf_2)]=0\,.
$$
Let $X_1 \in \kf_1$ and $X_2 \in \kf_2$.
By the Jacobi identity we have
\beg
    [A,\ad([X_1,X_2])]
       &=&
         \big[A,[\ad(X_1),\ad(X_2)]\big] \\
         &=&
         -\big[\ad(X_2),[A,\ad(X_1)]\big]
         -\big[\ad(X_1),[\ad(X_2),A]\big]\\
         &=& 0\,.
\en
Since a finite number of such iterated Lie brackets of elements
in $\kf_1$ and $\kf_2$ generate $\kf^*$,
it follows that $[A, \ad(\kf^*)]=0$.

By hypothesis, we know $\kf_1,\kf_2 \in \ker(A)$.
Then, for $X_1 \in \kf_1$, $X_2\in \kf_2$ as above, we have
$$
A([X_1,X_2])= A(\ad(X_1))(X_2)=\ad(X_1)(A(X_2))=0\,.
$$
An easy induction shows that
$\kf^* \subset \ker (A)$.
Thus, $\D(\kf_1) \cap \D(\kf_2) \subset \D(\kf^*)$, which proves our equality.
\end{proof}

\begin{definition}\label{def:toral-non-toral}{\em (Toral and non-toral subalgebras)}
For an intermediate subalgebra $\kf \in \Sub$, if  
$$\m_{\blds\kf}:=\m \cap \kf$$  
is an abelian subalgebra of $\g$ we say $\kf$ is {\em toral}. Otherwise, we say $\kf$ is {\em non-toral}. 
We let $\Sub_t$ denote the set of all toral subalgebras in $\Sub$ and let $\Sub_s$ denote the non-toral subalgebras in $\Sub$. Let $\Sub_{s,t}$ denote the subset of non-toral subalgebras $\kf$ for which there is some toral subalgebra 
$\tf \in \Sub_t$ with $\tf < \kf$. 
\end{definition}

If $\kf$ is toral then $\kf$ is an abelian extension of $\h$, but $\kf$ itself need not be abelian. 
If $\kf$ is a non-toral subalgebra, its semisimple part $\kf_s$ is strictly bigger than the one of $\h$.

Note that toral or non-toral subalgebras are not necessarily compact.
(A subalgebra $\kf$ of $\g$ is said to be compact if it is
the Lie algebra of a compact subgroup $K$ of $G$.)

A subalgebra $\kf$ of $\g$ of dimension $k$ can be considered as an element in
the Grassmannian $Gr_k(\RR^{\dim \bldsg \g })$ of $k$-planes in $\RR^{\dim \bldsg \g}$. 
Thus, we can speak about convergence
of sequences of subalgebras of fixed dimension.

\begin{lemma}\label{lem:Dkfi}
Let $(\kf_i)_{i\in \N}$ be a sequence of non-toral (resp.~toral) subalgebras in $\Sub$
of fixed dimension $k$  and suppose that 
$\lim_{i\to \infty}\kf_i=\kf_\infty \in Gr_k(\RR^{\dim \bldsg \g})$.  Then \\
$(i)$ $\kf_\infty$ is a non-toral (resp.~toral) subalgebra.\\
$(ii)$  ${\displaystyle \lim_{i\to\infty}\D(\kf_i) \subset \D(\kf_\infty)}$.
\end{lemma}

\begin{proof}
 By the continuity of the Lie bracket, it is clear that the limit space $\kf_\infty$ is a subalgebra
of $\g$ of dimension $k$.

For each $i\in \NN$, the semisimple part $\h_{s}$ of $\h$ is a proper subalgebra of $(\kf_i)_{s}$, the semisimple part of $\kf_i$,  
since each $\kf_i$ is a non-toral extension of $\h$. This follows, since up to conjugation, there exist only finitely many
semisimple subalgebras $\kf_s$ of $\g$ with $\h_s <\kf_s < \g$, by Proposition 4.1 and 4.2, \cite{BWZ}.  
As a result, a sequence $(\kf_i)_{i \in \N}$ of  non-toral Lie algebras $\kf_i$ with $\h < \kf_i < \g$ of fixed dimension cannot converge to a toral subalgebra, proving the first claim.

To see the second claim, let $(A_i)_{i \in \N}$ be a sequence in $\Sphb$ with 
$A_i \in \D(\kf_i)$ for all $i \in \NN$.  Suppose that $\lim_{i\to \infty}A_i=A_\infty \in \Sphb$. We show that $A_\infty \in \D(\kf_\infty)$.
It is clear that $A_\infty \geq 0$ and that $\tr A_\infty=1$. Suppose now that $\kf_\infty \not\subset \ker(A_\infty)$, so
 there exists some $X \in \kf_\infty$ such that $A_\infty(X) \neq 0$. It follows that we also
have $A_i(X) \neq 0$ for all $i \geq i_0$. Since $\kf_i$ converges to $\kf_\infty$ as $i\to \infty$,
there exist a sequence $(X_i)_{i\in \NN}$ with $X_i \in \kf_i$ 
and $\lim_{i \to \infty }X_i=X$.
Consequently $A_i(X_i)\neq 0$ for sufficiently large $i$. But this is a contradiction.

It remains to show that $[A_\infty,\ad(\kf_\infty)]=0$. If this is not true, then for some $X \in \kf_\infty$, 
we have $[A_\infty,\ad(X)]\neq 0$. As above, we can construct a sequence with $[A_i,\ad(X_i)]\neq 0$ for sufficiently large $i$,
again obtaining a contradiction. This shows $A_\infty \in \D(\kf_\infty)$.
\end{proof}

\begin{remark}\label{rem:disk-dim} We note while the dimension $\dim(\kf_i) =\dim(\kf_\infty)$ is constant, it is possible that in the corresponding sequence of disks, the disk $ \D(\kf_\infty)$ of the limit may have a larger dimension.
\end{remark}

\begin{lemma}\label{lem:sequsubalgsub}
Any sequence $(\kf_i)_{i \in \N}$ in $\Sub$ has a convergent subsequence.
\end{lemma}

\begin{proof}
We first note that there are only finitely many possible dimensions of intermediate subalgebras.
Thus, passing to a subsequence, we may assume that $k=\dim \kf_i$ for all $i\in \N$.
Since the Grassmannian $Gr_k(\RR^{\dim \bldsg\g})$ is compact, the claim follows.
\end{proof}

\subsection{{\it Flags of subalgebras}}\label{sec:flags}

In this section we define a novel   
partial ordering, due to Graev,  on the set of flags of intermediate subalgebras \cite{Gr}. With the help of this ordering the combinatorial 
properties of butterflies described below follow easily.

\medskip

Recall that we write $\kf_1 < \kf_2$ if $\kf_1 \subsetneq \kf_2$.

\begin{definition}[Flags of subalgebras]\label{def:flags}  
A {\em flag} is  an increasing sequence 
$$
  \varphi = (\kf_1 < \dots < \kf_r)
$$
of subalgebras $\kf_1,...,\kf_r$ where we always require $\h < \kf_i < \kf_{i+1} \leq \g$ for $i=1,...,r-1$. 
We say $\varphi$ has {\em length} given by $\ell(\varphi) = r \geq 1$. 
A subalgebra $\kf$ is a flag $(\kf)$ of length one. A flag $\varphi$ has a {\em maximum}   
$\max(\varphi) = \kf_r$, and {\em height}  $h(\varphi) = \dim(\max(\varphi))$.
We will say a flag $\varphi$ is {\em proper} if $h(\varphi) < \dim \g$, otherwise $\varphi$ is {\em improper}. 
 \end{definition}

 If every subalgebra in the flag $\varphi$ occurs in the flag $\psi$,  we write $\varphi \subset \psi$.

 \begin{definition}[Sets of flags]\label{def:Union Non-toral flags} 
We call $ \varphi=(\kf_1 < \cdots <\kf_r)$ a {\em non-toral flag} if $\kf_1$ is non-toral.
We call $\varphi=(\tf_1 < \cdots <\tf_r)$ a {\em toral flag} if $\tf_r$ is toral.
We write $\Fl$ for the {\em set of all flags}, by $\Fln$ the set of all non-toral flags, 
and by $\Flt$ the set of all toral flags. 
We write $\Flts$ for the set of all flags of type
$\varphi=(\tf_1 < \cdots < \tf_r< \kf)$ with $\tf_i$ toral for each $i=1,\dots,r$, $1\leq r$,
and $\kf \leq \g$ non-toral. 
\end{definition}

Notice that if $\kf_1$ is non-toral, so is $\kf_i$ for all $1 \leq i \leq r$, and that if $\tf_r$ is
toral so is $\tf_i$ for all $1 \leq i \leq r$.  Furthermore, while a subalgebra is either toral or non-toral, a flag can be neither toral nor non-toral.

\begin{definition}[Partial ordering of flags]\label{def:flag-order} 
Let $\varphi, \psi \in \Fl$. Then we write $\varphi \leq \psi$ if either $\varphi = \psi$ 
or there exists a sequence of flags $\varphi = \varphi_1, \dots, \varphi_k=\psi$ for some $k \geq 2$ such that  the lengths of adjacent flags differ by $\pm 1$, and for every $1\leq j < k$, we have either
\begin{enumerate}
\item[ {\rm (i)}] $\varphi_{j} \subsetneq \varphi_{j+1}$ and $\max (\varphi_j) < \max (\varphi_{j+1})$ or                            
\item[ {\rm (ii)}] $\varphi_{j+1} \subsetneq \varphi_{j}$ and $\max (\varphi_{j+1}) = \max (\varphi_j)$.
\end{enumerate}
If we have two flags with $\varphi \leq \psi$ and with $\varphi \neq \psi$, we write $\varphi < \psi$.
\end{definition}

Notice that we obtain $\varphi_{j+1}$ from $\varphi_j$ by adding a new maximum subalgebra in case (i), 
or by removing a non-maximal subalgebra from $\varphi_j$ in case (ii).

\begin{example} Let $\kf_1,...,\kf_6$ be subalgebras with $\kf_1< \cdots <\kf_6$.
Let $\varphi = (\kf_1 < \kf_2 < \kf_3 < \kf_4)$ and $\psi = (\kf_2 < \kf_5 < \kf_6)$.
Then  $\varphi \leq \psi$ since
\begin{align*}
\varphi = (\kf_1 < \kf_2 < \kf_3 < \kf_4) = \varphi_1 \supset \varphi_2=(\kf_2 < \kf_3 < \kf_4) \supset \varphi_3=(\kf_2 < \kf_4) \\
\subset \varphi_4=(\kf_2 < \kf_4 < \kf_5) 
\supset \varphi_5=(\kf_2  < \kf_5) \subset \varphi_6=(\kf_2 < \kf_5 < \kf_6) = \psi\,.
\end{align*}
\end{example}

\begin{lemma}\label{lem:flag-transitivity} Let $\varphi,\psi,\zeta \in \Fl$ with  $\varphi \leq \psi$ and  $\psi \leq \zeta$. 
Then  $\varphi \leq \zeta$.
\end{lemma}

\begin{proof} Given $\varphi \leq \psi$, there exists a sequence of flags, $\varphi = \varphi_1, \dots, \varphi_k=\psi$ satisfying the criteria in Definition \ref{def:flag-order}. And since $\psi\leq \zeta$, there exists a sequence of flags, $\psi = \psi_1, \dots, \psi_l=\zeta$ as in Definition \ref{def:flag-order}. We construct sequence of flags $\varphi_1, \dots, \varphi_k=\psi_1, \varphi_{k+1}=\psi_2, \dots, \varphi_{k+l-1}=\psi_l$. This is a sequence that starts at $\varphi$ and ends at $\zeta$, satisfying criteria $(i)$ and $(ii)$ at each increment in Definition \ref{def:flag-order}. With this sequence, we conclude $\varphi \leq \zeta$.
\end{proof}

\begin{lemma}\label{lem:max}
Let $\varphi, \psi \in \Fl$ with  $\varphi \leq \psi$. Then $\max(\varphi)\leq \max(\psi)$.
Moreover, if there
exists a sequence $\varphi = \varphi_1, \dots, \varphi_k=\psi$ 
as in Definition \ref{def:flag-order} with $k \geq 2$, then $\varphi \neq \psi$. 
\end{lemma}

\begin{proof} If $\varphi=\psi$ then $\max(\varphi)=\max(\psi)$. Thus we may assume $\varphi \neq \psi$
and that there exists a sequence $\varphi = \varphi_1, \dots, \varphi_k=\psi$, $k \geq 2$. 
By Definition  \ref{def:flag-order} 
we have $\max (\varphi_j) \leq \max (\varphi_{j+1})$ for all $1 \leq j \leq k-1$. This shows the first claim.

To show the second claim
suppose that for some $j<k$, 
$\max (\varphi_j) < \max (\varphi_{j+1})$. Then $\max(\varphi) < \max(\psi)$ thus $\varphi \neq \psi$.
We are left with the case $\max (\varphi_j) = \max (\varphi_{j+1})$ for all $j=1,...,k-1$.
But then we have that the length $\ell(\varphi_{j+1}) = \ell(\varphi_j) -1$ for all $j=1,...,k-1$.  
Hence $\ell(\varphi) > \ell(\psi)$, which tells us $\varphi \neq \psi$. 
\end{proof}

Here we have a non-inductive description of this partial ordering:

\begin{lemma} \label{lem:ordering}
Let $\varphi,\psi \in \Fl$. Then
we have
$$\varphi \leq \psi \Leftrightarrow \begin{cases} &\max (\varphi) \leq \max(\psi) \textrm{ and } \\
&\{\kf \not\in\varphi ~\ande ~ \kf \in\psi\} \Rightarrow \max(\varphi) < \kf.\end{cases}$$
\end{lemma}

\begin{proof} 
If $\varphi = \psi$, the statement is trivially true. If $\varphi < \psi$,   then by Lemma \ref{lem:max}, 
$\max (\varphi) \leq \max(\psi)$.
Suppose there exists some $\kf \in \psi$, $\kf \not\in \varphi$ and $\max(\varphi) \not< \kf$. 
By hypothesis we have a sequence $\varphi = \varphi_1, \dots, \varphi_k = \psi$.
Since $\kf \not\in\varphi= \varphi_1$ and $\kf \in \psi =\varphi_k$, for some $1 \leq j_0 \leq k-1$, $\kf \not\in\varphi_{j_0}$ and 
$\kf \in\varphi_{j_0 +1}$. By Definition \ref{def:flag-order}, we can only add a subalgebra to a flag if we add a new maximum, that is, only if $\max(\varphi_{j_0}) < \kf$. But then by Lemma \ref{lem:max} we would have
$\max(\varphi) \leq \max(\varphi_{j_0}) < \kf$, yielding a contradiction.

Conversely, suppose we have flags $\varphi, \psi \in \Fl$ with $\max(\varphi) \leq \max(\psi)$ and furthermore, for all $\kf \in \psi$ with $\kf \not\in \varphi$ we have 
$\max(\varphi) < \kf$. 
We show $\varphi \leq \psi$. 

If $\psi \subsetneq \varphi$ and  $\max(\varphi) =\max(\psi)$, then $\varphi \leq \psi$. 
If $\psi \subsetneq \varphi$ is not true then
for some $\kf_{i_0} \in \psi=(\kf_1 < \cdots < \kf_r)$ we have $\kf_{i_0}\not\in \varphi$.
By hypothesis, $\max(\varphi) < \kf_{i_0}$. Let $i_0 \in \{1,..., r\}$  be the smallest
index for which $\kf_{i_0} \in \psi$ but $\kf_{i_0}\not\in \varphi$. 
We build our sequence $\varphi = \varphi_1, \dots, \varphi_k = \psi$ as follows:
$\varphi_1 = \varphi$,  
$\varphi_2 = (\varphi_{1} < \kf_{i_0})$, 
$\varphi_3 = (\varphi_{2} < \kf_{i_0 +1})$, $\dots$, 
$\varphi_{k'}= (\varphi < \kf_{i_0}< \cdots < \kf_r)$. Notice that $(\kf_1< \dots < \kf_{i_0 - 1}) \subset \varphi$ by definition of $i_0$. In particular, $\psi$ is a subset of $\varphi_{k'}$. 
There may be finitely many subalgebras in $\varphi$ which are not in $\psi$. We  continue our sequence, 
removing these one at a time. 
Since $\max(\varphi)< \kf_{i_0}$,
$\varphi_{k'+1}$ is obtained by removing the smallest subalgebra in $\varphi_{k'}$, not in $\psi$,  $\varphi_{k'+2}$ is obtained by removing the smallest subalgebra in $\varphi_{k'+1}$, not in $\psi$, etc.  Thus $\varphi < \psi$.
\end{proof}

\begin{corollary}\label{cor:ordering}
Let $\varphi,\tilde \varphi \in \Fl$ with $\varphi \leq \ti \varphi$ and heights
$h(\varphi)=h(\tilde \varphi)$. Then $\tilde \varphi \subset  \varphi$.
\end{corollary}

\begin{proof} Suppose that there is some subalgebra $\tilde \kf \in \ti \varphi$, $\tilde \kf \not\in\varphi$. 
Since $\varphi \leq \ti \varphi$ by hypothesis, we have $\max(\varphi) < \tilde \kf$ by Lemma \ref{lem:ordering}. 
This means $\max(\varphi) < \tilde \kf \leq \max(\ti \varphi)$ and consequently 
$h(\varphi)<h(\tilde \varphi)$, a contradiction. Thus $\tilde \varphi \subset  \varphi$.
\end{proof}

In order to describe the combinatorial properties of butterflies (see section \ref{sec:butterflies}) the following product of flags, introduced in \cite{Gr},
is very helpful.

\begin{definition}[Product of flags] \label{def:productflags} 
The {\em product} of two flags $ \varphi , \tilde{\varphi} \in \Fl$ is defined as follows:
\begin{enumerate}
\item[{\rm (1)}] 
Suppose there exists $\kf \in \varphi=(\kf_1< \cdots < \kf_r)$ with 
$\kf  > \max(\tilde{\varphi})$ and let $\kf_{i_0}$ be the smallest such subalgebra.
Then we set
$$ 
 \varphi \, \tilde{\varphi} := \big( 
  (\varphi \cap \tilde{\varphi})  < \kf_{i_0} < \kf_{i_0 +1} < \cdots < \kf_r  \big)\subset \varphi\,.
$$
\item[{\rm (2)}] 
Suppose there exists $\tilde\kf \in \tilde\varphi=(\ti \kf_1< \cdots < \ti \kf_{\ti r})$ 
with $\tilde\kf    > \max({\varphi})$ and let $\tilde\kf_{\tilde i_0}$
be the smallest such subalgebra. Then we set
$$ 
 \varphi \, \tilde{\varphi} := \big( (\varphi \cap \tilde{\varphi})  
< \tilde\kf_{\tilde i_0} <\ kf_{i_0 +1} < \cdots < \tilde\kf_{\tilde{r}}  \big) \subset \tilde\varphi\,. 
$$
\item[{\rm (3)}] If there exists no $\kf \in \varphi$ with $\kf  > \max(\tilde{\varphi})$ 
and no  $\tilde\kf \in \tilde\varphi$
with $\tilde\kf  > \max({\varphi})$, then for $\kf^* = \langle \max(\varphi) , \max(\tilde{\varphi})\rangle$
we set
 $$
 \varphi \, \tilde{\varphi} :=  
((\varphi \cap \tilde{\varphi})  
 \leq \kf^*)\,.
$$
\end{enumerate}
\end{definition}

\begin{lemma}\label{lem:flag-product}  
Let $\varphi \,, \tilde{\varphi} \in \Fl$. Then their product $\varphi \, \tilde{\varphi} \in \Fl$, 
with $\varphi \leq \varphi \, \tilde{\varphi}$ and $\tilde{\varphi} \leq  \varphi \, \tilde{\varphi}$. 

\end{lemma}
\begin{proof}  Cases (1) and (2) are clearly flags, since a subflag of a flag is a flag. In case (3), since 
$\max(\varphi \cap \tilde{\varphi}) \leq \langle \max(\varphi) , \max(\tilde{\varphi})\rangle$, we (possibly) extend the subflag $(\varphi \cap \tilde{\varphi})$ by adding a new maximum subalgebra of $\g$, and this is a flag.  

In case (1) in Definition
\ref{def:productflags}  we have $\max(\varphi)=\max(\varphi\ti \varphi)$ and $\varphi\ti \varphi \subset \varphi$.
Thus $\varphi \leq \varphi\ti \varphi$. We also have $\max(\ti \varphi)< \kf_{i_0}\leq \max(\varphi\ti \varphi)$
and if $\kf \in \varphi\ti\varphi$ with $\kf \not\in \ti \varphi$ then $\kf > \max(\ti\varphi)$. Thus by
Lemma \ref{lem:ordering} we have $\ti \varphi \leq \varphi \ti \varphi$. The case (2) follows precisely the same way.

 When $\varphi\ti\varphi$ is defined by case (3), we have $\max(\varphi),\max(\ti \varphi) \leq \max(\varphi\ti \varphi)$. If there exists a $\kf \in \varphi\ti\varphi$
with $\kf \not \in \varphi$, we know $\max(\varphi) < \kf =\langle \max(\varphi),\max(\ti \varphi)\rangle$.  
By Lemma  \ref{lem:ordering} we deduce $\varphi \leq \varphi \ti \varphi$.
If $\max(\varphi)=\langle \max(\varphi),\max(\ti \varphi)\rangle=\kf$ then
$\max(\varphi\ti \varphi)=\max(\varphi)$ and $\varphi\ti \varphi \subset \varphi$.
Thus $\varphi \leq \varphi\ti \varphi$. Analogously one can show  $\ti \varphi \leq \varphi\ti \varphi$ in this case.
\end{proof}

\begin{example} It may help to see the ordering, and the product, of explicit examples of flags. Let $G=SU(4)$ and $H = SU(2)$ ($SU(2) \subset SU(3) \subset SU(4)$). Consider the flags $\varphi = (\un(2)< \su(3))$ and $\tilde{\varphi} =(\su(3)<\un(3))$. Then $\varphi < \tilde{\varphi} $: $\varphi_1 = (\un(2)< \su(3)) =\varphi$, $\varphi_2 = (\su(3))$, 
$\varphi_3 = (\su(3)<\un(3))=\tilde{\varphi}$.
The flags $\psi = (\un(2)< \sy(2))$ and $\tilde{\psi} = (\su(2)\oplus \su(2))$  cannot be ordered.
We get the product flags 
$\varphi\tilde{\varphi} = (\su(3)<\un(3))$ and $\psi\tilde{\psi} = (\sy(2))$.
\end{example}

\begin{corollary}\label{cor:phipsi}
 Let
$\varphi=(\tf_1 < \cdots < \tf_r< \kf)$ and 
$\tilde \varphi= (\tilde \tf_1 < \cdots < \tilde \tf_{\tilde r}< \tilde \kf)$ be two flags in $\Flts$, so that each  $\tf_i$ and  $\tilde \tf_j$ is a toral subalgebra, and $\kf,\tilde \kf$ are non-toral. 
Let $\kf^*:=\langle \kf,\tilde \kf \rangle \leq\g$. Then
either $\varphi \tilde \varphi = (\kf^*)$ or
$\varphi \tilde \varphi =(\varphi_t< \kf^*)$ with $\varphi_t:=(\tf_{i_1} < \cdots < \tf_{i_k})$, $1 \leq k \leq r$.
\end{corollary}

\begin{proof}
We set $\kf^*:=\langle \max(\varphi),\max(\tilde \varphi)\rangle =\langle \kf,\tilde \kf\rangle \leq \g$. 
Since non-toral subalgebras cannot be contained in toral subalgebras,
we deduce from Definition \ref{def:productflags} that
$\varphi \tilde \varphi = ((\varphi \cap \ti \varphi) \leq  \kf^\ast)$. This shows the claim.
\end{proof}

\begin{definition}\label{def:flag-conv}
A sequence of flags $(\varphi_i)_{i \in \NN} \subset \Fl$
  converges to a flag $\varphi_\infty \in\Fl$,  
 $$\lim_{i\to\infty}\varphi_i = \varphi_\infty\,,$$  
 if for all  $i$, the sequence of flag lengths is constant, $\ell(\varphi_i) = \ell(\varphi_\infty)$, and   for 
$$
 {\displaystyle \varphi_i = (\kf^i_1 < \kf^i_2 < \dots < \kf^i_{\ell} )}\,
$$
  for each $j=1,\dots,\ell$, the sequence  of subalgebras converges, 
${\displaystyle \lim_{i \to\infty} \kf^i_j = \kf^{\infty}_j}$: see Lemma \ref{lem:Dkfi}. 
\end{definition}

\begin{lemma}\label{lem:flag-conv} 
Any sequence of flags $(\varphi_i)_{i \in \NN}$  in $\Fl$ subconverges to a limit flag:
${\displaystyle \lim_{j \to \infty} \varphi_{i_j} = \varphi_\infty }$.
Furthermore, if the flags in the sequence are toral, (resp.~non-toral),  
the limit flag is also toral (resp.~non-toral). 
\end{lemma}
\begin{proof} Let  $(\varphi_i)$  in $\Fl$ be a sequence of flags.  
Since the set of lengths of flags  in $\Fl$ is finite, our 
sequence of flags has a subsequence of flags with constant length $\ell$. 
Passing to this subsequence, we have $m=1,\dots,\ell$ sequences $(\kf^i_m)_{i=1}^\infty$ of subalgebras. 
By Lemma \ref{lem:Dkfi} and Lemma \ref{lem:sequsubalgsub} the claim follows.
\end{proof}
%%%%%%%%%%%%%%%%%%%%%%%%%%%%%%%%%%%%%%%%%%%%%%%%%%%%
\subsection{{\it Butterflies}}\label{sec:butterflies}

With the help of $\kf$-disks and flags of intermediated subalgebras Graev was able to 
describe certain important convex subsets of $\Sphb$ which he called butterflies. They inherit 
a very nice intersection property from the $\kf$-disks which can be most easily described
by the product of flags introduced in Section \ref{sec:flags}.

Given sets $X$ and $Y$ in $\Sphb$, we write ${\rm conv}\{X,Y\}$ for the convex hull of $X$ and $Y$, and we write the join of $X$ and $Y$ as   
%%%% added definition 2023
$X*Y = \{tx +(1-t)y \mid x \in X, y\in Y, t \in [0,1]\}$. (When $X$ and $Y$ are convex, $X*Y = {\rm conv}\{X,Y\}$.)

\medskip

\begin{definition}[Butterfly]\label{def:butterfly} 
The {\em butterfly}, $\Br[\varphi]$, associated to a flag $\varphi =(\kf_1< \cdots< \kf_r) \in \Fl$ of length
$\ell(\varphi)=r\geq 1$ is defined as follows: 

\noindent {\em Type 1.} When $\ell(\varphi)=1$,  so that $\varphi =(\kf)$, a subalgebra with 
$\kf \in \Sub$ or $\kf=\g$,  
we set $\Br[(\kf)]:=\D(\kf)$, the $\kf$-disk. 
Recall, for $\kf=\g$, this gives  $\Br[(\g)] = \D(\g) = \emp$.

\noindent {\em Type 2.} When $\ell(\varphi) > 1$ and $\max(\varphi)=\g$, then $\Br[\varphi]$ is the $(r-2)$-dimensional simplex 
$$
  \Br[\varphi] :=  {\rm conv}\{A^{\blds k_1}, \dots, A^{\blds k_{r-1}}
  \}\,.
$$ 

\noindent {\em Type 3.} When $\ell(\varphi) > 1$ and $\max(\varphi)<\g$, then $ \Br[\varphi]$ is 
$$
 \Br[\varphi] := {\rm conv}\{A^{\blds k_1}, \dots, A^{\blds k_{r-1}}\} \ast \Br[(\kf_r)] 
= {\rm conv}\{A^{\blds k_1}, \dots, A^{\blds k_{r-1}}\} \ast \D(\kf_r)\,.
$$
\end{definition}
We want to highlight that all butterflies $\Br[\varphi]$ are convex sets, since
$\D(\kf_r)$ is convex and the cone over a convex set is convex.

\begin{lemma}\label{lem:butterfly-kernels} 
Let $\varphi = (\kf_1< \dots < \kf_r) \in \Fl$ be a flag. 
For any operator  $A \in \Br[\varphi]$, 
either $A\in\D(\kf_r)$ 
or $\ker(A) \in \{\kf_1, \dots, \kf_{r-1}\}$.
\end{lemma}
\begin{proof} We have $A \in \Br[\varphi]=  {\rm conv}\{A^{\blds k_1}, \dots, A^{\blds k_{r-1}}\}*\D(\kf_r)$ that is
 \beg
 A = (1-\kappa) \cdot \sum_{i=1}^{r-1} \lambda_i A^{\blds k_i} + \kappa \cdot A'                           
\en
for some $0 \leq \kappa \leq 1$ and $0 \leq \lambda_1,...,\lambda_{r-1} \leq 1\,,\,\, \sum_{i=1}^{r-1}\lambda_i=1$
and $A' \in \D(\kf_r)$.
Recall that $\ker(A^{\blds k_1})=\kf_1 \subsetneq \cdots \subsetneq \ker (A^{\blds k_{r-1}})=\kf_{r-1}
 \subsetneq \ker (A')$ for all $A' \in \D(\kf_r)$
and that $\kf_r \subset \ker(A')$. Recall also that $A^{\blds k_1},\dots, A^{\blds k_{r-1}}$ and $A'$ are non-negative endomorphisms.
Now if $\kf_r \subsetneq \ker(A)$ we must have $\kappa=1$, thus $A\in\D(\kf_r)$. 
We next suppose that $\kappa <1$.
If $\lambda_1 \neq 0$, then $\ker(A) =\kf_1$, since $A^{\blds k_1}$ is positive on $\g \ominus \kf_1$. 
 Let $i$ be the smallest index such that $\lambda_i \neq 0$. Then  $\ker(A) =\kf_i < \kf_{i+1}$ since $A^{\blds k_i}$  is positive on $\g \ominus \kf_{i}$. 
\end{proof}

\begin{lemma}\label{lem:butterfly-join} The butterfly $\Br[\varphi]$ of a flag $\varphi = (\kf_1<\dots <\kf_r)$ with length $\ell(\varphi)\geq 2$  is a {\em join}. For each $i =2,\dots,r$ we have 

$$\Br[\varphi] =  {\rm conv}\{A^{\blds k_1}, \dots, A^{\blds k_{i-1}}\}*\Br[(\kf_i <\dots < \kf_r)]. $$
\end{lemma}
\begin{proof} 
We show ${\rm conv}\{A^{\blds k_1}, \dots, A^{\blds k_{i-1}}\}*\Br[(\kf_i <\dots < \kf_r)] \subset \Br[\varphi]$. 
The reverse  inclusion is obtained analogously. 
Any operator $A \in  {\rm conv}\{A^{\blds k_1}, \dots, A^{\blds k_{i-1}}\}*\Br[(\kf_i <\dots < \kf_r)]$ can be written as a sum,
$$ 
 A = (1-\mu)\cdot \sum_{j=1}^{i-1}\lambda_j A^{\blds k_j} 
+ \mu \cdot \Big((1-\kappa) \cdot \sum_{j=i}^{r-1} \lambda_j A^{\blds k_j} + \kappa \cdot A'  \Big)\,,
$$
where $A' \in \D(\kf_r)$, $0 \leq \kappa, \mu \leq 1$, and $0 \leq \lambda_j \leq 1$ with 
$\sum_{j=1}^{i-1} \lambda_j  = \sum_{j=i}^{r-1} \lambda_j =1.$ 
If $\mu=\kappa=1$, $A=A' \in \D(\kf_r) \subset \Br[\varphi]$. Otherwise, 
we set 
$$
\ti\lambda_j = \begin{cases} 
\frac{1-\mu}{(1-\kappa\mu)}\cdot \lambda_j, & 1\leq j \leq i-1 \\ 
\frac{\mu (1-\kappa)}{(1-\kappa\mu)} \cdot \lambda_j, & i \leq j \leq r-1.\end{cases}
$$
Notice that $\sum_{j=1}^{r-1}\tilde \lambda_j=1$ and $\tilde \kappa :=\kappa\mu \in [0,1]$.
Then 
$$
{\displaystyle A= (1-\ti \kappa)\cdot \sum_{j=1}^{r-1} \ti\lambda_j A^{\blds k_j} + \tilde \kappa \cdot A'}\,,
$$
an element of $\Br[\varphi]$. 
\end{proof}

\begin{corollary}\label{cor:ker}
Let $\varphi=(\kf_1 < \psi)\in \Fl$ be a flag with $\psi= (\kf_2 < \cdots < \kf_r)$. Let $A \in \Br[\varphi]$.
Then $\ker (A)=\kf_1$ if and only if $A \not\in \Br[\psi]$.
\end{corollary}

\begin{proof}
By Lemma \ref{lem:butterfly-join}, we know $A=(1-\kappa)\cdot A^{\blds k_1}+ \kappa \cdot A'$
with $A' \in \Br[\psi]$. Since $\kf_1 \subsetneq \ker (A')$, the claim follows.
\end{proof}

\begin{definition}\label{def:simplex} For every proper flag $\varphi =(\kf_1<\cdots <\kf_r)$, 
 we define the {\em flag simplex}
$$
   \Delta_{\varphi}:=\Delta_{(\varphi<\bldss\g)}:= \Br[(\varphi<\g)]={\rm conv}\{A^{\blds{\kf}_1},...,A^{\blds{\kf}_r}\}\,.
$$
\end{definition}
Notice that we have $\Delta_{\varphi} \subset \Br[\varphi]$.

\begin{lemma} \label{lem:subset}
Let $\varphi, \ti\varphi \in \Fl$ be flags with $\varphi < \ti\varphi$. 
Then $\Br[\ti\varphi] \subset \Br[\varphi]$. If $\ti \varphi$ is a proper flag, then $\Br[\ti\varphi] \subsetneq \Br[\varphi]$.
\end{lemma}

\begin{proof} Let $\varphi < \ti\varphi$. 
By definition, we have a sequence of flags 
$\varphi = \varphi_1, \varphi_2,\dots,\varphi_k=\ti\varphi$ where, at each step, either ($i$) $\varphi_j \subsetneq \varphi_{j+1}$ 
and $\max(\varphi_j) < \max(\varphi_{j+1})$, or ($ii$) $\varphi_{j+1} \subsetneq \varphi_j$ and $\max(\varphi_j) = \max(\varphi_{j+1})$. 
We show $\Br[\varphi_{j+1}] \subset \Br[\varphi_j]$ for each $j=1,\dots,k-1$.

Case ($i$): Suppose $\varphi_j \subsetneq \varphi_{j+1}$ and $\kf_{r-1}=\max(\varphi_j) < \max(\varphi_{j+1}) =\kf_r \leq  \g$. 
We know that  
$$
 {\rm conv}\{A^{\blds k_1},\dots, A^{\blds k_{r-1}}\} \subset {\rm conv}\{A^{\blds k_1}, \dots, A^{\blds k_{r-2}}\}*\D(\kf_{r-1})\,.
$$ 
Here we have strict inequality unless $\D(\kf_{r-1})=\{A^{\blds k_{r-1}}\}$.
Furthermore, $\D(\kf_r) \subsetneq \D(\kf_{r-1})$ by Lemma \ref{lem:disk-inclusion}. As a result, $\D(\kf_{r-1})*\D(\kf_r)= \D(\kf_{r-1})$,
using that $\D(\kf_{r-1})$ is convex.  If $\kf_r=\g$ then
$\D(\kf_r)=\emptyset$ and we still set $\D \ast \emptyset = \D$ for any set $\D$.
This yields
\begin{align*}
\Br[\varphi_{j+1}] &=  
\Big({\rm conv}\{A^{\blds k_1},\dots, A^{\blds k_{r-1}}\}\Big) \ast \D(\kf_r) \subset 
\Big({\rm conv}\{A^{\blds k_1}, \dots, A^{\blds k_{r-2}}\}*\D(\kf_{r-1})\Big) \ast \D(\kf_r) \\ 
&\subset  
 {\rm conv}\{A^{\blds k_1}, \dots, A^{\blds k_{r-2}}\}*\D(\kf_{r-1}) = \Br[\varphi_j].
\end{align*}
The inequality is strict unless $\D(\kf_{r-1})=\{A^{\blds k_{r-1}}\}$ and $\kf_r=\g$.

Case ($ii$): Suppose $\varphi_{j+1} \subsetneq \varphi_j$ and $\max(\varphi_j) = \max(\varphi_{j+1})$. If $\ell(\varphi_{j+1}) =1$, so that $\varphi_{j+1} = (\kf)$, then $\varphi_j = (\kf_1< \kf)$, $\Br[\varphi_{j+1}] = \D(\kf)$, 
while $\Br[\varphi_j] = \{A^{\blds k_1}\}*\D(\kf)$. Clearly $\Br[\varphi_{j+1}] = \D(\kf) \subsetneq \Br[\varphi_j]$. 
Recall that we allow $\kf=\g$.

Now suppose  $\ell(\varphi_{j+1})=r >1$. That is, $\varphi_{j+1} = (\kf_1<\kf_2< \cdots <\kf_r=\kf)$. 
Then $$
 \varphi_j=(\kf_1< \cdots < \kf_{i_0}< \ti\kf < \kf_{i_0+1} < \dots <\kf_r)
$$  
and we deduce
$$
\Br[\varphi_{j+1}]= {\rm conv}\{A^{\blds k_1}, \dots, A^{\blds k_r}\}*\D(\kf) \subsetneq 
\Br[\varphi_j] =   {\rm conv}\{A^{\blds k_1}, \dots, A^{\blds k_r},A^{\bldstik k}\}*\D(\kf)\,,
$$
since the convex hull of more points is a larger set. This shows  $\Br[\ti\varphi] \subset \Br[\varphi]$.
\end{proof}

\begin{lemma}\label{lem:butphipsi}
 $\Br[\varphi] \cap \Br[\ti\varphi] = \Br[\varphi \ti\varphi]$.
\end{lemma}

\begin{proof}  
By Lemma \ref{lem:subset}, since $\varphi, \ti\varphi \leq \varphi \ti\varphi$, we know
$\Br[\varphi \ti\varphi] \subset \Br[\varphi] \cap \Br[\ti\varphi]$.

To prove the lemma, we show that $\Br[\varphi] \cap \Br[\ti\varphi] \subset \Br[\varphi \ti\varphi]$, by induction on the 
length $\ell(\ti\varphi)=\ti r$.
 Suppose the length $\ell(\ti \varphi)=1$, so that $\ti \varphi = (\ti \kf)$. We prove this claim by
another induction on the length  $\ell(\varphi)=r$. Suppose that $\ell(\varphi)=1$, that is $\varphi=(\kf)$.
In this case $\varphi \ti \varphi =(\kf^\ast)$ with $\kf^\ast=\langle \kf,\ti \kf\rangle$. Since by Lemma \ref{lem:disk-intersection}
we have
$$
  \Br[\varphi]\cap \Br[\ti \varphi]=\D(\kf)\cap \D(\ti \kf)= \D(\kf^\ast)=\Br[\varphi \ti \varphi]\,,
$$
the above claim follows. 

Suppose $\ell(\varphi)=r \geq 2$, $\varphi=(\kf_1< \psi)$ with  $\psi:= (\kf_2 < \cdots < \kf_r)$  
and let $A \in \Br[\varphi]\cap \Br[\ti \varphi]$. By Corollary \ref{cor:ker}
we have $\ker (A)=\kf_1$ or $A \in \Br[\psi]$. In the latter case we obtain
by induction hypothesis $A \in \Br[\psi \ti \varphi]$
since $\ell(\psi)< \ell(\varphi)$. Since $\max(\psi\ti \varphi)=\max(\varphi\ti \varphi)$
and $\psi\ti \varphi \subset \varphi\ti \varphi$ we have $\Br[\psi \ti \varphi] \subset \Br[\varphi \ti \varphi]$.
In the first case we deduce from $\ti \kf \subset \ker(A)$ that $\ti \kf \leq \kf_1$. Thus $\varphi \ti \varphi=\varphi$
and the above claim follows for $\ell(\ti \varphi)=1$.

Suppose now $\ell(\ti\varphi)  =\ti r \geq 2$ and let
$\ti \varphi = (\ti\kf_1 < \ti\psi)$ where $\ti\psi = (\ti\kf_2 < \dots < \ti\kf_{\ti r})$.
Let $A \in \Br[\varphi] \cap \Br[\ti\varphi]$.  If $A \in \Br[\ti\psi]$, then since $\ell(\ti\psi)=\ti r -1$, by our induction hypothesis, we are done: see above. Thus, we are left with the case
$A = (1-\ti\kappa)A^{\ti{\blds k_1}} + \ti\kappa \ti A'$, with 
$\ti A' \in \Br[\ti\psi]$, and $0 \leq \ti\kappa < 1$. It follows $\ker(A)=\ti\kf_1$. 
Since $A \in \Br[\varphi]$ as well, and by symmetry we may assume  $\ell(\varphi) \geq 2$, we know by Lemma \ref{lem:butterfly-kernels} 
that either $A \in \D(\kf_r)$, 
in which case we get $ \kf_r \leq \ti\kf_1$, 
or $\ker(A) = \ti\kf_1 \in \{\kf_1,\dots,\kf_{r-1}\}$. 

If $\kf_r \leq \ti\kf_1$, then $\varphi\ti\varphi = \ti\varphi$, and 
hence $A \in \Br[\varphi\ti\varphi]$. If, instead, $\ker(A) = \ti\kf_1 \in \{\kf_1,\dots,\kf_{r-1}\}$, then $\ti\kf_1 = \kf_i$ for some 
$1\leq i \leq r-1$. We can write $A = (1-\kappa)A^{\blds k_i} + \kappa A'$, with 
$0 \leq \kappa <1$ and $A' \in \Br[\psi_i]$, 
where $\varphi = (\kf_1 < \cdots < \kf_{i} < \psi_i)$ for $\psi_i = (\kf_{i+1} < \dots < \kf_r)$. 

In our two representations of $A$, we show now $\it\kappa = \ti \kappa$. 
We get
$$
  (\ti\kappa-\kappa) \cdot A^{\bldss{\ti\kf_1}} = \ti\kappa \ti A' -\kappa A'\,.
$$
Now if $\kappa \neq \ti \kappa$ the left hand side, restricted to $\ti \kf_1^\perp$ 
is a non-vanishing multiple of the identity. Thus either the restriction of $\kappa A'$ or $\ti\kappa \ti A'$ to  $\ti \kf_1^\perp$ 
would be a positive operator using $\kappa,\ti \kappa \geq 0$ and $A',\ti A' \geq 0$. But this is a contradiction
to $\ti \kf_1 \subsetneq \ker(A')$ and $\ti \kf_1 \subsetneq \ker(\ti A')$.

We conclude $\ti\kappa = \kappa$.  If $\kappa = 0$ then $A= A^{\bldstik\kf_1}$, 
an element of $\Br[\varphi\ti\varphi]$. If $\kappa > 0$ then $\ti A'=A'$, for $\ti A' \in \Br[\ti\psi]$ and $A' \in \Br[\psi_i]$. Since $\ell(\ti\psi) < \ell(\ti\varphi)$ and $\ell(\psi_i) < \ell(\varphi)$, we can apply induction: $\ti A' = A' \in \Br[\psi_i]\cap\Br[\ti\psi] = \Br[\psi_i\ti\psi]$. 
To complete the proof we still need to show that $\varphi\ti\varphi = (\ti\kf_1 < \psi_i\ti\psi)$. 
Since $\kf_{i-1}< \kf_i =\ti \kf_1$ we have $\varphi \ti \varphi=\varphi' \ti \varphi$
with $\varphi'=(\ti \kf_1 < \psi_i)$. Since $\ti \varphi=(\ti \kf_1 < \ti \psi)$ also this claim follows.
\end{proof}

%%%%%%%%%%%%%%%%%%%%

\subsection{{\it Definitions of important sets}} \label{sec:defWT}
%%%%%%%%%%%%%%%%%%%%%%%%%%%%%%%%%%%%%%%%%%%%%%%%%%%%%%%%

Recall,  $\Sub$ is the set of all intermediate subalgebras $\kf$ such that $\h < \kf < \g$, 
 $\Sub_t$ is the set of all toral subalgebras in $\Sub$, 
 while $\Sub_s$ is the set of all non-toral subalgebras in $\Sub$,  and $\Sub_{s,t}$ is the set of all non-toral subalgebras containing a toral subalgebra. Recall, also, that 
 $\Flts$ denotes the set of all flags $\varphi=(\tf_1 < \cdots < \tf_r < \kf)$; that is, flags in which $\tf_i$ is toral, while $\kf \leq \g$ is non-toral (Definition \ref{def:Union Non-toral flags}).
  
  Motivated by (asymptotic) scalar curvature estimates along the curves
$\gamma_v(t) \in \MGo$,
defined in section \ref{sec:ginvm}, we now define a subset of $ \Sph$ on which we can bound the scalar curvature from above. 
However, as Graev noticed, it is much more convenient to define the
corresponding set $\WW \subset \Sphb$.

\begin{definition}\label{defin-W}
Let $G/H$ be a compact homogeneous space. Then 
\begin{enumerate}
\item  the union of all disks $\D(\kf)$ over all intermediate subalgebras is 
$$
  \WW := \bigcup_{{\blds\kf} \in \Sub} \D(\kf) \subset \Sphb \,,
$$
\item the union of all disks $\D(\kf)$ over all non-toral intermediate subalgebras is 
$$
 \WWs := \bigcup_{{\blds\kf} \in \Sub_s}\D(\kf) \subset \WW\,,
$$
\item the union of disks $\D(\kf)$ over all non-toral subalgebras $\kf$  such that there exists
a toral subalgebra $\tf < \kf$ (hence $\D(\kf) \subset \partial \D(\tf) \subset \D(\tf)$) is
$$
  \WWst := \bigcup_{{\blds\kf} \in \Sub_{s,t}} 
 \D(\kf)\subset \WWs\,,
	$$
	
\item  the union of all flag simplices $\Delta_{\varphi}$  over all  toral flags  $\varphi \in \Flt$ is
$$
 \T:=\bigcup_{\varphi \in \Flt} \Delta_{\varphi} \,,
$$

\item the set of all butterflies of flags $\varphi = (\varphi_t < \kf) \in \Flts$ is 
$$
  \Bt= \bigcup_{\varphi\in \Flts} \Br[\varphi] \,,
$$
\item the {\em nerve} of $G/H$ is the union of all flag simplices over all non-toral flags
$$
  \XXGH:= \bigcup_{\varphi \in \Fln}  \Delta_{\varphi}\,. 
$$
\end{enumerate}
\end{definition} 

Notice that
${\displaystyle
  \WW = \{ A \in \SymgH \mid A \in \D(\kf) \textrm{ for some }\kf \in\Sub \}}$, 
and $\WWs$ is the union of all proper butterflies $\Br[\varphi]$, $\varphi\in \Fln$. We see 
$\WW_{s,t} \subset \WWs \subset \WW$, $\T \subset \WW$, and $\XXGH \subset  \WWs$.  
Since for any subalgebra $\kf$ properly containing $\h$, we
have $\D(\kf)\subset \Sphb$, clearly $\WW \subset \Sphb$.

\medskip

We can identify our subsets of $\Ad(H)$-invariant subspaces of $\g$
with corresponding sets of {\em projections}, $P \in \SymgH$ with $P^2=P$.  
Clearly the set of projections in $\SymgH$ is an algebraic (thus semi-algebraic) subset. 
We will use this to prove the proposition below. 

\begin{definition}\label{def:proj} The set of projections whose kernel is a
subalgebra $\kf \in\Sub$  is
\beg
  \Ps:
	&= &
	 \{ P \in \SymgH \mid P=P^2 \,,\,\, \h \subsetneq \ker(P) \subsetneq \g \textrm{ and }\\
	&&
   \quad\quad\quad \quad\quad
		P([(\Id_{\g}-P)(X),(\Id_{\g}-P)(Y)])=0\,\,\, \textrm{ for all }     X,Y \in \g \}\,.
\en
 Writing $\Id_{\m}:=\Id_{\g}-\Id_{\h}\in \SymgH$, we also define
$$
 \Psn:= \{ P \in \Ps \mid [(\Id_{\m}-P)(X),(\Id_{\m}-P)(Y) ]\neq 0  \textrm{ for some } X,Y \in \g\}\,,
$$
and
$$
  \Pst:= \{ P \in \Ps \mid  \,\, [(\Id_{\m}-P)(X),(\Id_{\m}-P)(Y) ]= 0  \textrm{ for all } X,Y \in \g\}\,.
$$
\end{definition}

Recall, for details on the theory of semi-algebraic sets, in particular the theorem of
Tarski-Seidenberg, we refer to Section \ref{sec:semialgebraic}.

\begin{lemma}\label{lem:Ps}
The sets $\Ps$, $\Psn$ and $\Pst$ are compact and semi-algebraic. 
\end{lemma}

\begin{proof}
The condition $\h \subsetneq \ker(P)$ is equivalent to
$$
 P|_{\h}=0, \textrm{ and  for some }  Y \in \m\bs\{0\}, P(Y) = 0\,. 
$$
Hence the set $\Ps$ is semi-algebraic. % see Section \ref{sec:semialgebraic}.
It is compact, since a sequence of subalgebras of $\g$
of fixed dimension subconverges to a subalgebra of $\g$: see Lemma \ref{lem:sequsubalgsub}.

Notice that an element $P \in \Psn$ is a projection with $\ker(P) = \kf \in \Sub_s$, non-toral. By Lemma \ref{lem:Dkfi}, the limit if a convergent sequence of non-toral subalgebras is a non-toral subalgebra. 
This shows that $\Psn$ is compact.
Clearly, $\Psn$ is semi-algebraic.

Notice that an element $P \in \Pst$ is a projection with $\ker(P) = \tf \in \Sub_t$ (toral). 
It is again clear that $\Pst$ is a compact, semi-algebraic subset of $\SymgH$.
\end{proof}

\begin{proposition}\label{prop: WW} Let $G/H$ be a compact homogeneous space. 
Each of the sets  $\WW$, $ \WWs$, $ \WWst$, $\T$, and $\XXGH$ in Definition \ref{defin-W} is compact and semi-algebraic.
\end{proposition}

\begin{proof}
(1) Let $(A_i)_{i \in \NN}$ be a sequence in $\WW$. We may assume that
$\lim_{i \to \infty}A_i=A_\infty \in \Sphb$. We will show
that $A_\infty \in \WW$ and thus that $\WW$ is compact. Suppose that $A_i \in \D(\kf_i)$
for all $i \in \NN$, $\kf_i$ subalgebra with $\h < \kf_i<\g$. Then, by passing to a subsequence,
we may assume that  $\kf_i \to \kf_\infty$ as $i \to \infty$, with $\kf_\infty \in\Sub$ a subalgebra (see Lemma \ref{lem:Dkfi}). Since by hypothesis we know $[A_i,\ad(\kf_i)]=0$ for all $i \in \NN$, it follows
$[A_\infty,\ad(\kf_\infty)]=0$. Thus $\WW$ is compact.

We have
\beg
  \WW&=&
  \big\{ A \in \SymgH \mid A \geq 0,\,\, \tr A=1 \textrm{ and }
   \exists P \in \Ps \textrm{ with } \\
     &&
      \,\,\, \ker (P) \subset \ker(A) \textrm{ and }
      [A,\ad((\Id_{\g}-P)(X))]=0 \,\,\,\forall\,\,X \in \g \big\}\,.
\en
This shows that $\WW$ is a semi-algebraic set.

\medskip
\noindent (2) The proof that $\WWs$ is compact and semi-algebraic follows precisely as above. To see  compactness, we use that
a sequence $(\kf_i)_{i \in \NN}$ of non-toral subalgebras subconverges to a non-toral subalgebra $\kf_\infty$, by Lemma \ref{lem:Dkfi} and Lemma \ref{lem:sequsubalgsub}.
 To see that $\WWs$ is semi-algebraic, we replace $\Ps$ by $\Psn$.

\medskip

\noindent (3) The proof of compactness of $\WWst$ is analogous to the compactness of $\WWs$. 
Any sequence of non-toral subalgebras $(\kf_i)_{i\in \NN}$ subconverges to
a non-toral subalgebra $\kf_\infty$. If $\tf_i$ are subalgebras of $\kf_i$ for all $i \in \NN$,
then we may assume that $\tf_i$ subconverges to a toral subalgebra $\tf_\infty \subset \kf_\infty$.

To see that $\WWst$ is a semi-algebraic set, we observe that 
\beg
  \WWst
	 &=&
  \big\{ A \in \SymgH \mid A \geq 0,\,\, \tr A=1 \textrm{ and }
   \exists P_1\in \Pst,\,\,  \exists P_2 \in \Psn \textrm{ with }\\
     &&
      \,\,\, \ker(P_1) \subset \ker (P_2) \subset \ker(A) \textrm{ and }
      [A,\ad((\Id_{\g}-P_2)(X))]=0 \,\,\,\forall\,\,X \in \g \big\}\,.
\en
Thus $\WWst$ is semi-algebraic.

\medskip

\noindent (4)  Since a sequence of toral subalgebras subconverges to a toral subalgebra, 
$\T$ is compact.

We have
\beg
 \T &=&
    \big\{ A \in \WW \mid  A=\sum_{i=1}^r \lambda_i \cdot P_i
		 \textrm{ for some } P_1,...,P_r \in \Pst \\
  &&
	  \quad \quad \quad\quad\quad
		\textrm{s.t. }
	  \ker(P_1) \subsetneq \cdots \subsetneq \ker(P_r)
		 \textrm{ and } \lambda_i \geq 0
		 \big\}\,.
\en
 More precisely we mean that
$\exists P_1,...,P_r \in \Pst$ and $\exists \lambda_1,...,\lambda_r\geq 0$ with
 $\ker(P_1) \subsetneq \cdots \subsetneq \ker(P_r)$ such that  $A=\sum_{i=1}^r \lambda_i \cdot P_i$.
This shows that $\T$ is a semi-algebraic set.

\medskip

\noindent (5) To prove that the nerve $\XXGH$ is a compact, semi-algebraic set, we first observe that  
$\XXGH \subset \WWs$.  
Let $(A_i)_{i\in \NN}$ be a sequence in $\XXGH$
and for each $A_i$, let $\varphi_i=(\kf_1^i<\cdots < \kf_{r_i}^i < \g)$ be a corresponding flag with $A_i \in \Br[\varphi_i]$.  
Passing to a subsequence we may assume that $r_i\equiv r$ and that as $i \to \infty$, 
$\kf_j^i \to \kf_j^\infty$ for all $1\leq j\leq r$. Since $\kf_1^i$ is non-toral for all $i\in\NN$,
so is $\kf_1^\infty$. This shows the compactness of $\XXGH$.

We have
\beg
 \XXGH &=&
\big\{ A \in \WWs \mid  A=\sum_{i=1}^r \lambda_i \cdot P_i
		 \textrm{ for some } P_1,...,P_r \in \Psn \\
  &&
	  \quad \quad \quad\quad\quad
		\textrm{s.t. }
	  \ker(P_1) \subsetneq \cdots \subsetneq \ker(P_r)
		 \textrm{ and } \lambda_i \geq 0
		 \big\}\,
\en
because for each $P \in \Psn$ we have $\Ak=\tfrac{P}{\tr P}$ with $\kf=\ker(P)$. 
By definition of $\Psn$ the subalgebra $\kf$ is non-toral.
This shows that $\XXGH$ is semialgebraic.
\end{proof}

\begin{remark} In Proposition \ref{prop:theta} we prove that the set $\Bt$ is also compact and semi-algebraic.
In Lemma \ref{lem:XGHXXGH}, 
we show  that $\XXGH$ is homeomorphic to the {\em nerve}, $\XGH$, a compact semi-algebraic set, defined purely Lie theoretically, whose non-contractibility implies the existence of a $G$-invariant Einstein metric on $G/H$. 
\end{remark}

%%%%%%%%%%%%%%%%%%%%%%%%%%%%%%%%%%%%%%%%%%%%%%%%%%%%%%%%%%%%%
\section{Homotopies}\label{homotopies}
\subsection{{\it The first homotopy}}\label{sec:homotopy1}
%%%%%%%%%%%%%%%%%%%%%%%%%%%%%%%%%%%%%%%%%%%%%%%%%%%%%%%%%%%%%%%

In this section we show in Theorem \ref{thm:homotopy2} that
$\WWs$ is a strong deformation retract of a compact, semi-algebraic set $\B_{\epsilon_0} \subset \Sphb$.
Here we follow \cite{Gr}.

The geometric significance of $\XXsre$ comes from the fact that for directions 
$v \in \Sph$ outside of an open neighborhood of the inverse image of $\B_{\epsilon_0}$ under ${\rm Gr}$,
we obtain uniform upper bounds for the scalar curvature  
along the geodesic $\gamma_v(t)$, for all $t \geq \bar t(\delta)$ (see Theorem  \ref{theo:scalestu}).

\medskip

In Definition \ref{defin-W} we defined the set $ \Bt= \bigcup_{\varphi\in \Flts} \Br[\varphi]$. We can view this as the set of  all butterflies joining (a subset of) $\T$ and $\WWst$: see Remark \ref{rem:tsubandnot}.
More precisely, for every $\Br[\varphi] \subset \Bt$ we have
$$
     \Br[\varphi]= \rm {conv}\{A^{\bldst{\tf}_1},...,A^{\bldst{\tf}_r}\} \ast D(\kf)=\Delta_{\varphi_t}\ast \D(\kf)
$$
where $\varphi=(\varphi_t<\kf)\in\Flts$, $\varphi_t=(\tf_1<\cdots <\tf_r)$ is a toral flag, 
 $\kf$ is non-toral with $\tf_r < \kf$, and $\Delta_{\varphi_t} = \rm {conv}\{A^{\bldst{\tf}_1},...,A^{\bldst{\tf}_r}\}$ is as in Definition \ref{def:simplex}. 
For every $A \in  \Br[\varphi]$, there exist $A_t \in \Delta_{\varphi_t} $,
 $A_s \in \D(\kf)$, and $\kappa \in [0,1]$ such that
\begn \label{eqn:joindescr}
  A= (1-\kappa)\cdot A_t + \kappa\cdot A_s\,.
\enn
Conversely, for any such $A$, assuming $\tf_r < \kf$,  we have $A\in  \Br[\varphi]$.

\begin{remark}\label{rem:tsubandnot}
It can happen that some toral subalgebras $\tf$ of $\g$ are contained in a non-toral subalgebra $\kf$,
and others are not. An easy example is $G=\SO(n+2)\times S^1 \times S^1$ and
$H=\SO(n)\cdot \Delta(S^1)$, where $\SO(n)$ is embedded as a block and 
$\Delta(S^1)$ is diagonally embedded into $\SO(2)\times S^1 \times \{e\}$. In this case
$\m = \m_0 \oplus \p$, where $\dim \m_0=2$ is abelian and $\p$ is $\Ad(H)$-irreducible.
The toral subalgebra $\n(\h)=\h \oplus \m_0$ is maximal in $\g$, thus certainly not contained in any proper non-toral subalgebra, whereas $\tf:= \so (n) \oplus \so(2) \oplus \RR$ is toral and is contained in 
$\kf=\so (n+2)\oplus \RR < \g$.
\end{remark}

By the following lemma, the description of $A$ in \eqref{eqn:joindescr}
is essentially unique.

\begin{lemma}\label{lem:esudescrA} %2023
Let $\varphi  \in\Flts$, let $\Br[\varphi]\subset \Bt$ and let $A, \tilde A \in \Br[\varphi]$. Suppose that
$$
 A= (1-\kappa)\cdot A_t + \kappa\cdot A_s
\textrm{ and }
\tilde A= (1-\tilde \kappa)\cdot \tilde A_t + \tilde \kappa\cdot \tilde A_s
$$
for some $0\leq \kappa,\tilde \kappa \leq 1$, $A_t, \tilde A_t \in \rm {conv}(A^{\bldst{\tf}_1},...,A^{\bldst{\tf}_r})$ 
and $A_s,\tilde A_s \in \D(\kf)$.
Then $A=\tilde A$ implies $\kappa = \tilde \kappa$, and moreover, 
\begin{enumerate}
\item[{\rm (1)}]  If $\kappa=1$, then $A_s=\tilde A_s$; 
\item[{\rm (2)}] if $\kappa=0$ then $A_t=\tilde A_t$; 
\item[{\rm (3)}] if $\kappa \in (0,1)$, then
$A_t=\tilde A_t$ and $A_s =\tilde A_s$. 
\end{enumerate}
\end{lemma}

\begin{proof}
Suppose that $\kappa=1$. Then $A=A_s$ and thus $\kf \subset \ker(A)$. If $\tilde \kappa<1$ then
$(1-\tilde \kappa)\tilde A_t\neq 0$, thus $\ker ((1-\tilde \kappa) \tilde A_t) \subset \tf_r$
and $\ker (\tilde A)=\ker (A) \subset \tf_r$. This contradicts  $\tf_r < \kf$ and
 shows $\kappa=\tilde \kappa=1$ and (1).

Suppose that $\kappa=0$. Then $A=A_t$ and $A\vert_{\tf_r^\perp}$ is a multiple of the identity.
If $\tilde \kappa>0$ then $(1-\tilde \kappa)\tilde A_t\vert_{\tf_r^\perp}$ is  a multiple of the identity
but $(1-\tilde \kappa)\tilde A_s\vert_{\tf_r^\perp}$ is not, contradicting $A=\tilde A$. This shows 
$\kappa=\tilde \kappa=0$ and (2).

We now assume that $0<\kappa,\tilde \kappa <1$. Since $A=\tilde A$, it follows that
$A\vert_{\tf_r}=\tilde A\vert_{\tf_r}$, thus $(1-\kappa) A_t\vert_{\tf_r}=(1-\tilde \kappa) \tilde A_t\vert_{\tf_r}$.
We know $A_t =\sum_{i=1}^r \lambda_i \cdot A^{\bldst{\tf}_i}$ and
$\tilde A_t =\sum_{i=1}^r \tilde \lambda_i \cdot A^{\bldst{\tf}_i}$ with 
$0 \leq \lambda_i,\tilde \lambda_i \leq 1$ for all $1\leq i\leq r$
and $\sum_{i=1}^r\lambda_i=\sum_{i=1}^r \tilde \lambda_i=1$. By the linear independence of $\{A^{\bldst{\tf}_1},...,A^{\bldst{\tf}_r}\}$, 
we conclude $(1-\kappa) \lambda_i =(1-\tilde \kappa) \tilde \lambda_i$ for each $i=1,...,r$. Summation over $i$ yields
$\kappa =\tilde \kappa$.
As a consequence, we deduce from $A=\tilde A$ that
$$
 (1-\kappa) \cdot (A_t -\tilde A_t)= -\kappa\cdot (A_s-\tilde A_s)\,.
$$
The map on the left hand side, restricted to $\tf_r^\perp$, is a multiple of the identity. Thus $A_s=\tilde A_s$
since $\kappa>0$. Since $\kappa<1$  we deduce $A_t=\tilde A_t$, which proves (3).
\end{proof}

The following corollary says that the description of $A$ in \eqref{eqn:joindescr}
does not depend on the butterfly $\Br(\varphi)$ with $A \in \Br(\varphi)$.

\begin{corollary}\label{cor:AdescBunique}
Let $\Br[\varphi],\Br[\tilde \varphi] \subset \Bt$ and let $A \in \Br[\varphi]\cap \Br[\tilde \varphi]$.
Then either $A \in \T$, or $A \in \WWst$, or there exists unique $A_t \in \T$, $A_s \in \WWst$
and $\kappa \in (0,1)$ such that $A=(1-\kappa)\cdot A_t + \kappa \cdot A_s$.
\end{corollary}

\begin{proof}
Since $A \in \Br[\varphi]\cap \Br[\tilde \varphi]$ we conclude that
$\emptyset \neq \Br[\varphi]\cap \Br[\tilde \varphi]$. Let
$\varphi=(\tf_1 < \cdots < \tf_r< \kf)$ and 
$\tilde \varphi= (\tilde \tf_1 < \cdots < \tilde \tf_{\tilde r}< \tilde \kf)$.
By Lemma \ref{lem:butphipsi} we know
that $A \in \Br[\varphi \tilde \varphi]$. Thus $\varphi\tilde \varphi \neq (\g)$.
By  Corollary \ref{cor:phipsi} we conclude therefore that
$\varphi \tilde \varphi =(\varphi_t<\g)$ with $\varphi_t:=(\tf_{i_1} < \cdots < \tf_{i_k})$, $1 \leq k \leq r$
and $1 \leq i_1 < \cdots < i_k \leq r$,
or that $ \varphi \tilde \varphi=(\kf^\ast)$ (recall, $\kf^\ast=\langle \kf,\tilde \kf \rangle$ and here $\kf^\ast<\g$), or that
$\varphi\tilde \varphi=(\varphi_t < \kf^\ast)$ with $\varphi_t$, $\kf^\ast$ as above.

We first consider the case that $\varphi\tilde \varphi=(\varphi_t < \kf^\ast)$. By Lemma \ref{lem:esudescrA}
we obtain an (essentially) unique $A_t \in \Delta_{\varphi_t}$,
$A_s \in \D(\kf^\ast)$ and $\kappa \in [0,1]$ with $A=(1-\kappa)\cdot A_t + \kappa\cdot A_s$. 
If $\kappa=0$ or $\kappa=1$ the claim follows. 
Now suppose $\kappa\in (0,1)$.
Since $\Br[\varphi\tilde \varphi] \subset \Br[\varphi]$, again by Lemma \ref{lem:esudescrA}
 we obtain an (essentially) unique description $A =(1-\kappa')\cdot A_t' + \kappa'\cdot A_s'$,
for some $A_t' \in {\rm conv}(A^{\bldst{\tf}_1},...,  A^{\bldst{\tf}_r})$,
$A_s' \in \D(\kf)$ and $\kappa' \in [0,1]$. Since $\D(\kf^\ast)\subset \D(\kf)$ we deduce that
$A_s \in \D(\kf)$ and $A_t \in
{\rm conv}(A^{\bldst{\tf}_1},...,  A^{\bldst{\tf}_r})$. This yields two descriptions of $A$ in $\Br[\varphi]$.
By Lemma \ref{lem:esudescrA}, (3), we obtain uniqueness. 

In the second case, if $ \varphi \tilde \varphi=(\kf^\ast)$, we have $\Br[\varphi\tilde \varphi]=\D(\kf^\ast)\subset \WWst$, because $(\kf^\ast)$ is non-toral  (since $\kf$ is non-toral). Finally, in the first case we have  $\Br[\varphi\tilde \varphi]=\Delta_{\varphi_t}\subset \T$.
This shows the claim.
\end{proof}

The above corollary  immediately implies

\begin{corollary}\label{cor:Adescunique}
Let $A \in \Bt$. Then either $A \in \T$, $A \in \WWst$ or there exist unique 
$A_t \in \T$, $A_s \in \WWst$
and $\kappa \in (0,1)$ such that $A=(1-\kappa)\cdot A_t + \kappa\cdot A_s$.
\end{corollary}

Recall that a choice of $A_t$ restricts the possible maps $A_s$. In particular there may exist
$A_t \in \T$ such that $A_t$ is not contained in any $\Br[\varphi]$, $\varphi=(\varphi_t< \kf) \in \Flts$:
see Remark \ref{rem:tsubandnot}.

\begin{definition}\label{def:Bth}
 Let $G/H$ be a compact homogeneous space. Then we set
$$
 \Bth :=\{ (A_t,A_s) \in \T \times \WWst \mid  (1) \textrm{ and } (2) \textrm{ are satisfied}\}\,,
 $$
where the conditions (1) and (2) are defined as follows:
We require for  $(A_t,A_s)$ that there exists a projection $P=P(A_t,A_s) \in \Psn$ satisfying:\\
(1) $\ker(P) \subset \ker (A_s)$ and 
      $[A_s,\ad((\Id_{\g}-P)(X))]=0 \,\,\,\textrm{ for all } X \in \g$; \\
 (2) $\tr(P)\cdot A_tP= \tr(P)\cdot PA_t=\tr(A_tP)\cdot P$.
\end{definition}

 Notice that 
condition (1) is the requirement that $A_s \in \D(\kf)$ 
for $\kf:= \ker(P)$ the non-toral subalgebra related to $P$. 
Condition (2) ensures that $A_t$ and $P$ commute, thus they can be diagonalized simultaneously, and $\ker(A_t) \subset \ker(P)$ (that is, $\tf_r < \kf$). Consequently,
on $\kf^\perp=\m_{\blds\kf}$, 
the map $A_t$ must be a multiple of the identity. 

\begin{lemma}\label{lem:Bth-cpt} The set  $\Bth$ is  compact and semi-algebraic.
\end{lemma}

\begin{proof}
Since the sets $\T$, $\WWst$ and $\Psn$ are semi-algebraic, $\Bth$ is semi-algebraic: see Lemma \ref{lem:Ps}.

 To see that $\Bth$ is compact, we consider a sequence $((A_t^i,A_s^i))_{i \in \NN}$ 
 in $\Bth$, and the corresponding sequence $(P_i)_{i \in \NN} = (P(A_t^i,A_s^i))$.  Passing to a subsequence, by compactness of $\T$, $\WWst$ and $\Psn$, we know as $i \to \infty$, $(A_t^i,A_s^i) \to (A_t^\infty, A_s^\infty)\in \T \times \WWs$ and $P_i \to P_\infty$. 
 Since condition (2) is algebraic,  (2) will hold for $(A_t^\infty,P_\infty)$. 
 Similarly,  the second part of condition (1) holds:  $[A_s^\infty,\ad((\Id_{\g}-P_\infty)(X))]=0 \,\,\,\textrm{ for all } X \in \g$. Furthermore, since the eigenvalues of the maps $P_i$ are in $\{0,1\}$, the kernel of $P_\infty$ cannot jump in dimension, an we conclude $\ker(P_\infty) \subset \ker (A_s^\infty)$. Thus condition (1) holds as well.  This shows $P_\infty=P(A_t^\infty,A_s^\infty)$, thus
  $\Bth$ is indeed compact.
\end{proof}

\begin{proposition}\label{prop:theta}
Let
${\displaystyle 
 \theta: [0,1] \times \Bth \to \Sphb}$ be given by 
 $$
 (\kappa,A_t,A_s)\mapsto (1-\kappa)\cdot A_t + \kappa\cdot A_s\,.
$$
Then $\theta([0,1]\times \Bth)=\Bt$. Thus the set $\Bt$ is compact and semi-algebraic.
\end{proposition}

\begin{proof}	
Corollary \ref{cor:Adescunique}	
implies $\Bt \subset \theta([0,1]\times \Bth)$, since we will have $A_t$ and $A_s$ and we can use   the projection $P$ onto $\kf^{\perp}$, which will yield $\ker(A_t) \subset \ker(P)$.  Then (1) is satisfied and it is clear that (2) holds true as well. 

To see the reverse containment, 
let $(A_t,A_s)\in \Bth$ and $P \in \Psn$ be given as in Definition \ref{def:Bth}. 									
Let $A_t =\sum_{i=1}^r \lambda_i \cdot A^{\bldst{\tf}_i}$ for a toral flag $(\tf_1<\cdots < \tf_r)$ with
$0< \lambda_i\leq 1$ for all $i=1,\dots,r$ and let
$\kf = \ker(P)$ for a non-toral subalgebra $\kf$. 
By condition (2),
we know that $A_t\vert_{\kf^\perp}$ is a positive multiple of the identity. Thus $\tf_r$ 
is a subalgebra of $\kf$, proper since $\tf_r$ is toral and $\kf$ is not.  
By condition (1), we know $A_s \in \D(\kf)$. 
Thus the image $\theta([0,1]\times \Bth) \subset \Bt$,  proving our equality.
\end{proof}

Next, we define an $\epsilon$-tube of $\T \cap \Bt$ in $\Bt$.

\begin{proposition}\label{prop:Teps}												
For $0<\epsilon \leq 1$ let
$$
 \T(\epsilon):=\theta([0,\eps)\times \Bth)=
 \big\{ A \in \Bt \mid A=(1-\kappa) \cdot A_t+\kappa\cdot A_s\,,\,\, 0 \leq \kappa <\epsilon \big\}\,.
$$
Then $\T(\epsilon)$ is a semi-algebraic,  open set in $\Bt$.
\end{proposition}

\begin{proof}
Since $\theta$ is a polynomial function, and since we know that $[0,\eps)\times \Bth$ is semi-algebraic,    by the theorem
of Tarski-Seidenberg, we conclude that $\T(\epsilon)$ is also semi-algebraic.  
Notice, furthermore, that
 the sets $\T(\epsilon)$ and $\theta([\eps,1]\times \Bth)$ are disjoint, by Lemma \ref{lem:esudescrA},
while their union is $\T(\epsilon) \cup \theta([\eps,1]\times \Bth)=\Bt$, by Proposition \ref{prop:theta}.
Since $\theta([\eps,1]\times \Bth)$ is compact, the claim follows.
\end{proof}

\medskip

The following observation will be important in Section \ref{secscalab},  where we define 
the subset  $\TSi:={\rm Gr}^{-1}(\T)$ of $\Sph$.   
In Theorem \ref{thm:Tbound}, we prove (not using any results of this section) that
there exists an open neighborhood $U_{\TSi}$ of $\TSi$ in $\Si$ and $t_1=t_1(U_{\TSi})>0$ such that for 
all $v \in U_{\TSi}$ and all $t \geq t_1$, we have
$\sc(\gamma_v(t)) \leq \tfrac{1}{2}\cdot b_{G/H}$.

\begin{corollary}\label{cor:choseeps0}
Let  $U_{\TSi}$ be the open neighborhood of $\TSi:={\rm Gr}^{-1}(\T)$ in $\Si$
 defined in Theorem \ref{thm:Tbound}. 
Let $\U_T:={\rm Gr}(U_{\TSi}) \subset \Sphb$, a neighborhood of $\T$. Then
there exists some $\epsilon_0=\epsilon_0(U_{\TSi})>0$ 
such that the image $\theta(\{\kappa=\epsilon_0\}\times \Bth)\subset \U_T$.
\end{corollary}

\begin{proof}
Since $\T={\rm Gr}(\TSi)$ and ${\displaystyle \bigcap_{0<\eps \leq 1} \T(\eps) \subset \T}$, the claim follows.
\end{proof}

We call $\theta(\{\kappa=\epsilon_0\}\times \Bth) \subset \Bt$ the bottom  of $\Bt \backslash \T(\epsilon_0)$ 
considering $\Bt$ a cone over $\WWst$ with cone tip $\T \cap \Bt$.

We come to our first main result in this section.

\begin{theorem}\label{thm:homotopy1}
For every $\epsilon \in (0,1]$, the set $\WWst$ is a strong deformation retract of $\Bt\bs \T(\epsilon)$.
\end{theorem}

\begin{proof}
We first note that
$$
 \Bt \bs \T(\epsilon)=
 \big\{ A \in \Bt \mid A=(1-\kappa)A_t+\kappa A_s\,,\,\,\epsilon \leq \kappa \leq 1\big\}\,.
$$
For $s, \kappa \in [0,1]$ we set $\kappa_s:=(1-s)(1-\kappa)$ and define $H:[0,1]\times \Bt\bs \T(\epsilon) \to \Bt\bs \T(\epsilon)$ by
$$
	    H(s,\kappa,A_t,A_s) =  \kappa_s\cdot A_t +(1-\kappa_s)\cdot A_s =\theta(1-\kappa_s, A_t,A_s)\,.
$$
We see that $H$ is well-defined and continuous, 
by Corollary \ref{cor:Adescunique} and the continuity of $\theta$. 
We verify that the image of $H(s,\kappa,A_t,A_s)$ lies in $\Bt$. 
Let $0 \leq  s \leq 1$.  We see that $0 \leq \kappa_s \leq 1-\epsilon$, 
hence $\epsilon \leq 1-\kappa_s \leq 1$; therefore, 
$H(s,\kappa,A_t,A_s) =\theta(1-\kappa_s, A_t,A_s) \in \Bt \bs  \T(\epsilon)$.  
Notice that when $\kappa = 1$ (i.e., when $A \in \WWst$), $\kappa_s=0$ for all values $s\in[0,1]$, so that 
$H(s,1,A_t,A_s)=\theta(1, A_t,A_s) = A_s \in \WWst$.  For a general $\kappa$, when $s=0$, $H(0,\kappa,A_t, A_s)=\theta(\kappa, A_t, A_s)$, and when $s=1$,  
$H(1,\kappa,A_t, A_s)=\theta(1, A_t, A_s)=A_s \in \WWst$. 
 Thus the map $H$ is our strong deformation retraction. 
\end{proof}
\medskip

In general there does not exist a strong deformation retraction from $\Bt$ to $\WWst$; it is necessary to remove $\T(\epsilon)$. An easy
example is $G/H=(\SU(2)\times \SU(2))/S^1_{k,l}$, where $S^1_{k,l}$ is embedded diagonally into
a maximal torus of $G$ with slope $(k,l)$ for generic $k,l \in \NN$. 
Let $\tf$ be the unique maximal torus of $\g$ with $\h < \tf$ and let
$\kf_1=\su (2) \oplus \RR$ and $\kf_2=\RR \oplus \su (2)$ be the non-toral subalgebras of $\g$ with 
$\h< \kf_1,\kf_2 < \g$. Let $\varphi_1=(\tf< \kf_1)$ and $\varphi_2=(\tf <\kf_2)$. 
Then $\Bt=\Br[\varphi_1]\cup \Br[\varphi_2]$, the union of two cones over $\At$, is contractible,
whereas $\WWst=\{A^{\blds{\kf}_1},A^{\blds{\kf}_2}\}$, which is not.

\begin{theorem}\label{thm:homotopy2}
We set $\B := \B_t \cup \WWs$, 
$\epsilon_0 > 0$ as in Corollary \ref{cor:choseeps0}, and
$$
   \XXsre:=\B \bs \T(\epsilon_0)\,.
$$
There exists a strong deformation retraction from
$\XXsre$  to $\WWs$. Moreover, $ \XXsre$ is compact and
semi-algebraic.
\end{theorem}

\begin{proof}
We claim that one can extend the strong deformation retraction  $H$
 from $\B_t\bs \T(\epsilon_0)$ to $\WWst$, 
 defined in Theorem \ref{thm:homotopy1},
to a strong deformation retraction $\ti H$ from 
$\XXsre=(\B_t \bs \T(\epsilon_0)) \cup (\WWs \bs \WWst)$
to $\WWs$.   
Let $\X_1:=\Bt \bs \T(\eps_0)$.  Let $\X_2:=\WWs \bs \WWst$. 
Notice $\XXsre = \X_1 \cup \X_2$ is a disjoint union, and $[0,1]\times\X_1$ 
is the domain of our original map $H$.  
The new retraction map $\tilde H$ 
is defined as $\ti H |_{[0,1]\times\X_1} := H$ and $\ti H |_{[0,1]\times\X_2} = \Id|_{\X_2}$. 
We note that $\X_1 \cap \WWs = \WWst$, and therefore our map $\ti H$ on
$[0,1]\times \X_1 \cap [0,1]\times \overline{\X}_2 $ is the identity.
 Thus $\ti H: [0,1] \times \XXsre \to \XXsre$ is indeed a strong deformation retraction. 
This shows the first claim.

By Proposition \ref{prop:Teps} we know that $\T(\eps_0)$ is open in $\B_t$, and 
therefore $\X_1$ is compact.
By Proposition \ref{prop: WW}, compactness of  $\WWs$ ensures that $\overline{\X}_2$ is a compact subset of $\WWs \subset \XXsre$. 
Since both $\X_1$ and $\overline{\X}_2$ are compact,  we know $\XXsre$ is compact. And since 
both $\X_1$ and  $\WWs$ are semi-algebraic, by Proposition \ref{prop: WW},  so is $\XXsre=\X_1 \cup \WWs$.
\end{proof}

\begin{corollary}\label{cor:Udefret}
There exists an open neighborhood $\UU$ of $ \B_{\epsilon_0}$ in $\Sphb$, 
such that $\B_{\epsilon_0}$ is a strong deformation retract of $\UU$. 
\end{corollary}

\begin{proof}
This follows from Theorem 1 of \cite{DK1}, see also \cite[III, Theorem 1.1]{DK2}.
\end{proof}

\medskip

We now give a characterization of the set $\WW \bs \B$: see Theorem
\ref{thm:homotopy2}.
To this end let $\kf \in \Sub$ be a subalgebra of $\g$  and
suppose that $ \{\Ak\} \subsetneq \D(\kf)$, so that $\D(\kf)$ has positive dimension.
 Recall that the boundary
of a convex set is the complement of its interior; 
furthermore, every convex set has  a dimension.
When $\kf$ is not a maximal subalgebra, 
we define the {\em algebraic part} $\D(\kf)_+$ of $\D(\kf)$
by
$$
  \D(\kf)_+:= \bigcup_{\ti \kf >\kf}  {\rm conv}(\Ak,\D(\ti \kf))\,.
$$
When $\kf$ is maximal we set $\D(\kf)_+=\{\Ak\}$.

We let $\D(\kf)_-:=\D(\kf)\bs \D(\kf)_+$ denote
the {\em non-algebraic part} of $\D(\kf)$. We define the {\em algebraic boundary} of $\D(\kf)$
$$
  \partial_+ \D(\kf):= \D(\kf)_+\cap \partial \D(\kf)\,,
$$
 and the {\em non-algebraic boundary}
$$
  \partial_- \D(\kf):= \partial \D(\kf) \bs \partial_+ \D(\kf)\,.
$$
Notice that
$$
   \partial_- \D(\kf) = \D(\kf)_- \cap  \partial \D(\kf)\,.
$$

\begin{lemma}\label{lem:Bdescr}
Let $A \in \WW$.  Then $A \in \B$ if and only if $A \in \WWs$ or there exists 
a flag $\varphi=(\varphi_t<\kf)\in\Flts$, 
where $\varphi_t=(\tf_1 < \cdots < \tf_r)$ is a toral flag with $\tf_r$ maximal in $\kf \in \Sub_s$,
 a non-toral subalgebra, such that
$$
     A \in {\rm conv}(\Delta_{\varphi_t},A')
$$
with $A' \in \D(\kf)$.
\end{lemma}

\begin{proof}
Let $A \in \B$. If $A \in \WWs$ then the claim follows, thus we may assume $A \in \B \bs \WWs$.
By definition of $\B$ we conclude $A \in \Bt$.
That is,  
there exists 
a toral flag $(\tf_1 < \cdots < \tf_r)$,
a non-toral subalgebra $\kf $ with  $\tf_r<\kf$, $A' \in \D(\kf)$
and $\lambda_1,...,\lambda_r,\lambda \in [0,1]$
with $\sum_{i=1}^r \lambda_i=1$ such that
$$
    A = \lambda \cdot \sum_{i=1}^r \lambda_i \cdot A^{\bldst\tf_i}  +
(1-\lambda)\cdot A'\,.
$$
Suppose now that there exists a non-toral subalgebra $\tilde \kf$ with
$\tf_r < \ti \kf <\kf$.
Then by Lemma \ref{lem:disk-inclusion}, we know $\D(\kf)\subset \D(\tilde \kf)$, so that we can replace $\kf$
 by $\ti \kf$, hence we may assume $\tf_r$ is maximal in $\kf$. 
Since $\sum_{i=1}^r \lambda \cdot \lambda_i
+(1-\lambda)=1$, the claim follows.  
The proof of the other direction follows from the definition of $\B$ and $\Bt$.
\end{proof}

\begin{corollary}\label{cor:WbsBdescr}
Let $A \in \WW \bs \B$. Then either $A \in \T$ or 
there exists a toral flag $(\tf_1< \cdots <\tf_r)$ and $A_- \in \D(\tf_r)_-$
such that
\begn
  A \in {\rm conv}\{A^{\bldst\tf_1},...,A^{\bldst\tf_{r-1}},A_-\}\bs \T\,.\label{case2}
\enn
\end{corollary}

\begin{proof}
Let $A\in\WW \bs \B$. Then, since $\WWs \subset \B$, we know $A \not\in\WWs$, thus there exists a toral
subalgebra $\tf_1$
with $A \in \D(\tf_1)$.
If $A= A^{\bldst\tf_1}$ then $A \in \T$ and we're done. Otherwise, we may assume $A \neq A^{\bldst\tf_1}$.
Now if the radial projection of the convex set $\D(\tf_1)$ to its boundary (with center
$A^{\bldst\tf_1}$) maps $A$
to $A_2\in \partial \D(\tf_1)$ such that
$A_2 \in \partial_-\D(\tf_1)$, then the claim holds for $r=1$ and $A_- = A_2$.

By our induction hypothesis we may assume that for some $r\geq 2$, either the claim is proved, or there
exists a toral flag $(\tf_1 < \cdots < \tf_{r-1})$, such that
$$
   A \in {\rm conv}\{A^{\bldst\tf_1},...,A^{\bldst\tf_{r-1}},A_r\}
$$
with $A_r \in \partial_+\D(\tf_{r-1})$.
It follows that there
exists a subalgebra $\kf$ with $\tf_{r-1} < \kf$ and $A_r \in \D(\kf)$.
Let us assume that
$\kf$ is a subalgebra of minimal dimension with that property. If this $\kf$ is
non-toral then  $\tf_{r-1}$ is
maximal in $\kf$  and by Lemma \ref{lem:Bdescr} we
obtain a contradiction.
Thus $\tf_r:=\kf$ must be toral. Now if $A_r=
A^{\bldst\tf_r}$ then $A \in \T$ and we are done. Suppose instead that $A_r\neq
A^{\bldst\tf_r}$ and let $A_{r+1} \in \partial \D(\tf_r)$ denote
the radial projection of $A_r$ onto $\partial \D(\tf_r)$. If $A_{r+1} \in \partial_- \D(\tf_r)$ the
claim follows.
Otherwise, we may assume that $A_{r+1}\in \partial_+ \D(\tf_r)$, which
yields  $A \in {\rm conv}\{A^{\bldst\tf_1},...,A^{\bldst\tf_{r}},A_{r+1}\}$.
Since the maximal length of toral flags is bounded, this process ends in finitely many steps, 
and we have proved our corollary.
\end{proof}

The next result shows that, for a toral subalgebra $\tf_r$, whenever there exist elements in the non-algebraic part $\D_-(\tf_r)$ of $\D(\tf_r)$,
 certain structure constants $[ijk]$ cannot vanish (see Section \ref{sec:scal}).    
 We briefly recall that $\g = \h \oplus \m$, where $\m$ is an $\Ad(H)$-invariant complement to $\h$ in $\g$ As such, $\m$ decomposes into $\Ad(H)$-irreducible summands. In the lemma below, the decomposition is chosen such that 
$$\m = \m_{I_1^{\bldst\tf_r}} \oplus \m_{I_2} \oplus (\m_{I_3} \oplus \dots \oplus \m_{I_{\ell'}})\,,$$ 
where $\tf_r=\h \oplus \m_{I_1^{\bldst\tf_r}}$, $\ker(A_r) = \tf_r \oplus \m_{I_2}$, and the remaining subspaces 
$\m_{I_j}$ are grouped into the eigenspaces for the positive eigenvalues of $A_r$. 
Notice that these summands may not be $\Ad(H)$-irreducible.
The structure constants are given by 
$$
{\displaystyle [ijk] = \sum_{\alpha, \beta, \gamma} Q([e_\alpha,e_\beta],e_\gamma)^2}\,,
$$
 where the sum runs over $Q$-orthonormal basis vectors 
$e_\alpha \in \m_i$, $e_\beta \in \m_j$ and $e_\gamma \in \m_k$, as in \cite{WZ2}. 

\begin{lemma}\label{lem:strconst222}
Let  
$(\tf_1< \cdots <\tf_r)$ be
a toral flag, $A_- \in \D_-(\tf_r)$ and
\beg
  A \in {\rm conv}\{A^{\bldst\tf_1},...,A^{\bldst\tf_{r-1}}, A_-\}\bs \T\,.
\en
Let $f$ be a good decomposition for $A$, which is $\tf_r$-adapted.
Let $A_r$ denote the radial projection of $A_-$ to $\partial_- \D(\tf_r)$. 
Let $I_1:=I_1^{\bldst\tf_r}$, $I_2 \subset \{1,...,\ell\}\bs I_1$ be given
by $\tf_r \oplus \m_{I_2} =\ker (A_r)$ and let
    $I_3,...,I_{\ell'}$ correspond to the positive eigenspaces of $A_r$.
     Then either there exist $i_0,j_0 \in I_2$ and 
		$k_0 \in \{I_1 \cup I_2\}^\co$ (the complement),
        or  there exist $i_0 \in I_2$,
     $j_0\in I_p$ and $k_0 \in I_q$ with $3 \leq p,q \leq \ell'$, $p \neq q$,
(provided that $\ell'>3$),
     with
     $$
       [i_0j_0k_0]_f> 0 \,.
     $$
\end{lemma}

\begin{proof} Given $A_-$, we obtain $I_1$ and $I_2$. 
If  there exist $i_0,j_0 \in I_2$ and $k_0 \in \{I_1 \cup
I_2\}^\co$  with $[i_0j_0k_0]> 0$, we are done. Consequently we may assume that for all $i_0,j_0\in I_2$ and all
$k_0 \in \{I_1 \cup I_2\}^\co$,  we have $[i_0j_0k_0]= 0$,
in short, 
$[I_2I_2I_p]=0$ for all $p \geq 3$.
Then we claim that $\kf:=\tf_r \oplus \m_{I_2}$ is a subalgebra.
To see this, recall that since $\tf_r$ is a subalgebra, $[I_1I_1I_p]=0$ for any $p\geq 2$. 
Moreover, by Corollary \ref{cor:submkf}
$[I_1I_2I_p]=0$ for all $p \geq 3$. Thus $[I_1I_1I_p]=[I_1I_2I_p]=[I_2I_2I_p]=0$ for all $p\geq3$,
which shows that $\kf:=\h \oplus \m_{I_1} \oplus \m_{I_2}$ is a subalgebra. 

If $\ell'=3$ then $A_r =A^{\blds\kf}$, thus $A_r \in \D(\kf)$, contradicting $A_- \in \D_-(\tf_r)$.
Thus we may assume that $\ell'>3$. As above, if there exist $i_0 \in I_2$,
     $j_0\in I_p$ and $k_0 \in I_q$, $p\neq q$,  with $[i_0j_0k_0]> 0$, we are done. 
Suppose now
that $[I_2I_pI_q]=0$ for all $p,q \geq 3$, $p\neq q$. Again, by
 Corollary \ref{cor:submkf}
we know $[I_1I_pI_q]=0$ for all $2\leq p\neq q$. 
Thus $[\ad(\kf),A_r]=0$ and hence $A_r \in \D(\kf)$,  again contradicting $A_- \in \D_-(\tf_r)$.
This proves our lemma.
\end{proof}

%%%%%%%%%%%%%%%%%%%%%%%%%%%%%%%%%%%%%%%%%%%%%%%%%%%%
\subsection{{\it The second homotopy}}\label{sec:homotopy2}
%%%%%%%%%%%%%%%%%%%%%%%%%%%%%%%%%%%%%%%%%%%%%%%%%%%%%%

In this section we show that for compact homogeneous spaces $G/H$ for which
$\WWs\neq \emptyset$, there exists a strong deformation retraction
from $\WWs$ to the nerve $\XXGH$ of $G/H$: see Theorem \ref{thm:homotopy3}.
We again follow \cite{Gr}.  Recall that $\Sub_s$ denotes the set of all non-toral subalgebras $\kf$ with $\h < \kf < \g$, that $\Fln$ denotes the set of all flags $\varphi$ of non-toral subalgebras, and that the height of a flag $h(\varphi)$ is the dimension of $\max(\varphi)$.

\medskip

\begin{definition}\label{def:nerveseq}
Suppose that $\Sub_s \neq \emptyset$ and let $d\geq 1$ be defined by
$$
 \{i_1< \cdots <i_d\}:=\{ \dim \kf \mid \kf \in \Sub_s \}\,.
$$
Moreover, for $1 \leq j \leq d$ let
$$
  \XXs^j
	:= 
	\bigcup_{\varphi \in \Fln, ~h(\varphi) \geq  i_j} \Br[\varphi]
$$	
and 
$$
  \XXs^{d+1}:= \bigcup_{\varphi \in \Fln, ~h(\varphi) =\dim \bldsg\g} \Br[\varphi] = \qquad \bigcup_{\varphi \in \Fln} \Delta_{\varphi}\,.
$$
\end{definition}

We clearly have
$$
  \WWs = \XXs^1 \supset \XXs^2 \supset \cdots \supset \XXs^{d+1}=\XXGH\,. 
$$

\begin{lemma}\label{lem:XXsj}
For $j=1,\dots,d$ the set $\XXs^j$ is compact and semi-algebraic.
\end{lemma}

\begin{proof}
We have
\beg
 \XXs^j&=&
\big\{ A \in \WWs \mid  A=(1-\kappa) \sum_{i=1}^{r-1} \lambda_i \cdot \tfrac{P_i}{\tr P_i}+ \kappa A'
		 \textrm{ for some } P_1,...,P_r \in \Psn \\
  &&
	  \quad \quad \quad\quad\quad
		\textrm{s.t. }
	  \ker(P_1) \subsetneq \cdots \subsetneq \ker(P_r)\,,
		 \kappa, \lambda_i \geq 0\,, \sum_{i=1}^{r-1} \lambda_i =1, r \geq 1\\
	 &&
	  \quad \quad \quad\quad\quad	
		\textrm{ and } A' \in \D(\kf_r) \textrm{ with } \kf_r=\ker(P_r)
		\textrm{ and } \dim (\kf_r) \geq i_j
		 \big\}\,.
\en
This shows that $\XXs^j$ is semi-algebraic.

To see compactness of $\XXs^j$, 
let $(A_m)_{m \in \NN}$
be a sequence in $ \XXs^{j}$ converging to $A_\infty \in \WWs$ as $m \to \infty$. 
For each $m\in\NN$, we have $A_m \in \Br[\varphi_m]$ for $\varphi_m \in \Fln$ with $h(\varphi_m)\geq i_j$.
By passing to a subsequence we may assume that $h(\varphi_m)=i_{\tilde j}$ for all $m \in \NN$, $\tilde j\geq j$
and that $r_m=r$ for all $m\in \NN$.
By Lemma \ref{lem:flag-conv} we may assume that $\varphi_m$
converges to $\varphi_\infty \in \Fln$ as $m \to \infty$, with $h(\varphi_\infty)=i_{\tilde j}$.
Moreover, each $A_m$ has a representation  
$$
 A_m=(1-\kappa_m) \sum_{i=1}^{r-1} (\lambda_i)_m \cdot A^{\blds\kf_i^m}
+ \kappa_m A'_m \,,
$$
$A'_m \in \D(\kf_r^m)$. 
By passing to a subsequence 
we may assume that all the data there, i.e., the scalars $\kappa_m$ and $(\lambda_i)_m$, converge as well.
We set $\kf_\infty:=\lim_{m \to\infty} \kf_r^m $. By Lemma 
\ref{lem:Dkfi} we know $\lim_{m\to\infty} \D(\kf_r^m) \subset \D(\kf_\infty)$, 
and hence $A_\infty \in \Br[\varphi_\infty]$, proving $\XXs^j$ is compact.
\end{proof}

For every $A \in \WWs$  there exists $\varphi \in \Fln$ with
$A \in \Br[\varphi]$. (Note, $\Br[\g] = \emptyset$ tells us we will not have $\varphi = (\g)$.) 
If $\ti \varphi \in \Fln$ is another non-toral flag with $A \in \Br[\ti \varphi]$,
then by Lemma \ref{lem:butphipsi} we have $A \in \Br[\varphi \ti \varphi]$ and at least one of $\varphi, \ti\varphi < \varphi\ti \varphi$ since $\varphi \neq \tilde \varphi$. 
We note that larger flags correspond to the smaller butterflies.

\begin{lemma}\label{lem:smallest-butterfly} For every $A \in \WWs$  there exists a largest flag $\varphi_A \in \Fln$ with $A \in \Br[\varphi_A]$. More precisely, there is a flag $\varphi_A \in \Fln$ with the smallest butterfly containing $A$:
$$\Br[\varphi_A] := \bigcap_{\varphi \in \Fln,~A \in \Br[\varphi]} \Br[\varphi]\,.$$
\end{lemma}
This allows us to give the following definition:

\begin{definition}\label{def:heightA}
For each $A \in \WWs$ let $\varphi_A \in \Fln$ denote the largest flag with
$A \in \Br[\varphi_A]$. We set 
$$
  h(A):= h(\varphi_A)\,.
$$
\end{definition}

\begin{proof} Let $A \in \WWs$ and let $\varphi\in\Fln$,  $\varphi = (\kf_1 < \dots < \kf_r)$, $r \geq 1$,
be a flag with $A \in \Br[\varphi]$.
Since $A \in \Br[\varphi]$, we can write 
$$A = (1-\kappa)\sum_{i=1}^{r-1} \lambda_i \cdot A^{\blds\kf_i} + \kappa A' \,,$$ 
where $0 \leq \kappa \leq 1$, $0 \leq \lambda_i \leq 1$, $\sum_{i=1}^{r-1} \lambda_i=1$, and  $A' \in \D(\kf_r)$. 
If $\kappa = 1$ (so that $A=A'$), by removing all subalgebras in $\varphi$ except $\kf_r$, we construct a new flag, $\ti\varphi\in\Fln$ such that $A \in \Br[\ti\varphi]$ and $\ti\varphi \geq  \varphi$. 
 If $\kappa < 1$, then for each $i$ for which the coefficient $\lambda_i=0$, by removing the corresponding $\kf_i$ from our flag $\varphi$, we construct a new flag $\ti\varphi\in\Fln$ such that $A \in \Br[\ti\varphi]$ and $\ti\varphi \geq  \varphi$.

Suppose now that  $\tilde \varphi = (\kf_1 < \dots < \kf_r)$, $r \geq 1$, with $\tilde \varphi$ as above.
If $\ti\varphi$ is not the largest flag with $A \in\Br[\ti\varphi]$, then there exists  $\varphi_* \in \Fln$ with $A \in \Br[\varphi_*]$ and  $\varphi_* > \ti\varphi$. 

We will show now that $\max(\varphi_*)>\max(\tilde \varphi)$. Notice first that this is clear if $\kappa=1$.
Thus we may assume that $\kappa<1$ and that $\max(\varphi_*)=\max(\tilde \varphi)$
since by Lemma \ref{lem:ordering} we have $\max(\varphi_*)\geq \max(\tilde \varphi)$.
Using that $\ti\varphi<\varphi_*$ we deduce from Corollary \ref{cor:ordering}
that $ \varphi_*  \subsetneq \ti \varphi$. As a consequence, there exists
$i_0 \in \{1,...,r-1\}$ such that $\kf_{i_0}$ does not appear in $\ti \varphi$. Suppose that $i_0$ is
minimal with that property. Then, $A \in \Br[\tilde \varphi] \cap \Br[\varphi_*]$
and  $ \varphi_*  \subsetneq \ti \varphi $ yields 
two representations of $A$:
\beg A &= (1-\kappa)\sum_{i=1}^{r-1} \lambda_i \cdot A^{\blds\kf_i} + \kappa A' \\
~&= (1-\kappa^*)\sum_{i=1}^{i_0-1} \lambda_i^* \cdot A^{\blds\kf_i} + \kappa^* A'' \,,
\en
where $A' \in \D(\kf_r)$,
$A'' \in \Br[\psi]$ and $\psi \subset (\kf_{i_0+1}< \cdots <\kf_r)$. 
This would imply that the dimension of the eigenspace corresponding to the smallest
eigenvalue of $A\vert_{\kf_{i_0-1}^\perp}$ 
has dimension $\dim \kf_{i_0}-\dim \kf_{i_0-1}$ on the one hand, 
since $\kappa<1$ and all the $\lambda_i>0$,
but dimension $\dim \min(\psi) -\dim \kf_{i_0-1}$ on the other hand, a contradiction. 
This shows the above claim.

Since the flag $\varphi_*$ has a strictly greater maximum,  the above 
process will end in finitely many steps. 
Thus we obtain a largest flag whose butterfly contains $A$.
\end{proof} 

\begin{lemma}\label{lem:phisubset}
Let $A \in  \WWs$ and let $\varphi \in \Fln$ with $A \in \Br[\varphi]$ and $h(\varphi)\geq h(A)$.
Then $h(\varphi)=h(A)$, $\varphi_A \subset \varphi$ and $\max(\varphi_A)=\max(\varphi)$. 
\end{lemma}

\begin{proof}
Since $\varphi \leq \varphi_A$ by our choice of $\varphi_A$, we know by Lemma \ref{lem:ordering} that 
$h(\varphi) \leq h(\varphi_A)=h(A)$.
Thus $h(\varphi)=h(A)$.  By Corollary \ref{cor:ordering} we use  $\varphi \leq \varphi_A$ again to conclude that $\varphi_A \subset \varphi$. 
This proves our claim.
\end{proof}

This yields the following consequence:

\begin{proposition}\label{prop:kA}
Let $1\leq j \leq d$ and $A \in  \XXs^{j}\bs \XXs^{j+1}$.
Then $h(A)=i_j$. Moreover, for any $\varphi\in \Fln$ with $h(\varphi)\geq i_j$ and $A \in \Br[\varphi]$, 
we have $\varphi_A \subset \varphi$ and $\max(\varphi)=\max(\varphi_A)$.
\end{proposition}
                                                                                            
\begin{proof}	
Since $A \in  \XXs^{j}$ we know for some flag $\varphi \in \Fln$ with $h(\varphi)\geq i_j$, $A \in \Br[\varphi]$, and therefore, $h(A)\geq h(\varphi)\geq i_j$. But since  $A \not\in \XXs^{j+1}$, we know $h(A)< i_{j+1}$ . Thus $h(\varphi)=h(A)=i_j$, and the claim follows, by Lemma \ref{lem:phisubset}.
\end{proof}

This allows us to associate to each $A \in  \XXs^{j}\bs \XXs^{j+1}$ a non-toral subalgebra 
$$
  \kf_A=\max(\varphi_A)\,.
$$
 For any flag $\varphi=(\kf_1 < \cdots < \kf_r)$ in $\Fln$
with $h(\varphi)\geq i_j$
and $A \in \Br[\varphi] \cap  \XXs^{j}\bs \XXs^{j+1}$,  
we must have $\max(\varphi)=\kf_r=\kf_A$ by Proposition \ref{prop:kA}.

\begin{corollary}\label{cor:kA}
Let $1\leq j \leq d$ and $A \in  \XXs^{j}\bs \XXs^{j+1}$. Then
for any $\varphi \in \Fln$ with $h(\varphi) \geq i_j$ and $A \in \Br[\varphi]$ we have $A^{\blds\kf_A} \in \Br[\varphi]$.
\end{corollary}

Recall now that for $\varphi=(\kf_1 < \cdots < \kf_{r_\varphi} <\kf_A)$ with $r_\varphi\geq 1$, and for $A \in \Br[\varphi]$, we have
\begin{eqnarray}
  A=(1-\kappa)\cdot \sum_{i=1}^{r_\varphi} \lambda_i \cdot A^{\blds\kf_i}+\kappa \cdot A_{r_\varphi+1}\label{eqn:Arep} \,,
\end{eqnarray}
for some $A_{r_\varphi+1} \in \D(\kf_A)$ and $0\leq \kappa,\lambda_1,...,\lambda_{r_\varphi}\leq 1$,
$\sum_{i=1}^{r_\varphi} \lambda_i=1$. 
(By an argument analogous to the one in the proof of Lemma \ref{lem:esudescrA}, the representation of $A$ is essentially unique.) Thus, by Corollary \ref{cor:kA} the following projection map is well-defined.

\begin{definition}[Projection]\label{def:projection}
Let $1\leq j \leq d$. Then, we call the map
$$
  P_j:\XXs^{j}\bs \XXs^{j+1} \to \XXGH\,,
$$
defined below the {\em $j^{th}$ projection map}.

For $A \in   \Br[\varphi]\subset \XXs^{j}\bs \XXs^{j+1}$  
with $\varphi=(\kf_A)$, we set $P_j(A):= A^{\blds\k_A}$.

For $A \in   \Br[\varphi]\subset \XXs^{j}\bs \XXs^{j+1}$  
with $\varphi=(\kf_1 < \cdots < \kf_{r_\varphi} <\kf_A)\in \Fln$ and  $r_\varphi\geq 1$, we define $P_j$ by
$$
  A= (1-\kappa)\cdot \sum_{i=1}^{r_\varphi} \lambda_i \cdot A^{\blds\kf_i}+\kappa \cdot A_{r_\varphi+1}
\quad	\mapsto    \quad
	(1-\kappa)\cdot \sum_{i=1}^{r_\varphi} \lambda_i\cdot A^{\blds\kf_i} + \kappa \cdot A^{\blds\kf_A}\,,
$$
where $A_{r_\varphi+1} \in \D(\kf_A)$ and $0\leq \kappa,\lambda_1,\dots,\lambda_{r_\varphi}\leq 1$,
$\sum_{i=1}^{r_\varphi} \lambda_i=1$.  
\end{definition}
Since $A \not\in \XXs^{j+1}$ we know 
that we cannot express $A_{r_{\varphi +1}}$ as a convex combination of maps $A^{\blds\kf}$. By replacing $A_{r_{\varphi +1}}$ (the component of $A$ in $\D(\kf_A)$) with $A^{\blds \kf_A}$, the image $P_j(A)$ is in $\XXGH$.  
Notice that if there exists a $\varphi \in \Fln$ with 
 $A \in \Br[\varphi]\subset \XXs^j \bs \XXs^{j+1}$
 then $A^{\blds\kf_A}$ is not the only element in the disk $\D(\kf_A)$.

Recall, that a map $f:(X_1,d_1)\to (X_2,d_2)$ between two metric spaces is continuous
in a point $x_1 \in X_1$ if for all sequences $(a_i)_{i \in \N}$ in $X_1$ with $\lim_{i\to \infty} a_i=x_1$
there exists a subsequence $(a_{i_k})_{k \in \N}$ of  $(a_i)_{i \in \N}$
 with $\lim_{k\to \infty}f(a_{i_k})=f(x_1)$.

\begin{proposition}\label{prop:Pcont}
Let $1\leq j \leq d$. Then,
$P_j:\XXs^{j}\bs \XXs^{j+1} \to \XXGH$ is continuous.
\end{proposition}

\begin{proof} Fix $j \in\{1,\dots,d\}$. 
Let $A_\infty \in \XXs^{j}\bs \XXs^{j+1}$ and $\varphi_\infty \in \Fln$ with 
$A_\infty \in \Br[\varphi_\infty]$ and $h(\varphi_\infty)=i_j$: see Proposition \ref{prop:kA}.

 Let $(A_m)_{m \in \NN}$
be a sequence in $ \XXs^{j}\bs \XXs^{j+1}$ converging to $A_\infty$ as $m \to \infty$. 
For each $m\in\NN$, we have $A_m \in \Br[\varphi_m]$ for $\varphi_m \in \Fln$ with $h(\varphi_m)=i_j$: 
see Proposition \ref{prop:kA}. 
By Lemma \ref{lem:flag-conv} we may assume that $\varphi_m$
converges to $\tilde \varphi_\infty \in \Fln$ with $h(\tilde \varphi_\infty)=i_j$.
Moreover, each $A_m$ has a representation as in \eqref{eqn:Arep}. By passing to a subsequence 
we may assume that all the data there, i.e., the scalars $\kappa_m$ and $(\lambda_i)_m$, converge as well:
see proof of Lemma \ref{lem:XXsj}.
Note: since $A_\infty \in \XXs^{j}\bs \XXs^{j+1}$,  in the representation of $A_\infty$ as in Equation \ref{eqn:Arep}, the scalar $\kappa_\infty > 0$; hence we may assume for each $A_m$ the scalar $\kappa_m>0$.

Let $\kf=\max(\varphi_\infty)$ and $\tilde \kf=\max(\tilde \varphi_\infty)$.
Since $A_\infty \in (\Br[\varphi_\infty] \cap  \Br[\tilde \varphi_\infty])$ 
by Lemma \ref{lem:butphipsi} it follows that $A_\infty \in \Br[\varphi_\infty\tilde \varphi_\infty]$.
By Definition \ref{def:productflags}, this implies   
$\max(\varphi_\infty\tilde \varphi_\infty)=\langle\kf,\tilde \kf\rangle$, and since $A_\infty \not\in \XXs^{j+1}$, 
we conclude $\dim \langle \kf,\tilde \kf\rangle =\dim \kf$, thus
$\kf=\tilde \kf$. By Definition \ref{def:flag-conv}, we know  $\lim_{m \to \infty}\max(\varphi_m)=\kf$. 
By Corollary \ref{cor:kA}, this proves the claim.
\end{proof}

We are now in position to prove for each $j=1\dots,d$ that $\XXs^j$ and $\XXs^{j+1}$ are homotopy equivalent.
Recall that $\XXs^{j+1}$ is a compact subset of $\XXs^j$.

\begin{definition}[Pre-Homotopy]\label{def:Hjdef}
Let $1\leq j \leq d$ and let $\sigma:\XXs^j\to [0,1]$ be a continuous map with $\sigma^{-1}(0)=\XXs^{j+1}$.
Then we call the map
$$
  H_j(\sigma):[0,1] \times \XXs^{j} \to \XXs^j\,\,;\,\,\,(\kappa,A)\mapsto (1-\sigma(A)\cdot \kappa)\cdot A 
	+ \sigma(A)\cdot \kappa \cdot P_j(A)
$$
the {\em $j$-th pre-homotopy} associated to $\sigma$, where for $A \in \XXs^{j+1}$ we set $\sigma(A)\cdot P_j(A):=0$.
\end{definition}

Notice, that since $\sigma(A)=0$ for $A \in \XXs^{j+1}$ we have 
\begn
 H_j(\sigma)\vert_{\{t\}\times \XXs^{j+1}}
={\Id}\vert_{\{t\}\times \XXs^{j+1}} \label{eqn:Hjprop}
\enn 
for all $t\in [0,1]$.
(By Proposition \ref{prop:Pcont},  the map $P_j$ is defined and continuous on $\XXs^j\bs \XXs^{j+1}$.)

\begin{proposition}\label{prop:Hjcont}
Let $1\leq j \leq d$ and let $\sigma:\XXs^j\to [0,1]$ be a continuous map with $\sigma^{-1}(0)=\XXs^{j+1}$.
Then, the map $H_j(\sigma)$ is continuous.
\end{proposition}

\begin{proof}
It is enough to show that $\sigma \cdot P_j:\XXs^j\to \SymgH$ is continuous. 
By Proposition \ref{prop:Pcont}, the map $\sigma\cdot P_j$  restricted to 
$\XXs^j \bs \XXs^{j+1}$ is continuous.
Consider a sequence $(A_i)_{i \in \NN}$ in $\XXs^j \bs \XXs^{j+1}$ converging to $A_\infty \in \XXs^{j+1}$. 
Since $\XXGH$ is compact, $P_j$ is a bounded map. Thus the sequence $(P_j(A_i))_{i \in \N}$ is bounded. 
By hypothesis, $\sigma:\XXs^j\to [0,1]$ is continuous and  
$\sigma( \XXs^{j+1})=0$. Thus  $\lim_{i\to\infty} \sigma(A_i)\cdot P_j(A_i) =0$ and
the claim follows.
\end{proof}

Although we would like to extend the projection map $P_j$ to $\XXs^{j+1}$, this does not seem to be possible in general.

\begin{definition}\label{def:Pj}
Let $1\leq j \leq d$.
For $A \in \XXs^j$ let 
$$
  \PP_A^j :=\{ A^{\max(\varphi)} \mid  \varphi \in \Fln \textrm{ with } i_{j} \leq h(\varphi)< \dim \g
	  \textrm{ and }A\in \Br[\varphi]\} \subset \XXGH\,.  
$$	  

\end{definition}

Notice that by Proposition \ref{prop:kA},
for $A \in \XXs^j\bs \XXs^{j+1}$ we have
$\PP_A^j=\{A^{\blds\kf_A}\}$.

\begin{lemma}\label{lem:PAcomp}
Let $1\leq j \leq d$. Then for $A \in \XXs^j$ the set $\PP_A^j$ is compact.  
\end{lemma}

\begin{proof} Fix $j \in\{1,\dots,d\}$, and let $A \in \XXs^j$. 
Let $(A^{\blds\kf_i})_{i \in \NN}$ be a sequence in $\PP_A^j$. Thus there exists a sequence
of flags $(\varphi_i)_{i \in \NN}$ in $\Fln$ such that for each $i\in \NN$, $A\in \Br[\varphi_i]$, 
$\max(\varphi_i)=\kf_i$, and $i_{j}\leq h(\varphi_i)<\dim \g$. By Lemma \ref{lem:flag-conv}, passing to a subsequence, we may assume that $\varphi_i$ converges to $\varphi_\infty \in \Fln$ with
$i_{j}\leq h(\varphi_\infty)< \dim \g$ and $\kf_\infty := \max(\varphi_\infty)$. 
 By Lemma \ref{lem:flag-conv},  $\lim_{i \to\infty} \kf_i = \kf_\infty$, and thus by Lemma 
\ref{lem:Dkfi} we know $\lim_{i\to\infty} \D(\kf_i) \subset \D(\kf_\infty)$, and hence $A \in \Br[\varphi_\infty]$.
\end{proof}

\begin{lemma}\label{lem:Pjcomp}
Let $1\leq j \leq d$. The sets
\beg
    \PP^{j} &:=&
		 \{(A,P) \mid A \in \XXs^{j} \textrm{ and }P\in \PP_A^j\} \subset \XXs^{j}\times \XXGH \,,\\
		 \PP^{j}_\ast &:=&
		 \{(A,P) \mid A \in \XXs^{j+1} \textrm{ and }P\in \PP_A^j\} \subset \XXs^{j+1}\times \XXGH
\en
are compact.
\end{lemma}

\begin{proof} We note that $\PP^{j}_\ast \subset \PP^j$. We will prove compactness of $\PP^j$. The proof of compactness of $\PP^{j}_\ast$ follows similarly. 
Let $((A_i,P_i))_{i \in \NN}$ be a sequence in $\PP^{j}$. Since $\XXs^{j}$ and $\XXGH$ are each 
compact, we may assume (by passing to a subsequence) that $\lim_{i\to \infty}A_i =A_\infty \in \XXs^{j}$, 
and $\lim_{i \to \infty}P_i=A^{\blds\kf_\infty}$.

Now let $(\varphi_i)_{i\in \NN}$ be a sequence in $\Fln$ such that for each $i \in \NN$, we have $A_i \in \Br[\varphi_i]$, $i_{j}\leq h(\varphi_i)< \dim \g$, and $P_i=A^{\max(\varphi_i)}$. By Lemma \ref{lem:flag-conv}, we may assume (passing to a subsequence) that 
$\lim_{i\to \infty}\varphi_i =\varphi_\infty \in \Fln$
with $i_{j}\leq h(\varphi_\infty)<\dim\g$ and $\kf_\infty=\max(\varphi_\infty)$. 
As in Lemma \ref{lem:PAcomp} it follows that
$A_\infty \in \Br[\varphi_\infty]$, proving our claim.
\end{proof}

For each $1\leq j \leq d$, we define the distance $d_{max}$ on the product space $\XXs^{j}\times \XXGH$ using the Euclidean distances $d$ on each component. 
Given points $(A,P), (\tilde A,\tilde P) \in \XXs^{j}\times \XXGH$, let
$$
  d_{max}((A,P),(\tilde A,\tilde P)):=\max(d(A,\ti A),d(P,\tilde P))\,.
$$

\begin{definition}\label{def:sigmaj}
Let $1\leq j \leq d$. Then we set
$$
  \sigma_j:\XXs^j\bs \XXs^{j+1}\to [0,\infty)\,\,;\,\,\,
	  A \mapsto d_{max}((A,A^{\bldss\kf_A}),\PP^{j}_\ast)\,.
$$
\end{definition}

Notice that for $A \in \XXs^j\bs \XXs^{j+1}$ we have
$ 0<d(A,\XXs^{j+1}) \leq d_{max}((A,A^{\bldss\kf_A}), \PP^{j}_\ast)$.

\begin{lemma}\label{lem:sigmaj}
Let $1\leq j \leq d$. Then we can extend $\sigma_j$ to a continuous map
$\sigma_j:\XXs^j\to [0,\infty)$ with $\sigma_j^{-1}(0)=\XXs^{j+1}$.
\end{lemma}

\begin{proof}
Recall that $\XXs^{j+1}$ is a compact subset of $\XXs^{j}$. We set 
$\sigma_j(A)=0$ for $A \in \XXs^{j+1}$.
Let $(A_i)_{i \in \NN}$ be a sequence in $\XXs^j\bs \XXs^{j+1}$ with
$\lim_{i\to \infty}A_i=A_\infty \in \XXs^{j+1}$. We need to show that $\lim_{i \to \infty}\sigma_j(A_i)=0$.
Since by Lemma \ref{lem:Pjcomp} the set $\PP^{j}$ is compact, by passing to a subsequence
we may assume that
$(A_i,P_i)_{i\in \NN}$ converges to $(\tilde A_\infty,P_\infty) \in \PP^j$, where $P_i =A^{\blds\kf_{A_i}}$ for all $i \in \NN$.
Since we know that $A_\infty=\tilde A_\infty \in \XXs^{j+1}$ we deduce
$(A_\infty,P_\infty) \in \PP^{j}_\ast$. This shows the claim.
\end{proof}

We now arrive at the main result of this section.

\begin{theorem}\label{thm:homotopy3}
Let $G/H$ be a compact homogeneous space. Then $\WWs$ and the nerve $\XXGH$ of $G/H$
are homotopy equivalent.
\end{theorem}

\begin{proof}
It is enough to show that for each $1 \leq j \leq d$,
$\XXs^{j}$ and $\XXs^{j+1}$ are homotopy equivalent.

Let $1 \leq j \leq d$.
Since $\XXs^{j}$ and $\XXs^{j+1}$ are compact and semi-algebraic 
with $\XXs^{j+1}\subset \XXs^{j}$, the claim will follow from Lemma \ref{lem:XY},
once we show that there exists a $\delta_0 >0$ such that 
for all $\delta \in (0,\delta_0)$ 
there is a continuous map
$$
  H_j^\delta:[0,1]\times \XXs^j \to \XXs^j
$$
with the following properties: 
\begin{align} H_j^\delta(0,A)=A \quad \textrm{  for all }& A \in \XXs^j \,,\notag \\
H_j^\delta(t,A)=A \quad \textrm{  for all }& A \in \XXs^{j+1} \textrm{ and all }  t \in[0,1] \,,\notag \\
  d(H_j^\delta(1,A),\XXs^{j+1}) < \delta  \textrm{ for all }& A \in \XXs^{j}\,.\label{eqn:Hjest}
\end{align}
Using the continuous map $\sigma_j:\XXs^j \to [0,\infty)$ in Definition \ref{def:sigmaj} and
Lemma \ref{lem:sigmaj}, 
we set 
$$
  \sigma^\delta_j:\XXs^{j}\to [0,1]\,\,;\,\,\,A \mapsto \min(1,\tfrac{1}{\delta}\cdot \sigma_j(A))
$$
and
$$
  H_j^\delta:=H_j(\sigma_j^\delta):[0,1]\times \XXs^j \to \XXs^j\,,
$$
as in Definition \ref{def:Hjdef}. 
Notice that since $\sigma_j$ is continuous so is $\sigma^\delta_j$ and
since $\sigma_j^{-1}(0)=\XXs^{j+1}$ we have $(\sigma_j^\delta)^{-1}(0)=\XXs^{j+1}$. Thus
by Proposition \ref{prop:Hjcont}, $H_j^\delta$ is continuous. Moreover, by \eqref{eqn:Hjprop}
we know the first two of the required properties of the maps $H_j^\delta$ are fulfilled. To see that the third property also holds, 
let $A  \in \XXs^j\bs \XXs^{j+1}$. If $\sigma^\delta_j(A)=1$ then $(H_j^\delta)(1,A)=A^{\blds\kf_A}\in \XXGH \subset \XXs^{j+1}$.
So now suppose that $\sigma^\delta_j(A)<1$.
Then by definition of $\sigma_j$, we have $d_{max}((A,A^{\blds\kf_A}),\PP^{j}_\ast)< \delta$. Thus there exists some
$(B,P) \in \PP^{j}_\ast$ with $d_{max}((A,A^{\blds\kf_A}),(B,P))<\delta$. Recall that $B \in \XXs^{j+1}$. By the definition
of $d_{max}$ we conclude that $d(A,B)<\delta$ and $d(A^{\blds\kf_A},P)<\delta$.

By our definition
of $H_j^\delta$ the point $(H_j^\delta)(1,A)$ lies on the line segment $[A,A^{\blds\kf_A}] \subset \XXs^{j}$. 
Furthermore, since $(B,P)\in \PP^{j}_\ast$, we know there exists a flag $\varphi \in \Fln$
such that $i_j\leq h(\varphi)<\dim \g$, $B \in \Br[\varphi]$, and $A^{\max(\varphi)}=P$. By choosing $\delta>0$ sufficiently small,
it follows from $d(A^{\blds\kf_A},P)<\delta$ and $\dim \kf_A=i_j$ that $\dim \max(\varphi)=i_j$ as well.

Since $B \in \XXs^{j+1}$ there exists a flag $\tilde \varphi$ such that $i_{j+1} \leq h(\tilde \varphi)$
and $B \in \Br[\tilde \varphi]$. By Lemma \ref{lem:butphipsi} we have
$B \in \Br[\varphi]\cap\Br[\ti\varphi]= \Br[\varphi \tilde \varphi]$, and by Lemma \ref{lem:flag-product}, $\varphi,\tilde \varphi \leq \varphi \tilde \varphi$.
In particular $\max(\varphi) < \max(\varphi\tilde \varphi)$. Moreover, since $\max(\tilde \varphi) < \max(\varphi)$
is not possible, by Definition \ref{def:productflags} there exist
$\varphi' \subset \varphi $ and $\varphi'' \subset \tilde \varphi$ with $\min(\varphi'')>\max(\varphi)$, such that
$\varphi\tilde \varphi =(\varphi' < \varphi'' < \max(\varphi \tilde \varphi))$. Notice that $\varphi'$ or $\varphi''$
might be the empty set. But since $\max(\varphi')\leq \max(\varphi)$ and $\max(\varphi)< \min(\varphi'')$
we have
$$
 \varphi\tilde \varphi \subset \varphi_\ast:=(\varphi' \leq  \max(\varphi) < \varphi'' < \max(\varphi \tilde \varphi))\,.
$$
By the definition of butterflies we know 
$B \in  \Br[\varphi\tilde \varphi] \subset \Br[\varphi_\ast] \subset \XXs^{j+1}$,
since  $\varphi\tilde \varphi \subset \varphi_\ast$ and $\max(\varphi\tilde \varphi)=\max(\varphi_\ast)$, 
and $P=A^{\max(\varphi)}\in \Br[\varphi_\ast]$.
Thus, we deduce that for the line segment
$[P,B] \subset \Br[\varphi_\ast] \subset \XXs^{j+1}$, since butterflies are convex.

Now we use the elementary fact that the maximal distance between the two line segments $[A,A^{\blds\kf_A}]$
and $[B,P]$ (in some Euclidean space) 
is one of $d(A,B)$ or  $d(A^{\blds\kf_A},P)$. Since both $d(A,B),d(A^{\blds\kf_A},P)<\delta$, 
we conclude $d(H_j^\delta(1,A),\XXs^{j+1}) < \delta$. This shows the claim.
\end{proof}

\begin{remark}
We mention that the simpler function $\sigma_j(A):=d(A,\XXs^{j+1})$ does not suffice.
\end{remark}

%%%%%%%%%%%%%%%%%%%%%%%%%%%%%%%%%%%%%%%%%%
\section{Scalar curvature functional} 

\subsection{{\it The scalar curvature}}\label{sec:scal}

%%%%%%%%%%%%%%%%%%%%%%%%%%%%%%%%%%%%%%%%%%%%%%%%%%%%%%%%%%%%%%%%%%%%%%

In this section we will provide a well known formula for the scalar curvature
functional along the curve $\gamma_v(t)=\exp(t\cdot v) \in \MGo$,
$v \in T_{\Id} \MGo$, for compact homogeneous spaces $G/H$ which are not isotropy irreducible: see
Remark \ref{rem:isotirr}.  

Recall that $\m$ is the ${\rm Ad}(H)$-invariant, $Q$-orthogonal complement to $\h$ in
$\g$. The set of $G$-invariant metrics on $G/H$ is $\MG$, and $\MGo$ is the space of $G$-invariant metrics of volume one. Furthermore, for every $g \in \MGo$, we identify $P_g \in \MGo$, an $\Ad(H)$-invariant scalar product, satisfying 
$g(\cdot,\cdot) = Q(P_g\cdot,\cdot)$.  Thus, $T_{\Id} \MGo = \{  v \in \SymmH\mid \tr v=0\}\,$. 
Here our setting is $\MGo$ and the unit tangent sphere $\Si$ in $T_{\Id} \MGo$, not in $\Sphb$, the image under the Graev map (see Remark \ref{rem:Sphbmodel2}).
  
\medskip

\begin{lemma}\label{lem:good-decomp}
For any $v \in T_{\Id} \MG$ there exists a $Q$-orthogonal decomposition
$  f=\oplus_{i=1}^\ell \m_i$ 
of $\m$ into irreducible, $\Ad(H)$-invariant summands
$\m_i$ and 
 $v_i \in \RR$ for $i=1,...,\ell$,  such that
\beg
   v = v_1 \cdot \Id_{{\m}_1}   + \cdots +
             v_{\ell} \cdot \Id_ {{\m}_\ell}.
\en
Any such decomposition $f$ will be called a {\em good decomposition} with respect to $v$. 
\end{lemma}

\begin{proof}
The eigenspaces of $v$ are $\Ad(H)$-invariant and pairwise $Q$-orthogonal.
Decomposing an eigenspace further into $Q$-orthogonal $\Ad(H)$-irreducible
summands shows the claim.
\end{proof}

 By  Lemma \ref{lem:expdiff}, for each $P_g \in \MGo$ 
there exists $v \in \Si \subset T_{\Id} \MGo$ and $t_0\geq 0$ such that $P_g=\gamma_v(t_0)$. 
By the above lemma,
\begn
  \gamma_v (t)=e^{tv_1}\cdot \Id_{{\m}_1}  + \cdots +
             e^{tv_{\ell}}\cdot  \Id_{\m_\ell}\label{gammav}\,.
\enn 
 For each $1 \leq i \leq \ell$, we set
$$
    d_i := \dim \m_i
$$
and we have
\begn
   \Vert v \Vert^2=\tr (v^2)=\sum_{i=1}^{\ell} d_i 
   \cdot \, v_i^2\,
	\quad \textrm{ and }\quad
	0=\tr v= \sum_{i=1}^{\ell} d_i 
   \cdot \, v_i\,.
	\label{l2metric}
\enn
Let $v \in T_{\Id}\MG$ and let $f$ be a good decomposition  with respect to
$v$. We denote by 
$$
  \hat v_1<\cdots <\hat v_{\ell_v}
	$$ 
the distinct eigenvalues of $v$ ordered by size,
$1 \leq \ell_v\leq  \ell$. For each eigenvalue  $\hat v_m$, $1 \leq m \leq \ell_v$, we define the index set (which depends on our choice of $f$) 	
\beg
 I_m^v := I_m^v(f) :=\{i \in\{1,...,\ell\} \mid v_i=\hat v_m\}\,.
\en 
Let
$$
 \lambda(v)=\hat v_1 \quad \textrm{ and }\quad \Lambda(v)=\hat v_{\ell_v}
$$ 
denote the smallest and the largest eigenvalues of $v$, respectively.
The proof of the following Lemma follows immediately from \eqref{l2metric}.

\begin{lemma}\label{lem:cGH}
Let $G/H$ be a compact homogeneous space. 
Any $v\in \Si \subset T_{\Id} \MGo$ must have at least $2$ distinct eigenvalues, and
there exists a constant $c_{G/H}<0$ such that the following holds: 
\beg
   \lambda(v) \leq  c_{G/H}
	\quad \textrm{and} \quad
   -c_{G/H} \leq \Lambda(v)\,.
\en
\end{lemma}

Let $f=\oplus_{i=1}^\ell \m_i$ be a decomposition of $\m$ into $Q$-orthogonal, $\Ad(H)$-irreducible summands. 
In what follows, we will simply say that $f$ {\em is a decomposition of $\m$}. 
Given such a decomposition $f$ of $\m$,
for any non-empty subset $I$ of $\{1,...,\ell\}$ we set
\beg
  \m _I:=\oplus_{i\in I}\m_i \quad \textrm{and} \quad
   d_I:= \dim \m_I\,.
\en
For $v \in T_{\Id}\MGo$ and a good decomposition $f$ with respect to $v$, 
the spaces $\m_{I_m^v}$ are the eigenspaces of $v$, thus we have
\beg
   \gamma_v (t)=e^{t\hat v_1}\cdot \Id_{{\m}_{I^v_1}}
                    + \cdots +
                e^{t\hat v_{\ell_v}} \cdot \Id_{{\m}_{I^v_{\ell_v}}}\,.
\en

We are ready to describe the formula for the scalar curvature of homogeneous metrics \cite{WZ2}.
Let $f$ be a fixed decomposition of $\m$ and let $\{e_1,\dots, e_n\}$ be a $Q$-orthonormal basis of $\m$ adapted to $f$.
Let $I,J,K$ be non-empty subsets of $\{1,...,\ell\}$. Following
\cite{WZ2}, we define
\beg
  [IJK] =[IJK]_f:=\sum_{\alpha,\beta,\gamma} Q([e_\alpha,e_\beta],e_\gamma)^2\,,
\en
where we sum over all indices $\alpha, \beta ,\gamma \in
\{1,...,n\}$ with $e_\alpha \in \m_I$, $e_\beta \in
\m_J$ and $e_\gamma \in \m_K$. Since $Q$ is $\Ad(G)$-invariant, we know the adjoint maps 
$$
  \ad(X):\g \to \g\,\,;\,\,\,
Y \mapsto [X,Y]
$$ 
are $Q$-skew-adjoint (by differentiation).  
That is, for all $X,Y,Z \in \g$ we have
\begn  \label{eqn:adXscew}
    Q([X,Y],Z) &=& Q(X,[Y,Z])\,.
\enn
It follows that $[IJK]$ is symmetric
in all three entries and is independent of the $Q$-orthonormal
bases chosen for $\m_I$, $\m_J$ and $\m_K$. In  case
$I=\{i\}$, $J=\{j\}$ and $K=\{k\}$ we write $[ijk]$. 

\begin{remark}
We have $[ijk]\geq 0$, with $[ijk]=0$ if and only if
$Q([\m_i,\m_j],\m_k)=0$. Furthermore, for any $[ijk]$, on the space of all ordered
decompositions of $\m$, the function $f\mapsto [ijk]_f$ is 
continuous.
% Added, 2023
The topology on the space of decompositions is defined as follows:
A sequence $(f_i)_{i \in \NN}$ with
 $f_i:=\m_1^i \oplus \dots \oplus \m_\ell^i $, $i \in \NN$, converges to
 $f_\infty:=\m_1^\infty \oplus  \dots \oplus \m_\ell^\infty$  (where both are $Q$-orthogonal  decompositions into irreducible summands) if, for all $j=1, \dots, \ell$, as $i \to \infty$, the subspaces $\m_j^i $ converge to $ \m_j^\infty$   in the Grassmannian. 
\end{remark}

Let $B$ denote the negative of the Killing form on $\g$. For
all $X,Y \in \g$ we have 
$$
 B(X,Y)=-\tr _{\g} (\ad (X)\circ \ad (Y))\,.
$$
Since both $Q$ and the Killing form are $\Ad(G)$-invariant, by \eqref{eqn:adXscew} we have
$$
 B\vert_{\m_i}= b_i\cdot Q\vert_{\m_i}
$$
for $1 \leq i \leq \ell$ with $b_i \geq 0$.  
Notice $B(X,X)=0$ if and only if $X\in \z(\g)$, the center of $\g$.

\begin{lemma}[\cite{WZ2}]
Let $v\in T_{\Id}\MG$ and let $f$ be a good decomposition with respect to $v$.
Then the scalar curvature of $\gamma_v(t)$ is given by
\begn
 \quad \quad
 \sc (\gamma_v (t))
   &=&
	\tfrac{1}{2} \sum_{i=1}^\ell 
 d_i b_i \cdot e^{t(-v_i)}   
      -\tfrac{1}{4} \sum_{i,j,k=1}^\ell
      [ i j k] \cdot e^{t(v_i-v_j-v_k)}\label{scal} \\
  &=& 
	\tfrac{1}{2} \sum_{i=1}^{\ell_v}\Big( \sum_{j\in I^v_i} d_j b_j
       \Big) \cdot e^{t(-\hat v_i)}\label{scalI}
			-\tfrac{1}{4} \sum_{i,j,k=1}^{\ell_v}
      [ I^v_i I^v_j I^v_k] \cdot
      e^{t(\hat v_i-\hat v_j-\hat v_k)} \,.
\enn
\end{lemma}

\begin{remark}
Notice that both the numbers $b_i$ and $[ijk]$ may depend on the choice of good decomposition $f$ with respect to $v$.
However, for ease of notation we will suppress this whenever it is clear that our decomposition $f$
is chosen. By contrast, the numbers $ \sum_{j\in I^v_i} d_j b_j$ and
$[ I^v_i I^v_j I^v_k]$ in (\ref{scalI}) are independent
of the choice of a good decomposition $f$. 
\end{remark}

%%%%%%%%%%%%%%%%%%%%%%%%%%%%%%%%%%%%%%%%%%%%%%%%%%%%%%%%%%%%%%%%%%%%%%%%%%%
\subsection{{\it Canonical directions}}\label{sec:candir}

%%%%%%%%%%%%%%%%%%%%%%%%%%%%%%%%%%%%%%%%%%%%%%%%%%%%%%%%%%%%%%%%%%%%%%%%%%%

We assign to each subalgebra $\kf\in \Sub$ a {\em canonical direction} $\vk$  in the unit sphere $\Si$ of $T_{\Id}\MGo$
and we investigate the scalar curvature functional along the curve 
$\gamma_{\vk}(t)$ emanating from $Q$, provided that $G/H$ is not isotropy irreducible.
If $\kf$ is a compact subalgebra and $K$ is the connected subgroup of $G$ with Lie algebra $\kf$, 
then the canonical direction $\vk$ comes from the canonical variation of $Q$ with respect to the
fibration $K/H \to G/H \to G/K$  \cite[9.67]{Bes}.

\medskip
In Sections \ref{sec:Ltbasics} and \ref{sec:LieII} the reader can find some well known facts about the geometry of homogeneous spaces.  Some specific results are referenced in what follows. 

As in Definition \ref{def:toral-non-toral}, we write $\m_{\blds\kf}:= \kf \cap \m$ for a subalgebra $\kf \in \Sub$, let 
 and we let $\m_{\blds\kf}^{\perp}$ denote the
$Q$-orthogonal complement of $\m_{\blds\kf}$ in $\m$.

\begin{definition}[Canonical directions] \label{def:candir}
For each subalgebra $\kf \in \Sub$, the {\em canonical direction}
$\vk\in \Si$ associated to $\kf$ is defined by
\beg
   \vk:=
	         \vk_1\cdot \Id_{\m_{\blds\kf}} +  
	         \vk_2\cdot \Id_{\m_{\blds\kf}^\perp}\,,
\en
where $(\dim \m_{\blds\kf}) \cdot \vk_1+(\dim \m_{\blds\kf}^\perp) \cdot \vk_2=0$,
$\Vert \vk \Vert=1$
and $\vk_1< 0 < \vk_2$.
\end{definition}

Given a decomposition $f=\oplus_{i=1}^\ell \m_i$ of $\m$ and given a subalgebra $\kf \in \Sub$, 
we say $\kf$ is  {\it $f$-adapted} if there exists
 $I_1:=I_1^{\blds\kf}\subset \{1,...,\ell\}$ with $\kf=\h \oplus \m_{I_1}$. 
 Notice that for every subalgebra $\kf$ there
exists a decomposition $f$ of $\m$ such that $\kf$ is $f$-adapted.
 Let $I_2=\{1,...,\ell\} \backslash I_1$.
Then, 
since $[I_1 I_1 I_2]=0$ (because $[\m_{I_1}, \m_{I_1}] \subset \kf \perp \m_{I_2}$),
by Lemma \ref{lem:1p5} we conclude that
\beg
  \sum_{j\in I_1} d_j b_j
     &=&[I_1 I_1 I_1]+
    [I_1 I_2 I_2]
     + \sum_{j\in I_1}2d_jc_j\\
		 \sum_{j\in I_2} d_j   b_j
   &=&[I_2 I_2 I_2]
    +2[I_1 I_2I_2] + \sum_{j\in I_2}2d_jc_j\,.
\en
(Here the constants $c_i$ are as in Section \ref{sec:LieII}.) %%%% added 2023
Now let
\beg
  a_1
	  &:=& 
		 \tfrac{1}{2} \big(\sum_{j\in I_1} d_j  b_j
        -\tfrac{1}{2} [ I_1 I_1 I_1]
        - [ I_1 I_2 I_2] \big)
				=
				 \sum_{j\in I_1} d_j  c_j
        +\tfrac{1}{4} [ I_1 I_1 I_1] \geq 0\\
				 a_2
	  &:=& 
		 \tfrac{1}{2} \big(\sum_{j\in I_2} d_j  b_j
        -\tfrac{1}{2} [ I_2 I_2 I_2] \big)
		=
				 \sum_{j\in I_2} d_j  c_j
        +\tfrac{1}{4} [ I_2 I_2 I_2]
        +[I_1 I_2I_2] \geq 0\,.
\en
Then, by (\ref{scalI}) we obtain
\begn  \label{eqn:scalcank}
  \sc (\gamma_v (t)) 
	= 
	\sum_{m=1}^{2} a_m \cdot e^{t(-\vk_m)}
    -\tfrac{1}{4}[I_1 I_2I_2] \cdot e^{t(\vk_1 -2\vk_2)}\,.
\enn
Since $\vk_1-2\vk_2 < -\vk_2$ and $a_2-\tfrac{1}{4}[I_1 I_2I_2]\geq 0$ we obtain 
\begn  \label{scanvar}
  \sc (\gamma_v (t))
	&\geq &
  a_1 \cdot e^{t(-\vk_1)} \,.
\enn

Notice that for a toral subalgebra $\tf$ we have $\tf =\h \oplus \af$
with $\af \subset \m_0$, where $\m_0$ is the isotypical summand of
$\m$ defined by $[\h,\m_0]=0$. (See  Definition \ref{def:m0} and  Lemma \ref{lemkmm}.)
As a consequence, for any $\Ad(H)$-invariant subspace $\ti \m$ of $\m \ominus \m_0$
 we have $0 \neq [\h, \tilde \m]\subset \ti \m$, thus by \eqref{eqn:adXscew} we conclude
$[\ti \m,\ti \m]\vert_{\h}\neq 0$. This implies
that in order for an $\Ad(H)$-invariant subspace in $\m$ 
to be a subalgebra of $\g$, it must lie in $\m_0$.  

Now we can state a first, well known result towards the asymptotic behavior of
the scalar curvature functional.

\begin{lemma}\label{lem:scalpos}
For any subalgebra $\kf \in \Sub$, 
\beg
  \lim_{t \to +\infty}\sc (\gamma_{\vk}(t))=
    \Big\{ \begin{array}{c}+\infty \\[-0.1cm]  0 \end{array} 
          \quad \Longleftrightarrow \quad
  \Big\{  \begin{array}{c} \kf\,\,\textrm{non-toral} \\[-0.1cm] \kf\,\,\textrm{toral}\,.  \end{array} 
\en
If in addition $G/H$ is not a torus, then
$\sc (\gamma_{\vk}(t))>0$ for all $t\geq 0$.
\end{lemma}

\begin{proof} If $\kf$ is a non-toral subalgebra, then 
 by Lemma \ref{lem:toral} we have $a_1>0$. By \eqref{scanvar}, since $-\vk_1 >0$,
we deduce
$\sc (\gamma_{\vk}(t))>0$ for all $t\geq 0$ and
 ${\displaystyle \lim_{t \to +\infty}\sc (\gamma_{\vk}(t))=\infty}$.

If $\kf$ is a toral subalgebra then $a_1=0$  by Lemma \ref{lem:toral}.
Thus,
${\displaystyle \lim_{t \to +\infty}\sc (\gamma_{\vk}(t))=0}$ by \eqref{eqn:scalcank}.
Moreover,  in this case for all $t \geq 0$, we have
\beg
  \sc (\gamma_v (t)) 
	&= &
	a_2 \cdot e^{t(-\vk_2)}
    -\tfrac{1}{4}[I_1 I_2I_2] \cdot e^{t(\vk_1-2\vk_2)} \\
	& \geq &	
   e^{t(-\vk_2)} \cdot \big(  \sum_{j\in I_2} d_j c_j
        +\tfrac{1}{4} [ I_2 I_2 I_2]
        +\tfrac{3}{4}[I_1 I_2I_2]  \big)\,.
\en
The term on the right-hand side is zero if and only if 
$\m$ is an abelian subalgebra of
$\g$ commuting with $\h$. (See Lemma \ref{lem:toral}.) Since we assume $G/H$ is almost effective, we deduce $\m$ is abelian if and only if $\h=\{0\}$ and $\g$ is abelian. As a consequence $G/H$ admits a flat metric. Since  the only compact homogeneous space admitting a flat metric is the torus \cite[7.61]{Bes},
the claim follows.
\end{proof}

%%%%%%%%%%%%%%%%%%%%%%%%%%%%%%%%%%%%%%%%%%%%%%%%%%%%%%%%%%%%%%%%%%%%%%%

\subsection{{\it Scalar curvature estimates from below}}\label{sec:scalestbe}
Recall that 
$\MGo$ is the space of $G$-invariant metrics on $G/H$ of volume one. Furthermore, for every $g \in \MGo$, 
we identify $P_g \in \MGo$, an $\Ad(H)$-invariant scalar product, satisfying 
$g(\cdot,\cdot) = Q(P_g\cdot,\cdot)$, and $T_{\Id} \MGo := \{  v \in \SymmH\mid \tr v=0\}$.

In this section we investigate flags of non-toral
subal\-gebras and directions $\tilde v \in T_{\Id}\MGo\bs\{0\}$ which are convex combinations of
the corresponding canonical directions. Finally, we
provide scalar curvature estimates along $\gamma_v(t)$, 
$v=\frac{\tilde v}{\Vert \ti v\Vert}$.

\medskip

In the following we will identify 
\beg
  v =v_1\cdot \Id_{{\m}_1}  + \cdots +
  v_{\ell} \cdot \Id_ {{\m}_\ell} \in T_{\Id} \MG
\en
with $(v_1,...,v_\ell) \in \RR ^\ell$. Recall that
$v \in T_{\Id} \MGo$ if and only if  $\sum_{i=1}^\ell
d_i v_i=0$. 

As in Section \ref{sec:kfdisc}, we write $\kf_1
< \kf_2$ to indicate $\kf_1\subsetneq \kf_2$. 
As in Section \ref{sec:butterflies}, we write ${\rm conv}(X,Y)$ to denote the convex hull of $X$
and $Y$ in $T_{\Id}\MGo$. 

\begin{lemma} \label{lem:linin}%%%%%% revised 2023
Let $\kf_1, ...,\kf_p$ be  subalge\-bras, $\kf_i \in \Sub$, with
$\kf_1< ... <\kf_p$. Then, the set of canonical directions
$(\vk_1, \dots ,\vk_p)$ form 
a linearly independent set. Furthermore, for any $\ti v \in {\rm conv}\{v^{\blds\kf_1}, \dots, v^{\blds\kf_p}\}$, 
$\ti v = \sum_{i=1}^p \lambda_i v^{\blds\kf_i}$, which we view as an element in $
(\tilde v_1,...,\tilde v_{p+1}) \in \RR^{p+1}$, 
we have  $\tilde v_1 \leq \cdots \leq \tilde v_{p+1}$, with strict inequality $\ti v_{i_0} < \ti v_{i_0+1}$ whenever $\lambda_{i_0}>0$.
%where $\lambda_i$ is the smallest eigenvalue of $v_i$.
\end{lemma}

\begin{proof} Let $f$ be a decomposition of $\m$, such that the
subalgebras $\kf_1, \dots,\kf_p$ are \mbox{$f$-adapted}. In
order to make notation as simple as possible let us assume that
$\kf_i
=\h \oplus \bigoplus_{m=1}^{j_i}\m_m$ with $1\leq j_1 <...< j_p <
\ell$. Hence \beg
  v^{\blds\kf_i} =(\underbrace{v^{\blds\kf_i}_1,\dots, v^{\blds\kf_i}_1}_{j_i},
            v^{\blds\kf_i}_2, \dots, v^{\blds\kf_i}_2)\in \RR ^\ell
\en
and for each $v^{\blds\kf_i}$ we have $v^{\blds\kf_i}_1<0$ and $v^{\blds\kf_i}_2>0$, $1 \leq i \leq p$.
Suppose that $(v^{\blds\kf_1},\dots , v^{\blds\kf_{i-1}})$ are linearly independent for some $2 \leq i \leq p-1$.
Then the restriction of any non-trivial linear combination of these vectors to $\kf_{i-1}^\perp$
has only one eigenvalue. But the restriction of $v^{\blds\kf_i}$ to $\kf_{i-1}^\perp$ has two different 
eigenvalues. This shows the first claim.

Let 
${\dsp \tilde v  = \sum_{i=1}^p \lambda _i  \cdot v^{\blds\kf_i}}$ be an element of the convex hull,
${\rm conv}\{v^{\blds\kf_1}, \dots, v^{\blds\kf_p}\}$, so that 
 $0 \leq \lambda_1,\dots, \lambda_p \leq 1$ and $\sum_{i=1}^p \lambda_i=1$.
By the special shape of $v^{\blds\kf_i}$ we can view each $v^{\blds\kf_i}$
as an element in $\RR ^{p+1}$ (rather than in $\RR ^\ell$); that is,
\beg
  v^{\blds\kf_i} =(\underbrace{v^{\blds\kf_i}_1, \dots , v^{\blds\kf_i}_1}_i,
                 v^{\blds\kf_i}_2, \dots,  v^{\blds\kf_i}_2)\,.
\en
Hence
\begn
 \tilde v 
  & =&
	 \sum_{i=1}^p \lambda _i  \cdot v^{\blds\kf_i}=\nonumber\\
   &&
	\!\!\!\!\!\!\!\left( \begin{array}{ccccccc}
         \lambda _1 v^{\blds\kf_1}_1&+&\lambda _2 v^{\blds\kf_2}_1&+& \dots&+&
                   \lambda _p v^{\blds\kf_p}_1 \\
         \lambda _1 v^{\blds\kf_1}_2 &+&\lambda _2 v^{\blds\kf_2}_1&+& \dots &+&
         \lambda _p v^{\blds\kf_p}_1 \\
                  \vdots & & \vdots
          & & \vdots & & \vdots \\
          \lambda _1 v^{\blds\kf_1}_2&+&\lambda _2 v^{\blds\kf_2}_2&+& \dots &+&
                                           \lambda _p v^{\blds\kf_p}_2
         \end{array} \right) 
				=:
				 \left( \begin{array}{c}
        \tilde v_1 \\ \tilde  v_2 \\ \vdots \\ \tilde v_{p+1}
         \end{array}  \right)\,.\hspace*{-1cm} 
				\label{mostshrink}
\enn
Since $\lambda_i\geq 0$ and $v^{\blds\kf_i}_1< v^{\blds\kf_i}_2$ for all $1 \leq i \leq p$
this implies $\tilde v_1 \leq \cdots \leq \tilde v_{p+1}$.  When $\lambda_{i_0} > 0$, we will have a strict inequality: $\ti v_{i_0} < \ti v_{i_0+1}$.
\end{proof}

 Recall in Lemma \ref{lem:cGH} we defined  the constant $c_{G/H}<0$. 
We define the positive constant $n_{G/H}$ in Corollary \ref{cor:uniformntest}.

\begin{proposition} \label{prop:posflat}
Let $\kf_1, \dots, \kf_p$ be  subalgebras with
$\kf_1<\dots<\kf_p$ and let
\beg
 \Delta_{\kf_1,\dots, \kf_p}=
     {\rm conv}\{v^{\blds\kf_1}, \dots, v^{\blds\kf_p}\}\,.
\en 
If $\kf_1$ is a non-toral subalgebra, then for all $\tilde v \in
\Delta_{\kf_1,\dots, \kf_p}$ and all $t\geq 0$ we have
$$
\sc (\gamma_v(t)) \geq e^{-t c_{G/H}}\cdot n_{G/H}\,,
$$ 
where $v:=\frac{\tilde v}{\Vert \tilde v\Vert} \in \Si$.
\end{proposition}

\begin{proof} 
Linear independence of  $\{v^{\blds\kf_1}, \dots, v^{\blds\kf_p}\}$ 
implies $\tilde v\neq 0$ (Lemma \ref{lem:linin}), and we also know  $\tilde v_1 \leq \cdots \leq \tilde v_{p+1}$ for the components of $\ti v$.
 Thus we can define
$v:=\frac{\tilde v}{\Vert \tilde v\Vert}$. Denoting by $\hat v_1,...,\hat v_{p+1}$
the corresponding eigenvalues of $v$, we still have 
\begn
   \hat v_1 \leq \cdots  \leq \hat v_{p+1}\,. \label{eqn:esthatv}
\enn 
Let $j_1,...,j_p$ be defined as in Lemma \ref{lem:linin} and let
$I_1:=\{1,...,j_1\},...,$ $I_{p+1}:=\{j_p+1,...,\ell\}$. Since
$\kf_i,\kf_j,\kf_k$ are subalgebras contained in each other, we
have $[I_iI_jI_k]=0$, whenever all three indices $i,j,k$ are
pairwise distinct, and $[I_iI_iI_k]=0$ for $k>i$.

With the help of Lemma \ref{lem:1p5} we conclude
\beg
  \sum_{j\in I_m} d_j b_j
   &=&[I_m I_m I_m]+
    \sum_{s=m+1}^{p+1}[I_m I_s I_s]
    +2 \sum_{s=1}^{m-1}[I_s I_mI_m] + \sum_{j\in I_m}2d_jc_j
\en
for $m\in \{1,...,p+1\}$. Now let
\beg
  a_m
	  &:=& 
		 \tfrac{1}{2} \big(\sum_{j\in I_m} d_j b_j
        -\tfrac{1}{2} [ I_m I_m I_m]
        -\sum_{s=m+1}^{p+1} [ I_m I_s I_s] \big)\\
		&=&
				 \sum_{j\in I_m} d_j c_j
        +\tfrac{1}{4} [ I_m I_m I_m]
        + \sum_{s=1}^{m-1}[I_s I_mI_m] \geq 0\,.
\en
By (\ref{scalI}) we obtain
\beg
  \sc (\gamma_v (t)) = \sum_{m=1}^{p+1} a_m \cdot e^{t(-\hat v_m)}
    -\tfrac{1}{4}\sum_{m=1}^{p+1}\sum_{s=1}^{m-1}[I_s I_mI_m] \cdot
   e^{t(\hat v_s-2\hat v_m)}\,.
\en
Now \eqref{eqn:esthatv}
yields $\hat v_s-2\hat v_m\leq -\hat v_m$ for $s\leq m-1$. Thus
\beg
  \sc (\gamma_v (t))\geq
  \sum_{m=1}^{p+1}e^{t(-\hat v_m)} \cdot
	\big(  \sum_{j\in I_m}d_jc_j+
      \tfrac{1}{4}[I_m I_m I_m]+\tfrac{3}{4}\sum_{s=1}^{m-1}[I_s I_mI_m]
     \big)\,.
\en
 Since by assumption $\kf_1$ is not toral, we obtain the claim
by Lemma \ref{lem:toral} and Corollary \ref{cor:uniformntest}, using
that for any $v \in \Si$ we have $\lambda(v) \leq -c_{G/H}$ by Lemma \ref{lem:cGH}.
\end{proof}

%%%%%%%%%%%%%%%%%%%%%%%%%%%%%%%%%%%%%%%%%%%%%%%%%%%%%%%%%%
\subsection{{\it Scalar curvature estimates from above}}\label{secscalab}
%%%%%%%%%%%%%%%%%%%%%%%%%%%%%%%%%%%%%%%%%%%%%%%%%%%%%

In this section we prove in Theorem \ref{theo:scalestu} that for each $\delta>0$
 there exists $\bar t(\delta)>0$ such that for every direction $v \in \Si \bs \Usre (\delta)$ (where 
 $\Usre (\delta)$ is the $\delta$-neighborhood of the set of extended non-toral directions $\Xsre$ in $\Si$) and for all 
 $t\geq \bar t(\delta)>0$,
the scalar curvature of $\gamma_v(t)$ is uniformly bounded from above. In particular we show
that in this non-compact region, there do not exist $G$-invariant metrics with large scalar curvature.

The scalar curvature estimates provided here are quite involved. In some sense what makes this situation difficult is when the most shrinking direction $v$ corresponds to a toral subalgebra (or more generally to a flag of toral subalgebras). 
The complication is that we need scalar curvature estimates to apply to a sequence of directions converging  to such a direction, for all large times $t \geq \bar t(\delta)$. Notice that
 the most shrinking direction need not correspond to a toral subalgebra for those approximating directions.
The key result here is purely Lie-theoretic, Proposition \ref{propinequ}, where we show how a priori independent structure constants relate to each other.

Recall, 
 $\Si$ is the unit sphere in $T_{\Id} \MGo$, 
and our homeomorphism ${\rm Gr}: 
\Si \to \Sphb$, the Graev map. In Section \ref{sec:defWT}, Definition \ref{defin-W}, we defined  $\WW$ (resp. $\WWs$), the union of all disks $\D(\kf)$ (resp. non-toral disks), and $\T$ (resp. $\XXGH$),  the union of flag simplices $\Delta_{\varphi}$  over all  toral flags (resp. non-toral flags).
There we saw that $\XXGH \subset \WWs \subset \WW$.
 Here we set $\WSi:={\rm Gr}^{-1}(\WW)$, $\WSi_s:={\rm Gr}^{-1}(\WWs)$, $\TSi:={\rm Gr}^{-1}(\T)$, $\XSGH :={\rm Gr}^{-1}(\XXGH)$, and 
 $\BSi :={\rm Gr}^{-1}(\B)$. 
In Section \ref{sec:homotopy1} we defined $\T(\epsilon_0)$  and $\XXsre=\B \bs \T(\epsilon_0)$ (see Theorem \ref{thm:homotopy2}), and we observed that $\WWs \subset \XXsre$. We set 
$ \TSi(\epsilon_0):= {\rm Gr}^{-1}(\T(\epsilon_0))$, and we set 
$\Xsre={\rm Gr}^{-1}(\XXsre)=\BSi \bs \TSi(\epsilon_0)$. Thus we have
\begin{equation}\label{eqn:containments}
  \XSGH \subset \WSi_s \subset \Xsre \subset \WSi \subset \Sph \,.
\end{equation}
Recall that in Theorem \ref{thm:Tbound} we show, not using any of the results of Section  \ref{sec:homotopy1},
that there is an open neighborhood $U_{\TSi}$ of $\TSi$ in $\Si$ with nice properties: see Theorem \ref{thm:Tbound}. 
Using $U_{\TSi}$,  we then show in Corollary \ref{cor:choseeps0} that there exists $\epsilon_0 >0$ 
 such that the bottom of $\Xsre$ is contained in the open neighborhood $U_{\TSi}$: see the remark
following Corollary \ref{cor:choseeps0}.

Given some $\delta >0$, let $\Usre (\delta)$ denote the open $\delta$-neighborhood of $\Xsre$  in $\Si$.
Recall that the
nonpositive number $-b_{G/H}$ is the trace of the Killing form.
The main result of this section is the following theorem:

\begin{theorem} \label{theo:scalestu}
For every $\delta>0$ there exists a corresponding $\bar t(\delta)>0$ such that for all
$v\in  \Si \bs \Usre (\delta)$ and for all $t \geq \bar t(\delta)$, we have
$$
  \sc (\gamma_v(t))\leq \tfrac{1}{2}b_{G/H}\,.
	$$
\end{theorem}

\begin{proof}
We prove the theorem by contradiction.
Suppose not: then for some $\delta_0>0$, 
for each $\alpha \in \NN$ there exists a direction $v_\alpha \in  \Si \bs \Usre (\delta_0)$
and some $t_\alpha \geq \alpha$ with 
$$
 \sc(\gamma_{v_\alpha}(t_\alpha))>  \tfrac{1}{2}b_{G/H} \,.
$$ 
For each $\alpha \in \NN$, let $f_\alpha$ be a good decomposition with respect to $v_\alpha$. 
Since $\Si  \bs \Usre (\delta_0)$ is compact, by passing to a subsequence,
we may assume that  $v_\alpha \to v \in \Si  \bs \Usre (\delta_0)$ as  $\alpha \to \infty$,
and that simultaneously, for some good decomposition $f$ with respect to $v$, our good decompositions $f_\alpha \to f$ as $\alpha \to \infty$, since a sequence of $Q$-orthonormal bases of $\m$ subconverges to a $Q$-orthonormal basis.

If $v \not \in \WSi$, then  by 
the scalar curvature estimates provided in Theorem \ref{theo:negdir}, we obtain a contradiction.
Suppose, instead, that $v\in \WSi$. By the definition of $\Usre (\delta)$,
this implies that either $v \in U_{\TSi}$,  
the open neighborhood of $\TSi$ given in
Theorem \ref{thm:Tbound}, or $v \in \WSi \bs \BSi$. In each of these cases, the curvature estimates
provided in Theorem \ref{thm:Tbound} and  Theorem \ref{theo:scalest}  yield our desired contradiction, 
proving the claim.
\end{proof}

We obtain
the following corollary.

\begin{corollary}\label{cor:nnontoral}
Let $G/H$ be a compact homogeneous space. Suppose that there are no non-toral subalgebras.
Then $\sc:\MGo \to \RR$ attains its global maximum.
\end{corollary}

\begin{proof}
Since there are no non-toral subalgebras, $\B = \emptyset$ (also any subset of $\B$ is empty). Thus we have $\Xsre=\emptyset$, and hence, 
by Theorem \ref{theo:scalestu}, $\sc:\MGo\to \RR$ is bounded above.
Now if $G/H$ is not a torus, then $\sc(Q)>0$. By choosing a maximizing sequence $(g_i)_{i \in \NN}$ we obtain
a Palais-Smale-sequence with a uniform positive lower bound for $\sc(g_i)$. By \cite{BWZ} this
sequence will subconverge to a global maximum of $\sc$. If $G/H$ is a torus, then every $G$-invariant metric is flat
and the claim follows as well.
\end{proof}

In fact, $\sc:\MGo \to \RR$ is bounded  above if and only if there exists no non-toral
subalgebra: see Lemma \ref{lem:scalpos}.

\begin{lemma} \label{lem:sigmaall} \cite{WZ2}
Let $G/H$ be a compact homogeneous space. If $\WSi=\Si$, then the scalar
curvature functional $\sc :\MGo \to \RR$ is bounded from below by zero.
If in addition $G/H$ has finite fundamental group or if $G/H$ is a torus, then
$\sc :\MGo \to \RR$ attains its global minimum.
\end{lemma}

Now we present our first scalar curvature estimates, essentially
due to \cite{WZ2}.
Let $\WSi(\delta)$ denote the open $\delta$-neighborhood of $\WSi$ in
$\Si$ and let $\D(\kf)^\Si:={\rm Gr}^{-1}(\D(\kf))\subset \WSi$ for $\kf \in \Sub$.

\begin{theorem}\label{theo:negdir}
Let $G/H$ be a compact homogeneous space. For each $\delta
>0$ there exists $t_0(\delta)>0$ such that for all $t \geq
t_0(\delta)$ and all $v \in \Si\bs \WSi(\delta)$, we have 
$$\sc (\gamma_v(t))\leq 0 \quad {\rm and }\quad \lim_{t\to +\infty}\sc(\gamma_v(t))=-\infty\,.$$
\end{theorem}

\begin{proof}  We claim that for all $\delta >0$ there exists $\eps
(\delta)>0$ such that for all $v \in \Si \bs \WSi(\delta)$, there
exists a good decomposition $f$ of $v$ with the following
property: There exists $(i_0,j_0,k_0)\in \{1,\dots,\ell \}^3$ with
\begn
  [i_0j_0k_0]_f \geq \eps(\delta)
	\quad \textrm{ and }\quad
	 v_{i_0}-v_{j_0}-v_{k_0}+\lambda(v)\geq \eps(\delta)\,,\label{eqn:uniijk}
\enn
where $\lambda(v)$ is the smallest eigenvalue of $v$.
This is seen as follows: Suppose no such
$\eps(\delta)>0$ exists. Then for each $\alpha \in \NN$ there
exists $v_\alpha \in \Si \bs \WSi(\delta)$ such that for all
good decompositions $f_\alpha$ of  $v_\alpha$ the
following holds: For all  $(i,j,k)\in \{1,\dots,\ell \}^3$,
we have either 
\beg 
[ijk]_{f_\alpha} < \tfrac{1}{\alpha}
   \quad \textrm{ or }\quad
    v_{\alpha,i} -v_{\alpha,j} - v_{\alpha,k}+\lambda(v_\alpha)
		 < \tfrac{1}{\alpha}\,.
\en
Passing to a subsequence, 
 we may assume that $v_\alpha \to w \in \Si \bs \WSi(\delta)$ 
and $f_\alpha \to f$ as $\alpha \to \infty$,
where $f$ is a good decomposition of $w$.
We obtain for all  $(i,j,k)\in \{1,\dots,\ell \}^3$
either 
\begn
 [ijk]_{f} =0
\quad \textrm{ or }\quad
 w_i-w_j-w_k+\lambda(w)\leq 0\,.\label{eqn:ijklambdaleqz}
\enn
Now let $I_1:=I_1^{w}(f)$,\dots, $I_{\ell_w}:=
I_{\ell_w}^{w}(f)$. Then for $j,k \in I_1$ and $i\in \{1,\dots,\ell\}\bs
I_1$ the inequality in \eqref{eqn:ijklambdaleqz}
cannot hold because $w_j = w_k=\lambda(w)$ and $ w_i-\lambda(w)>0$ for such $i$.
It follows that $[I_1I_1I_p]=0$ for $p >1$ and hence 
$\kf:=\h \oplus \m_{I_1}$ is a subalgebra. By a similar
argument, for $j\in I_1$, so that $w_j=\lambda(w)$, and for $i \in I_p$, $k \in I_q$ with $2\leq q< p$, 
again the inequality in \eqref{eqn:ijklambdaleqz}
cannot hold since $w_p - w_q > 0$.
We deduce that $[I_1I_pI_q]=0$ for all $1 \leq p \neq q$, which implies $\ad(\kf)$-invariance for $w$ (see Corollary \ref{cor:submkf}). 
But then $w \in \D(\kf)^\Si \subset \WSi$, which is a contradiction.
This shows the above claim.

The proof of the theorem now follows from \eqref{scalI}, here:
$$ \sc (\gamma_v (t))=
	\tfrac{1}{2} \sum_{i=1}^\ell d_ib_i \cdot e^{t(-v_i)}
      -\tfrac{1}{4} \sum_{i,j,k=1}^\ell
      [ i j k] \cdot e^{t(v_i-v_j-v_k)}\,.$$
Notice that $b_i\leq b_{G/H}$, see \eqref{corbGH}, thus
the first summand  
 is bounded from above by $\tfrac{1}{2}b_{G/H}\cdot e^{t\vert \lambda(v)\vert}$.
Moreover, by \eqref{eqn:uniijk} there exists some 
$(i_0,j_0,k_0)\in \{1,\dots,\ell \}^3$ such that the corresponding summand in the
second term  
is bounded from above
by $-\tfrac{1}{4}\cdot \eps(\delta)\cdot e^{t(\vert \lambda(v)\vert + \eps(\delta))}$.
It follows that
$$
  \sc (\gamma_v(t))
	   \leq 
	   e^{t\vert \lambda(v)\vert}\cdot \big(
		   \tfrac{b_{G/H}}{2}-\tfrac{1}{4}\cdot \eps(\delta)\cdot e^{t\eps(\delta)}\big)\,.
$$
We choose $t_0(\delta)>0$ such that the second factor is negative for all $t\geq t_0(\delta)$.
Then since $\lambda(v) \leq c_{G/H}$ (see Lemma \ref{lem:cGH}), the claim follows.
\end{proof}

We come now to our second main scalar curvature estimate.

\begin{theorem}\label{thm:Tbound}
Let $G/H$ be a compact homogeneous space. Then there exists an
open neighborhood $U_{\TSi}$ of $\TSi$ in $\Si$ and some $t_1>0$ such that for 
all $v \in U_{\TSi}$ and all $t \geq t_1$ we have
$$
   \sc(\gamma_v(t)) \leq \tfrac{1}{2}\cdot b_{G/H}\,.
$$
\end{theorem}

\begin{proof}
Suppose that the above claim is not true. Then for all $\alpha \in \NN$ 
there exists a sequence $v_\alpha \in \Si$ and $t_\alpha \geq \alpha$
with 
\begn \label{eqn:scalbigbGH}
 \sc(\gamma_{v_\alpha}(t_\alpha))> \tfrac{1}{2}\cdot b_{G/H}
\enn
and $\lim_{\alpha \to \infty} d_{\Si}(v_\alpha,\TSi)=0$.
Since $\TSi={\rm Gr}^{-1}(\T)$ is compact (Proposition \ref{prop: WW}), 
we may assume that $\lim_{\alpha \to \infty}v_\alpha=v \in \TSi$.
Moreover, for each $\alpha \in \NN$, let $f_\alpha$ be a good decomposition of $v_\alpha$. By passing to a subsequence,
we may assume
that as $\alpha \to \infty$ the decompositions $f_\alpha \to f$, a good decomposition  of $v$. 

By definition of $\TSi$ there exists a toral flag $(\tf_1 < \cdots < \tf_m)$ such that
$v=\pi_{\Si}(v^E)$ with
\beg
  v^E
	&=&
	 \sum_{j=1}^{m}\lambda_j \cdot v^{\bldss\tf_j}
	\en 
and $0 \leq \lambda_j \leq 1$ for all $1 \leq j \leq m$ and  $\sum_{j=1}^m \lambda _j=1$.  
(And $\pi_{\Si}$ denotes the radial projection from $T_{\Id}\MGo$ to the unit sphere $\Si$.) 
We may assume that $\lambda_m>0$; otherwise drop $\tf:=\tf_m$.	
Since $\lambda_m>0$, we conclude that the eigenvalues of
$v\vert_{\m_{\bldst\tf}}$ are strictly less than the eigenvalues of
$v\vert_{\m_{\bldst\tf}^\perp}$ (see Lemma \ref{lem:linin}).  This is true for $v^E$ and carries over to $v$. 
Notice also that $v\vert_{\m_{\bldst\tf}^\perp}$ has only one positive eigenvalue, 
$\Lambda(v)$ (the maximal eigenvalue of $v$), which  is greater than $-c_{G/H}>0$, by Lemma \ref{lem:cGH}.
It follows that  the subalgebra $\tf$ is $f$-adapted for the above good decomposition $f$ of $v$.

Let $I:=I_1^{\bldst\tf}(f)$ and $I^{\co}:=\{1,...,\ell\}\bs I$. Then we have
$v^E_k<v^E_j$ 
for all $k\in I$ and all $j\in I^{\co}$ and consequently
$v_k<v_j$ for all  $k\in I$ and all $j\in I^{\co}$.
(We write $v_{\alpha,i} = \Lambda(v)-\ti w_{\alpha,i}$  for some constants  $\ti w_{\alpha,i}$.)  
Thus, we now are in position to apply Lemma \ref{lem:torvcan}. This lemma tells us 
there exists some  
$t_3>0$ such that  for all $t\geq t_3$, and for all but finitely many $\alpha \in \NN$ we have
\begin{align*}
  \sc (\gamma_{v_\alpha}(t))
	 & \leq 
\tfrac{1}{2} \sum_{i\in I^{\co}}
                    d_ib_i \cdot e^{t(- v_{\alpha,i})}
       -\tfrac{1}{4}\sum_{i,j,k\in I^{\co}}[ijk]_{f_\alpha}
              \cdot    e^{t(v_{\alpha,i} - v_{\alpha,j} - v_{\alpha,k})} \\
& \leq  \tfrac{1}{2} \sum_{i\in I^{\co}}
                    d_ib_i \cdot e^{t(- v_{\alpha,i})}
=   \tfrac{1}{2}\cdot \sum_{i \in I^{\co}} d_i b_i \cdot 
      e^{t  (- \Lambda(v) +\ti w_{\alpha,i} )}\,.
\end{align*} 
Since the eigenvalues of $v_\alpha$ converge
to the eigenvalues of $v$ as $\alpha \to \infty$, we know $\ti w_{\alpha,i} \to 0$ as $\alpha \to \infty$.
Using that $-\Lambda(v) \leq c_{G/H}$ we deduce
for all $\alpha$ sufficiently large and for all $t\geq t_3$, 
\beg
   \sc (\gamma_{v_\alpha}(t)) 
	&\leq &
	 \tfrac{1}{2} \cdot \sum_{i\in I^{\co}}
    d_i b_i \cdot e^{t(\frac{1}{2}c_{G/H})} \leq \tfrac{1}{2}b_{G/H}\,.
\en
This contradicts \eqref{eqn:scalbigbGH}, 
since for $\alpha $ large we have $t_\alpha \geq \alpha >t_3$. This shows the claim.
\end{proof}

We come to our third scalar curvature estimate. Notice, that 
in the next theorem, the constant $t_2$ is allowed to depend on $v$, $f$, the sequence
$(v_\alpha)_{\alpha\in \NN}$ and so on.

\begin{theorem} \label{theo:scalest}
Let $G/H$ be a compact homogeneous space.  
Let $v \in  \WSi \bs (\BSi \cup \TSi)$ and let $(v_\alpha)_{\alpha \in \NN}$ be a sequence in $\Si$
converging to $v$. For each $\alpha \in \NN$, let $f_\alpha$ be a good decomposition of
$v_\alpha$, with $(f_\alpha)$ converging to a good decomposition $f$ of $v$.
Then there exists some $t_2>0$, such that for all $t \geq t_2$, and for all but finitely many $\alpha$,
\beg
 \sc (\gamma_{v_\alpha}(t))
&\leq  & 0 \,.
\en
\end{theorem}

\begin{proof}
We write $A={\rm Gr}(v)$,  equivalently, $v={\rm Gr}^{-1}(A)$. 
Then by Corollary \ref{cor:WbsBdescr} 
there exists some $m \geq 1$, 
a toral flag $(\tf_1< \cdots <\tf_m)$, and $A_- \in \D_-(\tf_m)$
such that
\begn
 A \in {\rm conv}(A^{\bldst\tf_1},\dots,A^{\bldst\tf_{m-1}},A_-)\bs \T\,.
\enn
Recall that ${\rm Gr}^{-1}$ first orthogonally projects  $A$ to $A-\tfrac{1}{n}\cdot \Id_{\m}$ in the set $\{\tr = 0\}$ and
then it projects to $\Si$ by radial projection ($\pi_\Si$). As a consequence,
convex sets in $\Sphb$ are mapped to convex sets in $\{\tr = 0\}$ and on to convex sets in $\Si$. 

Now let  $A_m$ denote the radial projection of $A_-$ to $\partial_- \D(\tf_m)$.
Thus $A_- \in {\rm conv}\{A^{\bldss\tf_m},A_m\}$. As a consequence we have
$$
  A \in {\rm conv}\{A^{\bldst\tf_1},\dots,A^{\bldst\tf_m},A_m\}\,.
$$
By Lemma \ref{lem:linin} ,$\{A^{\bldst\tf_1},\dots,A^{\bldst\tf_m}\}$ is a linearly independent set.
Since $\tf_m \subsetneq \ker(A_m)$ it follows that 
$\{A^{\bldst\tf_1},\dots,A^{\bldst\tf_m},A_m\}$ is linearly independent as well. We get 
\beg
  A=\lambda_{m+1}\cdot A_m+\sum_{j=1}^m \lambda_j \cdot A^{\bldst\tf_j}
\en
with $0 \leq \lambda _j \leq 1$ and $\sum_{j=1}^{m+1} \lambda_j=1$.
Notice that we have $\lambda_{m+1}>0$, since $A \not\in \T$. 
We may furthermore assume that $\lambda_m>0$; otherwise we could remove $\tf_m$ from the
above flag, provided $A\neq A_m$,  
a case we will consider below.

Since $\lambda_m>0$ and $\tf:=\tf_m \subsetneq \ker (A_m)$,
it follows that $\tf$ is $f$-adapted
for any good decomposition $f$ of $v={\rm Gr}^{-1}(A)$; 
the eigenvalues of $v\vert_{\m_{\bldst\tf}}$ are strictly smaller than
the eigenvalues of $v\vert_{\m_{\bldst\tf}^\perp}$ (see Lemma \ref{lem:linin}).
Thus $f$ is simultaneously a good decomposition for $v^{\bldss\tf}$ and for ${\rm Gr}^{-1}(A_m)$.

Now let $I:=I_1^{\bldss\tf}(f)$ and $I^{\co}:=\{1,\dots,\ell\}\bs I$. Then we obtain
$v_k<v_j$ for all $k\in I$ and all $j\in I^{\co}$.
Thus, we are in a position to apply Lemma \ref{lem:torvcan}. There exists some  
$t_2>0$  such that for all but finitely many $\alpha \in \NN$ we have
\beg
  \sc (\gamma_{v_\alpha}(t))
  \leq
	\tfrac{1}{2} \sum_{i\in I^{\co}}
                    d_ib_i \cdot e^{t(- v_{\alpha,i})}
       -\tfrac{1}{4}\sum_{i,j,k\in I^{\co}}[ijk]_{f_\alpha}
              \cdot    e^{t(v_{\alpha,i} - v_{\alpha,j} - v_{\alpha,k})}
\en
for all $t \geq t_2$.

Now let $I_2,\dots,I_{\ell'}$, $\ell' \geq 3$ be defined as in Lemma
\ref{lem:strconst222}. That is, we have $\ker (A_m)=\tf \oplus \m_{I_2}$
and that $I_3,\dots,I_{\ell'}$ correspond to the positive eigenvalues of $A_m$.
Applying Lemma \ref{lem:strconst222}, there exist either $i_0,j_0 \in I_2$
and $k_0 \in I_p$ for $p\geq 3$ with $[i_0j_0k_0]_f>0$ -- in the case $\ell'=3$
this holds -- or there exist $i_0\in I_2$ and $j_0 \in I_p$ and
$k_0 \in I_q$ with $3 \leq p< q$ such that $[i_0j_0k_0]_f>0$.
Let $a_2=0<a_3< \cdots <a_{\ell'}$ denote the eigenvalues of $A_m$.
Then clearly
$a_{k_0}-a_{j_0}-a_{i_0}+a_2 >0$.
Since, when restricted to $\tf^{\perp}$,  $\sum_{j=1}^m \lambda_j \cdot A^{\bldss\tf_j}$  is
a multiple of the identity, 
we conclude that
\begn
 v_{k_0}-v_{j_0}-v_{i_0}+v_\ast>0\,,\label{eqn:vijkast}
\enn
where $v_\ast$ 
denotes the smallest eigenvalue of $v\vert_{\m_{\bldss\tf}^\perp}$.

We write $v_{\alpha,i}=v_i+w_{\alpha,i}$ for all $i \in \{1,\dots,\ell\}$ and all $\alpha \in \NN$.
Since $\lim_{\alpha \to \infty}(v_\alpha)=v$ we deduce $\lim_{\alpha \to \infty}w_{\alpha,i}=0$
for  each $i \in \{1,\dots,\ell\}$.
This yields
\beg
 2 \sum_{i\in I^{\co}}
                    d_ib_i \cdot e^{t(- v_{\alpha,i})}
  =
	2\sum_{i\in I^{\co}} d_i b_i
    \cdot  e^{t(-v_i- w_{\alpha,i})}									
	\leq 	
	 2 b_{G/H} \cdot e^{-t v_\ast} \cdot \sum_{i\in I^{\co}}   e^{t(-w_{\alpha,i})}		\,.
\en
Moreover, we have (for constants $\ti w_{\alpha,ijk} = w_{\alpha,i} +w_{\alpha,j} +w_{\alpha,k}$, and $\ti w_\alpha = \ti w_{\alpha,i_0 j_0 k_0}$)
\beg
-\sum_{i,j,k\in I^{\co}}[ijk]_{f_\alpha}
              \cdot    e^{t(v_{\alpha,k} - v_{\alpha,j} - v_{\alpha,i})}
	&= &
	-\sum_{i,j,k\in I^{\co}}[ijk]_{f_\alpha}
              \cdot    e^{t(v_k-v_j-v_i+\tilde w_{\alpha,ijk})} \\
  &\leq &							
	-[i_0j_0k_0]_{f_\alpha}
	 \cdot    e^{t(v_{k_0}-v_{j_0}-v_{i_0}+\tilde w_\alpha)}\\
	&\leq &							
	-  e^{-t v_\ast} \cdot [i_0j_0k_0]_{f_\alpha}  
	 \cdot    e^{t(v_{k_0}-v_{j_0}-v_{i_0}+v_\ast+\tilde w_\alpha)}
\en
with $\lim_{\alpha \to \infty} \tilde w_\alpha=0$.
Putting everything together we arrive at
\beg
  4\cdot \sc (\gamma_{v_\alpha}(t))
   \leq 
	   e^{-t v_\ast}\cdot
		\Big(  2 b_{G/H} \cdot \sum_{i\in I^{\co}} 
    e^{t(-w_{\alpha,i})}		
		-[i_0j_0k_0]_{f_\alpha} 
	 \cdot    e^{t(v_{k_0}-v_{j_0}-v_{i_0}+v_\ast+\tilde w_\alpha)} \Big)
\en
for all $t \geq t_2$ and $\alpha$ sufficiently large.

Now, by \eqref{eqn:vijkast} we know $v_{k_0}-v_{j_0}-v_{i_0}+v_\ast>0$.
Since the structure constants $[ijk]_f$ depend continuously on $f$
and since $f_\alpha \to f$ for $\alpha \to \infty$,   
we deduce $\sc(\gamma_{v_\alpha}(t))< 0$ for all $t\geq t_2$ and $\alpha$
sufficiently large.

Finally, it remains to handle the case $A=A_m$. In this case we obtain a contradiction simply by
applying \eqref{scal}. As above, let $0=a_2<a_3< \cdots < a_{\ell'}$ denote the eigenvalues
of $A$.  Let $v={\rm Gr}^{-1}(A)$.
Then the first sum in  \eqref{scal} is bounded from above by
$\tfrac{1}{2}b_{G/H}\cdot e^{t\vert \lambda(v)\vert}$: see proof of Theorem
\ref{theo:negdir}. However, since the smallest eigenvalue of $v\vert_{\m_{\bldss\tf}^\perp}$ is now
indeed the smallest eigenvalue of $v$, that is $\lambda(v)=v_\ast$,
we conclude as above that the second term in  \eqref{scal} is bounded  above
by $-\tfrac{1}{4} e^{t\vert \lambda(v)\vert}\cdot [i_0j_0k_0]_f \cdot e^{t(v_{k_0}-v_{j_0}-v_{i_0}+v_\ast)}$.
By \eqref{eqn:vijkast} since $[i_0j_0k_0]_f>0$, we obtain a contradiction in this case, as above.
\end{proof}

The second and third scalar curvature estimates proved above
use the following important fourth scalar curvature estimate.

\begin{lemma} \label{lem:torvcan}
Let $v \in \Si$, let $f$ be a good
decomposition of $v$, and let $\tf$ be a toral subalgebra which is $f$-adapted. Let $I=I^{\bldss\tf}_1(f)$,
$I^{\co}:=\{1,\dots,\ell\}\bs I$ and suppose $(v_\alpha)_{\alpha \in \NN}$ is a
sequence in $\Si$ such that for each $\alpha$, $f_\alpha$ is a good
decomposition for $v_\alpha$. Suppose as $\alpha \to\infty$, $v_\alpha \to v$ and $f_\alpha \to f$. If for
all $k\in I$ and all $j\in I^{\co}$,
$
 v_k<v_j
$ 
holds true, then there exists some $t_3>0$,
such that for all $t\geq t_3$ and for all but finitely many $\alpha \in \NN$,
we have
\begn
  \sc (\gamma_{v_\alpha}(t))
  &\leq&
	\tfrac{1}{2} \sum_{i\in I^{\co}}
                    d_ib_i \cdot e^{t(- v_{\alpha,i})}
       -\tfrac{1}{4}\sum_{i,j,k\in I^{\co}}[ijk]_{f_\alpha}
              \cdot    e^{t(v_{\alpha,i} - v_{\alpha,j} - v_{\alpha,k})}\,.\label{scalbound}
\enn
\end{lemma}

\begin{proof}
For each $1 \leq i \leq \ell$, by Lemma \ref{lem:1p5} we have
${\dsp d_ib_i=2d_ic_i+\sum_{j,k\in I \cup I^{\co}} [ijk]_{f_\alpha}}$. Recall that $b_i$ and $c_i$ may
depend on $\alpha$. Since $\tf$ is a toral subalgebra, we have $\m_{\bldss\tf} \leq \m_0$.
Thus, $\m_{\bldss\tf}$ is a direct sum of $1$-dimensional $\Ad(H)$-irreducible summands $\m_1,...,\m_r$
and 
$$
  f=\m_1 \oplus \cdots \oplus \m_r \oplus \m_{r+1} \oplus \cdots \oplus \m_\ell\,.
$$
Since $\lim_{\alpha \to \infty} f_\alpha =f$ we conclude that
$$
  f_\alpha=(\m_1)_\alpha \oplus \cdots \oplus (\m_r)_\alpha \oplus (\m_{r+1})_\alpha \oplus \cdots 
	\oplus (\m_\ell)_\alpha
$$
with $\dim (\m_i)_\alpha=1$ for $i=1,...,r$. Thus $ (\m_i)_\alpha \subset \m_0$ which implies $c_i=0$
for $i \in I$. We conclude for $i \in I$ 
\beg
   d_ib_i=\sum_{j,k \in I } [ijk]_{f_\alpha}+2
            \sum_{j\in I,k \in I^{\co}} [ijk]_{f_\alpha}+
            \sum_{j,k \in I^{\co}} [ijk]_{f_\alpha}\,.
\en
With the help of (\ref{scal}) we obtain
\beg
\lefteqn{\sc (\gamma_{v_\alpha}(t))=}&&\\
 &=& 
    \tfrac{1}{2}\sum_{i\in I\cup I^{\co}} d_i b_i \cdot  e^{t(- v_{\alpha,i})}
     -\tfrac{1}{4}\sum_{i,j,k \in I \cup I^{\co}} [ijk]_{f_\alpha}
          \cdot e^{t(v_{\alpha,i} - v_{\alpha,j} - v_{\alpha,k})}\\
 &=&
  \tfrac{1}{2}\sum_{i\in I} e^{t(- v_{\alpha,i})} \cdot \Big\{
        \sum_{j,k\in I} [ijk]_{f_\alpha}
         \cdot (1- \tfrac{1}{2} e^{t(v_{\alpha,j}-v_{\alpha,k})} ) \\
  &&   
	+\sum_{j\in I,k\in I^{\co}} [ijk]_{f_\alpha}
        \cdot (2-\tfrac{1}{2}e^{t(v_{\alpha,k} - v_{\alpha,j})} - e^{t(v_{\alpha,j} - v_{\alpha,k})}  )\\
  && 
		+\,\,\,\sum_{j,k\in I^{\co}} [ijk]_{f_\alpha} \cdot 
         (1-\tfrac{1}{2} e^{t(v_{\alpha,j} - v_{\alpha,k})}
           -\tfrac{1}{2} e^{t(v_{\alpha,k} - v_{\alpha,j})})\\
  &&
		-\tfrac{1}{2}\sum_{j,k \in I^{\co}} [ijk]_{f_\alpha} \cdot 
          e^{t(2v_{\alpha,i} - v_{\alpha,j} - v_{\alpha,k})}\Big\} \\
 && 
   +\tfrac{1}{2}\sum_{i\in I^{\co}} d_i b_i  \cdot e^{t(- v_{\alpha,i})}
     -\tfrac{1}{4}\sum_{i,j,k \in I^{\co}} [ijk]_{f_\alpha} \cdot 
          e^{t(v_{\alpha,i}-v_{\alpha,j}-v_{\alpha,k})}\,.
\en
To prove the claim, we need to show that contributions from the first four summands are nonpositive.  
The fourth summand is clearly nonpositive. 
That the third summand is nonpositive follows from $1 \leq \frac12(e^{x}+e^{-x})$. 
We will show that 
$$
 \sum_{j,k\in I} [ijk]_{f_\alpha}
         \cdot (1- \tfrac{1}{2} e^{t(v_{\alpha,j} - v_{\alpha,k})})
	+\sum_{j\in I,k\in I^{\co}} [ijk]_{f_\alpha}
        \cdot (2-\tfrac{1}{2}e^{t(v_{\alpha,k} - v_{\alpha,j})}) \leq 0\,. 			
$$
By assumption there exists $\eps>0$ such that $v_{\alpha,k} - v_{\alpha,j}>\eps$
for all $\alpha$ sufficiently large and for all $j\in I,k\in I^{\co}$. Suppose now that there
exists a positive constant $C$ with
\begn
 \sum_{j,k\in I} [ijk]_{f_\alpha}  \leq C \cdot \!\!\!\!
  \sum_{j\in I,k\in I^{\co}} [ijk]_{f_\alpha}  \label{estadhinv}
\enn
for all $i\in I$ and all $\alpha \in \NN$. Then
\beg
   \lefteqn{\hspace{-4cm}\sum_{j,k\in I} [ijk]_{f_\alpha}
         \cdot  (1- \tfrac{1}{2} e^{t(v_{\alpha,j} - v_{\alpha,k})})
          \,\,+\!\!\sum_{j\in I,k\in I^{\co}} [ijk]_{f_\alpha}
          \cdot (2-\tfrac{1}{2}e^{t(v_{\alpha,k} - v_{\alpha,j})})}&&\\
   &&\hspace{-4cm}\leq  \sum_{j,k\in I} [ijk]_{f_\alpha}
         \,+\!\!\sum_{j\in I,k\in I^{\co}} [ijk]_{f_\alpha}
         \cdot  (2-\tfrac{1}{2}e^{t\eps})\\
    &&\hspace{-4cm}\leq \sum_{j\in I,k\in I^{\co}} [ijk]_{f_\alpha}
          \cdot (C+2-\tfrac{1}{2}e^{t\eps})
\en
and the desired estimate follows.

It remains to prove the estimate (\ref{estadhinv}). Recall $\m_{\bldss\tf} \leq \m_0$.
Since $\m_0$ is a compact subalgebra of $\g$ by Lemma \ref{lemnor}
and since by definition $\m_0$ consists of all the trivial 
summands of $\m$, (\ref{estadhinv})
can be viewed an estimate for a sequence of ${\rm Ad}(H)$-invariant
decompositions $(f_\alpha)_{\alpha \in \NN}$ of $\m_0$. 
Let $m_0=\dim(\m_0)$. 
By the definition of $[ijk]_{f_\alpha}$, the inequality (\ref{estadhinv}) is precisely
an estimate on a sequence $(\{e_1^\alpha,\dots,e_{m_0}^\alpha \})_{\alpha\in \NN}$ of $Q$-orthonormal bases
of $\m_0$ converging to a $Q$-orthonormal basis
$\{e_1,\dots,e_{m_0}\}$ of $\m_0$. As a consequence, we are in position to
apply Proposition \ref{propinequ} and the claim follows.
\end{proof}

%%%%%%%%%%%%%%%%%%%%%%%%%%%%%%%%%%%%%%%%%%%%%%%%%%%%%
\subsection{{\it Variational methods}}\label{sec:var}

%%%%%%%%%%%%%%%%%%%%%%%%%%%%%%%%%%%%%%%%%%%%%%%%%%%%%%%%%%%%%%%%%%%%%%%%%
%%%%%%%%%%%%%%%%%%%%%%%%%%%%%%%%%%%%%%%%%%%%%%%%%%%%%%%%%%%%%%%%%%%%%%%%%

In this section we apply variational methods to the scalar
curvature functional
$$
 \sc :\MGo \to \RR
$$
of compact homogeneous spaces $G/H$ admitting a non-contractible nerve $\XGH$, 
in order to prove the existence of a
critical point, an Einstein metric. 
More precisely, we show that there exists a
critical point $g$ such that the augmented coindex $m^*(g)$ of  $g$ is bounded from below by $q+1$, where
$q$ denotes the homological dimension of $\XGH$. Finally, we indicate how
Lyusternik-Schnirelmann theory can be applied.

% added 2023
Recall, that in Definition \ref{defin-W} we defined the nerve $\XXGH$ 
and just above Equation \eqref{eqn:containments} we set 
$$
 {\rm Gr}^{-1}(\XXGH)=:\XSGH \subset S \subset T_{\Id} \MGo
$$ 
where ${\rm Gr}$ denotes the Graev map,
a homeomorphism, see Lemma \ref{lem:Graevmap}, and $S$ denotes
the unit sphere in $T_{\Id} \MGo$: see Lemma \ref{lem:expdiff}.
 Then in Definition \ref{def:XGH} the nerve $\XGH$ is defined, 
as after Theorem \ref{theoA},
and in Lemma \ref{lem:XGHXXGH}  we show all of these versions of the nerve are homeomorphic. Notice also that in the proof of Theorem \ref{theomain},
we will consider elements in $\MGo$ as $G$-invariant metrics $g$ on $G/H$ and
not as endomorphisms $P_g$ anymore: see Section \ref{sec:ginvm}.

For proving existence, a key property of the scalar curvature functional, restricted 
to
$$
  \{\sc \geq \epsilon\}:=\{g \in \MGo \mid \sc(g) \geq \epsilon \}\,,
$$
$\epsilon >0$, is that it satisfies the Palais-Smale Condition (C) by Theorem A in \cite{BWZ}:
Any sequence $(g_i)_{i \in \N}$ in $\MGo$ with $\lim_{i \to \infty}\sc(g_i)=c \in (0, \infty)$ 
and $\lim_{i\to \infty} \Vert (\nabla_{L^2}\sc)_{g_i}\Vert_{g_i}=0$ has a convergent subsequence (converging to
a critical point).
Recall that by Equation (\ref{eqn:L2grad}) 
$ \Vert (\nabla_{L^2}\sc)_g \Vert^2_{g} = \tr (\Ric_0(g))^2$, $\Ric_0(g)$ denoting the
traceless Ricci endomorphism of $g$.

\medskip

The following theorem is our main existence result on compact
homogeneous Einstein manifolds. Recall that if the
 nerve $\XGH$ is empty, it is non-contractible.

\begin{theorem} \label{theomain}
Let $G/H$ be a compact homogeneous space. If the nerve
$\XGH$ of $G/H$ is not contractible, then $G/H$ admits a
$G$-invariant Einstein metric.
\end{theorem}

\begin{proof}  We may assume that $G/H$ has finite fundamental group.
If $\XGH=\emp$, then the claim
follows from Corollary \ref{cor:nnontoral}. 
Now suppose $\XGH \neq \emp$. 
By Lemma \ref{lem:sigmaall} we may assume that the compact set
$\Xsre$ is a proper subset of  $\Si$.
 By Corollary \ref{cor:Udefret},
 $\Xsre$ is a strong deformation retract of the open neighborhood $\USi:={\rm Gr}^{-1}(\UU)\subset \Sph$
of $\Xsre$ in $\Sph$. Recall also that  
we have the following chain of inclusions (see Equation (\ref{eqn:containments})):
\beg
 \XSGH\subset \WSis\subset \Xsre\subset
\USi \subset \Sph\,.
\en
There exist strong 
deformation retractions from $\USi$ to $\Xsre$, see above, and from 
$\Xsre$ to  $\WSis$, by Theorem \ref{thm:homotopy2},
and from $ \WSis$ to  $\XSGH$ by Theorem \ref{thm:homotopy3}.
In Section \ref{sec:nerve}, we show
that $\XGH$ is homeomorphic to $\XSGH={\rm Gr}^{-1}(\XXGH)\subset \Sph$ (see Lemma \ref{lem:XGHXXGH}).

Let $\delta>0$ be sufficiently small that that the open $\delta$-neighborhood $\Usre(\delta)$
of $\Xsre$ in $\Si$ is contained in $\USi$. We define the open truncated cone
\beg
  C(\USi,\bar t(\delta)):=\{ \gamma_v(t) \mid v \in
   \USi,\,\,t> \bar t(\delta)\}\subset \MGo
\en 
over $\USi$ with $\bar t(\delta)>0$, as in Theorem
%%%% changed Q --> \Id 2023
\ref{theo:scalestu}. Recall that $\exp:T_{\Id}\MGo 
\to \MGo\,;\,\,v \mapsto \gamma_v(1)$
is an diffeomorphism: see Lemma \ref{lem:expdiff}.
Since $ \Usre(\delta) \subset  \USi$
we have $\Si \bs \USi \subset \Si \bs  \Usre(\delta) $,
thus by  Theorem \ref{theo:scalestu},
for all $v \in \Si \bs \USi$ we obtain
$$
 \sc (\gamma_v(t))\leq \tfrac{1}{2}b_{G/H}
$$ 
for all $t\geq \bar t(\delta)$. Since the closed geodesic $\bar t(\delta)$-ball around $Q$ 
 in $\MGo$ is compact, there exists a positive constant 
$\sc_+>2\cdot \sc (Q)>0$ such that 
\begin{eqnarray}
\{\sc\geq \sc _+\} \subset
C(\USi,\bar t(\delta))\,.\label{eqn:scsc+C}
\end{eqnarray}
By Proposition \ref{prop:posflat} there exists some 
$T> \bar t(\delta)$ such that for all $t\geq T$ and all $v\in \XSGH$,
we have $\sc (\gamma_v(t))> \sc _+$. We define the cycle
\begn
  B:=\{\gamma_v(T) \mid v \in  \XSGH\} \subset \{\sc >\sc_+\} 
  \label{Bscalp}
\enn
and the contractible cone
\beg
 A_B:=\{\gamma_v(t) \mid v \in  \XSGH,\,\,0\leq t \leq T \}
\en
over $B$. 
 By Proposition \ref{prop:posflat}, we have $\sc (g)>0$ for
all $g \in A_B$. Since $A_B$ is compact, there exists some $\eps >0$ with
\begn \label{eqn:Aeps}
 A_B \subset \{\sc \geq \eps \}\,.
\enn
We claim that $B$ is not contractible in $ C(\USi,\bar t(\delta))$.

If $B$ were contractible in $ C(\USi,\bar t(\delta))=\USi\times (\bar t(\delta),\infty)$,
then by applying the radial projection from  $\USi\times (\bar t(\delta),\infty)$ to
$\USi \times \{2\bar t(\delta)\}$ it would follow that 
$\XSGH$ is contractible in $\USi$. 
Let us denote this contraction by $r:[0,1] \times\XSGH\to \USi$, $r|_{\{0\}\times \XSGH}=\Id_{\XSGH }$.
Notice that by the above there
exists a continuous map $\rho:\USi \to \XSGH$ with $\rho\vert_{ \XSGH}=\Id_{\XSGH }$. 
Then  
$$
  \rho \circ r:[0,1] \times\XSGH\to \XSGH
$$ 
is a contraction of $\XSGH$, contradicting our assumption. Thus $B$ is not contractible. 

As in Section \ref{sec:ginvm}, we write $\Vert \cdot \Vert_g$
to denote the norm on $T_g\MGo$ induced by the $L^2$-metric.  
Let $Z=\{g\in \MGo \mid (\nabla _{L^2} \sc)_g=0\}$
be the set of critical points of $\sc:\MGo \to \RR$. By \cite{BWZ} 
we know $Z$ is compact. 
Thus there exists a smooth positive function $\Psi$ on $\MGo$ with 
$\Psi(g) \equiv 1$ close to $Z$, and with
$$
\Psi (g)= \tfrac{1}{\Vert (\nabla _{L^2} \sc )_g  \Vert_g}
$$
outside a compact neighborhood $K$ of $Z$, and $\Psi(g)\leq
\frac{1}{\Vert (\nabla _{L^2} \sc )_g\Vert_g}$ for all $g \in \MGo$. 
We obtain
a smooth vector field 
$$
 X(g)= \Psi (g) \cdot (\nabla _{L^2} \sc )_g
$$
on $\MGo$ with $\Vert X(g) \Vert_g \leq 1$ for all $g \in \MGo$.

Let $g(s)$ be an integral curve of $X$. Since $(\MGo,L^2)$ is complete and
$\Vert X \Vert \leq 1$, $g(s)$ is defined for all $s\in\RR$.
We have for $s\geq 0$
\beg
  \tfrac{d}{ds}\, \sc (g(s))= \Psi (g(s)) \cdot \Vert
  ( \nabla _{L^2} \sc )_{g(s)}\Vert^2_{g(s)}\geq 0\,.
\en 
Let $(\Phi_s)_{s \in \RR}$ denote the flow of $X$. For all $s\geq 0$,  
$$
 \Phi_s(B)\subset \{\sc \geq \sc_+\} \subset C(\USi,\bar t(\delta)) \,,
$$ 
thus for every \mbox{$s\geq 0$} there exists some $g_s \in A_B$ with $\sc (\Phi_s(g_s))<\sc _+$. 
Otherwise there would exist $s_0>0$
such that $\Phi_{s_0}(A_B)\subset \{\sc\geq \sc _+\}\subset
C(\USi,\bar t(\delta))$, and we would obtain the following explicit contraction of
$B$ in $C(\USi,\bar t(\delta))$: First, we deform $B$ to $\Phi_{s_0}(B)$ by 
$\Phi\vert_{s \in [0,s_0]}$,
next we contract $\Phi_{s_0}(B)$  to the point $\Phi_{s_0}(Q)$, 
using the radial
contraction from $B$ to $Q$  
in the cone $A_B$. 

Now consider the sequence $(g_j)_{j \in\NN}$ in $A_B$. 
After passing to a subsequence, we may assume that $(g_{j})_{j\in \NN}$ converges to $g_\infty \in A_B$. 
Using that for all $j,k \in \NN$ with $k>j$ we have $\sc(\Phi_j(g_k))<\sc_+$,
it follows that $\sc (\Phi_{j}(g_\infty))\leq \sc _+$ for all $j \in \NN$.
As a consequence, for all $j \in \NN$, we have
\beg
     \int_0^{j} \Psi(\Phi_s (g_\infty) )\cdot \Vert (\nabla _{L^2}
          \sc)_{\Phi_s(g_\infty)}\Vert^2_{\Phi_s(g_\infty) }\, ds
      \leq  \sc _+ -\sc (\Phi_0(g_\infty))\,.
\en
Since the integrand is non-negative, there exists a sequence $(s_j)_{j \in \NN}$
with $s_j \to \infty$ as $j \to \infty$, such that as $j \to \infty$,
\begn \label{eqn:PS}
  \Psi(g_j' )\cdot \Vert (\nabla _{L^2}
          \sc)_{g_j'}\Vert^2_{g_j'} \to 0
\enn
with $g_j':=\Phi_{s_j} (g_\infty)$.
We show that in fact, $\Vert (\nabla _{L^2} \sc)_{g_j'}\Vert^2_{g_j'} \to 0$.  
For a contradiction, suppose there exists a subsequence of $(s_j)_{j \in \NN}$ with
$\Vert (\nabla _{L^2} \sc)_{g_j'}\Vert^2_{g_j'} \geq \varepsilon_1 > 0$
for all $j\in \NN$. Recall that
by construction, the smooth positive function $\Psi$ restricted to the compact neighborhood $K$ of $Z$ 
has a positive lower bound $\varepsilon_2 >0$. Suppose now that for infinitely many $j$ we have
$g_j'\in K$. Then
${\dsp  \Psi(g_j' )\cdot \Vert (\nabla _{L^2}
          \sc)_{g_j'}\Vert^2_{g_j'} \geq \varepsilon_1\varepsilon_2 >0}$, a contradiction.
If, alternatively, for infinitely many $j$ we have
$ g_j' \not\in K$,
by construction the left-hand side of  \eqref{eqn:PS} equals 
${\dsp \Vert (\nabla _{L^2}\sc)_{g_j'}\Vert_{g_j'} \geq \sqrt{\varepsilon_1} >0}$, 
and we again obtain a contradiction.
In either case, 
this yields a Palais-Smale-sequence of $\sc :\MGo \to \RR$ in
$\{ \sc \geq \epsilon\}$, see \eqref{eqn:L2grad} and \eqref{eqn:Aeps}, hence by Theorem A in
\cite{BWZ} we obtain the existence of a $G$-invariant Einstein metric on $G/H$.
\end{proof}

\medskip

For the compact homogeneous space $G/H=\SU(4)/\Un(2)$,
where ${\rm U(2)}\subset \SO(4) \cap{\rm Sp}(2)$, the next
picture shows the graph of the scalar curvature functional $\sc
:\MGo \to \RR$.
\begin{figure}[ht]
\begin{center}
\includegraphics[width=2.5in,height=2in]{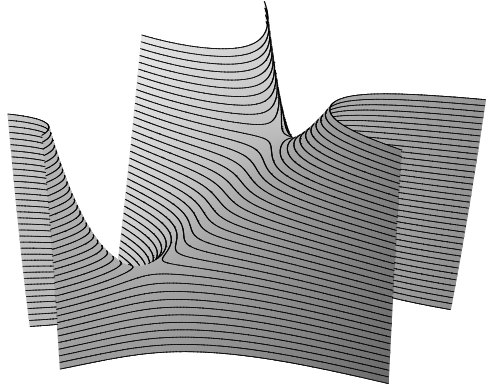}
\end{center}
\caption{Scalar curvature functional on $\SU(4)/\Un(2)$.}
\end{figure}
In this case, the set $B$ consists of three points $g_1,g_2,g_3$,
such that $(g_1,\sc(g_1)),\,(g_2,\sc(g_2))$ and $(g_3,\sc(g_3))$
are located on the three connected components of the highest visible energy
level. The base point $(Q,\sc(Q))$ is a point on a lower
energy level (say below the lower saddle point), and the cone $A_B$ 
corresponds to a union
of three curves joining $(Q,\sc(Q))$ with $(g_i,\sc(g_i))$, $i=1,2,3$.
When we apply the gradient flow $\Phi$ to the cone
$A_B$, here the points $p_s\in \Phi_s(A)$, $s\geq 0$, having the lowest
energy among all $p \in \Phi_s(A)$,  converge, as $s\to +\infty$,
to the lower saddle point.

\begin{definition}\label{def:coindex}
The {\em coindex} of a critical point
$g\in \MGo$ is $m(g)=\dim(W_{>0})$, the dimension of the maximal subspace $W_{>0}$ of $T_g\MGo$ on which the Hessian of $\sc:\MGo\to\RR$ is positive definite. The {\em augmented coindex} of $g$ is 
$m^*(g)=\dim(W_{\geq 0})$, the dimension of the maximal subspace $W_{\geq 0}$ on which the 
Hessian is positive semi-definite.
\end{definition}

Recall that given a topological space $X$ and a field ${\mathbb F}$, and for $q \in \NN_0$, 
the reduced homology group $\tilde H_q(X,{\mathbb F})$ 
of $X$ with coefficients in ${\mathbb F}$ equals the homology group
$H_q(X,{\mathbb F})$ for $q \geq 1$ and satisfies $H_0(X,{\mathbb F})=\tilde H_0(X,{\mathbb F}) \oplus 
{\mathbb F}$. This implies that $X$ cannot be contractible
if $\tilde H_*(X,{\mathbb F}) \neq 0$.

In what follows we apply a version of \cite[Corollary (13)]{G.N}, which we
restate here in our context for convenience. Notice that we state a version for only a finite-dimensional space and with no
group action involved, and also that instead of the given functional $f$ we consider $-f$,
so that we have replaced infimum by supremum and so on.

\begin{theorem}\cite[see Corollary (13)]{G.N}\label{thm:gn}
 Let $M$ be a $C^2$-smooth Riemannian manifold
and $f:M\to \RR$ be $C^2$-smooth. Let $B$ be a compact subset of $M$ and 
suppose that there exists
a homological family $\Fl(\alpha)$ with boundary $B$ of dimension $q+1$. Set
$$ 
 c:=\sup_{A\in \Fl(\alpha)} \big\{\,\inf\{f(p)\mid p \in A\}  \,\big\} \in \RR\,.
$$ 
Suppose now that $\min \{f(b) \mid b \in B\}>c$ and that $f$ satisfies
the Palais-Smale condition (C) at the level $\{f =c\}$.
Then there exists a critical point $p$ of $f$ with $f(p)=c$ and $m^*(p) \geq  q+1$.
\end{theorem}

A family $\Fl(\alpha)$ of compact subsets $A$ of $M$ is said to be a homological family of dimension $q+1$ with
boundary $B$ if for some non-trivial class $\alpha \in H_{q+1}(M,B)$
we have that
$$
 \Fl(\alpha)=\{A \subset M \mid \alpha \textrm{ is in the image of } i_*:H_{q+1}(A,B)\to H_{q+1}(M,B)\}
$$
where $i_*$ is the homomorphism induce by the inclusion $i:A\to M$. Notice that $B \subset A$ for all $A\in\Fl(\alpha)$.

\begin{lemma} \label{lemcoindex}
Let $G/H$ be a compact homogeneous space. Suppose that for
a field ${\mathbb F}$
and for $q\in \NN_0$ we have $\ti H_q(\XGH,{\mathbb F})\neq 0$.
Then $G/H$ carries an Einstein metric $g\in \MGo$ whose
augmented coindex $m^*(g)$ is greater than or equal to $q+1$.
\end{lemma}

\begin{proof} We will invoke Theorem \ref{thm:gn}.
Let $B$ be as in the proof of Theorem \ref{theomain}
and let $\alpha'$ be a non-trivial class in $\ti H_q(B)\cong \ti H_q(\XGH)$. Since $\MGo$ is contractible,
 the long exact homotopy sequence for the pair $(B,\MGo)$ yields an
isomorphism $j_*: H_{q+1}(\MGo,B)\to  \ti H_q(B)$ \cite[2.16]{Hatch}. 
Consequently, we have $\alpha:=(j_*)^{-1}(\alpha')$, a non-trivial class
in $H_{q+1}(\MGo,B)$. For compact sets $A$ with
$B\subset A\subset \MGo$,
let $i:A\to \MGo$ denote the inclusion map and let 
$$
 i_*:H_{q+1}(A,B)\to H_{q+1}(\MGo,B)
 $$ 
 denote the induced relative homology map. As above we define the homological family
\beg
  \FF (\alpha)=\{A \subset \MGo \mid A \textrm{ compact},\,\,
                 B\subset A \textrm{ and } \alpha \in {\rm Im}(i_*)  \}
\en 
with boundary $B$ of dimension $q+1$. 
We have $\FF (\alpha)\neq \emp$,
since, for $A_B$ as defined in the proof of Theorem \ref{theomain},  $A_B\in {\mathscr F}(\alpha)$. 
Both $\MGo$ and $A_B$ are contractible, thus  $\alpha' \in \ti H_{q}(B)$ can be
viewed as a class in $ H_{q+1}(A_B,B)$, 
and $i_*$  maps this class onto $\alpha$  \cite[2.16]{Hatch}. 
Let
\beg
  c:=\sup_{A\in \FF(\alpha)}\,
  \big\{\inf \,\,\{\,\sc (g)\mid g \in A\}\,\big\}\,.
\en 
From the proof of Theorem \ref{theomain} we deduce $c>0$.
Moreover we claim $c\leq \sc_+$. To see this, suppose not:  
there exists a set $A\in
\FF(\alpha)$ with $A\subset \{\sc \geq \sc_+\} \subset
C(\USi,\bar t(\delta))$ such that $\alpha=i_*(\hat \alpha)$ for some 
$\hat\alpha \in  H_{q+1}(A,B)$. Let $i^1:A\to C(\USi,\bar t(\delta))$
and $i^2:C(\USi,\bar t(\delta))\to \MGo$ denote inclusion maps. Then
we have $i_*=i^2_* \circ i^1_*$, thus $\alpha =i_*^2(i_*^1(\hat \alpha))$. But since we showed in the 
proof of Theorem \ref{theomain} that $B$ is a strong deformation retract of
$C(\USi,\bar t(\delta))$, we conclude 
$ H_{q+1}(C(\USi,\bar t(\delta)),B)=0$ by \cite[p.118]{Hatch}.  
As a consequence,  $i_*^1(\hat \alpha)=0$,
contradicting $\alpha \neq 0$. 

By \eqref{Bscalp} and the compactness of $B$ we deduce $\min \{\sc(g) \mid g \in B\}>\sc_+ \geq c$.
Moreover, by Theorem A in \cite{BWZ} the Palais-Smale condition (C) holds for $\sc:\MGo \to \RR$ 
at the level $\{\sc=c\}$, since $c>0$. Now, by Theorem \ref{thm:gn}, the claim follows.
\end{proof}

\medskip

We turn to the question of when the existence of more than one critical point can
be guaranteed.
Since the corresponding Lyusternik-Schnirelmann-theory is
well known, we give a brief presentation of this part, essentially following
\cite[pp.225-226]{BTZ}.

\begin{lemma}\label{lemlst}
Let $G/H$ be a compact homogeneous space. Then the scalar curvature
functional has at least as many  critical points as the cup length of $\XGH$.
\end{lemma}

\begin{proof} We repeat the proof described in \cite{BTZ},
with one minor modification. Consider
$$
 f:(\MGo, L^2) \to \RR\,\,;\,g \mapsto -\sc (g)
$$
(to adjust notation) and the continuous deformation $\Phi
:\MGo\times\RR \to \MGo$ with $\Phi$ as in the proof of Theorem
\ref{theomain}. Then we have $f(\Phi_s(q))\leq f(q)$ for all $q\in
\MGo$ and $s\geq 0$. By \cite{BWZ} we obtain conclusion ($\star$)
in \cite{BTZ} (the Palais-Smale condition (C))  for $\kappa<0$. We
define the subsets $Y=\MGo$ and $Z=\{f \leq -\sc _+-1\}$ of
$\MGo$. Both $Y$ and $Z$ are invariant under the deformation
$\Phi$.

Let $h\in \ti H_*(\XGH)$ be non-zero. We may view
$h$ as a non-trivial class in $H_*(Z)$. Therefore, $h$ corresponds
to a non-trivial relative class in $H_*(Y,Z)$, again denoted by
$h$. (Here, we consider relative cohomology classes, since $Y$
is contractible.)
Conclusion (1.2) on page 255 in \cite{BTZ} follows, since
there exists a cycle $z$ in $h$, such that $\vert z \vert \subset
\{f<0\}$ (cf.~proof of Theorem \ref{theomain}). Now let
$h_1,h_2\in H_*(Y,Z)$ be non-zero, and suppose that there exists a
cohomology class $\omega \in H^{\geq 1}(Y,Z)$ such that $\omega
\cap h_2 = h_1$; we say that $h_1$ is subordinate to $h_2$.
 The claim follows now as in \cite{BTZ}.
\end{proof}

\medskip

Note that in general the above lemma does not imply that these distinct
critical points are non-isometric.

\begin{corollary}\label{corlst}
Let $G/H$ be a compact homogeneous space. Suppose that there  exist only
finitely many (different but possibly isometric) Einstein metrics in
$\MGo$. Then $G/H$ admits at least 
as many non-isometric Einstein metrics as the cup length of
$\XGH$.
\end{corollary}

\begin{proof}
Suppose we have $h_1, h_2 \in H_*(Y)$ and there exists a $k$-form $\omega  \in H^*(Y)$ of degree $d\geq 1$ such that $h_1$ is subordinate to $h_2$. 
 In \cite[1.3]{BTZ} they prove an an inequality for the energy on the two critical points you get from $h_1$ and $h_2$:
 $$\kappa(h_2) \geq \kappa(h_1).$$
Here, we set $X=Y=\MGo$, our energy function is $\sc$, and $\kappa(h)$ denotes the infimum  over all cycles $z\in h$ of the maximum of $\sc$ on $|z|$.
Since our hypothesis is that the set $K$ of critical points of $\sc$ 
is finite, the set $K_\kappa$ of critical points with $\sc=\kappa$ is also finite.  Thus we may assume that an open neighborhood $U$ of $K_\kappa$ is a finite union of open (disjoint) disks in $\MGo$.  Moreover,  in  \cite[1.3]{BTZ}   they prove that  if $\kappa(h_2) = \kappa(h_1)$, the restriction of $\omega$ to $U$ is nonzero. 
However, for a $k$-form $\omega  \in H^*(Y)$ of degree $d\geq 1$,
 the restriction of $\omega$ to a disk must vanish. Therefore, $\kappa(h_2) = \kappa(h_1)$ is impossible; we must have strict inequality.
\end{proof}

The hypothesis of the corollary is that the set of critical points of the scalar curvature functional is finite. We note that for homogeneous spaces $G/H$ with ${\rm rk}\,
G={\rm rk}\, H$, it is conjectured  that this is true: see \cite[p.683]{BWZ}.

\begin{remark}%%%% added 2023
There exist (many) homogeneous spaces $G/H$ such that infinitely many isometric
$G$-invariant metrics $g \in \MGo$ are solutions to the Einstein equation:
see Remark \ref{rem:familiesofsolutions}.
This happens whenever $\dim N_G(H)\geq \dim H$ and an Einstein metric $g_E\in \MGo$ is not
fixed by the action of $N_G(H)$ on $\MGo$: see \eqref{eqn:NGH}. 
Notice however, that at this moment
there is no compact homogeneous space $G/H$ known, admitting infinitely many non-isometric Einstein metrics $g \in \MGo$ with $\sc(g)>0$.
 
There exist compact homogeneous spaces $G/H$ admitting many non-isometric,
$G$-invariant Einstein metrics: for instance $\SO(n)$, $n\geq
12$, admits at least $n$ non-isometric Einstein metrics \cite[p.62]{DZ}. All but the biinvariant metric
$Q$ come in high-dimensional families. Showing that two homogeneous metrics $g_1,g_2 \in \MGo$ 
are non-isometric is in general quite difficult if $\sc(g_1)=\sc(g_2)$. 
For left-invariant metrics on simple, compact Lie groups see  \cite[Thm.3]{DZ}.
\end{remark}

%%%%%%%%%%%%%%%%%%%%%%%%%%%%%%%%%%%%%%%%%%%%%%%%%%%%%%%%%%%%%%%%%%%%%%%%%
%%%%%%%%%%%%%%%%%%%%%%%%%%%%%%%%%%%%%%%%%%%%%%%%%%%%%%%%%%%%%%%%%%%%%%%%%

\section{The nerve of $G/H$ and its invariants}
Here we will first define the nerve $\XGH$ and give several examples. We next define the simplicial complex $\DGH$. When  $\dim N_G(H)=\dim H$, the nerve $\XGH$ and the simplicial complex $\DGH$ agree.  We describe some topological properties of $\DGH$ and some examples. 
%%%%%%%%%%%%%%%%%%%%%%%%%%%%%%%%%%%%%%%%%%
\subsection{{\it The nerve of $G/H$}}\label{sec:nerve} 
In this section we define the nerve $\XGH$ of a compact homogeneous space. It can be viewed as 
the space of $G$-invariant fibrations of $G/H$ with non-toral leaves. 
Recall that $\Psn$ denotes the set of projections $P$ such that $\ker(P)$ 
is a non-toral subalgebra $\kf$ of $\g$.

\begin{definition}[Nerve of $G/H$]\label{def:XGH}
Let
\begin{align*}
 \XGH =
    \big\{ B \in \SymgH &\mid
  B=\sum_{i=1}^r \lambda_i \cdot P_i   \textrm{ for some }  P_1,\dots,P_r \in \Psn\\
	&
	\quad\textrm{ s.t. } \ker(P_1) \subsetneq \cdots \subsetneq \ker(P_r)
  \textrm{ and }  \sum_{i=1}^r \lambda_i=1,\,\, \lambda_i\geq 0 \big\}\,.
\end{align*}
\end{definition}

The condition $\ker(P_i)\subsetneq \ker(P_{i+1})$ means that for the corresponding subalgebras $\kf_i,\kf_{i+1}$ 
we have $\kf_i < \kf_{i+1}$. Moreover, in the definition of $\XGH$ we may assume that $\lambda_1,\dots,\lambda_r>0$.

\begin{lemma}\label{lem:XGHXXGH}
The nerve $\XGH$ and $\XSGH$ are homeomorphic.
\end{lemma}

\begin{proof}
Since $\XSGH$ and $\XXGH$ are homeomorphic, it suffices to show that
$$
  h:\XGH\to \XXGH\,,\,\,B\mapsto \tfrac{B}{\tr B}
$$ 
is a homeomorphism. Clearly, $h$ is continuous and surjective. We show $h$ is also injective. Suppose now that 
$$ 
 h(B)=\sum_{i=1}^r \lambda_i \cdot P_i =\sum_{j=1}^{\tilde r} \tilde \lambda_j \cdot  \tilde P_j=h(\tilde B)
$$
for $P_1,\dots,P_r,\tilde P_1,\dots,\tilde P_{\tilde r} \in \Psn$ and $0< \lambda_1,\dots,\lambda_r, \tilde \lambda_1,\dots, \tilde \lambda_{\tilde r}$ with 
$$
  \sum_{i=1}^r \lambda_i \cdot \tr(P_i)=\sum_{j=1}^{\tilde r} \tilde \lambda_j \cdot  \tr(\tilde P_j)=1\,.
$$
We show $B = \ti B$. Since $\ker(P_1) \subsetneq \ker(P_2)$ and so on (in the expressions for $h(B)$ and $h(\ti B)$), and since both $\lambda_1,\tilde \lambda_1>0$, we see $\ker(P_1)=\ker(\tilde P_1)$. 
Thus $P_1=\tilde P_1$ and $\lambda_1=\tilde \lambda_1$. Now by induction, $B=\ti B$ and the claim follows.
\end{proof}

In our construction of the nerve $\XGH$, there may be some redundancies;  that is, we may consider more intermediate
subalgebras than we need. 

\begin{definition}[\cite{Gr}]
Let $\kf$ be a subalgebra of $\g$ with $\h <\kf <\g$. Then we call $\kf$
{\em almost semisimple} if $\kf=[\kf,\kf]+\h$.
\end{definition}

Notice that this is equivalent to  $\kf=\kf_s + \z(\h)$, where $\kf_s$ is the semisimple component of $\kf$, 
and $\z(\h)$  is the center of $\h$. As a consequence, the  projection from $\z(\kf)$ 
onto $\z(\h)$ is surjective. Therefore, if $\h$ is semisimple, then an almost semisimple subalgebra $\kf$ must
be semisimple as well. By contrast, toral subalgebras $\tf=\h \oplus \af$, $\af \subset \m$ abelian,
are never almost semisimple.

For each $\kf \in \Sub$, let  $[\kf]$ denote the connected component of $\kf$ in the subset of all $\dim(\kf)$-dimensional intermediate subalgebras of $\g$. 
These were  investigated in \cite{BWZ}.

\begin{lemma}\label{lem:assfinite}
Up to conjugation there exist only finitely many almost semisimple subalgebras $\kf$ with $\h < \kf <\g$.
Moreover
$$
  [\kf]=(N_G(H))_0\cdot \kf\,.
$$
\end{lemma}

\begin{proof}
By \cite{BWZ}, up to conjugation there exist only finitely many semisimple subalgebras $\kf_s$ of $\g$
with $\h < \kf_s$. Thus, up to conjugation there exist only finitely many almost semisimple subalgebras $\kf$ of $\g$.
The second claim follows from \cite[Proposition 4.2]{BWZ}.
\end{proof}

It follows that $[\kf]$ is a compact homogeneous space.

\begin{lemma}\label{lem:genass}
Let $\kf,\tilde \kf$ be almost semisimple. Then $\langle \kf,\tilde \kf\rangle$ is also almost semisimple.
\end{lemma}

\begin{proof}
Using that $\kf=\kf_s + \z(\h)$, $\tilde \kf=\tilde \kf_s + \z(\h)$ and $[\z(\h),\z(\h)]=0$
we have
\beg
 [\kf,\tilde \kf]
  &=&
	 [\kf_s,\tilde \kf_s]+[\kf_s,\z(\h)]+[\tilde \kf_s,\z(\h)]\\
	&\subset &
	 [\kf_s,\tilde \kf_s]+\kf + \tilde \kf \\
	&=&
	  [\kf_s,\tilde \kf_s]+\kf_s + \tilde \kf_s +\z(\h)\,. 
\en
Thus $\langle \kf,\tilde \kf\rangle=\langle \kf_s,\tilde \kf_s\rangle + \z(\h)$. Since 
$\langle \kf_s,\tilde \kf_s\rangle$ is semisimple, the claim follows.
\end{proof}

\begin{definition}\label{def:min-sub}
We call a non-toral subalgebra $\kf \in \Sub_s$ {\em minimal} if there exists no $\ti \kf \in \Sub_s$
with $\ti \kf < \kf$.  Let 
\begin{eqnarray*}
 \Submin &=&\{\kf \in \Sub_s \mid \kf \textrm{ is a minimal subalgebra}\}\,, \\
  \Submg &=& \{ \kf \in \Sub_s \mid \kf \textrm{ is generated by finitely many minimal subalgebras}\}\,,\\
  \Subass &=& \{ \kf \in \Sub_s \mid \kf  \textrm{ is almost semisimple}\}\,.
\end{eqnarray*}	
\end{definition}

Notice that for a minimal subalgebra $\kf$, $\h$ does not have to be maximal in $\kf$. An example is
$G/H=\SU(4)/\SU(2)$, where $\SU(2)$ is embedded into $\SU(4)$ as a lower $(2 \times 2)$-block.
Then $\su(3)$ is minimal, where $\SU(3)$ is embedded into $\SU(4)$ as a lower $(3 \times 3)$-block,
but $\h=\su(2) \subset \un(2) \subset \kf:=\su(3)$.

The next lemma is crucial for applications.

\begin{lemma}\label{lem:minareass}
Let $\kf\in \Sub_s$ be a minimal subalgebra. 
Then $\kf$ is almost semisimple.
\end{lemma}

\begin{proof}
We write $\h=\h_s \oplus \z(\h)$ and $\kf=\kf_s \oplus \z(\kf)$.
Since $\kf$ is non-toral we have $\h_s < \kf_s$. Let $\af$ denote
the natural projection of $\z(\h)\subset \kf$ onto 
$\z(\kf)$ with respect to the decomposition $\kf=\kf_s \oplus \z(\kf)$.
If $\af=\z(\kf)$, then $\kf=\kf_s + \z(\h)$ and the claim follows.
If $\af < \z(\kf)$, then $\tilde \kf=\kf_s \oplus \af$ is a non-toral subalgebra with
$\h < \tilde \kf < \kf$. This is a contradiction, since $\kf$ is minimal.
\end{proof}
\begin{remark}\label{rem:intersections}
We note that the classes of minimal (or non-toral or almost semisimple) subalgebras are not  closed under intersections, generally (when $\dim H < \dim N(H)$).  Let $\g = \so(n+4)$, let $\h = \so(n) < \so(n) \oplus \so(4)$,  and consider $\kf_1 = \so(n)\oplus \so(3)$ where $\so(3) \subset \so(4)$ is the upper-left block, and $\kf_2 =  \so(n)\oplus \so(3)'$ where $\so(3)' \subset \so(4)$ is the lower-right block. While $\kf_1$ and $\kf_2$ are each non-toral subalgebras containing $\h$,  their intersection $\kf_1 \cap \kf_2 = \so(n) \oplus \so(2)$ is toral. 
\end{remark}

Recall, in Definition \ref{def:proj}, we defined the set $\Psn$ of projections $P \in \SymgH$ for which the kernel is  $\ker(P)\in \Sub_s$, a non-toral subalgebra with $\h < \kf < \g$. 
 In Lemma \ref{lem:Ps} we proved that 
$\Psn$ is compact and semi-algebraic.

\begin{definition}\label{def:Psnm}
We describe the following subsets of $\Psn$:
\begin{align*}
 \Psnm &=\{ P \in \Psn \mid  \kf=\ker(P) \in \Submin \}\,,\\
  \Psnmg&=\{ P \in \Psn \mid \kf=\ker (P) \in \Submg \}\,,\\
  \Psnass&=\{ P \in \Psn \mid \kf=\ker (P) \in \Subass \}\,.
\end{align*}
\end{definition}
Note that $\Psnm \subset \Psnmg \subset \Psnass$ and $\Psnmg$ is a connected component of $\Psnass$.  
\begin{lemma}\label{lem:Psnm}
The sets $ \Psnm$,  $\Psnmg$, and $\Psnass$ are compact and semi-algebraic. 
\end{lemma}

\begin{proof}
To see that $  \Psnass$ is compact, we note that by Lemma \ref{lem:assfinite}, up to conjugation by $(N_G(H))_0$ there exist at most finitely many almost semisimple subalgebras $\kf \in\Sub$.  
Since there exist only finitely many semisimple subalgebras $\kf_s$ of $\g$ with $\h_s<\kf_s$,
we know that $ \Psnm$  is compact.  
Using that for an almost semisimple Lie algebra $\kf$ we
have $\kf=\kf_s \oplus \z(\h)=\kf_s + \h$,
we observe that 
$$
 \Psnass = \big\{ 
   P \in \Psn \mid\ker(P) = \{ [(\Id_{\g}-P)(X),(\Id_{\g}-P)(Y)]+Z \mid X,Y \in \g,\,,\,Z \in\h\} \big\}\,.
$$
This shows $\Psnass$ is semi-algebraic, and since
$$
  \Psnm:=\{P \in \Psn \mid \not \exists\tilde P\in \Psn  \textrm{ such that } \ker(\tilde P) \subsetneq \ker(P)\}\,
$$
we have that $\Psnm$ is also semi-algebraic.  
Since $\Psnmg$ is a connected component of the compact and semi-algebraic set $\Psnass$, we conclude
$\Psnmg$ is compact and semi-algebraic (see Section \ref{sec:semialgebraic}).
\end{proof}

This yields the following: 

\begin{definition}
We let  $\XGHass$ (resp.~$\XGHmin$) denote the {\em almost semisimple nerve} (resp.~  {\em minimal nerve}) of $G/H$, defined analogously to $\XGH$, except 
instead of $P\in \Psn$ we require $P \in \Psnass$ (resp.~$P \in \Psnmg$): 
\begin{align*} \XGHass :=
    \big\{ B \in \SymgH 
		&\mid
  B=\sum_{i=1}^r \lambda_i \cdot P_i   \textrm{ for some } P_1,\dots,P_r \in \Psnass\,,\\
	&
	\quad\textrm{ s.t. } \ker(P_1) \subsetneq \cdots \subsetneq \ker(P_r), \,\,
	\sum_{i=1}^r \lambda_i=1\,,\,\, \lambda_i\geq 0  \big\}\,, \\
  \XGHmin =
    \big\{ B \in \SymgH 
		&\mid
  B=\sum_{i=1}^r \lambda_i \cdot P_i   \textrm{ for some } P_1,\dots,P_r \in \Psnmg\,,\\  
	&
	\quad \textrm{ s.t. } \ker(P_1) \subsetneq \cdots \subsetneq \ker(P_r),\,\, 
	\sum_{i=1}^r \lambda_i=1\,,\,\,	\lambda_i\geq 0,   \big\}\,.
\end{align*}
\end{definition}

In our definition of ${\dsp \WWs= \bigcup_{{\blds\kf} \in \Sub_s}\D(\kf)}$ (Definition \ref{defin-W}), 
 it is enough to consider the union of discs $\D(\kf)$ where $\kf $ is minimal, 
 by Lemma \ref{lem:disk-inclusion}.

\begin{proposition}\label{prop:nerv}
Let $G/H$ be a compact homogeneous space. Then the nerve $\XGH$, the almost semisimple nerve $\XGHass$,
and the minimal nerve $\XGHmin$ are homotopy equivalent.
\end{proposition}

\begin{proof}
First we show that $\XGH$ and $\XGHass$ are homotopy equivalent
by applying the same methods as in Section \ref{sec:homotopy2}.

In Section \ref{sec:homotopy2}, we denoted by $\{i_1,\dots,i_d\}$
 the set of dimensions of subalgebras in $\Sub_s$.  
For $j=1,\dots,d$, we defined the compact, semi-algebraic sets
${\dsp \XXs^j= 	\bigcup \Br[\varphi]}$, taking the union over all flags $\varphi \in \Fln$, where $\Fln$ is the set of all flags $\varphi$ consisting of non-toral intermediate subalgebras, $\kf\in\Sub_s$, of flag height $h(\varphi) \geq  i_j$: $h(\varphi)=\dim(\max(\varphi))$.  

The set of dimensions of subalgebras in $\Subass$ is a subset 
$\{i_1',\dots,i_{\tilde d}'\}$ of $\{i_1,\dots,i_d\}$.   
For simplicity let us denote them by  $\{i_1,\dots,i_{\tilde d}\}$.
We can define the analogous sets 
$$
\XXas^j = \bigcup_{\varphi \in \Flnas, ~h(\varphi) \geq  i_j} \Br[\varphi]
$$
to be unions of butterflies 
of flags $\varphi \in \Flnas$, $j=1,\dots,\tilde d$ and
$$
  \XXas^{\tilde d+1}:= \bigcup_{\varphi \in \Flnas, ~h(\varphi) =\dim \bldsg\g} \Br[\varphi] = \qquad \bigcup_{\varphi \in \Flnas} \Delta_{\varphi}\,.
$$
Here we denote by  $\Flnas$ the set of all flags $\varphi=(\kf_1 < \cdots < \kf_r)$ with
$\kf_i \in \Subass$ for $1 \leq i \leq r-1$ and $\kf_r \in \Subass \cup \{\g\}$.

One key property of almost semisimple subalgebras is that a sequence
$(\kf_i)_{i \in \NN}$ of almost semisimple subalgebras of the same dimension will
subconverge to a limit subalgebra $\kf_\infty$ which is again almost semisimple: see Lemma \ref{lem:Dkfi}.
This follows from the identity $\kf_i=(\kf_i)_s +\h$. Moreover, this shows a second key property,
namely, that a sequence 
of flags in $\Flnas$ subconverges to a limit flag in $\Flnas$: see Lemma \ref{lem:sequsubalgsub}.
A third  key property of these flags is that
if $\varphi, \ti\varphi \in \Flnas$, then the product $\varphi\ti\varphi$ is also in $\Flnas$. 
It follows that the intersection of two butterflies $\Br[\varphi] \cap \Br[\tilde \varphi]
=\Br[\varphi\ti \varphi]$ and $\varphi\ti\varphi \in \Flnas$ if $\varphi, \ti\varphi \in \Flnas$.

In this way, as in Section \ref{sec:homotopy2}, we obtain nested sets 
$$
 \WWs = \XXas^1 \supset \XXas^2 \supset \cdots \supset \XXas^{\tilde d+1}=\XXGHass\,.
$$ 
It is straightforward to adapt the construction of the
homotopies between $\XXs^j$ and $\XXs^{j+1}$ in Section \ref{sec:homotopy2}
to obtain homotopies between $\XXas^{j}$ and $\XXas^{j+1}$, now for $j=1,\dots,\tilde d$.
We conclude that $\WWs$ is homotopy equivalent to $\XXGHass$. 
As in Lemma \ref{lem:XGHXXGH} one can now show that $\XXGHass$ and $\XGHass$ are 
homeomorphic. Thus $\XGH$ and $\XGHass$ are homotopy equivalent.

 The claim that $\XGH$ and $\XGHmin$ are homotopy equivalent follows precisely the same way.
Notice that the subalgebras generated by minimal subalgebras satisfy the first key property mentioned above.
This follows from Lemma \ref{lem:assfinite}, which guarantees that there exist only
finitely many components $[\kf]$ of almost semisimple subalgebras such that $\h < \kf <\g$, and that
$[\kf]=(N_G(H))_0 \cdot \kf$. This shows that if a sequence $(\kf_i)_{i \in \NN}$ of almost
semisimple subalgebras converges to $\kf_\infty$, then, for all large $i$,  $\kf_i$ is conjugate to $\kf_\infty$
 by an element in $N_G(H)$. This in turn shows that a sequence
of minimal subalgebras will subconverge to a minimal subalgebra, and that the same is true for 
a sequence of subalgebras generated by finitely many minimal ones. Furthermore, the second key property, 
mentioned above, now follows immediately. The third key property follows from the fact that if $\kf_1,\kf_2$ are two
subalgebras of $\g$, each generated by finitely many minimal subalgebras, then $\langle \kf_1,\kf_2\rangle$
is also a subalgebra generated by finitely many minimal ones.
\end{proof}

%%%%%%%%%%%%%%%%%%%%%%%%%%%%%%%%%%%%%%%%%%%%%%%%%%%%%
\subsection{{\it Examples of the nerve $\XGH$}} \label{sec:exanerve}

In this section we will provide several examples of compact homogeneous
spaces $G/H$, some with a  non-contractible nerve $\XGH$ and others with a contractible nerve $\XGH$. Recall that by
Proposition \ref{prop:nerv}, we know
$\XGH$, $\XGHass$ and $\XGHmin$ are homotopy equivalent.

\begin{example}\label{ex:SO(n+k)/SO(n)}
Let $n\geq 3$, $k\geq 2$ and
$G/H=\SO(n+k)/\SO(n)$ where $\SO(n)$ is embedded into
$\SO(n+k)$
as a lower right $(n \times n)$-block.
Then there exists a non-contractible cell $A \subset \XGH$
which is homeomorphic to the sphere $S^{\frac{1}{2}k(k+1)-2}$.
\end{example}

\begin{proof}
The isotropy representation of $G/H$ is given by $\m=\m_0 \oplus
\bigoplus_{i=1}^k \m_i$ where
$\m_0=\so(k)$ and $\m_1,...,\m_k$ are equivalent $\SO(n)$-modules,
isomorphic to the
standard representation of  $\SO(n)$ on $\RR^n$.
We consider all flags $\varphi=(\kf_1 < \cdots < \kf_{k-1})$ where
$\kf_i \cong \so(n+i)$, $1\leq i \leq k-1$.
Given such a flag $\varphi$, for each $i=1,\dots,k-1$, let $P_i \in
\Psn$ denote the projection corresponding to $\kf_i$. We consider the
cell $A$ to be the union of all simplices
$$
   \dsp{\rm conv}
\left\{B=\sum_{i=1}^{k-1} \lambda_i \cdot P_i
   \mid \sum_{i=1}^{k-1}  \lambda_i=1, \lambda_i \geq 0\right \} \subset \XGH\,,
$$
and we claim that $A$ is homeomorphic to $S^{\frac{1}{2}k(k+1)-2}$.

To prove this, we define a homeomorphism $h$ between the two sets.
Consider an orthonormal basis $(e_1,...,e_k)$ 
corresponding to
$(P_1,...,P_{k-1})$.  E.g., $\kf_1 \cong \so(n+1)$, with $\RR^{n+1}=\RR^n
\oplus \langle e_1\rangle_{\bldss{\RR}}$. 
We start with an element $B$ in $A$, with corresponding scalars $0\leq
\lambda_i \leq 1$, $i=1,...,k-1$,
with $\sum \lambda_i=1$. 
The basis $(e_1,\dots,e_k)$ will be the eigenbasis of the symmetric $(k
\times k)$-matrix $h(B)$
with trace zero, norm one and ascending eigenvalues.
To achieve this we will now define the eigenvalues
$\ti \lambda_1 \leq \dots \leq \ti \lambda_k$ of $h(B)$ with
$\sum_{i=1}^k \ti \lambda_i=0$ and $\sum_{i=1}^k \ti \lambda_i^2=1$.
We first map
$(\lambda_1,\dots,\lambda_{k-1})$ onto
  its trace-free projection $(\tilde \lambda_1',\dots,\tilde \lambda_k')$
   with
\begin{align*}
\tilde \lambda_j':= \sum_{i=1}^{j-1}  \lambda_j -\tfrac 1k \cdot
\sum_{i=1}^{k-1}\lambda_{k-i}\cdot i \,,
\end{align*}
$1 \leq j \leq k$. Notice that $\sum_{i=1}^k \tilde \lambda_i'=0$ and that
$\tilde \lambda_1' \leq \dots \leq \tilde \lambda_k'$. Finally, 
projecting to the unit sphere in  $\RR^k$, we obtain
$(\ti \lambda_1,\dots,\ti \lambda_k)$.

The inverse map $g$ is given as follows: Let $S$ be a symmetric,
traceless $(k\times k)$-matrix
of norm one with eigenvalues $(\ti \lambda_1,\dots,\ti \lambda_k)$, $\ti
\lambda_1 \leq \dots \leq \ti \lambda_k$
and corresponding eigenbasis $(e_1,\dots,e_k)$ of $\RR^k$. As above we
obtain the corresponding
projections $P_1,\dots,P_{k-1}$. The scalars $\lambda_i$, $1\leq i \leq
k-1$, are defined by
$$
   \lambda_i:=\tfrac{\ti \lambda_{i+1}-\ti \lambda_i}{\ti \lambda_k -\ti
\lambda_1}\,.
$$
It is easy to check that $0 \leq \lambda_1,\dots, \lambda_{k-1} \leq 1$ and
that $\sum_{i=1}^{k-1} \lambda_i=1$.

Now let $B$ be a point in $A$ with generic eigenvalues $0< \lambda_1 <
\dots< \lambda_{k-1}$,
$\sum \lambda_i =1$. Let $(B_m)_{m \in \NN}$ be a sequence in $\XGH$
with ${\dsp \lim_{m\to \infty} B_m=B}$.
We claim that for all large $m$, we have $B_m \in A$. Suppose this
is not the case.
By passing to a subsequence, we may assume that $B_m \not \in A$ for
all $m$.
Passing to a further subsequence, we may assume that each of the corresponding sequences of scalars
and projections defining  
 $B_m$ converges. It follows that the  eigenspace of the smallest eigenvalue of
$B$, given by some
$\so(n+1)$ since $0<\lambda_1$, will be approximated by a sum of
eigenspaces of the $B_m$,
and that sum of eigenspaces must be a subalgebra. Now since $\so(n)$ is maximal in $\so(n+1)$
it follows that
this smallest eigenspace is given by $\so(n+1)_m \in [\so(n+1)]=\SO(k)\cdot \so(n+1)$ converging
to $\so(n+1)$. Inductively the above claim follows.

Since $A$ is homeomorphic to $S^{\frac{1}{2}k(k+1)-2}\!$, it is non-contractible in $\XGH$ by Lemma \ref{lem:example}.
\end{proof}

\begin{lemma}\label{lem:example}
Let $X$ be a topological space and let $A\subset X$ be a compact $q$-dimensional 
topolo\-gical submanifold which is orientable over some
coefficient ring $R$.
If there exists a point $a_0\in A$ and
an open neighborhood $U_1$ of $a_0$ in $X$ with $U_1 \subset A$,
then $H_q(X) \neq 0$.
\end{lemma}

\begin{proof}
Let $[A] \in H_q(A) \bs \{0\}$ denote the orientation class of $A$.
For $i:A \hookrightarrow X$, the embedding of $A$ into $X$,
it suffices to show  that $[c]:=i_*([A]) \neq 0$ in $H_q(X)$. For a contradiction, suppose
we have some $z \in S_{q+1}(X)$ 
such that $\partial z =c$. Let $U_2 := X \bs \{a_0\}$. 
Then there exist chains $u_i\in
S_{q+1}(U_i)$, 
with $\tilde z = u_1 + u_2$ and $\partial z=\partial \tilde z$
(we use a barycentric subdivision).
Thus from $c =\partial \tilde z$ we deduce
    $$
\partial u_2= c - \partial u_1\,.
$$
Since $u_1 \in S_{q+1}(U_1)$ and $U_1 \subset A$, we know 
$\partial u_1\in S_q(U_1) \subset
S_q(A)$.
Thus the chain on the right hand side  is in $S_q(A)$.
On the left, 
$\partial u_2 $ is a chain in $S_q(U_2)$. We let  $A':=U_2 \cap A= A \bs \{a_0\}$ and we know 
we have $\partial u_2 \in S_q( A')$, and consequently, 
 $c-\partial u_1 \in S_q(A')$. 
Now $[c]=[c-\partial u_1]\neq 0$ in $H_q(A)$ but we must have some $z' \in S_{q+1}(A')$
with $\partial z'=c-\partial u_1 $, because,
since $A'$ is a non-compact $q$-dimensional manifold, $H_q(A')=0$.
Thus $[c]=0$ in $H_q(A)$, a contradiction. Thus $H_q(X) \neq 0$.
\end{proof}

\begin{remark}
For  $G/H=\SO(n+2)/\SO(n)$ we have $\XGH=S^1=S^1\cdot \so(n+1)$.
Moreover, in this case
$\n(\h)$ is toral and a maximal subalgebra of $\g$ and
$\Delta_{G/H}=\emptyset$.

When $G/H=\SO(n+3)/\SO(n)$ we have $\XGHass= \{*\} \cup S^4$.
Here $\{*\}$ is given by
$\kf=\so(3)\oplus \so(n)$, the only other almost semisimple subalgebra.
Clearly, $\kf$ is maximal in $\g$
and a minimal almost semisimple subalgebra. Notice also that there exist
non-toral subalgebras of
type $\so(2)\oplus \so(n+1)$; however these are  not almost semisimple.
It follows that $\XGHass$, and thus $\XGH$,
are disconnected. The simplicial complex $\Delta_{G/H}$ 
consists of three singletons, given by $\so(3)\oplus \so(n)$, 
$\so(n+1)$ and $\so(n+2)$. 

When $G/H=\SO(n+k)/\SO(n)$, $k \geq 4$, the nerve $\XGH$ is connected with
$H_{q(k)}(\XGH)\neq 0$ for $q(k)=\tfrac 12 k(k+1)-2$.
By Lemma \ref{lemcoindex}, we obtain the existence of an Einstein metric
$g \in \MGo$
whose augmented coindex is at least $q(k)+1$. Notice that
$\dim \so(k)=\dim N_G(H) - \dim H=\tfrac 12 k(k-1)$ and that
$q(k)+1-\tfrac 12 k(k-1)=k-1$. Thus, we obtain (at least) $k-1$
dimensions which are orthogonal
to the $\SO(k)$-orbit of $g$ on which the Hessian of $\sc$ is
non-negative.
If one considers instead the simplicial complex $\Delta_{G/H}$,
then by Lemma \ref{lemcoindex} one obtains a critical point $g'$ (with
possibly larger isometry group)
whose augmented coindex is greater than or equal to $\tfrac 12 (k-2)$, for $k$ even,
and greater than or equal to $\tfrac 12 (k-1)$, for $k$ odd: see Theorems \ref{thm:RB} and \ref{thm:RC}. We note
that it is possible that $g=g'$ (up to isometry).
\end{remark}

The following example, due to Graev, demonstrates that the nerve $\XGH$ can be a finer detector than the simplicial complex $\DGH$.

\begin{example}\label{exa:Graev}
The space $G/H=\Un(N)/\Un(2)$ admits a non-contractible nerve $\XGH$ although its simplicial complex
$\DGH$ is contractible. For the (non-standard) embedding of $H$ in $G$ see below.
\end{example}

\begin{proof}
Let $p \geq 2$ and let $\CC^N:=S^p (\CC^3)$, the space of $p$-symmetric tensor
products of $\CC^3$, where $N= \binom{p+2}{2}$.
 For $e_3 \in \CC^3$, the third standard basis vector, let $v_0:= \otimes_p e_3 \in \CC^N$,
 and let $G_{v_0}$ denote the stabilizer of
$v_0$ under  the standard action of
$G:=\Un(N)$ on $\CC^N$,  
the unitary group.  
Notice that $G_{v_0}$ is isomorphic to $\Un(N-1)$.

The standard representation of $\Un(3)$ on $\CC^3$ induces an irreducible complex
representation $\rho:\Un(3) \to \Un(N)$ on $S^p(\CC^3)$ \cite[13.1]{F.H}. 
Let $K :=\rho(\Un(3))$ and notice $\kf \cong \un(3)$.
Let $\Un(2)$ denote the stabilizer in $\Un(3)$ of $e_3$ and let
$H:=\rho(\Un(2))$. Then  $H < G_{v_0}$ and $\h \cong \un(2)$. 

Although $\h$ is not maximal in $\kf$,  $\kf$ is a minimal non-toral subalgebra of $\g$.
Moreover, since $K$ is complex and acts irreducibly on $\CC^N$ \cite{F.H}, we know 
$\kf$ is a maximal subalgebra of $\g$ by \cite[Thm.1.3]{D.E.1}. 
Furthermore, since $\h < \g_{v_0} < \g$, it follows that
$\XGH$ is not connected.

Next, we show that $\Delta_{G/H}$ is contractible. To this end, 
we decompose $\CC^3=
\langle
e_3\rangle_{\bldss\CC} \oplus 
\langle e_1,e_2\rangle_{\bldss\CC}$. Then it is easy to check that
$$
  S^p(\CC^3)=S^p(\langle
e_3\rangle_{\bldss\CC} \oplus
\langle e_1,e_2\rangle_{\bldss\CC})= \bigoplus_{k=1}^{p+1} V_k
$$
where, for each $k=1,\dots,p+1$, 
$\dim_{\bldss\CC} V_k=k$, 
$V_k$ is an irreducible complex $\Un(2)$-representation (see \cite[p.76,(3.1)]{B-D}), and the $V_k$ are pairwise
inequivalent modules (having different dimensions).
In fact, $V_1$ is a trivial module,   
 $V_2$ is the standard representation of $\Un(2)$ on $\CC^2=S^1(\CC^2)$ 
 (which confirms that the representation $\rho$ is effective), 
and for $k \geq 3$, $V_k$ is isomorphic to $S^{k-1}(\CC^2)$.
Notice, for $k= 2,\dots,p+1$, that the realizations of the complex $H$-representations $V_k$,  which
 we consider here, are irreducible real representations, by Lemma \ref{lem:repcomplex}.

This shows that 
$$
  H < \Un(1)\times \Un(2) \times \cdots \times \Un(p+1) \subset \Un(1) \times \Un(N-1) \subset \Un(N)\,,
$$
with $\sum_{k=1}^{p+1}k=N$.  
Moreover, for $k= 2,\dots,p+1$, the projection $\rho_k(H)$  of $\rho(H)$ onto $\Un(k)$
is given by the standard representation of $\Un(2)$ on $V_k \cong S^{k-1}(\CC^2)$.

As a first consequence of this,
the connected component of the centralizer of $H$ in $G$ is a $(p+1)$-dimensional torus
$$
 T^{p+1}=\Un(1) \times {\rm Z}(\Un(2))\times \cdots \times  {\rm Z}(\Un(p+1))\,,
$$
where $ {\rm Z}(\Un(k))$ denotes the center of $\Un(k)$. To see this we note that 
 Lemma \ref{lem:repcomplex} guarantees that the realization of the irreducible complex representation of $\Un(2)$ on $V_k$ is still irreducible. By \cite[Thm.1.5]{D.E.1},
it follows that $\rho_k(\Un(2))$ is maximal in $\Un(k)$, $2 \leq k \leq p+1$. 
As a consequence, the
centralizer of $\rho_k(\Un(2))$ in $\Un(k)$ is precisely the center 
$ {\rm Z}(\Un(k))$ of $\Un(k)$. This shows the above claim.

Furthermore, we know that the center of ${\rm U(2)}$ acts on $V_k$ by
$\rho_k(e^{i\varphi}\cdot I_2)=e^{i\varphi (k-1)}\cdot I_k$. This shows that on the complement
of $\un(k_1)\oplus \un (k_2)$ in $\un(k_1+k_2)$, $1 \leq k_1 < k_2 \leq p+1$,
$\Ad(H)$ acts non-trivially. Here $\un(k_1), \un(k_2)$ denote the Lie algebras of the above
blocks $\Un(k_1), \Un(k_2)$ and $\Un(k_1+k_2)$ the corresponding $(k_1+k_2)$-block.

Since the center of $H$
is $1$-dimensional, we conclude that $\dim N_G(H)-\dim H=p$ and that
$A:=(N_G(H))_0/H$ is abelian. We will also view $A$ as a maximal torus in
a compact complement of $H$ in $N_G(H)$.

We can now show that $\Delta_{G/AH}$ is contractible. We will prove that
each minimal (non-toral) subalgebra $\kf$ of $\g$ with $\h \oplus \af < \kf$
satisfies $\kf \leq \un(1) \oplus \un(N-1)$.
With this, we can conclude that $\Delta_{G/AH}$ is a cone, thus contractible.

To show this we first examine the isotropy representation of $H$ on 
$\m = \un(N)\ominus (\un(1)\oplus \un(N-1))$, where $\m \cong \CC^{N-1}$.
Since $H \subset \Un(N-1)$, $\m$ is $\Ad(H)$-invariant. Moreover 
we know that $\m = \bigoplus_{k=2}^{p+1} \m_k$, with $\m_k \cong \CC^k$ and
that $\Ad(H)$ acts on $\m_k$ via $\rho_k(H)$ irreducibly, using the above embedding of $H$ in $G$.
Since $A <A^{p+1}$, each of the subspaces $\m_k$ is $\Ad(AH)$-invariant and $\Un(1) < A$
acts non-trivially on all of them. Since $\Un(1)$ acts on $\un(N-1)$ trivially, the summands
$\m_k$ are $\Ad(AH)$-isotypical summands of the isotropy representation of $G/AH$.

Suppose now that $\h \oplus \af < \kf$ but $\kf \not \leq \un(N-1)$.
Then by the above we have $\kf \cap \m \neq 0$, thus $\m_k \subset \kf$
for some $2 \leq k \leq p+1$. 
But then  $[\m_k,\m_k]=\un(k) \ominus \af_k \subset \kf$, where $\af_k$ 
denotes a maximal torus of $\un(k)$, and $\un(k)$ denotes
the Lie algebra of the  block $\Un(k)$, above.
By taking further Lie brackets, we conclude that $\su(k) \subset \kf$. Thus $\h \oplus \af \oplus \un(k)$
is a non-toral subalgebra of $\kf$, contradicting the minimality of $\kf$.
\end{proof}

\begin{remark} One can easily generalize 
 Example \ref{exa:Graev}, replacing $\CC^3$ by $\CC^m$ and $\Un(3)$ by $\Un(m)$, $m\geq 4$,
to construct further examples with the same property that  the simplicial complex
$\DGH$ is contractible while the nerve $\XGH$ is not.
\end{remark}

\begin{remark}\label{rem:familiesofsolutions}
We would like to mention that
in \cite{B-K} we already described  such an example:
$G/H=\SU(3)\times SU(3)/\Delta \SU(2)\Delta_{p,q} \Un(1)$,
for $p=q \in \NN$: see also Examples 5.13 and 5.14 in \cite{Bo05}.
Moreover, all these provide examples of $G$-homogeneous Einstein metrics
which are not fixed by the action of $N_G(H)$; that is, these Einstein metrics
appear in a one-dimensional family. This follows since these Einstein
metrics are not invariant under the right action of  $\Un(1)=\Un(1)\Un(1)/\Delta_{p,p}\Un(1)$: among those metrics there exist no invariant Einstein metric.
\end{remark}

We include a brief overview of Lie groups and their real (resp.~complex, symplectic) representation theory in Section \ref{sec:Ltbasics}. 

\begin{example}\label{exa:<}
%"2023"
Let $G/H$ be a compact homogeneous space with $G$ simple such that
there exists an intermediate subgroup $K$ such that (i) $\dim G/K>1$ 
and $\g \ominus \kf$ is $\Ad(H)$-irreducible, (ii) 
$K/H=K_1/H_1 \times \cdots \times K_r/H_r$ is a product of isotropy irreducible spaces 
with $\dim K_i/H_i >1$ for all $i=1,...,r$, and (iii) $G/H \neq {\rm Spin}(8)/G_2$.
Then there exist only finitely many non-toral subalgebras $\kf$, and 
$\XGH$ is contractible.
\end{example}

\begin{proof}
By assumption $K/H \cong K_1/H_1 \times \cdots \times K_r/H_r$, where each of the factors $K_i/H_i$ is isotropy irreducible, $i=1,...,r$.
It follows that $\kf=\kf_1 \oplus \cdots \oplus \kf_r \oplus \rf$ 
and $\h =\h_1 \oplus \cdots \oplus \h_r \oplus \rf$, meaning that $K/H$ may not be presented effectively.
Writing $\kf_i =\h_i \oplus \m_i$, $\m_i$  $\Ad(H)$-irreducible and isotypical, for $1 \leq i \leq r$, 
we deduce 
$$
 \m=\m_1 \oplus \cdots \oplus \m_r \oplus \hat \m\,.
$$
Note that the $\m_i$ are pairwise $Q$-orthogonal, $1 \leq i \leq r$, 
but that $\hat \m$ may not be $Q$-orthogonal to say $\m_1$.
This leaves us with at most finitely many
non-toral minimal subalgebras $\kf_1,...,\kf_r$ contained in $\kf$.
Also, since $\dim G/K>1$, the module $\hat \m$ cannot be $\Ad(H)$-trivial. 

Suppose now that $\hat \m$ is not equivalent to any of the $\m_i$, $i=1,...,r$.
Suppose furthermore, 
that there exists a minimal non-toral subalgebra $\ti \kf< \g$ which is not a subgroup of $\kf$.
We conclude that $\hat \m \subset \ti \kf$, since $\hat \m$ is an (irreducible) 
isotypical summand of $\m$.
Now, if $[\hat \m,\hat \m] \subset \h$
then $G$ is not simple by Lemma \ref{lem:splitt}, a contradiction.
Thus, we may assume that there exists a non-empty subset $I = \{1,...,\tilde r\}$, such that for
all $i\in I$ we have ${\rm pr}_{\bldss{\m_i}}([\hat \m,\hat \m])\neq 0$. If $\tilde r <r$
then set $\p_1:= \oplus_{i=1}^{\ti r} \m_i \oplus \hat \m$ and $\p_2:=\oplus_{j=\ti r+1}^r \m_j$.
Then, since $\hat \m$ is $\Ad(K_1 \times \cdots \times K_{\ti r})$-invariant, $[\p_1,\p_1]\subset \h \oplus
\p_1$, $[\p_2,\p_2]\subset \h \oplus \p_2$ and $Q(\p_1,\p_2)=0$, by Lemma \ref{lem:splitt}
$G$ would not be simple, a contradiction. Thus $\tilde r=r$. 
But this would imply $\ti \kf=\g$, also a contradiction.

This shows that in this case all the minimal non-toral subalgebras of $\g$ are subalgebras
of $\kf$. As a consequence, the minimal nerve $\XGHmin$ is contractible and so is $\XGH$.

Now suppose that $\hat \m$ is equivalent to $\m_1$.
Then  $\hat \m$ is not equivalent to $\m_2,...,\m_r$ since $H_1$
acts trivially on those modules but non-trivially on $\m_1$. Moreover,
$H_2,..,H_r$ act trivially on $\m_1$ thus on $\hat \m$, and consequently 
we have ${\rm pr}_{\bldss{\m_i}}[\hat \m,\hat \m]=0$ for all $i=2,...,r$. As above,
this shows that $r=1$, since otherwise $G$ would not be simple.

Finally we deduce that $\m =\m_1 \oplus \hat \m$. By \cite{DK}, \cite[Thm A.1]{He} we conclude that
there exists only one such space, namely
$G/H={\rm Spin}(8)/G_2$, proving the claim.
\end{proof}

The nerve $\XGH$ of ${\rm Spin}(8)/G_2$ consists of three singletons,
each one isomorphic to ${\rm Spin}(7)$
(under the triality isomomorphism). Note also that when $G$ is not simple then there exist
further compact homogeneous spaces with two equivalent summands, e.g. $G/H=(L \times L \times L)/\Delta_3 L$,
for any simple group $L$.

\begin{remark}
The assumptions in Example \ref{exa:<} are very restrictive, meaning that
$\dim \MGo$ has to be very small. One may try to generalize the above by relaxing
various conditions.
\end{remark}

For completeness we show the following result: see proof of Theorem 2.1 in \cite{WZ2}.

\begin{lemma}[\cite{WZ2}]\label{lem:splitt}
Let $G/H$ be an almost effective compact homogeneous space and suppose that $\m=\m_1 \oplus \m_2$, $\m_1,\m_2$
$\Ad(H)$-invariant with
$Q(\m_1,\m_2)=0$, $[\m_1,\m_1]\subset \h \oplus \m_1$ and $[\m_2,\m_2]\subset \h \oplus \m_2$. 
Then $G/H=G_1/H_1 \times G_2/H_2$.
\end{lemma}

\begin{proof}
Let $\h_1 :={\rm pr}_{\bldss\h}[\m_1,\m_1]$ and set $\g_1:=\h_1 \oplus \m_1$.
Using the Jacobi indentity we see that $\h_1$ is a subalgebra of $\h$. Moreover,
$\g_1$ is a subalgebra of $\g$ because $[\h_1,\m_1]\subset \m_1$.
We define $\g_2$ analogously. We claim that  $[\m_1,\m_2]=0$. To see this note 
$$
 Q([\m_1,\m_2],\m_1)=Q(\m_2,[\m_1,\m_1])=Q(\m_2,\h\oplus \m_1)=0
$$ 
and 
$$
 Q([\m_1,\m_2],\h)=Q(\m_2,[\m_1,\h])=Q(\m_2,\m_1)=0\,.
$$
Using $Q(\m_1,\m_2)=0$, $[\m_1,\m_2]=0$ and the Jacobi identity we obtain
$$
Q([\m_1,\m_1]_{\bldss\h},[\m_2,\m_2]_{\bldss\h})=Q([\m_1,\m_1],[\m_2,\m_2])=
Q(\m_1,[\m_1,[\m_2,\m_2]])=0\,.
$$
This shows $Q(\h_1,\h_2)=0$. Thus $\g_1$ and $\g_2$ are $Q$-orthogonal ideals in $\g$.
By effectiveness we deduce $\g=\g_1 \oplus \g_2$. This shows the claim.
\end{proof}

\begin{example}\label{exa:>}
Let $G/H$ be a compact homogeneous space with $G$ simple such that
there exists a subgroup $K$, such that (i) $K/H$ is isotropy irreducible with $\dim K/H>1$ 
(ii) all $G$-invariant metrics on $G/H$ are submersion metrics with respect to 
$K/H \to G/H \to G/K$ (iii) $\dim N_G(K) = \dim K$ and (iv) for every maximal subgroup $L$ of $G$ 
with $H<L$ we have $K<L$. Then there exist only finitely many non-toral subalgebras 
and $\XGH$ is contractible.
\end{example}

\begin{proof}
Let 
$\m = \m_{\bldss{\kf}} \oplus \m_2 \oplus \cdots \oplus \m_{\ell}$ 
be a decomposition of $\m$
into $Q$-orthogonal $\Ad(H)$-irreducible summands. Then by (ii) the summands $\m_2,...,\m_\ell$
are  $\Ad(K)$-irreducible, see \cite[9.79]{Bes},  and by (iii) they are $\Ad(K)$-nontrivial. 

We now show this implies $\dim N_G(H)=\dim H$.
For a contradiction, suppose $\dim \m_0>0$. Then we must have $\m_0 \cap \kf = 0$; otherwise 
$\h < \h\oplus (\m_0 \cap \kf) \leq \kf$ either shows that $\h$ is not maximal in $\kf$, which is impossible,
or that  $\kf = \h \oplus  (\m_0 \cap \kf) $, which would imply $\dim K/H=1$, contradicting (i).
Since $\m_0 \cap \kf =0$, we have $Q(\m_0, \m_{\bldss{\kf}})=0$, because 
$\m_0$ (the sum of all trivial $\Ad(H)$-summands in $\m$) is an $\Ad(H)$-isotypical summand of $\m$. 
Then,by (ii), any $1$-dimensional subspace of $\m_0$ is $\Ad(K)$-invariant, contradicting (iii).

Thus by \cite{BWZ} we know there exist only finitely many intermediate subalgebras $\lf$
with $\h < \lf < \g$ (all non-toral), and by  Corollary \ref{cor:LKmax}, $\XGH$ is contractible.
\end{proof}

Examples of $G/H \in \Nl_>$ are plentiful. 

\begin{example}\label{exa:mixedtype}
Let $G:=\SO(2n)$ where $n=\sum_{i=1}^r n_i$ with $n_1,\dots,n_r \geq 2$ and $r \geq 2$, and let
$$
   H:= \Un(n_1) \times \SO(2n_2) \times \cdots  \times \SO(2n_r)\,.
$$
Then $G/H$ satisfies all the assumptions of the above lemma: the group $K$ is
given by  
$$
 K=\SO(2n_1)\times \cdots \times  \SO(2n_{r}) \,.
$$
\end{example}

\begin{remark}
We can replace condition (iv) in Example \ref{exa:>} by the condition that $\m_{\bldss{\kf}}$ is an $\Ad(H)$-isotypical summand of $\m$. To show this,
let $\lf$ be a maximal subalgebra of $\g$. We claim that $\kf \leq \lf$. 
If not, then we may assume that $\lf=\h \oplus \m_{\bldss{\lf}}$,
$\m_{\bldss{\lf}}= \m_2 \oplus \cdots \oplus \m_{r}$ as in Example \ref{exa:>}, $r \leq \ell$,
using that $\m_{\bldss{\kf}}$ is an isotypical summand. 
Then $\lf + \kf$ is a subalgebra of $\g$, since  $[\kf,\m_{\bldss{\lf}} ]\subset \m_{\bldss{\lf}}$ 
using that $\m_{\bldss{\lf}}$
is $\Ad(K)$-invariant by (ii). If $\lf + \kf < \g$ we obtain a contradiction since $\lf$ is
maximal. So we are left with the case $\kf + \lf =\g$, that is $r=\ell$. Since $\m_{\bldss{\kf}}$ is an isotypical
summand we deduce again by Lemma \ref{lem:splitt} that $G$ is not simple, another contradiction.
\end{remark}

\begin{example}
Let $G/H=\SO(9)/\SO(3)\Un(2)$,
where the embedding of $H$ into $G$ is given by
$$
  \SO(3)\Un(2)\subset
  \SO(3)\SO(4)\subset \SO(7)\subset \SO(2){\rm
SO}(7)\subset \SO(9)\,.
$$
Then the nerve $\XGH$ is contractible.
\end{example}

\begin{proof}
The isotropy representation of $G/H$ decomposes into
$$
   \m = \m_0 \oplus (\RR^3_1 \oplus \RR^3_2) \oplus (\CC^2_1\oplus \CC^2_2)
   \oplus \RR^2 \oplus \RR^{12}\,.
$$
Here $\m_0=\so(2)$, the centralizer of $\so(7)$, $\RR^3_1,\RR^3_2$ are
equivalent to the standard representation of
$\SO(3)$, of real type, $\CC^2_1, \CC^2_2$ are equivalent to the
standard representation of $\Un(2)$, of complex type,
  $\RR^2$ denotes the orthogonal complement of $\un(2)$ in $\so(4)$ and
  $\RR^{12}$ denotes the orthogonal complement of $\so(3)\oplus \so(4)$
in $\so(7)$, which is irreducible.

  The minimal non-toral subalgebras are
  $$
   \kf_\varphi:=\so(4)_\varphi \oplus \un(2) \,,
      \quad
     \kf_2:=\so(3) \oplus \so(4) \,,
  \quad
     \kf_3:=\so(3)\oplus \un(3)
         \quad \textrm{ and }\quad
         \kf_4:= \so(3)\oplus \tilde \un(3)\,,
$$
where $\varphi \in S^1=(N_G(H))_0=\SO(2)$. Notice $\kf_\varphi=\h
\oplus (\RR^3)_\varphi$,
$\kf_2 =\h \oplus \RR^2$, $\kf_3=\h \oplus \m_0 \oplus {\tilde \CC}_1^2$
and $\kf_4=\h \oplus \m_0 \oplus {\tilde \CC}_1^2$. 
 There are precisely two subalgebras isomorphic to $\un(3)$ between
$\un(2)\oplus \un(1)$ and $\so(6)$.

We first compute  the subalgebras generated by any two of our minimal non-toral
subalgebras.
For $\varphi \neq \ti \varphi$ we have $\langle \kf_\varphi,\kf_{\ti
\varphi}\rangle
=\so(5)\oplus \un(2)$, $\langle \kf_\varphi,\kf_2\rangle=\so(4)_\varphi
\oplus \so(4)$,
$\langle \kf_\varphi,\kf_3\rangle=\langle \kf_\varphi,\kf_4\rangle=\g$,
$\langle \kf_2,\kf_3\rangle=\langle \kf_2,\kf_4\rangle=\so(3)\oplus \so(6)$
and finally $\langle \kf_3,\kf_4\rangle=\so(3)\oplus \so(6)$.

We see that $\kf_3,\kf_4$ are only related to $\kf_2$ via the three flags
$(\kf_3< \so(3)\so(6))$, $(\kf_4< \so(3)\so(6))$ and $(\kf_2 <
\so(3)\so(6))$.
Clearly, this part of $\XGHmin$ is contractible, thus we may disregard
$\kf_3$ and $\kf_4$.

Next, for all $\varphi \neq \ti \varphi$ we consider
$$
\langle \kf_\varphi,\kf_{\ti \varphi},\kf_2\rangle
=\langle \so(5)\un(2),\so(3)\so(4)\rangle=\so(5)\so(4)
$$
which is maximal in $\g$. This shows that $\XGHmin$ is a cone, thus
contractible.
\end{proof}

By varying the numbers $9,3,2$ it is easy to describe many further examples of the above kind.

%%%%%%%%%%%%%%%%%%%%%%%%%%%%%%%%%%%%%%%%%%%%%%%%%%%%%%%%%%%%%%%%%%%%%%%%%%%%%%%%%%%%%%%%%%%%
\subsection{{\it The simplicial complex of $G/H$}}\label{sec:tricks}

In this section we apply the methods proved in previous sections to subspaces of
$\MGo$ which are invariant under the volume normalized Ricci flow. The main result of
this section is Theorem \ref{thm:trick}, which is particularly useful for classification results: see
Remark \ref{rem:finiteproblem}. 

\medskip

Let $G/H$ be a compact homogeneous space. Recall that the Lie algebra of $N_G(H)$, the normalizer of $H$ in $G$,
is given by $\n(\h)=\h \oplus \m_0$: see Lemma \ref{lemnor}. If $\m_0=0$, then by \cite{BWZ} there exist only finitely many
subalgebras $\kf$ with $\h < \kf < \g$. As a consequence $\XGH$ is indeed a simplicial complex $\DGH$ with
finitely many faces.

\begin{lemma}\label{lem:DGH}
Let $G/H$ be a compact homogeneous space with $\dim N_G(H)=\dim H$. Then
$\XGH$ is a finite simplicial complex $\DGH$ which is homeomorphic to the flag complex of all subalgebras $\kf$
with $\h < \kf < \g$.
\end{lemma}

\begin{proof}
Since $\dim N_G(H)=\dim H$,
by \cite{BWZ} there exist only finitely many subalgebras $\kf$ with $\h< \kf <\g$.
Let $\varphi=(\kf_1< \cdots < \kf_r)$ be a flag of such subalgebras. Let $P_1,\dots,P_r \in \Psn$ denote the corresponding projections, 
i.e., $\kf_i = \ker(P_i)$ for each $i=1,\dots,r$. 
Then the simplex $\Delta_{\varphi}^P:={\rm conv}\{P_1,\dots,P_r\}$ is a realization of $\varphi$ (cf. Definition \ref{def:simplex}). Suppose 
$\tilde \varphi=(\tilde \kf_1< \cdots < \tilde \kf_{\tilde r})$ is another flag.
The claim follows from the equality  
$$
 \Delta_{\varphi}^P \cap \Delta_{\tilde \varphi}^P 
 =\Delta^P_{\varphi \cap \tilde \varphi}\,.
$$
It is clear that $ \Delta^P_{\varphi \cap \tilde \varphi} \subset \Delta_{\varphi}^P \cap \Delta_{\tilde \varphi}^P $.
The proof of the other inclusion is analogous to the proof of injectivity in Lemma \ref{lem:XGHXXGH}.
\end{proof}

When considering the full flag complex $\DGH$ there may be some redundancy. 
Instead of considering all flags $\varphi$ we may consider 
those flags $(\kf_1<\cdots <\kf_r)$ where, for each $1\leq i \leq r$, $\kf_i \in \Submg$ (Definition \ref{def:min-sub}) is  generated by
minimal non-toral subalgebras of $\g$.  
We denote by $\Deltaf_{G/H}$ the corresponding flag complex. Notice that $\Deltaf_{G/H}$ is a subcomplex of $\DGH$,
which,  by Lemma \ref{lem:rsimhomsim}, is homotopy equivalent to $\DGH$.

Next, we describe how to associate a simplicial complex to a
compact homogeneous space $G/H$ for which $\dim N_G(H)> \dim H$.
Recall that our choice of $Q$ guarantees the Lie algebra $\m_0$ is compact: see Lemma \ref{lemnor}.

\begin{lemma}\label{lemAH}
Let $G/H$ be a compact homogeneous space. Let $\af$ be a Cartan subalgebra
of $\m_0$ and let $A$ denote the corresponding maximal torus. Then
$AH$ is a compact subgroup of $N_G(H)$ with Lie algebra $\h\oplus \af$ and $\n(\h \oplus \af)=\h \oplus \af$.
\end{lemma}

\begin{proof}
First of all, because $[\h,\m_0]=0$, it is clear that $\h \oplus \af$ is a subalgebra of $\g$ and
 that $A$ and $H$ commute. Since $A$ and $H$ are each compact, we know $AH$ is 
a compact subgroup of $N_G(H)$. Since $H$ acts non-trivially on $\m \ominus \m_0$, so does $AH$.
And since $A$ is a maximal torus $A$ acts non-trivially on $\m_0 \ominus \af$.
Thus $\n(\h \oplus \af)= \n \oplus \af$ by Lemma \ref{lemnor}.
\end{proof}

Since $AH \subset N_G(H)$ we have $\Ad(AH)(\h)\subset \h$ and
consequently, by $\Ad(G)$-invariance of $Q$, $\Ad(AH)(\m)\subset \m$.
We denote by 
$$
  (\MGo)^A:=\{ P_g \in \MGo \mid \Ad(AH)\vert_{\m}\cdot P_g = P_g \cdot \Ad(AH)\vert_{\m}\}
$$
the set of $\Ad(AH)$-equivariant  elements in $\MGo$. 
The metrics in $(\MGo)^A$ are submersion metrics with respect to $(AH)/H \to G/H \to G/(AH)$
(see \cite[9.80]{Bes}). Notice that since $A \subset N_G(H)$, this submersion is induced by the isometric
right action of $A$ on $G/H$ given by $(a,gH)\mapsto gaH$.

\begin{lemma} \label{lemnv2}
Let $G/H$ be a compact homogeneous space. Let $\af$ be a Cartan subalgebra
of $\m_0$ and let $A$ denote the corresponding maximal torus. Then $(\MGo)^A$
is a totally geodesic subspace of $\MGo$, invariant under the volume normalized Ricci flow.
\end{lemma}

\begin{proof}
The group of $G$-equivariant diffeomorphisms of $G/H$ acts on $(\MGo,L^2)$ by pulling back $G$-invariant
metrics. 
By \cite[I, Corollary 4.3]{B.G}, we know $G$-equivariant
diffeomorphisms of $G/H$ are precisely the right trans\-lations $R_n:G/H\to
G/H\,\,;\,\,\,gH \mapsto gnH$, for some \mbox{$n\in N_G(H)$}. Since
$N_G(H)$ normalizes $H$, we can restrict the adjoint action of
$N_G(H)$ to $\m$. Thus, ${\rm Ad}(n)\vert_{\m} \in {\rm O}
(\m,Q\vert_{\m})$ for all $n\in N_G(H)$. This yields an isometric action
of $N_G(H)$ on  $(\MG,L^2)$ given by 
\begn
  P_g \mapsto (R_n)_*(P_g)={\rm Ad}(n)\vert_{\m} \cdot 
                       P_g\cdot {\rm Ad}(n^{-1})\vert_{\m}\,,\label{eqn:NGH}
\enn
where we identify $g$ with $P_g$, an ${\rm Ad}(H)$-equivariant endomorphism  of $\m$: see Section \ref{sec:ginvm}.
To see this, let $X,Y \in \m =T_{eH}G/H$ and let $g \in \MG$. Then
\beg
  ((R_n)_*g)(X,Y)_{eH}
	 &=&
	  g((dR_n)_{eH}\cdot X,(dR_n)_{eH}\cdot Y)_{nH} \\
		&=&
		 g((dL_{n^{-1}}\cdot dR_n)_{eH}\cdot X,(dL_{n^{-1}}\cdot dR_n)_{eH}\cdot Y)_{eH}\\
		&=&
		Q( (P_g \cdot \Ad(n^{-1})\vert_{\m}) \cdot X,\Ad(n^{-1})\vert_{\m}\cdot Y)\\
		&=&
		 Q((\Ad(n)\vert_{\m}\cdot P_g \cdot \Ad(n^{-1})\vert_{\m})\cdot X, Y)\,.
\en
For any subgroup $L\leq N_G(H)$, we let 
\beg
 \MGoL =\{g\in \MGo \mid  g= (R_l)_*(g) \textrm{ for all }  l\in L\}\,.
\en
Notice we have $Q\in \MGoL$. It follows as in Lemma \ref{lem:MGL2} that $\MGoL$ is a totally geodesic subspace
of $\MGo$.  

Recall here that the Ricci flow preserves isometries. That is if $g_0 \in \MGo$ is an initial metric
then the solution $(g(t))_{t \in [0,T)}$ of the volume normalized Ricci flow with $g(0)=g_0$
satisfies $g(t) \in \MGo$ for all $t \in [0,T)$.

By uniqueness of Ricci flow solutions on compact manifolds,
this implies that when we start at an initial metric $g_0 \in \MGoL$ which is fixed by 
the diffeomorphisms $R_l$  for all $l \in L$,
then the corresponding volume normalized Ricci flow solution will also be fixed by  $R_l$ for all $l \in L$.
Thus $\MGoL$ is invariant under the volume normalized Ricci flow. Choosing 
% added 2023 T --> A
$L=A$ shows the claim.
\end{proof}

We turn to the main result of this section which also provides of proof for Theorem \ref{theoB}.

\begin{theorem}\label{thm:trick}
Let $G/H$ be a compact homogeneous space with finite fundamental group.
Let $\af$ be a Cartan subalgebra
of $\m_0$ and let $A$ denote the corresponding maximal torus.
If  $\Delta_{G/AH}$ is non-contractible, then
$G/H$ admits a $G$-invariant Einstein metric.
\end{theorem}

\begin{proof}
We consider the restriction of the scalar
curvature functional to the subspace $(\MGo)^A$ of $G$-invariant 
metrics which are invariant under the adjoint action of $AH$. 
All results proved for the $\sc:\MGo \to \RR$ can be applied to $\sc:(\MGo)^A \to \RR$,
since  the totally geodesic space  $(\MGo)^A$ of $\MGo$ 
is invariant under the normalized Ricci flow: see Lemma \ref{lemnv2}.

We now define the corresponding simplicial complex by $\DeltaT_{G/H}$.
We call a non-toral $\Ad(AH)$-invariant subalgebra $\kf_A$ with $\h < \kf_A< \g$
{\em minimal, non-toral} if any smaller $\Ad(AH)$-invariant subalgebra $\kf'_A$, with $\h<\kf'_A<\kf_A$,
must be toral. By Lemma \ref{lemfinite},
there exist only finitely many minimal, non-toral subalgebras $\kf_T$. 
Moreover, by Proposition \ref{propdelta},
the set of all $\Ad(AH)$-invariant, non-toral subalgebras, generated by minimal, non-toral subalgebras $\kf_A$, 
is also finite, and of course it is partially ordered by the inclusion relation.
The simplicial complex $\DeltaT_{G/H}$ attached to $G/H$ and $A$ is
the corresponding flag complex.
Notice that the definitions above still make sense if $\n(\h)=\h$, so that $A=\{e\}$. In this case, $\DeltaT_{G/H}:=\Deltaf_{G/H}$.

The claim follows provided that $\DeltaT_{G/H}$ is non-contractible.
To this end we note that
by Proposition \ref{propdelta}, 
$\DeltaT_{G/H}=\Deltaf_{G/AH}$. Moreover, by Lemma \ref{lem:rsimhomsim}, we have that 
 $\Deltaf_{G/AH}$ is homotopy equivalent to $\Delta_{G/AH}$.
Since by hypothesis, $\Delta_{G/AH}$ is non-contractible, so is $\DeltaT_{G/H}$.
This shows the above claim and the theorem is proved.
\end{proof}

To indicate that the above theorem is quite useful, let $G=\SO(2n)$ and let $A$ be a maximal Torus of $G$.
Let $A'$ be any compact subtorus of $A$. Then by the above theorem $G/A'$ admits an $G$-invariant
Einstein metric, since $\Delta_{\SO(2n)/A}$ is not contractible: see Theorem \ref{thm:RB}.
Notice that for a generic  $A'$ (with a fixed dimension), the isotropy representation does not depend on $A'$,
whereas for a non-generic $A'$ this is no longer true. 

\begin{remark}\label{rem:finiteproblem}
Let $M^n=G/H$ be a compact, simply connected homogeneous space. In dimension $n=7$ and each $n\geq 9$,
there exist infinitely many homotopy types of simply connected
compact homogeneous spaces: see \cite{Kl}. Nevertheless, computing the corresponding
simplicial complexes $\Delta_{G/AH}$, up to a fixed dimension of $G/H$, is a finite task.  
To see this, note that when $M^n=G/H$, the semisimple part $G_{s}$ of $G$ also acts transitively on $M^n$: see \cite{B-K}.
The dimension of a compact, connected, simply connected, semisimple
Lie group $G$ acting transitively and almost effectively on $M^n$ is bounded between $n$ and
$\frac{1}{2}n(n+1)$. Hence there exist only finitely many such
groups acting transitively on $M^n$. For each of these Lie groups there
exist (up to conjugation) at most finitely many connected, semisimple Lie
subgroups. Therefore, computing the corresponding
simplicial complexes $\Delta_{G/AH}$, up to a fixed dimension of $G/H$, is a finite task.
\end{remark}

\begin{remark}\label{rem:T}
 The simplicial complex $\Delta_{G/AH}$
does not depend on  $A$, since all maximal tori of 
$\m_0$ are conjugate.
\end{remark}

Let $\af$ be a Cartan subalgebra of $\m_0$ and let $A$ denote the corresponding maximal torus.
We have 
$$
 \m =   \af \oplus (\af^\perp \cap \m_{0})\oplus \m_0^\perp\,.
$$
The compact Lie  group $AH$  acts via the adjoint action on these three
summands of $\m$ in the following manner: $AH$ acts trivially on $\af$,
$H$ acts trivially on $\af^\perp \cap \m_0$ but $A$ does not,
and finally $H$ does not act trivially on $\m_0^\perp$ by definition of $\m_0$.
Recall also that $\m=\m_0$ is possible. In any case it follows $\n(\h \oplus \af)=\h \oplus \af$.

We note that $\af$ is the $\Ad(AH)$-trivial isotypical summand of $\m$;
that $\af^\perp \cap \m_{0}$ is the direct sum of $2$-dimensional (non-trivial) isotypical summands
of $\Ad(AH)$, the root spaces of $\af$ in $\m_0$; and that $\m_0^\perp$ is the direct sum of
the $\Ad(AH)$-isotypical summands of $\m$, on which $\Ad(H)$ acts non-trivially. Of course if $H=\{e\}$, then $\m_0^\perp = \{0\}$.

For each $\Ad(AH)$-invariant subalgebra $\kf_A$ with $\h < \kf_A$, we can decompose $\kf_A$ as
\beg
  \kf_A=\h \oplus \m_{\blds\kf_A}^1 \oplus \m_{\blds\kf_A}^2 \oplus \m_{\blds\kf_A}^3 \,,
\en
where $\m_{\blds\kf_A}^1 \subset \af$,
$\m_{\blds\kf_A}^2 \subset \af^\perp \cap \m_{0}$ and $\m_{\blds\kf_A}^3 \subset \m_0^\perp$.
Notice that in particular this applies to any subalgebra $\kf $ with $\h \oplus \af < \kf$.

\begin{lemma}\label{lemfinite}
Let $G/H$ be a compact homogeneous space with finite fundamental group. 
Then the minimal, non-toral, $\Ad(AH)$-invariant
subalgebras $\kf_A$ with $\h < \kf_A < \g$ are in one-to-one correspondence with
minimal subalgebras $\kf$ with $\h \oplus \af < \kf < \g$.
\end{lemma}

\begin{proof}  
Let $\kf_A=\h \oplus \m_{\blds\kf_A}^1 \oplus \m_{\blds\kf_A}^2 \oplus \m_{\blds\kf_A}^3<\g$
be a non-toral, $\Ad(AH)$-invariant subalgebra with $\h < \kf_A < \g$. By hypothesis, since $\kf_A$ is nontoral,
we cannot have $\m_{\blds\kf_A}^2 = \m_{\blds\kf_A}^3=\{0\}$.
Then we can associate to $\kf_A$ 
$$
\kf:=\h\oplus \af \oplus \m_{\blds\kf_A}^2 \oplus \m_{\blds\kf_A}^3\,,
$$ 
 an $\Ad(AH)$-invariant subalgebra of $\g$ with $\h \oplus \af < \kf$. Moreover,
since the fundamental group of $G/H$ is  finite, $\kf$ is a proper subalgebra of $\g$ 
(see the proof of Proposition \ref{propDeltacon}). 

Conversely,
let 
$$
 \kf=\h\oplus \af \oplus \m_{\blds\kf}^2 \oplus \m_{\blds\kf}^3
$$ 
be a subalgebra with $\h \oplus \af < \kf < \g$,
$\m_{\blds\kf}^2\subset \af^\perp \cap \m_0$ and $ \m_{\blds\kf}^3 \subset \m_0^\perp$.
Notice that we cannot have $\m_{\blds\kf}^2=\m_{\blds\kf}^3=\{0\}$. 
Let
$$
\m_{\blds\kf}^1:=[\m_{\blds\kf}^2, \m_{\blds\kf}^2]\vert_{\af} \oplus [\m_{\blds\kf}^3, \m_{\blds\kf}^3]\vert_{\af}\,.
$$
We claim that
$$
 \kf_A:=\h \oplus \m_{\blds\kf}^1 \oplus \m_{\blds\kf}^2 \oplus \m_{\blds\kf}^3
$$ 
is a non-toral, $\Ad(AH)$-invariant subalgebra. 
By construction, $\h \oplus  \m_{\blds\kf}^1 \subset \h \oplus \af$, and clearly  $\h \oplus  \m_{\blds\kf}^1$ 
is a subalgebra. 
Moreover $[\h \oplus  \m_{\blds\kf}^1,\m_{\blds\kf}^2] \subset \m_{\blds\kf}^2$
and $[\h \oplus  \m_{\blds\kf}^1,\m_{\blds\kf}^3] \subset \m_{\blds\kf}^3$
using the $\Ad(AH)$-invariance of 
 $\m_{\blds\kf}^2$ and  $\m_{\blds\kf}^3$.  Furthermore
$[\m_{\blds\kf}^2, \m_{\blds\kf}^2],\,[\m_{\blds\kf}^3, \m_{\blds\kf}^3] \subset \kf_A$ by definition of $\m_{\blds\kf}^1$.
Finally $[ \m_{\blds\kf}^2, \m_{\blds\kf}^3] \subset  \m_{\blds\kf}^3$ because 
$\kf_0:=\h \oplus  \af  \oplus \m_{\blds\kf}^2=\n(\h) \cap \kf$ 
is a subalgebra of $\kf$. This proves $\kf_A$ is indeed an $\Ad(AH)$-invariant subalgebra.
If $ \m_{\blds\kf}^3 \neq \{0\}$ then $\kf_A$ is non-toral. If $ \m_{\blds\kf}^3 = \{0\}$
then  $\m_{\blds\kf}^2 \neq \{0\}$ and
 $[ \m_{\blds\kf}^2, \m_{\blds\kf}^2   ]\vert_{\af}  \oplus \m_{\blds\kf}^2 \subset \m_0$ is
a non-toral subalgebra since $\m_{\blds\kf}^2$ is a direct sum of root spaces.
It follows that $\kf_A$ is a non-toral subalgebra, with $\h < \kf_A < \g$.

We now show that minimality is preserved in this correspondence. We assume we have some $\kf_A$, minimal. We claim that the corresponding $\kf$ is minimal, as well.
If not then there exists $\h \oplus \af < \tilde \kf < \kf$ and we must
have $\m_{\bldstik \kf}^2 \subsetneq \m_{\blds\kf}^2 $ or 
$\m_{\bldstik \kf}^3 \subsetneq \m_{\blds\kf}^3 $. It would follow that
$\tilde \kf_A < \kf_A$, a contradiction.

Alternatively, when we begin with $\kf$, minimal,
to see that the corresponding $\kf_A$ is minimal, suppose instead we have some $\ti\kf_A$, an $\Ad(AH)$-invariant non-toral subalgebra  with $\ti \kf_A < \kf_A$. Then 
$\ti\kf_A = \h \oplus  \m_{\bldstik\kf}^1  \oplus \m_{\bldstik\kf}^2 \oplus \m_{\bldstik\kf}^3$. 
Let $\ti \kf$ be the associated subalgebra with $\h \oplus \af < \ti \kf$. Since $\kf$ is minimal
we must have $\m_{\bldstik\kf}^2 = \m_{\blds\kf}^2$ and $\m_{\bldstik\kf}^3 = \m_{\blds\kf}^3$.
Thus it would follow that $\m_{\bldstik\kf}^1 \subsetneq  \m_{\blds\kf}^1$, a contradiction.
\end{proof}

It is necessary to assume that $G/H$ has finite fundamental group, as the
homogeneous space $G/H=(S^1\times S^3)/\{e\}$ shows: $K_A=S^3$ is minimal and non-toral, but
the corresponding group $K$ with $A=S^1 \times S^1< K$ is $K=G$.

\begin{proposition}\label{propdelta}
Let $G/H$ be a compact homogeneous space with finite fundamental group. Then
$\DeltaT_{G/H}=\Deltaf_{G/AH}$.
\end{proposition}

\begin {proof} By Lemma \ref{lemfinite} the non-toral, $\Ad(AH)$-invariant subalgebras $\kf_A$
and the subalgebras $\kf$ with $\h \oplus \af < \kf< \g$ are in one-to-one correspondence. 

Let $\kf,\lf$ be two minimal subalgebras with $\h \oplus \af < \kf,\lf < \g$.
Recall $\langle \kf,\lf \rangle$ denotes the subalgebra generated by $\kf$ and $\lf$. We claim that $\langle \kf,\lf \rangle_A = \langle \kf_A,\lf_A \rangle$. Let
$\kf=\kf_A \oplus \af_{\bldss\kf}$ and 
$ \lf=\lf_A \oplus \af_{\bldss\lf}$
with $\af_{\bldss\kf}, \af_{\bldss\lf}\subset \af$.
Since $\kf_A$ and $\lf_A$ are each $\Ad(AH)$-invariant, we obtain
$$
 [\kf,\lf]\subset \kf_A+\lf_A+[\kf_A,\lf_A]\subset \langle \kf_A,\lf_A \rangle \,.
$$ 
Proceeding inductively yields $\langle \kf ,\lf \rangle=\langle \kf_A, \lf_A \rangle
\oplus \af' $ for some $\af' \subset \af$. 
Therefore $\langle \kf ,\lf \rangle$
 corresponds to $\langle \kf_A, \lf_a \rangle$ via the above assignment.

We conclude that the flag complex $\Deltaf_{G/AH}$ can be viewed as a subcomplex of
$\DeltaT_{G/H}$. Since $G/H$ has finite fundamental group,
by Lemma \ref{lemfinite}, every  
$\Ad(AH)$-invariant subalgebra $\kf_A$, generated by minimal, non-toral $\Ad(AH)$-invariant 
subalgebras, is a proper subalgebra of $\g$,  
thus $\Deltaf_{G/AH}=\DeltaT_{G/H}$. 
\end{proof}

Finally, we turn to a simple criterion  which guarantees that the
simplicial complex $\DeltaT_{G/H}$ of a compact homogeneous space
is contractible.

\begin{proposition}\label{propDeltacon}
Let $G/H$ be a compact homogeneous space. If the fundamental group of $G/H$ is infinite,
 then either $\DeltaT_{G/H}$ is contractible or
$\DeltaT_{G/H}=\emp$ and $G/H$ is a torus.
\end{proposition}

\begin{proof} If $\g$ is abelian, then $\DeltaT_{G/H}=\emp$ and
$G/H=T^n$. %change to $G/H =A^n$?
If instead,  $\g$ is not abelian, let $A$ be a
maximal torus of the compact, connected Lie group $G$. 
Since by hypothesis $G/H$ has infinite fundamental group,
$\g=\af \oplus \g'$, where $\af \subset \z(\g)$, 
and $\g'$ is a non-toral 
subalgebra of $\g$ containing $\h$, which is invariant under
the adjoint action of $A$. We conclude $\DeltaT_{G/H}\neq \emp$.
Furthermore, all minimal, non-toral $\Ad(AH)$-invariant subalgebras $\kf_A$ 
with $\h < \kf_A<\g$ are contained in $\g'$. It follows that the subalgebra $\g''$,
generated by all such subalgebras $\kf_A$, is a proper, $\Ad(AH)$-invariant subalgebra of $\g$.
 This implies that $\DeltaT_{G/H}$ is a cone over $\g''$.
\end{proof}

\begin{corollary}\label{cor:gmax}
Let $G/H$ be a compact homogeneous space. If there exists
a subalgebra $\g'$ with $\h < \g' <\g$ containing all minimal, non-toral $\Ad(AH)$-invariant 
subalgebras of $\g$, then $\DeltaT_{G/H}$ is contractible.
\end{corollary}
%%%%%%%%%%%%%%%%%%%%%%%%%%%%%%%%%%%%%%%%%%%%%%%%%%%%%%%%%%%
\subsection{{\it Topological properties of the simplicial complex $\DGH$}} \label{sec:simcom}

In this section we highlight some 
properties of the simplicial complex, most
notably and usefully, we compute the homotopy type of the simplex $\DGH$ for $G/H$ a homogeneous product space.

\medskip

Let $G/H$ be a compact homogeneous space with $\dim N_G(H)=\dim H$.
Then  there exist only finitely many subalgebras $\kf$
with $\h < \kf < \g$, and all must be non-toral, by \cite{BWZ}.

\begin{definition}[Simplicial complex]\label{def:DGH}%%% added 2023
Let $G/H$ be a compact homogeneous space with finite fundamental group.
Let $A$ be a  maximal torus of a compact complement of $H$ in $N_G(H)$. 
Then, the simplicial complex $\DGH$ of $G/H$ is the flag complex
 of (connected) intermediate subgroups $K$  with $AH<K<G$.
\end{definition}

Since $\dim N_G(AH)=\dim AH$ by Lemma \ref{lemAH}, there exist only finitely many
 subalgebras 
 $\kf$
with 
$$
 \h \oplus \af< \kf < \g
 $$ 
 and all such $\kf$ must be non-toral.
Notice furthermore,
that $\Delta_{G/AH}$ is independent of the choice of $A$, see Remark \ref{rem:T},
and that $A=\{e\}$ is allowed; this is the case when $\dim N_G(H)=\dim H$. 
For each $m \geq 1$,  the $(m-1)$-dimensional faces of 
$\Delta_{G/H}$ correspond to the flags 
of length $m$.

Let $\Delta_1$ and $\Delta_2$ be abstract (finite) simplicial
complexes with realizations $\vert \Delta_1 \vert \subset \RR^{p_1}$ and
$\vert \Delta_2 \vert \subset \RR^{p_2}$. The {\em join} of $\Delta_1$
and $\Delta_2$, denoted by $\Delta_1 \join \Delta_2$, is the
abstract simplicial complex which has the realization 
\beg
  \vert \Delta_1 \join \Delta_2 \vert =
   \big\{((1-\lambda)\cdot x_1,\lambda \cdot x_2,\lambda) \mid 0\leq \lambda \leq 1,\,\,
   x_i \in \vert \Delta_i \vert, \, i=1,2 \big\}
\en
in $\RR^{p_1+p_2+1}$: see \cite[2.3.18]{M.C}.  
Notice that $\dim \vert \Delta_1 \join \Delta_2 \vert =
\dim \vert \Delta_1 \vert +\dim \vert \Delta_2 \vert +1$.
If $\Delta_2 =S^0$, then $ \Delta_1 \join \Delta_2$ is
the suspension of $\Delta_1$. Moreover, we have $\emp \join \Delta_1=\Delta_1$,
$\Delta_1 \join \Delta_2 =\Delta_2 \join \Delta_1$ and
$(\Delta_1 \join \Delta_2) \join \Delta_3
=\Delta_1 \join (\Delta_2 \join \Delta_3)$.

\begin{theorem}[\cite{Bo}, \mbox{\cite[Corollary 2.20]{Rau}}] \label{theoprod}
Let $G_1/H_1$ and $G_2/H_2$ be two compact homogeneous spaces with $\dim N_{G_i}(H_i)=\dim H_i$
for $i=1,2$. Then
$$
\Delta_{(G_1 \times G_2)/(H_1 \times H_2)}   
 \approx 
  \Delta_{G_1/H_1} \join \Delta_{G_2/H_2}\join S^0\,.
$$
\end{theorem}

\begin{proof}
We set $\Delta_1:=\Delta_{G_1/H_1}$ and $\Delta_2:=\Delta_{G_2/H_2}$.
If $\Delta_1=\Delta_2=\emptyset$ then 
$$
\Delta_{(G_1 \times G_2)/(H_1 \times H_2)}=\{(\g_1\times \h_2),(\h_1 \times \g_2)\}=S^0\,.
$$
Let $C(\tilde \Delta)$ denote the cone over a simplicial complex $\tilde \Delta$,
that is, the simplicial complex with realization 
$$
 \vert C(\tilde \Delta) \vert =\big\{  (t,t \cdot \tilde x)\mid 0\leq t \leq 1, \, x \in \vert \tilde \Delta \vert\big\}\,.
$$ 
Furthermore, let $0_{\tilde \Delta}:=(0,0)$ denote the vertex of $C(\tilde \Delta)$ which is not contained in
$\tilde \Delta$.

If $\Delta_1=\emptyset$ but $\Delta_2 \neq \emptyset$, then
$\Delta:=\Delta_{(G_1 \times G_2)/(H_1 \times H_2)}$ contains two copies of $\Delta_2$,
namely $\{\h_1\} \times \Delta_2$ and $\{\g_1\} \times \Delta_2$. This covers all vertices in $\Delta$ except
$0_1:=\{\h_1 \times \g_2\}$ and $0_2:=\{\g_1\times \h_2\}$. Now  all (maximal) flags in $\{\h_1\} \times \Delta_2$
can be extended by $0_1$. Thus, the union of all these flags creates a simplicial complex homeomorphic to 
$C_{0_1}(\Delta_2)$. Notice that no other vertex in $\Delta$ is contained in $0_1$ and that $0_1$
is not contained in any other vertex. Similarly, $0_2$ and $\{\g_1\} \times \Delta_2$ form
a simplicial complex homeomorphic to $C_{0_2}(\Delta_2)$. 

Now let
$(\kf_1 < \cdots < \kf_r)$ be a maximal flag in $\Delta_2$ which is realized 
by a (maximal) face $F^{r-1}_2$ of $\Delta_2$. 
Then it can be extended to a maximal flag in $\Delta$
in precisely $r$ ways, namely by
$$
   (\h_1 \times \kf_1 <    \cdots < \h_1 \times \kf_{s}< \g_1 \times \kf_s < \cdots < \g_1 \times \kf_r)
$$
for $s=1,\dots,r$. It is easy to check that the union of the corresponding simplices is homeomorphic
to $F^{r-1}_2\times [0,1]$. This shows the claim in the case when precisely one simplicial complex is empty.

We may now assume that $\Delta_1,\Delta_2 \neq \emptyset$ and  that
$\vert \Delta_1 \vert \subset \RR^{p_1}$ and $\vert \Delta_2 \vert \subset \RR^{p_2}$ 
are realizations of $\Delta_1$ and $\Delta_2$, respectively.
We consider the flag complex $\Delta$ defined by  all flags of subalgebras $\kf < \g_1\times \g_2$ 
with $\h_1 \times \h_2 \leq \kf \leq \g_1\times \g_2$. Notice that when we remove the minimal vertex
$\{\h_1 \times \h_2\}$ and the maximal vertex $\{\g_1 \times \g_2\}$ from $\Delta$ we obtain 
$\Delta_{(G_1 \times G_2)/(H_1 \times H_2)}$.
Now by \cite[(1.2)]{Qui} we know there is a homeomorphism between 

$$\vert \Delta \vert 
\quad\textrm{ and }\quad
  \vert C_{\h_1}(C_{\g_1}(\Delta_1)) \vert \times \vert C_{\h_2}(C_{\g_2}(\Delta_2))\vert \,,
	$$
where we abbreviate $\h_i:=0_{\h_i}$ and $\g_i:=0_{\g_i}$, $i=1,2$. More precisely we may consider
\beg
  \vert C_{\h_1}(C_{\g_1}(\Delta_1))\vert
	&=&
	  \big\{(s_1,s_1(t_1,t_1x_1)) \mid 1 \leq s_1,t_1\leq 1\,,\,\,x_1 \in \vert \Delta_1\vert \big\}\in \RR^{1+1+p_1}\,,\\
	 \vert C_{\h_2}(C_{\g_2}(\Delta_2)) \vert
	&=&
	  \big\{(s_2,s_2(t_2,t_2x_2)) \mid 1 \leq s_2,t_2\leq 1\,,\,\,x_2 \in \vert \Delta_2\vert \big\}\in \RR^{1+1+p_2}\,.	
\en
We have $\{\h_1\times \h_2\}=(0,0,0,0,0,0)$,  $\{\g_1 \times \h_2\} = (1,0,0,0,0,0)$,
$\{\h_1 \times \g_2\} = (0,0,0,1,0,0)$, and  $\{\g_1 \times \g_2\} = (1,0,0,1,0,0)$. 
Next, by \cite[Proposition 1.9]{Qui} we know  there is a canonical homeomorphism between 
\begn \label{eqn:canhom}
  \big\vert (C_{\h_1}(C_{\g_1}(\Delta_1))\times C_{\h_2}(C_{\g_2}(\Delta))  \big\vert
	\bs
  ({\h_1}\times {\h_2})
	\quad\textrm{ and }\quad
 \vert C_{\g_1}(\Delta_1) \join C_{\g_2}(\Delta_2)\vert \,.
\enn
It is simply obtained by considering the curve $\{s_1=1\} \cup \{s_2 =1\} \subset [0,1]\times [0,1]\subset \RR^n$, 
 projecting that curve to the diagonal line segment in $[0,1]\times [0,1]$ joining the points $(0,1)$ and $(1,0)$, 
 parametrized by $\lambda \in [0,1]$.
Notice that under this parametrization, 
$(t_1,t_2,\lambda)=(0,0,0)$ corresponds to $\{\g_1 \times \h_2\}$, while $(0,0,1)$ corresponds to $\{\h_1 \times \g_2\}$, and 
$(0,0,\tfrac{1}{2})$ corresponds to $\{\g_1 \times \g_2\}$.

We now define the simplicial complex 
$\Delta_\ast:= C_{\g_1}(\Delta_1) \join C_{\g_2}(\Delta_2)$, and we remove the vertex $\{\g_1 \times \g_2\}$ from $\Delta_\ast$. 
We have
\beg
\vert \Delta_\ast \vert
 &=&
 \big\{ G(t_1,t_2,\lambda,x_1,x_2) \mid
            0 \leq t_1,t_2,\lambda \leq 1\,,\,\,x_i \in \vert \Delta_i\vert\,,\,\,i=1,2\big\}\,,
\en
where
$$
  G:[0,1]^3\times \vert\Delta_1 \vert\times \vert\Delta_2\vert
	\to \RR^{1+p_1+1+p_2+1}\,\,;\,\,\,
	  (t_1,t_2,\lambda,x_1,x_2)\mapsto  \left(\begin{array}{c} 
		 (1-\lambda)\cdot (t_1,t_1x_1)\\ \lambda\cdot (t_2,t_2x_2) \\\lambda \end{array}\right)\,.
$$
Notice that $G$ is continuous, that $G\vert_{\{\lambda=0\} }$ is independent of $t_2$ and $x_2$,  and that  
$G\vert_{\{\lambda=1\}}$ is independent of $t_1$ and $x_1$.
The image of $G\vert_{ \{t_1=t_2=0\}}$ is a compact line segment in $\Delta_*$, and it
contains  $\{\g_1 \times \g_2\}$, when $\lambda=\frac12$.

Next, we show that
there exists a strong deformation retract from $\vert \Delta_\ast \vert \bs \{\g_1 \times \g_2\}$ onto 
$G(\{t_1=1\}) \cup G(\{t_2=1\})$. To this end notice that it suffices to show that in the parameter space
$W:=[0,1]^3$ there exists a strong deformation retract from $W\bs (0,0,\tfrac{1}{2})$ onto
the union of the four closed faces $F_{\{\lambda=0\}}\cup F_{\{\lambda=1\}} \cup F_{\{t_1=1\}}\cup F_{\{t_2=1\}}$
of $W$. This is  because $G(F_{\{\lambda=0\}}) \subset G(F_{\{t_2=1\}})$
and  $G(F_{\{\lambda=1\}}) \subset G(F_{\{t_1=1\}})$, by the  properties of $G$ above.
This  strong deformation retract is obtained by using a radial projection from $(0,0,\tfrac{1}{2})$
onto this union.  
Now we can extend the deformation retract to $[0,1]^3 \times \vert \Delta_1\vert \times \vert \Delta_2\vert$. 
We deduce that $\vert \Delta_\ast\vert $ is homotopy equivalent to $X:= X_1\cup X_2$, where
\beg
 X_1
  &:=&
	  \big\{((1-\lambda)(1,x_1),\lambda (t_2,t_2x_2),\lambda) \mid
       0\leq \lambda,t_2 \leq 1\,,\,\,x_i \in
       \vert \Delta_i \vert\,,\,\,i=1,2 \big\}\,,\\
 X_2
&:= &
   \big\{((1-\lambda)(t_1,t_1 x_1),\lambda (1,x_2),\lambda) \mid
       0\leq \lambda,t_1 \leq 1\,,\,\,x_i \in
       \vert \Delta_i \vert\,,\,\,i=1,2 \big\}\,.
\en
Now let $\alpha(s)=1$ for $0\leq s \leq 1$,
$\alpha(s)=2-s$ for $1 \leq s \leq 2$,
$\beta(s)=s $ for $0\leq s\leq 1$,
$\beta(s)=1$ for $1 \leq s \leq 2$ and
$$
f(\lambda,\tau) = (1-\lambda)\alpha (2\tau)
\quad \textrm{ and }\quad
g(\lambda,\tau) = \lambda \beta(2\tau)\,.
$$ 
This yields
\beg			
X
 &=&
 \big\{(f(\lambda,\tau) (1,x_1),g(\lambda,\tau)
  (1,x_2),\lambda) \mid 
       0\leq \lambda, \tau \leq 1\,,\,\, x_i \in
       \vert \Delta_i \vert\,,\,\,i=1,2\big\}\,.
\en
The functions $f$ and $g$ are continuous maps $[0,1]^2 \to [0,1]$, both positive on $(0,1)^2$. 
We note $f^{-1}(0)=\{\lambda = 1\}\cup \{\tau = 1\}$ and $f(0,0)=1$ while 
$g^{-1}(0)=\{\lambda=0\}\cup \{\tau=0\}$ and $g(1,1)=1$.
For fixed $(x_1,x_2)\in \vert \Delta_1\vert \times
\vert \Delta_2\vert$ and $(\lambda,\tau),(\lambda',\tau')\in (0,1)^2$ two different parameters
$(\lambda,\tau) \neq (\lambda',\tau')$ parametrize different points in $X$.
On $f^{-1}(0)\bs \{ (0,1),(1,0)\}$ for all $x_1 \in \vert \Delta_1 \vert$ we obtain the same
point in $X$ and correspondingly on
$g^{-1}(0)\bs \{ (0,1),(1,0)\}$ for all $x_2\in \vert \Delta_2 \vert$ we obtain the same
point in $X$. Finally, each of  $(\lambda,\tau)=(0,1)$ and $(\lambda,\tau)=(1,0)$
parametrizes only one point in $\RR^{1+p_1+1+p_2+1}$, namely $(0,0,0,0,0)$ and
$(0,0,0,0,1)$, respectively.

Consider the functions 
$$
 \ti f(\lambda, \tau)=1-\tfrac 12 (\lambda +\tau)- \tfrac 12 |\lambda - \tau | 
\quad \textrm{ and }  \quad
\ti g(\lambda,\tau)=\tfrac 12 (\lambda +\tau)- \tfrac 12 |\lambda - \tau |\,.
$$ 
The functions $\ti f$ and $\ti g$ share several properties with $f$ and $g$: 
they are continuous maps $[0,1]^2 \to [0,1]$, positive on the interior.  
They satisfy $\ti f^{-1}(0)=\{\lambda = 1\} \cup \{\tau = 1\}$ and $\ti f(0,0)=1$, 
while $\ti g^{-1}(0)=\{\lambda=0\} \cup \{\tau=0\}$ and $\ti g(1,1)=1$. 
Notice 
$$
\ti f(\lambda,\tau) = \begin{cases} 1-\lambda & \text{ if } \lambda \geq \tau, \\
      1-\tau & \text{ if } \lambda \leq \tau,
\end{cases}  
 \qquad \text{and} \qquad
\ti g(\lambda,\tau) = \begin{cases} \tau & \text{ if } \lambda \geq \tau, \\
      \lambda & \text{ if } \lambda \leq \tau.
\end{cases}
$$
To understand $\ti f,\ti g$ let us introduce new coordinates $\tilde \lambda = \tfrac 12 (\lambda +\tau)$
and $\tilde \tau =\tfrac 12 (\tau -\lambda)$. Then
$$
   \tilde f(\tilde \lambda,\tilde \tau)=1-2\vert \tilde \tau\vert-(\ti \lambda -\vert \tilde \tau\vert)
	\qquad \text{and} \qquad
	 \tilde g(\tilde \lambda,\tilde \tau)=\ti \lambda -\vert \tilde \tau\vert
$$
with $\tilde \tau \in [-\tfrac 12,\tfrac 12]$ and 
$\ti \lambda \in [\vert \ti \tau\vert,1-\vert \ti \tau \vert]$.
For fixed $\tilde \tau_0 \in (-\tfrac 12,\tfrac 12)$ we have
$$
   \tilde f(\tilde \lambda,\tilde \tau_0)=(1-2\vert \tilde \tau_0\vert)\cdot 
	\big(1-\tfrac{\ti \lambda -\vert \tilde \tau_0\vert}{1-2\vert \tilde \tau_0\vert} \big)
	\qquad \text{and} \qquad
	 \tilde g(\tilde \lambda,\tilde \tau_0)=
	(1-2\vert \tilde \tau_0\vert)\cdot
	\tfrac{\ti \lambda -\vert \tilde \tau_0\vert}{1-2\vert \tilde \tau_0\vert}\,.
$$
Notice this provides the scaling functions of a join for all $\tilde \tau_0$,
scaled by $1-2\vert \tilde \tau_0\vert$.

Let us consider the diagonal $X_{\{\lambda=\tau\} } \subset X$,
parametrized by $\{\lambda =\tau \}$ and set $f(\lambda):=f(\lambda,\lambda)$
and $g(\lambda):=g(\lambda,\lambda)$:
$$
X_{\{\lambda=\tau\} } 
:=
 \big\{(f(\lambda) (1,x_1),g(\lambda)
  (1,x_2),\lambda) \mid 
       0\leq \lambda  \leq 1\,,\,\, x_i \in
       \vert \Delta_i \vert\,,\,\,i=1,2\big\}\,,
$$
We claim that $X_{\{\lambda =\tau\}}$ is homeomorphic to 
\begin{eqnarray*}
\ti X_{\{\lambda =\tau\}}
 &=& 
\big\{(\tilde f(\lambda)(1,x_1), \tilde g(\lambda) (1,x_2),\lambda) 
         \mid 0\leq \lambda  \leq 1\,,\,\, x_i \in
       \vert \Delta_i \vert\,,\,\,i=1,2 \big\}\\
&=&
\left\{((1-\lambda)(1,x_1), \lambda (1,x_2),\lambda) 
     \mid 0\leq \lambda  \leq 1\,,\,\, x_i \in
       \vert \Delta_i \vert\,,\,\,i=1,2 \right\}\,,
\end{eqnarray*}
where $\ti f(\lambda):=\ti f(\lambda,\lambda)$
and $\ti g(\lambda):=\ti g(\lambda,\lambda)$.
To this end let $h_{\{\lambda =\tau\}}:X_{\{\lambda =\tau\}}\to \ti X_{\{\lambda =\tau\}}$ be defined by
mapping 
$(f(\lambda) (1,x_1),g(\lambda)
  (1,x_2),\lambda)$ onto
$(\tilde f(\lambda) (1,x_1),\tilde g(\lambda)
  (1,x_2),\lambda)$, for
	$0\leq \lambda  \leq 1\,,\,\, x_i \in
       \vert \Delta_i \vert\,,\,\,i=1,2$. Since the functions $f,g,\tilde f,\tilde g$ are non-negative
	and $f,\ti f$ and $g,\ti g$ do have the same zero locus, see above, we indeed obtain a map from
	$X_{\{\lambda =\tau\}}$ to $\ti X_{\{\lambda =\tau\}}$ this way. Moreover, it is clear that
	the map $h_{\{\lambda =\tau\}}$ is bijective. Continuity follows since the map $f(\lambda)\to \ti f(\lambda)$
	is continuous using that $f$ has a continuous inverse. The continuity of the inverse map follows the same way.
	Thus $X_{\{\lambda =\tau\}}$ is homeomorphic to $\vert \Delta_1 * \Delta_2 \vert$.
	
	To show that $X$ is homeomorphic to the suspension of $\vert \Delta_1 * \Delta_2 \vert$ we 
	show now that $X$ is homeomorphic to $\tilde X$ with
	\beg			
\ti X
 &=&
 \big\{(\ti f(\lambda,\tau) (1,x_1),\ti g(\lambda,\tau)
  (1,x_2),\lambda) \mid 
       0\leq \lambda, \tau \leq 1\,,\,\, x_i \in
       \vert \Delta_i \vert\,,\,\,i=1,2\big\}\,.
\en
As above let $h:X\to \ti X$ be defined by
mapping 
$(f(\lambda,\tau) (1,x_1),g(\lambda,\tau)
  (1,x_2),\lambda)$ onto
$(\tilde f(\lambda,\tau) (1,x_1),\tilde g(\lambda,\tau)
  (1,x_2),\lambda)$, for
	$0\leq \lambda,\tau  \leq 1\,,\,\, x_i \in
       \vert \Delta_i \vert\,,\,\,i=1,2$. 	
			Using the properties of $f,g,\tilde f,\tilde g$ we see that indeed $h$ is a bijective map
			from $X$ to $\tilde X$. Continuity of $h$ and its inverse follow also as above.
This shows the claim.
\end{proof}

\begin{remark}
 A realization of
$\Delta_\ast \bs \{\g_1 \times \g_2\}$ is the subcomplex of $\vert \Delta_\ast \vert$,
obtained by removing the (open) simplicial neighborhood of the vertex $\g_1 \times \g_2 \in \vert \Delta_\ast\vert$. Clearly,
$\vert \Delta_\ast \bs \{\g_1 \times \g_2\}\vert$ and $\vert \Delta_\ast \vert \bs \{\g_1 \times \g_2\}$
are homotopy equivalent.
Although it is possible to improve the above result to show that $\vert \Delta_\ast \bs \{\g_1 \times \g_2\}\vert$
is homeomorphic to $X_1 \cup X_2$, this requires precise bookkeeping of the vertices under the
canonical homeomorphism in \eqref{eqn:canhom}.
\end{remark}

Since $S^0 \join S^0=S^1$ and more generally, the join of $p\geq 2$ copies of $S^0$, is $S^0 \join \cdots \join S^0=S^{p-1}$, $p\geq 2$ copies of $S^0$,
we  immediately obtain

\begin{corollary} \label{corprod}
Let $G/H$ be a compact homogeneous space. If
$G=G_1 \times \cdots \times G_p$ and
$H=H_1 \times \cdots \times H_p$, $p \geq 2$, with $H_i \subsetneq G_i$ and $\dim N_{G_i}(H_i)=\dim H_i$ for each $i=1,\dots,p$,
then 
$$
 \Delta_{G/H} \approx 
 \Delta_{G_1/H_1} \join \cdots \join
 \Delta_{G_p/ H_p} \join S^{p-2}\,.
$$
\end{corollary}

Even when $\Delta_{G_i/H_i}=\emp$ for certain $i$, this gives
interesting results. For instance in the case
$\Delta_{G_i/H_i}=\emp$ for all $1\leq i \leq p$, we get
$\Delta_{G/H}=S^{p-2}$.

\begin{corollary} \label{cortop}
Let $G/H=G_1/H_1\times G_2/H_2$ be a
compact product homogeneous space with $\dim N_{G_i}(H_i)=\dim H_i$ for $i=1,2$. Then, if either $\Delta_{G_1/H_1}$ or
$\Delta_{G_2/H_2}$ is contractible, then $\Delta_{G/H}$ is
contractible too.
\end{corollary}

Conversely, we expect that if $\Delta_{G_1/H_1}$ and $\Delta_{G_2/H_2}$ are non-contractible so is $\Delta_{G/H}$.
To some extent this is known. The reduced homology groups of joins
have been computed by \mbox{J. Milnor} \cite{J.M}: If $\ti
H_q(\Delta_1,{\mathbb F})$ and $\ti H_{\ti q}(\Delta_2,{\mathbb
F})$ are non-trivial for a field ${\mathbb F}$, then 
\begn \label{eqn:Htildejoin}
\ti H_{q+\ti q+1}(\Delta_1 \join \Delta_2,{\mathbb F})
\enn
is also non-trivial.

\begin{remark}\label{remtop}

Note that Milnor's result does not guarantee non-contractibility in general. 
We do not have a flag example to illustrate this, but here we describe an example from \cite{M.C}. 
There exists a
\mbox{2-dimensional}, non-contractible simplicial complex $\Delta$
such that $\Delta \join S^0$ is contractible.
The space $\Delta$ obtained from five points
$a_0,a_1,a_2,a_3,a_4$, ten $1$-cells (joining each pair of points)
and six 2-cells given by the edge-loops $a_0a_1a_2a_3a_4a_0$,
$a_0a_4a_1a_3a_2a_0$, $a_0a_2a_4a_3a_1a_0$, $a_0a_4a_2a_1a_3a_0$,
$a_0a_3a_2a_4a_1a_0$, $a_0a_2a_1a_4a_3a_0$ is non-contractible with 
$\ti H_n(\Delta,\ZZ)=0$ for all $n \in \NN_0$: see
Example 3.3.22. and Remark 8.3.12 in \cite{M.C}. 
Since $\Delta \join S^0$ is simply connected with 
 $\ti H_n(\Delta \join S^0,\ZZ)=0$ for all $n \in \NN_0$ it follows that $\Delta \join S^0$
is contractible: see Corollary 8.3.11 in \cite{M.C}.
\end{remark}

Continuing with the hypothesis  $\dim N_G(H)=\dim H$,
in this setting we know that maximality of $\h$ in $\kf$ is equivalent to minimality of $\kf$ in $\g$ (see Definition \ref{def:min-sub}). 
There are no toral intermediate subalgebras, and both 
$\Submin =\{\kf \in \Sub_s \mid \kf \textrm{ is a minimal subalgebra}\}$ and  $\Submg$,  
the set of subalgebras generated by minimal subalgebras, are finite sets. 
We denote by $\Deltaf_{G/H}$ the corresponding flag complex.
Clearly $\Deltaf_{G/H}$ is a subcomplex of $\Delta_{G/H}$. 
By Lemma \ref{lem:DGH}, the nerve $\XGH$ is the simplicial complex $\DGH$, and we know 
$\XGH$ is homotopy equivalent to the minimal nerve $\XGHmin$.

\begin{lemma}\label{lem:rsimhomsim}
Let $G/H$ be a compact homogeneous space with $\dim N_G(H)=\dim H$. Then 
$\Deltaf_{G/H}$ is  homotopy equivalent to  $\Delta_{G/H}$. 
\end{lemma}

\begin{proof} 
We know the set $\Submin$   of minimal proper intermediate subalgebras of $\g$ is a finite set: $\{\lf_1,\dots,\lf_t\}$, and the set of all proper intermediate subalgebras they generate, $\Submg$, is also finite. Let $i:\Submg \to \Sub$ denote the inclusion map. 
We show that for each $\kf \in \Sub$, the set $i_{\bldss\kf} :=\{\lf \in \Submg \mid \lf \leq \kf\}$   
 has a maximal element. It is clear that $i_{\bldss\kf}$ cannot be empty.  Suppose $\lf_{1},\dots, \lf_{s}$ are all the minimal subalgebras contained in $\kf$ (renumbering, possibly). Let $\lf^* := \langle \lf_1,\dots,\lf_s\rangle$, the subalgebra 
 generated by $\lf_{1},\dots, \lf_{s}$.  Clearly $\lf^* \in \Submg$, and since each  $\lf_i \leq \kf$ ($i=1,\dots,s$), we know $\lf^* \leq \kf$.  

We claim that $\lf^*$ is the maximal element of $i_{\bldss\kf}$.  To this end
suppose $\kf' \in i_{\bldss\kf}$, yet $\kf' \nleqslant \lf^*$.  By construction,  $\kf' \in \Submg$ means $\kf' = \langle \lf_{i_1},\dots, \lf_{i_m}\rangle$. If  there exists $i_j \not\in \{1,\dots,s\}$, then
$\kf' \nleqslant \kf$ since the containment $\lf_{i_j} \leq \kf$ does not hold.  Thus for all $\kf' \in i_{\bldss\kf}$, $\kf' \leq \lf^*$.

By \cite[1.5]{Qui}, since for any $\kf \in \Sub$, the set $i_{\bldss\kf}$  has a maximal element, $\Delta(i_{\bldss\kf})$ is contractible, and since this holds  for every $\kf \in \Sub$, by \cite[Prop.1.6]{Qui} 
the inclusion map $i: \Submg \to \Sub$ induces a homotopy equivalence  $\Deltaf_{G/H} \approx \Delta_{G/H}$. 
\end{proof}

Note that we obtain Corollary \ref{cor:gmax} immediately from Lemma \ref{lem:rsimhomsim}.

With this purely topological proof of Lemma \ref{lem:rsimhomsim}, we obtain
Corollary \ref{cor:LKmax}, which we could not have deduced otherwise.
We replace  $\leq$ with $\geq$, minimal with maximal, and {\sl generated from} by {\sl intersection}.

\begin{corollary}\label{cor:LKmax}
Let $G/H$ be a compact homogeneous space with $\dim N_G(H)=\dim H$. Suppose that
there exists a subalgebra $\kf \in\Sub$ such that $\kf\leq \lf$ for every $\lf \in \Sub$ which is maximal in 
$\g$. Then $\DGH$ is contractible.
\end{corollary}

%%%%%%%%%%%%%%%%%%%%%%%%%%%%%%%%%%%%%%%%%%%%%%%%%%%%%%%%%%%%%
\subsection{{\it Examples of the simplicial complex $\DGH$}}  \label{sec:examples}

In this section we will provide several examples of compact homogeneous spaces $G/H$ for which the  
 simplicial complex $\DGH$ is non-contractible and contractible. 
We will work with the (connected) subgroups corresponding to intermediate subalgebras.
To simplify our results we declare $\ti H_{-1}(\emptyset) \neq 0$; recall that
the empty set was declared to be non-contractible.

\medskip

Let $G$ be one of the classical simple groups, that is $G={\rm SU}(n)$, $n\geq 1$, $G={\rm SO}(2n)$, $n\geq 3$, 
$G={\rm SO}(2n+1)$ or $G={\rm Sp}(n)$, $n\geq 2$.

\begin{theorem}[\cite{Bo}, \cite{Rau}]\label{thm:RA}
Let $G={\rm SU}(n)$, $n\geq 2$ and let 
$$
  H={\rm S}({\rm U}(n_1)\times \cdots \times {\rm U}(n_k))
$$
with $n =\sum_{i=1}^k n_i$, $n_1,...,n_k\geq 1$, $k\geq 2$.
Then $\ti H_{k-3}(\DGH) \neq 0$.
\end{theorem}

Note that all subgroups $H$ of ${\rm SU}(n)$ of maximal rank are covered by the above theorem and 
that all these spaces are known admit a homogeneous K\"ahler-Einstein metric (see \cite[8.111]{Bes}).

\begin{theorem}[\cite{Rau}]\label{thm:RB}
Let $G={\rm SO}(2n)$, $n\geq 3$, and let $H$ be a compact subgroup of $G$ with $\rk G=\rk H$.
Then $\DGH$ is non-contractible if and only if
$$
  H= {\rm U}(n_1)\times \cdots \times {\rm U}(n_k)
	\quad \textrm{ or }\quad
	H= {\rm SO}(2n_1)\times \cdots  \times {\rm SO}(2n_k)
$$
where $\sum_{i=1}^k n_i=n$, $n_1,\dots,n_k\geq 1$, $k \geq 1$ in  the first case
and $k\geq 2$ in the second. Moreover,
in the first case we have $\ti H_{k-2}(\DGH)\neq 0$ and in  the second $\ti H_{k-3}(\DGH)\neq 0$. 
\end{theorem}

In certain cases the spaces of the first kind admit a homogeneous K\"ahler-Einstein metric: see \cite[8.113]{Bes};
here the embeddings of ${\rm U}(n_i)$ into ${\rm SO}(2n_i)$ are important ($1 \leq i \leq k$). The spaces of
the second kind do not admit  homogeneous K\"ahler-Einstein metrics unless $n_1=...=n_k=1$,  
but the spaces with 
$H= {\rm U}(n_1)\times \cdots \times {\rm U}(n_{k-1}) \times {\rm SO}(2n_k)$ do (see  \cite[8.113]{Bes}).
When $k \geq 2$ and $n_1,n_k>1$, this is an example where $H$ has mixed type.

\begin{theorem}[\cite{Rau}]\label{thm:RC}
Let $G={\rm SO}(2n+1)$, $n\geq 2$ and let $H$ be a compact subgroup of $G$ with $\rk G=\rk H$.
Then $\Delta_{G/H}$ is non-contractible if and only if
$$
  H= {\rm SO}(2n_1)\times \cdots  \times {\rm SO}(2n_k)
	\quad \textrm{ or }\quad
	H= {\rm SO}(2n_1)\times \cdots  \times {\rm SO}(2n_{k-1})\times {\rm SO}(2n_k+1) 
$$
where $\sum_{i=1}^k n_i=n$, $n_1,...,n_k\geq 1$,
 $k \geq 1$ in  the first case
and $k\geq 2$ in the second. Moreover,
in the first case we have $\ti H_{k-2}(\DGH)\neq 0$ and in  the second $\ti H_{k-3}(\DGH)\neq 0$. 
\end{theorem}

Except for the Grassmannians, these spaces do not admit homogeneous K\"ahler-Einstein metrics:
see \cite[8.115]{Bes}.

\begin{remark}\label{rem:mixtype}
In the situation of Theorem \ref{thm:RC} we declare the subgroups
$$
  H= {\rm U}(n_1)\times \cdots \times {\rm U}(n_j) \times {\rm SO}(2) \times \cdots \times {\rm SO}(2),
$$	
where $n=\sum_{i=1}^k n_i =n$, $1 \leq j \leq  k$, $n_1,...,n_j \geq 2$ to be of mixed type, e.g. 
$H={\rm U}(n)$ or $H={\rm U}(n-1){\rm SO}(2)$, $n \geq 3$ are of mixed type.
\end{remark}

\begin{theorem}[\cite{Rau}]\label{thm:RD}
Let $G={\rm Sp}(n)$, $n\geq 2$ and let $H$ be a compact subgroup of $G$ with $\rk G=\rk H$.
Then $\Delta_{G/H}$ is non-contractible if and only if
$$
  H= {\rm U}(n_1)\times \cdots \times {\rm U}(n_k)
	\quad \textrm{ or }\quad
	H= {\rm Sp}(n_1)\times \cdots  \times {\rm Sp}(n_k)
$$
where $\sum_{i=1}^k n_i=n$, $n_1,...,n_k\geq 1$,
 $k \geq 1$ in  the first case
and $k\geq 2$ in the second. Moreover,
in the first case we have $\ti H_{k-2}(\DGH)\neq 0$ and in  the second $\ti H_{k-3}(\DGH)\neq 0$. 
\end{theorem}

These spaces do not admit homogeneous K\"ahler-Einstein metrics except for ${\rm Sp}(n)/{\rm U}(n)$:
see \cite[8.115]{Bes}.

\begin{proposition}[\cite{Rau}]\label{prop:DGHcontr}
Let $G$ be one of the classical Lie groups and let $H$ be a subgroup of $G$
with $\rk G=\rk H$ of mixed type. Then $\DGH$ is contractible.
\end{proposition}

\begin{proof}
The claim follows from Corollary \ref{cor:gmax} once we show that there exists a subgroup $K$ 
such that for every maximal $L < G$ containing $H$, we have $H<K \leq L$. Since the intermediate subgroups $L$ must contain $H$,
they must be adapted to the block decomposition dictated by $H$. 

If $G={\rm SU}(n)$, $n\geq 1$, then there exists no such $H$ of mixed type.

If $G={\rm SO}(2n)$, $n\geq 3$, then it follows $n \geq 4$ and 
$$
 H= {\rm U}(n_1)\times \cdots \times {\rm U}(n_j)
 \times {\rm SO}(2n_{j+1})\times \cdots  \times {\rm SO}(2n_k),
$$
where $k \geq 2$, 
$n=\sum_{i=1}^k n_i =n$, $1 \leq j < k$, $n_1,...,n_j \geq 2$, and $n_k>1$, while (if $k > j+1$), 
each $n_{j+1},...,n_{k-1}\geq 1$ 
(after renumbering possibly). We set
$$
 K :={\rm SO}(2n_1)\times \cdots \times {\rm SO}(2n_j)
 \times {\rm SO}(2n_{j+1})\times \cdots  \times {\rm SO}(2n_k)\,.
$$
Now let $L$  be a maximal subalgebra of $G$ with $H <L$. Then 
we know, since $n_1,n_k>1$, that $L={\rm SO}(2p)\times {\rm SO}(2(n-p))$ for some $n_1 \leq p \leq n-n_k$  
(see Section \ref{sec:eq} and \cite[Table 5.1]{Rau}). 
The condition $n_k>1$ prevents $L \cong {\rm U}(n)$. Since $H < L$ we deduce $K \leq L$.

If $G={\rm SO}(2n+1)$, $n\geq 2$, there are two possibilities. One is
$$
  H= {\rm U}(n_1)\times \cdots \times {\rm U}(n_j) \times {\rm SO}(2n_{j+1}) \times \cdots \times {\rm SO}(2n_k),
$$	
	where $n=\sum_{i=1}^k n_i =n$, $1 \leq j \leq  k$, $n_1,...,n_j \geq 2$ and  $n_{j+1},...,n_{k}\geq 1$.
 We again set 
$$
 K :={\rm SO}(2n_1)\times \cdots \times {\rm SO}(2n_j)
 \times {\rm SO}(2n_{j+1})\times \cdots  \times {\rm SO}(2n_k)\,.
$$	
As above the claim follows in this case. The other possibility is
$$	
	H= {\rm U}(n_1)\times 
		 \cdots \times {\rm U}(n_j) \times {\rm SO}(2n_{j+1}) \times
		\cdots \times {\rm SO}(2n_k+1), 
$$
	where $n=\sum_{i=1}^k n_i =n$, $1 \leq j <  k$, $n_1,...,n_j \geq 2$ and  $n_{j+1},...,n_{k}\geq 1$.
We set 
$$
 K :={\rm SO}(2n_1)\times \cdots \times {\rm SO}(2n_j)
 \times {\rm SO}(2n_{j+1})\times \cdots  \times {\rm SO}(2n_k+1)
$$	
and the claim follows as above.

If $G={\rm Sp}(n)$, $n\geq 2$ then
$$
 H= {\rm U}(n_1)\times \cdots \times {\rm U}(n_j)
	\times {\rm Sp}(n_{j+1})\times \cdots  \times {\rm Sp}(n_k),
$$
where $k\geq 2$, $\sum_{i=1}^k n_i=n$, $1 \leq j<k$, $n_1,...,n_k\geq 1$. 
We set 
$$
 K :={\rm Sp}(n_1)\times \cdots \times {\rm Sp}(n_j)
 \times {\rm Sp}(n_{j+1})\times \cdots  \times {\rm Sp}(n_k)
$$	
and the claim follows as above. 
\end{proof}

\begin{remark}
The above results of Rau\ss e can be summarized as follows. For a classical 
simple group $G$ and a homogeneous space $G/H$ for which $\rk G=\rk H$, the simplicial complex $\DGH$ is non-contractible if and only if $H$ has simple factors of the same type: see Remark \ref{rem:mixtype}.
For exceptional groups the situation is much more complicated: see \cite{Rau}. For instance, it should be possible
to find compact homogeneous spaces $G/H$ such that $\ti H_q(\DGH)\neq 0$ for
two different $q \geq 0$. So far we know of no examples of this type.
\end{remark}

%%%%%%%%%%%%%%%%%%%%%%%%%%%%%%%%%%%%%%%%%%%%%%%%%%%%%%%%%%%%

\section{Miscellaneous}\label{sec:miscell}

\subsection{{\it An approach to non-existence}}\label{sec:eq}

In this section we describe examples of compact homogeneous spaces $G/H$ 
with contractible simplicial complex $\DGH$ (of equal rank) which admit or do not admit
$G$-invariant Einstein metrics. We also indicate how purely Lie theoretic assumptions
can help to consider certain combinations of the Einstein equation,
which essentially can be written as a sum of quadratic polynomials.

We will assume from now on that the isotropy representation of $G/H$
$$
  \m=\m_1 \oplus \cdots \oplus \m_\ell
$$
decomposes into isotypical summands and that $G/H$ has finite fundamental group.
For example, when  $\rk G =\rk H$ this is the case.
Under this assumption, any metric $g \in \MG$ is uniquely determined by positive numbers
$x_1,...,x_\ell>0$ with $g\vert_{m_i}=x_i \cdot Q\vert_{m_i}$, $i=1,...,\ell$, where $Q$ denotes our
background metric on $G/H$. In this case the eigenvalues $r_i$ of the Ricci tensor
$\Ric(g)\vert_{\m_i}$ are given  by
\begin{eqnarray}\label{eqn:einstein}
  r_i =\tfrac{b_i}{2x_i}-\tfrac{1}{2d_i}\sum_{j,k=1}^\ell [ijk] \tfrac{x_j}{x_ix_k}
	  +\tfrac{1}{4d_i}\sum_{j,k=1}^\ell [ijk] \tfrac{x_i}{x_jx_k}
\end{eqnarray}
for $i=1,...,\ell$: see e.g. \cite{WZ2}. Recall, the definitions of the
structure constants $[ijk]\geq 0$ and $b_i \geq 0$ can be found in Section \ref{sec:scal}.
Recall $[ijk]$ is invariant under permutation of $i,j,k$.

If $G/H$ is isotropy irreducible, that is if $\ell=1$, then
by Schur's Lemma, up to scaling, there exists only
one symmetric $G$-invariant bilinear form on $G/H$. Consequently,
each $G$-invariant metric is Einstein.

If the isotropy representation $\m$ can be decomposed into
two irreducible, inequivalent summands,
 the Einstein equation with unknowns $x_1,x_2>0$ and positive Einstein constant
$\lambda>0$ is given by
\beg
  \big(\tfrac{b_1}{2}-\tfrac{[111]}{4d_1}-\tfrac{[122]}{2d_1}\big)\cdot \tfrac{1}{x_1}
  -\tfrac{[112]}{2d_1}\cdot \tfrac{x_2}{x_1^2}
  +\tfrac{[122]}{4d_1}\cdot \tfrac{x_1}{x_2^2} &=&\lambda\\
 \big(\tfrac{b_2}{2}-\tfrac{[222]}{4d_2}-\tfrac{[112]}{2d_2}\big)\cdot \tfrac{1}{x_2}
  -\tfrac{[122]}{2d_2}\cdot \tfrac{x_1}{x_2^2}
  +\tfrac{[112]}{4d_2}\cdot \tfrac{x_2}{x_1^2} &=&\lambda \,.
\en
If $[112],\,[221]>0$, then $\h$ is a maximal subalgebra of $\g$, hence
there exists at least one solution since $\DGH= \emptyset$.

Next we consider the case when $\kf=\h\oplus \m_1$ is a subalgebra of $\g$, and thus $[112]=0$.
If $[122]=0$ as well, then there exists a solution as well, since $\DGH=\{\kf, \lf\}$ is non-contractible.
In fact this tells us $G/H$ is (on Lie algebra level) a product of isotropy irreducible homogeneous spaces.

Hence we may assume that $[112]=0$ and $[122]>0$. If $\kf$ is a toral subalgebra,
then again there exists a solution since $\DGH=\emptyset$.
Thus we are left with the case that $\kf$ is a non-toral subalgebra and $[221]>0$. 
Note that $\DGH=\{\kf\}$ is contractible in this case. By \cite{WZ2}, 
the above system admits  a solution if and only if
\begn
D:=
  \big(b_2-\tfrac{[222]}{2d_2}\big)^2
-4\big(b_1-\tfrac{[111]}{2d_1}-\tfrac{[122]}{d_1} \big) [122]
 \big(\tfrac{1}{2d_1}+\tfrac{1}{d_2}\big)
    \geq 0\,. \label{m2}
\enn
Thus in the two summand case, when $\DGH$ is contractible,
the algebraic invariant $D$ determines whether 
there exists a $G$-invariant Einstein metric on $G/H$.

\begin{remark}
For $G$ simple, the two summand case is classified, see \cite{DK}, \cite{He}.
We indicate here briefly how to achieve this classification
in case $\rk G=\rk H$. For a compact, simple Lie group $G$, the maximal subgroups $K$ 
with $\rk K=\rk G$  are well known, see e.g. \cite{WZ1}.
\begin{itemize} 
\item $G=\SU(n)$: $K=S({\rm U}(k){\rm U}(n-k))$, $1 \leq k \leq \lfloor\frac n2\rfloor$
\item $G=\SO(2n+1)$: $K=\SO(2n)$ or $K=\SO(2k)\SO(2n-2k+1)$, $1 \leq k < n$ 
\item $G=\Sp(n)$: $K={\rm U}(n)$ or  $K=\Sp(k)\Sp(n-k)$, $1\leq k \leq \lfloor\frac n2\rfloor$
\item $G=\SO(2n)$: $K={\rm U}(n)$ or $K=\SO(2k)\SO(2n-2k)$, $1 \leq k \leq \lfloor\frac n2\rfloor$  
\item $G=G_2$: $K=\SU(3),~\SO(4)$
\item $G=F_4$: $K=\Spin(9),~\Sp(1)\Sp(3),~(\SU(3))^2$ 
\item $G=E_6$: $K=\SU(6)\Sp(1),~(\SU(3))^3,~\Spin(10){\rm U}(1)$
\item $G=E_7$: $K=E_6 \SO(2),~\SU(8),~\Spin(12)\Sp(1),~\SU(3)\SU(6)$ 
\item $G=E_8$: $K=\SU(9),~\Spin(16),~(\SU(5))^2,~E_6\SU(3),~E_7\Sp(1)$. 
\end{itemize}
Now if $G/K$ is isotropy irreducible then  $K$ is a maximal subgroup in $G$ and we look for a maximal subgroup $H$ in $K$ such that not only is $K/H$ isotropy irreducible but also the isotropy representation of $K$ on $G/K$, when restricted to $H<K$, remains irreducible as a representation of $H$. There exist 
two infinite families and 11 more examples \cite{DK},\cite{He}, 
all admitting two $G$-homogeneous Einstein metrics. There are 
two spaces  which do not admit $G$-invariant Einstein metrics:
$E_7/\Spin(6)\Spin(6)\Sp(1)$ ($K=\Spin(12)\Sp(1)$) and
$E_8/ \Sp(1)\times \SU(8)$ ($K=\Sp(1)\times E_7$).

There is exactly one homogeneous space $G/K$ where $G$ is compact, connected, simple, and $K$ is a maximal subgroup of maximal rank in $G$ which is {\em not} isotropy irreducible \cite[Ex.5]{WZ1}: $G/H=E_8/ \SU(5)\SU(5)$. 
The isotropy representation has two irreducible modules, and since they are equivalent by an outer automorphism of 
$E_8$, the  space $E_8/\SU(5)\SU(5)$ admits a normal $G$-invariant Einstein metric \cite{WZ1}. 
\end{remark}

\begin{remark}
In the three summands case we comment on a possible classification in the equal rank case.
When $G$ is a classical compact Lie group, see Section \ref{sec:examples},  
Theorems  \ref{thm:RA}, \ref{thm:RB}, \ref{thm:RC}, \ref{thm:RD}, we have the following infinite families of homogeneous spaces whose isotropy representation has three irreducible summands (all admitting homogeneous Einstein
metrics):
\begin{itemize}
\item $G=\SU(n)$: $H=S({\rm U}(k_1){\rm U}(k_2){\rm U}(k_3))$ \quad ($k_1+k_2+k_3=n$, $k_i \geq 1$) 
\item $G=\SO(n)$: $H=\SO(2k_1)\SO(2k_2)\SO(m)$  \quad ($2k_1+2k_2+m=n$, $m\geq 3$, $k_i \geq 1$) 
\item $G=\Sp(n)$: $H=\Sp(k_1)\Sp(k_2)\Sp(k_3)$  \quad ($k_1+k_2+k_3=n$, $k_i \geq 1$). 
\end{itemize}
For the exceptional Lie groups we will only present the examples with 
contractible $\DGH$. Due to  Rau\ss e \cite{Rau} there 
exist only three such spaces: 
\begin{itemize}
\item
$E_7/\Spin(6)\Spin(6)\Sp(1)$, $H < K=\Spin(12)\Sp(1)$
\item 
$E_8/\Spin(10)\Spin(6)$, $H < K=\Spin(16)$   
\item
$E_8/\Spin(12)\SO(2)\Sp(1) = E_8/\Spin(12){\rm U}(2)$, $H< K_1=\Spin(12)\Spin(4),~K_2=E_7\SO(2)$ while $K_2<K_3= E_7\Sp(1)$, $K_1<K_3,~K_4=\Spin(16)$.
\end{itemize}
\end{remark}

Finally, we come back to the Example \ref{exa:mixedtype}, $G/H \in \Nl_>$.
We show how a special subgroup structure can be used to obtain
purely algebraic invariants which can (possibly) determine whether or
not $G/H$ admits a $G$-invariant Einstein metric.
Recall $G:=\SO(2n)$ where $n=\sum_{i=1}^r n_i$ with $n_1,\dots,n_r \geq 2$, $r \geq 2$, and 
$$
   H:= \Un(n_1) \times \SO(2n_2)\times \cdots \times  \SO(2n_{r}) \,.
$$
Let 
$$
 K=\SO(2n_1)\times \cdots \times  \SO(2n_{r}) 
\quad \textrm{ and }\quad
  K_*=\Un(n_1)\times\SO(2(n_2+\cdots + n_{r})) \,.
$$
We compute the isotropy representation of $G/H$. We have
$$
  \m = \m_d \oplus \m_{\blds{\kf}_*} 
	\quad
	\textrm{ and }\quad
	\m_d= \m_1 \oplus \m_2 \oplus \cdots \oplus \m_r 
$$ 
with $\kf = \h \oplus \m_1$,  and for each $i \in I:=\{2,...,r\}$, $\m_i$ corresponds to the
complement of $\SO(2n_1)\times \SO(2n_i)$ in $\SO(2(n_1+n_i))$. 
The summands $\m_1,...,\m_r$ are irreducible, but, when $r \geq 4$, the summand $\m_{\blds{\kf}_*} $ is not.

Note for each $i \in 1,...,r$ we have $[iii]=0$, $[1jk]=0$ for all $j \neq k$,
and also $[ijk]=0$ for all $i,j,k \in I$. We have $[1ii]>0$ for all $i\in I$
and $[iik]=0$ for all $i \in I$ and $k \neq 1,...,r$.
Moreover, for all $i,j \in I$, $i \neq j$, there exists precisely one $k_{ij}$ with $[ijk_{ij}]>0$.
Note $k_{ij} \neq 1,...,r$.  

We have
 $d_1=n_1(n_1-1)$, $d_i=4n_1 n_i$ for $i \in I$, and (for $Q(X,Y)=-\frac12 \tr(XY)$), $b_i=b=4(n-1)$. We see for each $i\in I$, $[1ii]=4n_1(n_1-1)n_i$. For convenience, we note that $\tfrac{[1ii]}{d_i} = \tfrac{4n_1(n_1-1)n_i}{4n_1 n_i}=(n_1-1)$.

 We compute
the Ricci curvature $\Ric(g)$ restricted to $\m_1$ and $\m_i$ for $i \in I$ using \eqref{eqn:einstein} and
Lemma \ref{lem:1p5}: 
\begin{eqnarray*}
   r_1 &=& \tfrac{c_1}{x_1} 
	+\tfrac{1}{4d_1} \sum_{i \in I}[1ii]\tfrac{x_1}{x_i^2}  = \tfrac{c_1}{x_1} +\sum_{i \in I} n_i \tfrac{x_1}{x_i^2}\,,\\
	 r_i & =& \tfrac{4(n-1)}{2x_i}
	             - \tfrac{n_1-1}{2} \tfrac{x_1}{x_i^2}
							  -\tfrac{1}{2d_i} \sum_{j \in I\bs \{i\} }[ij k_{ij}]\tfrac{1}{x_i}
								    \big(\tfrac{x_j}{x_{k_{ij}}} + \tfrac{x_{k_{ij}} }{x_j}  \big)
							+\tfrac{1}{2d_i} \sum_{j \in I\bs \{i\}}[ij k_{ij}]
								   \tfrac{x_i}{x_j x_{k_{ij}}}\,.
\end{eqnarray*}
Next, we compute the trace free part of $\Ric(g)\vert_{\m_d} \in {\rm End}(\m_d)$ restricted
to $\m_1$. (For an Einstein metric, the trace free part is zero.) We see
$$
  r_1 \sum_{i=1}^r d_i - \sum_{i=1}^r d_i r_i = r_1 \sum_{i \in I} d_i -\sum_{i \in I} d_i r_i \,,
$$
and
$$
\sum_{i \in I} d_i r_i
= \sum_{i \in I} (\tfrac{d_ib}{2x_i}
	             - \tfrac{1}{2} [1ii]\tfrac{x_1}{x_i^2})
							  -\tfrac{1}{2} \sum_{i \neq j \in I }[ij k_{ij}] \tfrac{x_{k_{ij}} }{x_ix_j}\,.
$$
Setting $d:=\sum_{i=1}^r d_i$ and $\alpha_i:=\tfrac{x_1}{x_i}$  for $i \in I$ we arrive at
\begin{eqnarray*}
 \lefteqn{4(  r_1 \sum_{i=1}^r d_i - \sum_{i=1}^r d_i r_i ) =}&&\\
	  &=&
		(\tfrac{4c_1}{x_1} +\tfrac{1}{d_1} \sum_{i \in I}[1ii]\tfrac{x_1}{x_i^2} ) 
		\sum_{i \in I} d_i
		-2\sum_{i \in I} ( \tfrac{d_ib}{x_i} - [1ii]\tfrac{x_1}{x_i^2} )
							  + 2\sum_{i \neq j \in I }[ij k_{ij}] \tfrac{x_{k_{ij}} }{x_ix_j} \\
	& > &
	 \tfrac{1}{x_1}   \sum_{i \in I}
    \big(
		 4d_ic_1 -2d_ib \alpha_i 	+[1ii]  \cdot 
		 \big( \tfrac{d}{d_1} +1 \big)\cdot \alpha_i^2 \big)\,.						
\end{eqnarray*}
Now if the structure constants of $G/H$ have the property that for all $i \in I$ the quadratic polynomials
$$
  q_i(\alpha) :=4c_1 -2b \cdot \alpha 	+\tfrac{[1ii]}{d_i} 
		 \big( \tfrac{d}{d_1} +1 \big)\cdot \alpha^2
$$
are positive for all $\alpha>0$ then $G/H$ could not admit a $G$-invariant Einstein metric. Note however,
$$
q_i(\alpha) =q(\alpha)= 2\left(4(n_1-1) - 4(n-1)\alpha +(2n-n_1-1)\alpha^2 \right)
$$
has (two positive) roots
$$
\frac{2(n-1) \pm 2\sqrt{(n-n_1)^2+2n_1}}{2n-n_1-1}\,.
$$

\begin{remark}
The  quadratic polynomial  in the above example does not imply non-existence of
$G$-invariant Einstein metrics. But clearly one already obtains  upper estimates for 
$\alpha_1,...,\alpha_r$. Moreover, when $n_1$ is large
compared to $n$, then one obtains lower bounds as well. It is an interesting problem 
to refine the methods above to obtain finer invariants and improved estimates. 
We mention that one can obtain similar estimates for a larger class of homogeneous spaces,
 yielding a huge class of non-existence examples: see  \cite{Bo05}, in particular equation (4.6).
\end{remark}

%%%%%%%%%%%%%%%%%%%%%%%%%%%%%%%%%%%%%%%%%%%%%%%%%%%
%%%%%%%%%%%%%%%%%%%%%%%%%%%%%%%%%%%%%%%%%%%%%%%%%%%%%%

\subsection{{\it Disconnected groups}} \label{sec:disc}

In this section we briefly discuss compact homogeneous spaces $M^n=G/H$ such that
the compact Lie groups $G$ and $H$ are  not necessarily connected. If
$M^n$ is simply connected then the connected component $G_0$ of $G$ containing the identity 
acts transitively on $M^n$, with $M^n=G_0/H_0$.  
We note that these spaces are discussed in the case
that $G/H$ is isotropy irreducible, when $H$ acts irreducibly on $\g/\h$: see \cite{WZ4}.  The smaller class for which 
the action of $H_0$ is irreducible on $\g/\h$ is called strongly isotropy irreducible. 

\medskip

Essentially, all our results also hold for connected compact homogeneous spaces
$M^n=G/H$ in this more general case with $G,H$ possibly disconnected: 
see \cite{Gr} and \cite{Bo}. Without going into detail,
we note that certain adjustments are necessary.
For instance, we can no longer consider arbitrary intermediate subalgebras $\kf$.

\begin{definition}[$H$-subalgebra]
Let $G/H$ be a compact homogeneous space. We say a subalgebra $\kf$ of $\g$ is an \emph{$H$-subalgebra}
if $\h < \kf <\g$ and  $\kf$ is $\Ad(H)$-invariant.
\end{definition}

When $H$ is connected, every intermediate subalgebra $\kf$ is an $H$-subalgebra, 
but for $H$  disconnected, this is no longer true.

We review some well-known facts about Lie groups. Let $L$ be a Lie group, not necessarily connected.  Let $L_0$ denote 
the connected component of $L$ containing the identity. As is well-known, $L_0$ is not only
a Lie group itself, in fact, $L_0$ is a normal subgroup of $L$. 
Recall  that the connected Lie subgroups $P$ of $L$ are
in one-to-one correspondence with the Lie subalgebras $\p$ of $\lf=T_e L=T_e L_0$. 
Finally, recall that a subalgebra $\p$ is said to be compact if the corresponding Lie subgroup $P$ is compact.

\begin{lemma}
Let $G$ be a compact Lie group and let $H$ be a compact subgroup of $G$.
Furthermore, let $\kf$ be an $H$-subalgebra of $\g$ corresponding to $K_0$.
Then $K(\kf):=H \cdot K_0$ is an
immersed subgroup of $G$ with $H \subsetneq K(\kf) \subsetneq G$,
and $\kf$ is its Lie subalgebra.
\end{lemma}

\begin{proof}
We show that $K(\kf)$ is closed under multiplication and taking inverses. Let
$h_1,h_2 \in H$
and $k_1^0,k_2^0 \in K_0$. Then we have
$$
   (h_1 k_1^0)(h_2 k_2^0) =
     h_1 h_2 (h_2^{-1}k_1^0h_2)k_2^0 \in H\cdot K_0 = K(\kf)
$$
because $h_2^{-1}k_1^0h_2 \in K_0$ and $K_0$ is a subgroup.

Next, let $h \in H$ and $k^0 \in K_0$. Then
$$
    (hk^0)^{-1}=(k^0)^{-1}h^{-1}=h^{-1}(hk^0 h^{-1})\in H\cdot K_0 = K(\kf)\,.
$$
Thus $K(\kf)$ is indeed closed under multiplication and inverses.
\end{proof}

When $H$ is connected, then $K(\kf)=K_0$. But if $H$ is
disconnected, then $H$ need not be contained in $K_0$. In any
case,  $K(\kf)$ is a (possibly disconnected) Lie subgroup of $G$ with
$T_e K(\kf)=\kf$.

The proof of Theorem \ref{theoA} works the same way in this more general case. 
The nerve $\XGH$ of $G/H$ is defined using flags of only 
non-toral $H$-subalgebras $\kf$.

\medskip

We indicate here why there are instances when considering non-connected groups $G$ and $H$ can be extremely useful.

\begin{example}
Suppose that $G$ is connected but $H$ is not. 
Then $G/H_0$ is a finite 
covering space of $G/H$. In particular the set of  $G$-invariant metrics on $G/H$ may be smaller than 
the set of  $G$-invariant metrics on $G/H_0$.
As a consequence it may be easier to show the existence of a $G$-invariant Einstein metric on $G/H$.
Of course such an Einstein metric can be lifted to a $G$-invariant Einstein metric on $G/H_0$. 
\end{example}

\begin{example}
Let $G$ be a simple, compact Lie group, let $H$ be a compact Lie group 
and $\varphi:H \to  {\rm Aut}(G)$ be a homomorphism. 
The semi-direct product $\hat G:=H \ltimes_\varphi G$ is defined as follows:
$$
   (h_1,g_1) \times_\varphi (h_2,g_2):=(h_1\cdot h_2,g_1 \cdot \varphi(h_1)(g_2))\,.
$$
It is not hard to check that this is a group action. Let $\hat H:= H \times \{e\}\subset \hat G$ and consider 
$\hat G/\hat H$. Then we have
$$
(h,e)(e,g)(h^{-1},e)= (h,\varphi(h)(g))(h^{-1},e)=(e, \varphi(h)(g))\,.
$$
This is also true for a smooth curve $g(t)$ in $G\simeq \{(e,g)\mid g \in G\} \subset \hat G$
with $g(0)=e$. As a consequence, $\hat \m:=\{0 \} \times \g$ is $\Ad(\hat H)$-invariant, so that
$\hat \g=\hat \h \oplus \hat \m$ is a reductive decomposition of $\hat \g$. Note also
that $\hat G/\hat H=\{\hat g_0 \hat H \mid \hat g_0=(h_0,g_0) \in \hat G\}$ and
$$
  \hat g_0\hat H=\{(h_0,g_0) \cdot (h,e) \mid h \in H\} = \{(h_0h,g_0)\mid h \in H\}= H \times \{g_0\}\,.
$$
Thus as smooth manifold, $\hat G/\hat H =G$.

Suppose now $H \subset G$ and $\varphi(h)(g)=h g h^{-1}$. Then $\Ad(\hat H)$
is given by the adjoint action of $H$ on $\g$.
For a first concrete example let us consider $G=S^3$ and $H=\{\pm 1,\pm i,\pm j,\pm k\}$. 
The Lie algebra of $S^3$ is spanned by $i,j,k$. Moreover, conjugation by $i\in H$ 
acts on $\langle i\rangle_{\bldss{\RR}}$ trivially, but we have
$$
  i j i^{-1}=- i j i = - k i = - j \quad \textrm{ and }\quad
  -i k i  = - i j = -k.
$$
(Conjugation by $j$ and $k$ is analogous.)
Thus $\langle i\rangle_{\bldss{\RR}}$, $\langle j\rangle_{\bldss{\RR}}$ and $\langle k\rangle_{\bldss{\RR}}$
are pairwise inequivalent $\Ad(\hat H)$-invariant subspaces.
 This shows that for a left-invariant metric $g$ on $S^3$ which is diagonal with respect to 
$\langle i\rangle_{\bldss{\RR}} \oplus \langle j\rangle_{\bldss{\RR}} \oplus \langle k\rangle_{\bldss{\RR}}$, 
 the Ricci tensor of $g$ must again be diagonal.

In our second example let $G$ be a compact, simple Lie group, connected and simply connected, such that the roots of
$\g$ are all the same length,  and let $H:=N(T)$ be the
normalizer of a maximal torus $T$ of $G$. We  compute the adjoint action of $H$ on $\g$.
 The adjoint action of $T \subset H$ is  well-known: $\tf$ is the trivial summand
and the other isotypical summands of $\Ad(T)$ are the (real) root spaces. Note that $\Ad(N(T))$ preserves
$\tf$ and consequently also preserves $\p:=\tf^\perp$, the sum of the root spaces. Since $W:=N(T)/T$ acts
simply-transitively on the roots, each of $\tf$ and $\p$ is $\Ad(H)$-irreducible.
Thus, there exist no non-toral $H$-subalgebras $\kf$, and we obtain the existence
of a $\hat G$-invariant Einstein metric on $G$.  
This family of examples builds on the family $G/T$ of isotropy irreducible spaces, in Table I, line 1 of  \cite{WZ4}.
\end{example}

%%%%%%%%%%%%%%%%%%%%%%%%%%%%%%%%%%%%%%%%%%%%%%%%%%%%%%%%%%
\section{Appendix}\label{sec:appendices}

In this section we will
make a brief digression, first into semi-algebraic geometry. We also provide helpful background on the isotropy representation of a compact homogeneous space $G/H$. 

\subsection{{\it Semi-algebraic sets}}\label{sec:semialgebraic}

We refer the reader to  Benedetti and Risler for more details \cite{B-R}. A set $X\subset
\RR^m$ is called {\em semi-algebraic} if $X$ is defined by finitely
many polynomial equations and inequalities.  
The inequalities are allowed to be both strict and  non-strict.

Given two semi-algebraic sets  $X\subset \RR^m$ and $Y\subset \RR^n$, we say a map $f:X\to Y$ is {\em semi-algebraic}
if the graph of $f$ is a semi-algebraic subset of $\RR^{m+n}$: see \cite[Definition 2.3.2]{B-R}. 
A semi-algebraic map is not necessarily continuous, see \cite[Example 2.7.3]{B-R},  though some
authors include a requirement of continuity in their definition of a semi-algebraic map: see \cite{H.H}, \cite[Theorem 7.6]{DK0}, 
\cite[I, Proposition 3.13]{DK2}. 
The paradigm for semi-algebraic maps are polynomials.

A fundamental result in semi-algebraic geometry is the following:

\begin{theorem}[Tarski-Seidenberg]
The image of a semi-algebraic subset of $\RR^n\times \RR^m$ under the projection
$\pi:\RR^n \times \RR^m \to \RR^n$ such that $(x,y)\mapsto x$ is a semi-algebraic set.
\end{theorem}

The theorem implies that intersections, unions, and complements of semi-algebraic sets are  semi-algebraic. 
Moreover, it is possible to use
first-order formulae (in the language of ordered fields with parameters in $\RR$) to obtain new semi-algebraic sets from known ones.
For instance,  given $X_1\subset \RR^n,X_2 \subset \RR^m$ semi-algebraic sets and a semi-algebraic map $F:\RR^n\times \RR^m \to \RR^k$, the set
$$
  \{ x_1 \in X_1\mid \exists x_2 \in X_2 \textrm{ s.t. } F(x_1,x_2) \leq 0\}
$$
is semi-algebraic. As another example, given endomorphisms $A,B$ of $\RR^n$,  the condition $\ker(A) \subsetneq \ker(B)$ is a semi-algebraic
condition because it is equivalent to
$$
  \{ \forall x \in \RR^n \mid A(x)=0 \Rightarrow B(x)=0\} \wedge \{ \exists y \in \RR^n \mid B(y)=0 \wedge A(y) \neq 0\}\,.
$$
Connected components of semi-algebraic sets are semi-algebraic, and
the closure, interior, or boundary of a semi-algebraic set is semi-algebraic.

The key property of compact semi-algebraic sets used in Sections \ref{sec:homotopy1} and \ref{sec:homotopy2}
is that such sets are absolute neighborhood retracts. The following result states a useful generalization of
that.

\begin{lemma}[\cite{Gr}]\label{lem:XY}
Let $X,Y \subset \RR^N$ be compact, semi-algebraic sets with $X \subset Y$ and $X \neq \emptyset$.
Suppose that for some $\delta_0>0$,  for all $\delta \in (0,\delta_0)$
there exists a continuous map
$$
  H_\delta:[0,1]\times Y \to Y
$$
with the following properties: $H_\delta(0,y)=y$ for all $y \in Y$,
$H_\delta(t,x)=x$ for all $x \in X$ and all $t \in[0,1]$, and
$$
  d(H_\delta(1,y),X)< \delta  \textrm{ for all } y \in Y\,.
$$
Then $X$ is a strong deformation retract of $Y$.
\end{lemma}

\begin{proof}
It follows from Theorem 1 of \cite{DK1}, see also \cite[III,
Theorem 1.1]{DK2}, that there exists an open semi-algebraic
neighborhood $U$ of $X$ in $Y$ and a semi-algebraic,
continuous map $G:[0,1]\times \overline{U}\to \overline{U}$, such
that the restriction $G\vert_{[0,1]\times U}$ yields a strong
deformation retraction from $U$ to $X$.

By compactness of $X$,  there exists $\delta_0>0$ such that for all $0<\delta\leq \delta_0$,
we have 
$$
  T_\delta(X):=\{y \in Y \mid d(y,X)<\delta\} \subset U\,.
	$$
	Hence we can define
$$
  H:[0,1]\times Y \to Y \textrm{ s.t. } (t,y) \mapsto \left\{
	 \begin{array}{ll}
	     H_{\delta}(2t,y)   & \textrm{ for } 0\leq t \leq \tfrac{1}{2} \\
			 G(2t-1,H_\delta(1,y))   &\textrm{ for } \tfrac{1}{2}\leq t \leq 1\,.
			\end{array} \right.
$$
The function $H$ is continuous since for $t=\tfrac{1}{2}$ we have $G(2t-1,H_\delta(1,y))=H_{\delta}(2t,y)$.
Moreover $H(t,x)=x$ for all $t \in [0,1]$ and all $x \in X$ and $H(1,Y)=X$.
This shows the claim.
\end{proof}

%%%%%%%%%%%%%%%%%%%%%%%%%%%%%%%%%%%%%%%%%%%%%%%%%%%

\subsection{{\it Lie-theory I}}\label{sec:Ltbasics}

In this section we provide details concerning the isotropy representation of a compact homogeneous
space $G/H$ and the normalizer $N_G(H)$ of $H$ in $G$. We assume, as we have  throughout the paper, that $G$ and $H$ are connected.

\medskip

\begin{definition}\label{def:m0}
For any compact homogeneous space $G/H$,
we denote by $\m_0$ the subspace of $\m$ on which $\Ad(H)$ acts trivially.
\end{definition}

Notice that if $\rk G=\rk H$, then $\m_0 = \{0\}$,  while if $H=\{e\}$, $\m_0=\g$. 
Recall also that 
\begin{eqnarray}
 \exp(\ad(X))=\Ad(\exp_G(X))\label{eqn:exp}
\end{eqnarray}
for all $X \in \g$, 
where $\exp_G:\g \to G$ denotes the exponential of $G$ and $\exp(D)=\sum_{k=0}^\infty \tfrac{D^k}{k!}$
for any endomorphism $D:\g \to \g$ \cite[II]{H.S}.

\begin{lemma}\label{lemkmm}
A subspace $\ti \m$ of $\m$ is $\Ad(H)$-invariant if and only if $[\h,\ti {\m}]\subset \ti {\m}$.
Moreover, $[\h,\m_0]=0$ and for 
any $\Ad(H)$-invariant subspace $\tilde \m$ of $\m \ominus \m_0$, 
we have $\{0\} \neq [\h,\tilde \m]$.
\end{lemma}

\begin{proof} 
If $\ti {\m}$ is  ${\Ad}(H)$-invariant,
then differentiation yields $[\h,\ti {\m}]\subset \ti {\m}$. Conversely, suppose we know 
$[\h,\ti {\m}]\subset \ti {\m}$. Since $H$ is connected, for any $Z \in \h$, 
${\Ad}(\exp (Z))=\exp ({\ad}(Z))$. Hence, $\ti {\m}$ is
${\Ad}(H)$-invariant. This shows the first claim. The second and the third claim follow immediately.
\end{proof}

Let $\n(\h)$ denote the  Lie algebra of the normalizer $N_G(H)$  of $H$ in $G$.

\begin{lemma}\label{lemnor}
Let $G/H$ be a compact homogeneous space. Then
$\n(\h)= \h \oplus \m_0$. Moreover, $\m_0$ is a compact subalgebra.
\end{lemma}

\begin{proof}
Since $[\h,\m_0]=0 \in \h$, see above, we conclude
$\m_0 \subset \n(\h)$. Conversely, suppose that $\ti \m \subset \n(\h) \cap \m$.
Then on the one hand, $[\h, \ti \m]\subset \h$ by definition of $\n(\h)$,
while on the other hand, $[\h, \ti \m] \subset \m$ because $[\h,\m]\subset \m$.
This shows $[\h, \ti \m]=0$, and thus $\m \cap \n(\h)\subset \m_0$, proving our equality.

Since $[\m_0,\m_0] \subset \n(\h)=\h \oplus \m_0$ and $[\h,\m_0]=0$, then  because $Q$ is self-adjoint 
 \eqref{eqn:adXscew}, we know that
$\m_0$ is a subalgebra of $\n(\h)$ and  thus is a subalgebra of $\g$. 
In order to show that the corresponding connected subgroup of $G$ is compact we now use the
special property of the biinvariant metric $Q$ chosen in Section \ref{sec:ginvm}.
We decompose the compact Lie algebra $\n(\h)=\h \oplus \m_0$ into the
$Q$-orthogonal sum of its semisimple part $(\h \oplus \m_0)_s=
\h_s \oplus (\m_0)_s$ and its center $\z(\h \oplus \m_0)$. Each of the Lie
algebras $\h_s$, $(\m_0)_s$ and $\z(\h \oplus \m_0)$ is compact, and so 
 $\h \cap \z(\h \oplus \m_0)$ is also compact. It remains to show that
the $Q$-orthogonal complement of $\h \cap \z(\h \oplus \m_0)$ in
$\z(\h \oplus \m_0)$ is compact. 

By the definition of $Q$,
the compact, abelian subalgebras $\h \cap \z(\h \oplus \m_0)$ and
$\z(\h \oplus \m_0)$ can be viewed as subalgebras of $\so (6N)$. 
Notice we can choose a $Q$-orthonormal (standard) basis
$(\hat e_1,...,\hat e_{3N})$ for the standard maximal (diagonal) torus
$\tf_{6N}$ of $\so(6N)$  so that every compact subtorus $\tf$ of $\tf_{6N}$
has a basis $(e_1,...,e_{\dim \bldss\tf})$, for which each $e_i$ is a rational linear combination
of the above basis elements $(\hat e_i)$ of  $\so(6N)$. Then, the $Q$-orthogonal complement in $\tf_{6N}$ of the
a subtorus $\tf$ must also have such a basis, and hence it corresponds to a compact subtorus in $\so(6N)$.  
As a consequence, the $Q$-orthogonal complement of $\h \cap \z(\h \oplus \m_0)$ in
$\z(\h \oplus \m_0)$ is the intersection of two compact subalgebras, thus a compact subalgebra.
\end{proof}

\begin{remark}\label{rem:iso}
We decompose $\m$ into its $\Ad(H)$-invariant isotypical summands (see \cite[II, Proposition 6.9]{B-D}):
$$
 \m=\bigoplus_{i=1}^{\ell_{\rm iso}}\,\p_i \,.
$$
Each $\p_i$ is a direct sum of $\Ad(H)$-irreducible summands which
are equivalent (as $\Ad(H)$-representations), while, for $i \neq j$, irreducible summands in $\p_i$ and $\p_j$
are inequivalent. Note by Lemma \ref{lemkmm},  if $\m_0$ is non-trivial, then $\m_0$ is one of the 
isotypical summands of $\m$.

Thus by Schur's Lemma,  for every $P_g \in \MG$ we have
\begn
 Q(P_g\,\cdot \,,\,\cdot \,)=
 Q((P_g)_1\,\cdot \,,\,\cdot \,)\vert_{\p_1} \perp \cdots \perp
 Q((P_g)_{\ell_{\rm iso}}\,\cdot \,,\,\cdot \,)
   \vert_{\p_{\ell_{\rm iso}}}	
\enn
where, for each $1 \leq i \leq \ell_{\rm iso}$, $(P_g)_i:\p_i \to \p_i$ is an $\Ad(H)$-equivariant, self-adjoint,
positive definite endomorphism.
Consequently, every $P_g \in \MG$ respects the decomposition $\m =\oplus _{i=1}^{\ell_{\rm iso}}\,\p_i$.
Notice this is also true for every $v \in \SymmH$. 
\end{remark}

For $\kf$ a subalgebra in $\Sub$, let $\m_{\blds\kf}:=\m \cap \kf$. For convenience, $\m_{\bldss{\h}} := \m$.

\begin{lemma}\label{lem:ad-invar}
Let $\kf$ be a subalgebra, $\h \leq \kf < \g$, let $K$ be the connected subgroup of $G$  
corresponding to $\kf$, and let $A \in \SymgH$ with $\kf\subset \ker (A)$. Then, the condition  
$[A,\ad(\kf)]=0$ is equivalent to the $\Ad(K)$-equivariance of $A$.
Furthermore, $[A,\ad(\kf)]=0$ 
if and only if $[A\vert_{\blds\kf^\perp},\ad(\m_{\blds\kf})\vert_{\blds\kf^\perp}]=0$
if and only if $[A\vert_{\bldss\m},\ad(\m_{\blds\kf})\vert_{\bldss\m}]=0$.
\end{lemma}

\begin{proof}
Let  $K$ be the connected subgroup of $G$ with Lie algebra $\kf$. We know by differentiation that 
$\Ad(K)$-equivariance of $A$ implies $[A,\ad(\kf)]=0$. To see the converse, let $A \in \SymgH$
and suppose that $[A,\ad(\kf)]=0$. Using \eqref{eqn:exp} shows that $A$ is $\Ad(K)$-equivariant.

Next, let $A \in \SymgH$ with $\kf\subset\ker(A)$. Since $\Ad(K)$, $\ad(\kf)$ respect
the decomposition $\g = \kf \oplus \kf^\perp$, it follows that $A\vert_{\blds\kf^\perp}$
is $\Ad(K)\vert_{\blds\kf^\perp}$-equivariant if and only if 
$[A\vert_{\blds\kf^\perp},\ad(\kf)\vert_{\blds\kf^\perp}]=0$.

Since $A \in \SymgH$ with $\kf\subset\ker(A)$ is $\Ad(K)$-equivariant if and only if 
$A\vert_{\blds\kf^\perp}$ is  $\Ad(K)\vert_{\blds\kf^\perp}$-equivariant
we deduce that $[A,\ad(\kf)]=0$ is equivalent to
$[A\vert_{\blds\kf^\perp},\ad(\kf)\vert_{\blds\kf^\perp}]=0$.

Since for $A \in \SymgH$ with $\kf\subset\ker(A)$ $A$ preserves $\m$ we deduce from $[A,\ad(\kf)]=0$ that
$[A\vert_{\bldss\m},\ad(\kf)\vert_{\bldss\m}]=0$, which implies $[A\vert_{\bldss\m},\ad(\m_{\blds\kf})\vert_{\bldss\m}]=0$ since
$\m_{\blds\kf}\subset \kf$.

Suppose now $[A\vert_{\bldss\m},\ad(\m_{\blds\kf})\vert_{\bldss\m}]=0$.
All elements of $\SymgH$ are $\Ad(H)$-equivariant, by definition, thus
$[A\vert_{\bldss\m},\ad(\h)\vert_{\bldss\m}]=0$. Using $\kf=\h \oplus \m_{\blds\kf}$ we conclude
$[A\vert_{\bldss\m},\ad(\kf)\vert_{\bldss\m}]=0$.
Since  $\ad(\kf)$ preserves the decomposition $\g=\kf \oplus \kf^\perp$ it follows that
$\ad(\kf)\vert_{\bldss\m}$ preserves the decomposition $\m=\m_{\blds\kf} \oplus \kf^\perp$.
We deduce $[A\vert_{\blds\kf^\perp},\ad(\kf)\vert_{\blds\kf^\perp}]=0$.
Using that for $A \in \SymgH$ with $\kf \subset \ker (A)$, this is equivalent to $[A,\ad(\kf)]=0$, 
proving the claim.
\end{proof}

Any such $A$ induces a submersion metric on $G/H$ with respect to $K/H \to G/H \to G/K$.
Notice however, that $K$ need not be a compact subgroup of $G$.

\begin{corollary}\label{cor:submkf}
Let $A \in \SymgH$ ($v \in \Si$) and let $\m_{I_1},\dots, \m_{I_\ell} \subset \m$ denote the eigenspaces of $A\vert_{\m}$
corresponding to distinct eigenvalues in increasing order.  
\begin{itemize}
\item[$(i)$] Then $[\m_{I_1},\m_{I_1}]\perp \m_{I_p}$ (i.e. $[I_1 I_1 I_p]=0$) for all $1<p \leq \ell$ 
if and only if $\kf = \h\oplus \m_{I_1}$ is a subalgebra. 
\item[$(ii)$] Furthermore,  when $\kf = \h\oplus \m_{I_1}$ is a subalgebra, 
$A \in \D(\kf)$ if and only if $\kf  \leq \ker(A)$ and for all 
$1 \leq p \neq q \leq \ell$, 
we have $[\m_{I_1},\m_{I_p}]\perp \m_{I_q}$ (i.e. $[I_1 I_p I_q]=0$).
\end{itemize}
\end{corollary}

\begin{proof}
$(i)$ If for all $1 < p $ we have $[\m_{I_1},\m_{I_1}]\perp \m_{I_p}$, then since $[\h,\m_{I_1}] \subset \m_{I_1}$ we have that $\kf = \h \oplus \m_{I_1}$ is a subalgebra of $\g$. Conversely, if $\kf = \h \oplus \m_{I_1}\in \Sub$ is a subalgebra, then   
we know $[\m_{I_1},\m_{I_1}] \subset \kf \perp \m_{I_p}$ for all $1<p$.

\noindent $(ii)$ Suppose $\kf = \h \oplus \m_{I_1}\in \Sub$ is a subalgebra, and $A \in \D(\kf)$ or equivalently, $v \in \D(\kf)^\Si$. Since $A \in \D(\kf)$ we know $[A,\ad(\kf)]=0$. Thus, by Lemma \ref{lem:ad-invar}
we deduce $[A\vert_{\blds\kf^\perp},\ad(\m_{\blds\kf})\vert_{\blds\kf^\perp}]=0$. 
This shows that $\ad(\m_{\blds\kf})$ respects
the eigenspaces of $A\vert_{\bldss\m}$. That is, $[\m_{I_1}, \m_{I_p}] \subset \m_{I_p} \perp \m_{I_q}$ for $p \neq q$. 
Conversely, suppose $\kf = \h \oplus \m_{I_1} \leq\ker(A)$ and for all $p \neq q$, 
we have $[\m_{I_1},\m_{I_p}]\perp \m_{I_q}$. 
Thus $\ad(\kf)\vert_{\blds\kf^\perp}$ preserves the eigenspaces of $A\vert_{\blds\kf^\perp}$.
By Lemma \ref{lem:ad-invar}, to see that $A \in \D(\kf)$ we need that 
$[A\vert_{\blds\kf^\perp},\ad(\kf)\vert_{\blds\kf^\perp}]=0$ which follows from
 $[A|_{\m_{I_p}},\ad(\m_{I_1})|_{\m_{I_p}}] =0$ for every $p>1$.
Let $X_1 \in \m_{I_i}$ and $X_p \in \m_{I_p}$, $p>1$. Then
$$
[A|_{\m_{I_p}},\ad(X_1)|_{\m_{I_p}}](X_p) = A[X_1,X_p] - [X_1,A X_p]=0\,,
$$ 
since by hypothesis,  $[\m_{I_1},\m_{I_p}] \subset \m_{I_p}$. Thus our equality holds. 
\end{proof}

\medskip

Here we provide a brief overview of representation theory of compact Lie groups $K$.
Let $V$ be a vector space over $\CC$. We call a representation $\rho_{\bldss{\CC}}:K \to \Un(V) \subset {\rm Gl}(V,\CC)$
an \emph{irreducible complex representation} if there exists no proper, non-trivial 
$\CC$-vector subspace $\tilde V$ of $V$ which
is $\rho_{\bldss{\CC}}(K)$-invariant. The \emph{realization} of an irreducible complex representation
$\rho_{\bldss{\CC}}:K \to  \Un(V,\CC)$ is given by $\rho_{\bldss{\RR}}:K\to \SO(W) \subset {\rm Gl}(W,\RR)\,;\,\,
k \mapsto \rho_{\bldss{\CC}}(k)$,
where $W$ is the $\RR$-vector space induced by $V$; that is, $W=V$ as sets, but $\dim_{\bldss{\RR}} W
=2 \dim_{\bldss{\CC}} V$.
Such representations are called \emph{real representations}. Notice that the realization of an irreducible
complex representation may not be irreducible.

An irreducible real representation $\rho_{\bldss{\RR}}:K \to \SO(W)$ is called of \emph{real type}, \emph{complex type}
or \emph{quaternionic type}, respectively, if the group of intertwining operators 
$$
 G(W,K):=\{ A \in {\rm End}(W,\RR) \mid A \cdot \rho_{\bldss{\RR}}(k) =  \rho_{\bldss{\RR}}(k) \cdot A \textrm{ for all } k\in K\} 
$$
is isomorphic to $\RR$, $\CC$ or $\HH$, respectively: see \cite[II, 6.2, II, Theorem 6.7]{B-D}.

\begin{lemma}\label{lem:repcomplex}
Let $\rho_{\bldss{\CC}}:\SU(n)\to \Un(V,\CC)$ be an irreducible complex representation, $n \geq 2$.
Suppose that $\rho_{\bldss{\CC}}$ can be extended to a representation of $\Un(n)$ such that
the center of $\Un(n)$ acts non-trivially on $V$ by multiples of the identity.
Then the realization $\rho_{\bldss{\RR}}:\Un(n)\to \SO(W)$, $W=V$, is an irreducible real
representation.
\end{lemma}                                           

\begin{proof}
For a contradiction, suppose that $\tilde W$ is a non-trivial proper real subspace of $W$ which is 
$\rho_{\bldss{\RR}}(\Un(n))$-invariant. By assumption there exists some $\alpha \in \RR$ such that
for every $e^{i\varphi} \cdot I_n$ in the center of $\Un(n)$, we have
$\rho_{\bldss{\CC}}(e^{i\varphi} \cdot I_n)=e^{i\varphi \alpha}\cdot I_V$. Thus whenever $\tilde w \in \tilde W$,
then $i\cdot \tilde w \in \tilde W$ as well. Consequently $\tilde W$ is an $\SU(n)$-invariant complex subspace of $W=V$,
contradicting our hypothesis. Consequently $W$ is an irreducible real representation.
\end{proof}

%%%%%%%%%%%%%%%%%%%%%%%%%%%%%%%%%%%%%%%%%%%%%%%%%%%%%%%%%%%%%%%%%%%%%%
\subsection{{\it Lie-theory II}}\label{sec:LieII}

In this section we describe the Casimir operator of the isotropy representation of a compact homogeneous
space $G/H$ and its relation to the Killing form and the structure constants.

\medskip

Let $(Z_m)_{1\leq m \leq \dim H}$ denote a $Q$-orthonormal basis of $\h$. 
Let $\m_i$ be an $\Ad(H)$-irreducible summand of $\m$. 
The {\em Casimir operator} on $\m_i$ is given by
\beg
  C_{\bldss\m_i,Q\vert_{\bldss\h}}:=-\sum_{m}{\ad}(Z_m)\circ {\ad}(Z_m):\m_i \to \m_i
\en
and it satisfies
\beg
   C_{\bldss\m_i,Q\vert_{\bldss\h}}=c_i\cdot {\Id}_{\bldss\m_i}\,.
\en
A short computation using \eqref{eqn:adXscew} shows $c_i\geq 0$. 
Clearly, $c_i=0$ if and only if $\m_i \subset \m_0$: see Lemma \ref{lemnor}.

Recall, the definitions of $d_i$, $b_i$ and $[ijk]$ can be found in Section \ref{sec:scal}
and that these numbers depend on the decomposition $f=\oplus_{i=1}^\ell \m_i$ of $\m$ choosen.

\begin{lemma}[\cite{WZ2}] \label{lem:1p5}
 Let $G/H$ be a compact homogeneous space and
let $f$ be a decomposition of $\m$.
Then for all $1 \leq i \leq \ell$,
\beg
  d_ib_i-\tfrac{1}{2}\sum_{j,k=1}^\ell [ijk]
   =2d_ic_i+\tfrac{1}{2}\sum_{j,k=1}^\ell [ijk] \geq 0\,.
\en
Moreover, the expression 
is zero if and only if $\m_i \subset \z(\g)$.
\end{lemma}

\begin{proof}
The equality is proved in \cite[La.1.5]{WZ2}. 

Suppose that $i=1$ and the left-hand side is zero.
Then $c_1=0$, thus $[\h,\m_1]=0$ and $\m_1 \subset \m_0$ by Lemma \ref{lemkmm}. Since $[11j]=0$ for
all $1\leq j\leq \ell$, the module $\m_1$ is an abelian subalgebra of $\m_0$. Moreover,
since for all $2 \leq j,k\leq \ell$ we have $[1jk]=0$, the subspace
$\g':=\h \oplus (\bigoplus_{j=2}^\ell \m_j)$
is a subalgebra of $\g$ commuting with $\m_1$: see Lemma \ref{lem:splitt}.
 It follows that $\m_1 \subset \z(\g)$.

The converse direction is clear.
\end{proof}

\begin{lemma} \label{lem:toral}
Let $\kf \in \Sub$ and let $f$ be a
decomposition of $\m$ such that $\kf$ is $f$-adapted,
and let $I_1=I^{\kf}_1$. Then
$$
 a_{\blds\kf}:=\sum_{j\in I_1}d_jc_j+\tfrac{1}{4}[I_1I_1I_1] \geq 0\,.
$$
Furthermore, $\kf$ is toral if and only if $a_{\blds\kf} =0$.
\end{lemma}
 
\begin{proof}
Suppose that $\kf=\h \oplus \af$  is a toral subalgebra of $\g$, that is, $\af \subset \m_0$. Then clearly $a_{\blds\kf}=0$.
Conversely, when $a_{\blds\kf}=0$, then $\sum_{j\in I_1}d_jc_j=0$, so we know $\m_{I_1}:=\m_{\kf}\subset \m_0$.
Because $\kf $ is a subalgebra we conclude $[\m_{I_1},\m_{I_1}]\subset \m_{I_1}$. Since   $[I_1I_1I_1]=0$, the claim follows.
\end{proof}

\begin{corollary}\label{cor:uniformntest}
For every a compact homogeneous space $G/H$, there 
 exists a constant $n_{G/H}>0$ such that the following holds:
given any non-toral subalgebra $\kf$,
with $f$ a decomposition of $\m$ such that $\kf$ is $f$-adapted,
and letting $I_1=I^{\blds\kf}_1$, then
$$
  a_{\blds\kf}=\sum_{j\in I_1}d_j c_j+\tfrac{1}{4}[I_1I_1I_1]  \geq n_{G/H}\,.
$$
\end{corollary}

\begin{proof}
For a contradiction, suppose that there exists a sequence of non-toral subalgebras $(\kf_\alpha)_{\alpha \in \NN}$
and a sequence $(f_\alpha)_{\alpha \in \NN}$ of decompositions of $\m$, such that for each $\alpha \in \NN$, 
$\kf_\alpha$ is $f_\alpha$-adapted, with 
$a_{\blds\kf_\alpha} \to 0$ as $\alpha \to \infty$. Because the Casimir operator $C\vert_{\m}:\m \to \m$
is non-negative with kernel $\m_0$, for $\alpha$ large we must have $\m_{\bldss\kf_\alpha}\subset \m_0$.
But since $\m_0$
and $\kf_\alpha$ are subalgebras of $\g$, so is 
$\kf_\alpha^0:=\m_{\bldss\kf_\alpha}=\m \cap \kf_{\alpha}$. Since each $\kf_\alpha$ is non-toral by
assumption, $\kf_\alpha^0$ is a non-abelian subalgebra of $\m_0$ and  consequently, $\kf_\alpha$ has
a nonvanishing semisimple part $(\kf_\alpha^0)^{ss}$.  Passing to a subsequence, we may assume that
 these semisimple subalgebras  converge to a limit subalgebra
$\kf_\infty^0$ of $\m_0 \subset \g$, and also that $f_\alpha \to f_\infty$ as $\alpha \to \infty$. 
Up to conjugation,  $\g$ has only finitely many semisimple subalgebras, see \cite[Corollary 4.5]{BWZ}; thus 
 $\kf_\infty^0$ must be semisimple. Hence $[I_1I_1I_1]_{f_\infty}>0$, a contradiction. 
\end{proof}

\begin{corollary} \label{corbipos} Let $G/H$ be a compact homogeneous space
with finite fundamental group and let $f$ be a
decomposition of $\m$ into $\Ad(H)$-irreducible summands. 
For each $1\leq i \leq \ell$,
$$
  d_ib_i-\tfrac{1}{2}\sum_{j,k=1}^\ell [ijk]>0\,,
$$
and consequently
\begn
  b_{G/H}:=\sum_{i=1}^{\ell}d_ib_i> 0\,.\label{corbGH}
\enn
\end{corollary}

\begin{proof}
By Lemma \ref{lem:1p5} we know if $2d_ic_i+\frac{1}{2}\sum_{j,k=1}^\ell [ijk]=0$, then 
$\m_i$ is an abelian subalgebra in $\g$ and $\m_i$ commutes with the subalgebra
$\h \oplus (\bigoplus_{j=1,j\neq i}^\ell \m_j)$. This would imply that
$G/H$ has infinite fundamental group, see Lemma \ref{lem:splitt},
contradicting our hypothesis. This proves the claim.
\end{proof}

%%%%%%%%%%%%%%%%%%%%%%%%%%%%%%%%%%%%%%%%%%%%%%%%%%%

\subsection{{\it Lie theory III: A \L ojasiewicz inequality for structure constants}}\label{sec:Loj}

We recall the \L ojasie\-wicz inequality in semi-algebraic geometry (cf.~\cite[Proposition 2.3.11]{B-R}): 
Let $K \subset \RR^n$ be a compact semi-algebraic set 
and let $f,g:K \to \RR$ be continuous and semi-algebraic. Suppose that for all $x \in K$ we have: 
$f(x)=0 \Rightarrow g(x)=0$. Thus $f^{-1}(0) \subset g^{-1}(0)$. Then there exists $N\in\NN$ and $C=C(f,g,K)\geq 0$ 
such that for all $x \in K$ we have
$$
 \vert g(x)\vert^N \leq C\cdot \vert f(x) \vert\,.
$$ 
The \L ojasiewicz inequality stated in Proposition \ref{propinequ} provides the key estimate for the proof
of Theorem \ref{theo:scalestu}.  Note that standard estimation methods cannot be
applied to the above situation.

\medskip

Let $I^{\co}=\{1,...,\dim \g\}\bs I$.

\begin{proposition} \label{propinequ}
Let $G$ be a compact, connected Lie group endowed with a  biinvariant
metric $Q$. Let  $I\subset \{1,...,\dim \,\g\}$ with $1\in I$. Let 
$Z_{\vert I\vert}$ denote the set of all $Q$-orthonormal bases  
$b:=(e_1,...,e_{\dim \bldss\g})$ of $(\g,Q)$
such that
$$
  \tf:=\oplus_{i\in I} \langle e_i \rangle
$$ 
is an abelian subalgebra. Then there exists a constant $C=C(G,Q)>0$
and an open neighborhood $U_{\vert I\vert}$ of $Z_{\vert I\vert}$ in the space of
all $Q$-orthonormal bases, such that for all  
$\tilde b:=(\tilde e_1,...,\tilde e_{\dim \bldss\g}) \in U_{\vert I\vert}$ we have
\begn \label{eqn:Loin}
  \sum_{j,k \in I}Q([\tilde e_1,\tilde e_j],\tilde e_k)^2
  \leq
  C \cdot \!\!\!\!\sum_{j\in I,k \in I^{\co}}
  Q([\tilde e_1,\tilde e_j],\tilde e_k)^2\,.
\enn
\end{proposition}

\begin{proof}  First note that if $I^{\co}=\emp$, then $\g$ is abelian and
the above claim is trivially true.  
Hence we may assume that
$I^{\co}\neq \emp$ and that $\g$ is not abelian. 
We prove the result by induction on $\vert I\vert$.
It is clear that 
for $\vert I\vert =1,2$ the above claim is true, since the left hand side of \eqref{eqn:Loin}
is zero by $\Ad$-invariance of $Q$.  
Thus we may assume that $\vert I\vert \geq 3$.

We first show that for an arbitrary choice of 
$Q$-orthonormal basis $b \in Z_{\vert I\vert}$
and an arbitrary choice of  sequence 
$(b_\alpha)_{\alpha\in \NN}:=(e_1(\alpha),...,e_{\dim \bldss\g}(\alpha))_{\alpha\in \NN}$ of
$Q$-orthonormal bases for $\g$ with $\lim_{\alpha \to \infty}b_\alpha=b$, 
 such a constant $C>0$ exists (a priori, $C$ may depend on $b$ and
the sequence $(b_\alpha)$).

For a contradiction,  suppose that no such constant $C>0$ exists. In what follows we will pass 
to a subsequence whenever convenient, without explicitly mentioning it.
For $1 \leq j,k \leq \dim \g$, we set
$$
  [1jk]_\alpha :=Q([e_1(\alpha),e_j(\alpha)], e_k(\alpha))^2 \,.
$$
We set $I':= I\bs \{1\}$ and get
\begn \label{div1}
 \sum_{j,k \in I'} [1jk]_\alpha  
  &>&
	 g(\alpha) \cdot  \sum_{j\in I'\,,k \in I^{\co}} [1jk]_\alpha
\enn
where $g:\RR \to \RR$ with $\lim_{\alpha \to \infty}g(\alpha)=+\infty$.

Suppose that for each $\alpha \in \NN$, $e_1(\alpha)\in \tf(\alpha)$, where $\tf(\alpha)$
denotes a maximal abelian subalgebra of $\g$.
 Since maximal abelian subalgebras of $\g$ are conjugate, there exists 
a maximal torus $\tf(\infty)$ and
a sequence $(g_\alpha)_{\alpha \in
\NN}$ of group elements in $G$ with 
${\rm Ad}(g_\alpha)(\tf(\alpha))=\tf(\infty)$. We obtain a sequence $(b_\alpha)$ of
$Q$-orthonormal bases of $\g $ converging to $b$, with $e_1(\alpha)\in
\tf(\infty)$ for all $\alpha \in \NN$, which we again denote by  
$b_\alpha=(e_1(\alpha),..., e_{\dim \bldss\g}(\alpha))$. 
Since all terms in (\ref{div1}) are ${\rm Ad}(G)$-invariant, 
Equation (\ref{div1})  holds for each $b_\alpha$ in this sequence as well.

\medskip

{\it Step 1:} 
Let $V_1(\alpha):=\langle e_2(\alpha),...,
e_{\vert I\vert}(\alpha)\rangle_{\bldss\RR}$,
and let $V_1(\alpha)^\perp$ denote the $Q$-orthogonal complement of $V_1(\alpha)$ in $\g$.
Let $j\in I'$. Since $e_1(\alpha)\in \tf (\infty)$ we have 
\beg
  [e_1(\alpha),e_j(\alpha)]
	&=&
	 {\rm pr}_{V_1(\alpha)} ([e_1(\alpha),e_j(\alpha)]) +
	 {\rm pr}_{V_1(\alpha)^\perp}([e_1(\alpha),e_j(\alpha)])\\
	&=&
  \sum_{k\in I'\bs \{j\}}s_{jk}(\alpha) \cdot
  e_k(\alpha)+ {\rm pr}_{V_1(\alpha)^\perp}([e_1(\alpha),e_j(\alpha)])
	\,\,\,\in \,\,\,\tf (\infty)^\perp\,.
\en
As mentioned above, we prove our proposition by induction on $\vert I\vert$.
Suppose that, after renumbering,
$I=\{1,...,\vert I\vert\}$, and that (after passing to a subsequence)
$$
  \vert s_{23}(\alpha)\vert \geq \vert s_{jk}(\alpha)\vert
$$
for all $j,k\in I'$ and all $\alpha \in \NN$.
By (\ref{div1}) we know that 
$\vert s_{23}(\alpha)\vert>0$ for all  sufficiently large $\alpha$. Thus, for all $j,k\in I'$
\begn
 \lim_{\alpha \to \infty}\tfrac{\vert s_{jk}(\alpha)
\vert}{\vert s_{23}(\alpha)\vert}\in [0,1]\,. \label{eqn:limalphajk}
\enn
 Note  that
$\lim_{\alpha\to \infty} s_{23}(\alpha)=0$, since
$\tf=\oplus_{i=1}^{\vert I \vert}\langle e_i \rangle$ is
abelian.

By (\ref{div1}) we claim that for
each $j\in I'$ there exists some $k\in I'$ 
such that the limit in \eqref{eqn:limalphajk} is positive.
For $\vert I\vert =3$ this is trivially true.
To see the claim,  suppose that there exists $j_0 \in I'\bs \{2,3\}$ such 
that this limit is zero for all $k\in I'$. 
It follows that there exists a function $\tilde g:\RR \to \RR$ 
with ${\dsp \lim_{\alpha \to \infty} \tilde g(\alpha)=+\infty}$ and
\begn \label{eqn:j0k23}
  \tfrac{1}{ \vert I\vert} \cdot [123]_\alpha=
	 \tfrac{1}{ \vert I\vert}\cdot \vert s_{23}(\alpha) \vert^2 
	\geq 
	\vert s_{j_0k}(\alpha)\vert^2 \cdot \tilde g(\alpha)
	=[1j_0k]_\alpha\cdot \tilde g(\alpha)
\enn
for all $k \in I'$ and all $\alpha \in \NN$. Let $I_{j_0}:=
I \bs \{j_0\}$ and $I'_{j_0}:= I'\bs \{j_0\}$. Then
$$
\sum_{j,k \in I'} [1jk]_\alpha
=\sum_{j,k \in I'_{j_0}} [1jk]_\alpha
+2 \sum_{k \in I'_{j_0}}[1j_0k]_\alpha \,,
$$
and
\beg
  \sum_{j\in I', k\in I^{\co}} [1jk]_\alpha
	&=&
	 \sum_{j \in I_{j_0}', k \in I^{\co}} [1jk]_\alpha +
	   \sum_{k\in I^{\co}} [1j_0k]_\alpha \\
	&=&	
	\sum_{j \in I_{j_0}', k \in (I_{j_0})^\co} [1jk]_\alpha +
	   \sum_{k\in I^{\co}} [1j_0k]_\alpha 	-
		  \sum_{j \in I'_{j_0}} [1jj_0]_\alpha \,.
\en
We deduce from \eqref{div1}
\beg
  \sum_{j,k \in I'_{j_0}} [1jk]_\alpha
+  (2 +g(\alpha)) \cdot \sum_{k \in I'_{j_0}} [1j_0k]_\alpha
&>&
 g(\alpha) \cdot \sum_{j \in I'_{j_0}, k \in (I_{j_0})^\co} [1jk]_\alpha\,.
\en
By \eqref{eqn:j0k23} this yields
\beg
 \tilde g(\alpha) \cdot  \sum_{j,k \in I'_{j_0}} [1jk]_\alpha
+ (2 + g(\alpha)) \cdot [123]_\alpha
&>&
\tilde g(\alpha) \cdot g(\alpha) 
\cdot \sum_{j \in I'_{j_0}, k \in (I_{j_0})^\co} [1jk]_\alpha\,.
\en
Since  $\sum_{j,k \in I'_{j_0}} [1jk]_\alpha \geq [123]_\alpha $, using
$j_0 \neq 2,3$, this yields
\beg 
    \sum_{j,k \in I'_{j_0}} [1jk]_\alpha
&>&
\tfrac{\tilde g(\alpha) \cdot g(\alpha) }{\tilde g(\alpha)+2+g(\alpha)}
\cdot \sum_{j \in I'_{j_0}, k \in (I_{j_0})^\co} [1jk]_\alpha\,.
\en
This shows that for the torus 
${\dsp \tf_{j_0}:=\oplus_{i\in I_{j_0}} \langle e_i \rangle}$ of
dimension $\vert I\vert-1$
we obtain an inequality like \eqref{div1}, because as $\alpha \to \infty$,
$$
\tfrac{\tilde g(\alpha) \cdot g(\alpha) }{\tilde g(\alpha)+2+g(\alpha)} \to +\infty\,.
$$
 By the induction hypothesis,
we obtain a contradiction. This proves for each $j\in I'$ there exists some $k\in I'$ 
such that the limit in \eqref{eqn:limalphajk} is positive.

\medskip

{\it Step 2:} For every $j\in I'$ there exists some $k(j)\in I'\bs \{j\}$
with 
$$
  \vert s_{jk(j)}(\alpha)\vert \geq \vert s_{jk}(\alpha)\vert
$$
for all $k\in I'$ and $\alpha \in \NN$. By Step 1
we have
\begn
 \lim_{\alpha \to \infty} \tfrac{\vert s_{jk(j)}(\alpha)\vert}{\vert s_{23}
(\alpha)\vert} >0\, , \label{eqn:sjk-ratio-lower-bound}
\enn
which yields $\vert s_{jk(j)}(\alpha)\vert >0$ for all sufficiently large 
$\alpha \in \NN$.
Thus we obtain, for all $j \in I'$,
\begn
  [e_1(\alpha),e_j(\alpha)]= s_{jk(j)}(\alpha)\cdot
  (E_j^1(\alpha)+X_j^1(\alpha))\in \tf (\infty)^\perp \,,\label{Xjinf}
\enn
with
$$
  E_j^1(\alpha):=\tfrac{ {\rm pr}_{V_1(\alpha)}([e_1(\alpha),e_j(\alpha)])}{ s_{jk(j)}(\alpha)}
  \quad \textrm{ and }\quad
  X_j^1(\alpha):=\tfrac{ {\rm pr}_{V_1(\alpha)^\perp}([e_1(\alpha),e_j(\alpha)])}{ s_{jk(j)}(\alpha)}\,.
$$
The left hand side of \eqref{div1} 
is bounded from above by  $C' \cdot  (s_{23}(\alpha))^2$ by the very definition of $s_{23}(\alpha)$,
where $C'$ is independent of $\alpha$.
Meanwhile,  $\Vert X_j^1(\alpha)\Vert^2 \cdot   s_{jk(j)}(\alpha)^2 $ contributes to the right hand side.
Since by \eqref{eqn:sjk-ratio-lower-bound},
 the limit behavior of $s_{jk(j)}(\alpha)$ is like that of  $s_{23}(\alpha)$,
in order for the inequality in  \eqref{div1} to hold while $g(\alpha) \to \infty$,  we must have 
\begn
 \lim_{\alpha \to \infty} X_j^1(\alpha)= 0\,. \label{eqn:Xtozero}
\enn
Next, let $\cf(e_1)\leq \g$ denote
the centralizer of $e_1$.  
Since $e_1 \in \tf(\infty)$ and 
$[e_1,\tf']=0$, where ${\dsp \tf':=\oplus_{i\in I'} \langle e_j\rangle \subset \tf}$ 
is abelian, we obtain
\begn
\tf(\infty)
\subset \tf (\infty) \oplus 
\tf'
 \subset \cf(e_1)\,.\label{tstrcf}
\enn
Using that $1 \leq \Vert E_j^1(\alpha)\Vert \leq C''$ by the definition of $s_{jk(j)}(\alpha)$
this implies
\begn
    \underbrace{E_j(\infty)}_{\neq 0}:=
		\lim_{\alpha \to \infty} E_j^1(\alpha)=\lim_{\alpha \to \infty} (E_j^1(\alpha)+X_j^1(\alpha))
		\in
    \cf(e_1)\cap \tf (\infty)^\perp 
		\cap \tf'\,.\label{Ejinf}
\enn
This is seen as follows: By \eqref{Xjinf} we have $E_j^1(\alpha)+X_j^1(\alpha) \in \tf(\infty)^\perp$.
Because we know $\lim_{\alpha\to \infty}X_j^1(\alpha)=0$  this implies $E_j(\infty) \in \tf(\infty)^\perp$.
Moreover,  we have $E_j^1(\alpha)\in V_1(\alpha)$ by  definition. Since $V_1(\alpha)\to \tf'$ as 
$\alpha \to \infty$, we deduce $E_j(\infty)\in \tf'$, so it follows that $E_j(\infty) \in \cf(e_1)$ by  
 \eqref{tstrcf}.

An important consequence is that
\begn
  \tf(\infty) \subsetneq
  \tf (\infty) \oplus \tf'
 \subset \cf(e_1)\,.\label{cent1}
\enn
In fact we have that 
$$
  e_2,...,e_{\vert I\vert}, E_2(\infty),...,E_{\vert I\vert}(\infty) \in \cf(e_1)\,.
$$ 

\medskip

{\it Step 3:} We show by induction on $p$, for $1 \leq p \leq \dim \tf(\infty)$, that
$\tf(\infty)\subsetneq \cf (\tf(\infty))$,
which is clearly impossible. This will prove our contradiction for the specific choice of 
$Q$-orthonormal basis $b \in Z_{\vert I\vert}$
and sequence 
$(b_\alpha)_{\alpha\in \NN}:=(e_1(\alpha),...,e_{\dim \bldss\g}(\alpha))_{\alpha\in \NN}$ of
$Q$-orthonormal bases for $\g$ with $\lim_{\alpha \to \infty}b_\alpha=b$.  
Our induction hypotheses for each $p$, 
denoted by $(A1)_p$,...,$(A4)_p$,  are defined here:

Assumption $(A1)_p$:
There exist a $Q$-orthonormal frame
$(e_1^1:=e_1,...,e_1^p)$ in $\tf(\infty)$, positive numbers
$d_1(\alpha),...,d_p(\alpha)>0$, 
and $r_1^{p+1}(\alpha)\in \tf(\infty)$,  
$r_1^{p+1}(\alpha) \perp e_1^1,...,e_1^p$,
such that
\begin{eqnarray}
 e_1(\alpha)=d_1(\alpha)\cdot e_1^1+\cdots +d_p(\alpha)\cdot e_1^p
  +r_1^{p+1}(\alpha)\,.\label{eqn:rpplus1}
\end{eqnarray}
Before we define $(A2)_p$, we set $\cf(0):=\g$ and 
$$
 \cf(p) := \bigcap_{q=1}^p \cf(e_1^q)\,.
$$
For $j \in I'$ let
$$
 e_j^1=e_j
\quad \textrm{ and }\quad
 e_j^{p}(\alpha)= {\rm pr}_{\cf(p-1)}(e_j(\alpha))\,.
$$
Let
$$
 V_p(\alpha):=\langle e_j^p(\alpha),...,
e_{\vert I\vert}^p(\alpha)\rangle_{\bldss\RR}
$$
and let $V_p(\alpha)^\perp$ denote the $Q$-orthogonal complement of $V_p(\alpha)$
in $\cf(p-1)$. Let
$$
  E_j^p(\alpha):=\tfrac{ {\rm pr}_{V_p(\alpha)}([e_1(\alpha),e_j^p(\alpha)])}{ s_{jk(j)}(\alpha)}
  \quad \textrm{ and }\quad
  X_j^p(\alpha):=\tfrac{ {\rm pr}_{V_p(\alpha)^\perp}([e_1(\alpha),e_j^p(\alpha)])}{ s_{jk(j)}(\alpha)}\,,
$$
recalling, as we showed in Step 2, that $s_{jk(j)}(\alpha)>0$ for all sufficiently large $\alpha \in \NN$. 
%see line before \eqref{Xjinf}.

Assumption $(A2)_p$: We have
\begn
    \tf (\infty)\subsetneq
    \tf (\infty) \oplus \tf' \subset
    \cf(p)\,.\label{A2}
\enn
Notice that because  $e_1(\alpha)\in \tf(\infty)\subset \cf(p)\subset \cf(p-1)$, we have
$$
{\rm pr}_{\cf(p-1)}([e_1(\alpha),e_j(\alpha)])=
 [e_1(\alpha),e_j^{p}(\alpha)]=s_{jk(j)}\cdot 
( E^{p}_j(\alpha)+ X^{p}_j(\alpha)) \in \cf(p-1) \cap \tf(\infty)^\perp\,.
$$

Assumption $(A3)_p$: For every $j \in I'$ we know 
${\dsp \lim_{\alpha\to \infty}X_j^p(\alpha)=0}$ and
\begn  \label{A3}
 \lim_{\alpha \to \infty} (E^{p}_j(\alpha)+X_j^p(\alpha))=E_j(\infty) \in (\cf(p)\cap \tf'\cap \tf(\infty)^\perp)\bs \{0\}\,.
\enn

Assumption $(A4)_p$: For every $j \in I'$, we have
\begn \label{A4}
 \lim_{\alpha
\to \infty} e_j^{p}(\alpha)=e_j \in \cf(p) \cap \tf'\,.
\enn 

The base case $p=1$ of our induction was essentially established in Step 2: By definition we have ${\dsp \lim_{\alpha\to \infty}e_1(\alpha)=e_1}$,
that is, we can write $e_1(\alpha)=d_1(\alpha)\cdot e_1 + r_1^2(\alpha)$ with $d_1(\alpha)>0$,
$r_1^2(\alpha)\in \tf(\infty)$ and $e_1 \perp r_1^2(\alpha)$. Thus  $(A1)_1$ holds.
Next, $(A2)_1$ is established in \eqref{cent1}.  Equations 
\eqref{eqn:Xtozero} and (\ref{Ejinf}) give $(A3)_1$. And then,
because $e_j^1(\alpha)=e_j(\alpha) \to e_j \in \tf' \subset \cf(e_1)$, $(A4)_1$ is also clear.

We now prove the induction step $p \to p+1\leq \dim \tf(\infty)$:
By the induction hypothesis $(A2)_p$ we have $\tf(\infty) \subset \cf(p)$, 
thus $[\tf_\infty,\cf(p)^\perp]\subset \cf(p)^\perp$.
For $j \in I'$ we define $e_j^{p+1}(\alpha):={\rm pr}_{\bldss{\cf}(p)}(e_j(\alpha))$.
Then by $(A1)_p$ we have
\beg
 {\rm pr}_{\bldss{\cf}(p)}( [e_1(\alpha),e_j(\alpha) ] )
  =
  [e_1(\alpha),e_j^{p+1}(\alpha) ]
	=
 [ r_1^{p+1}(\alpha),e_j^{p+1}(\alpha)  ] \in \cf(p) \cap \tf(\infty)^\perp\,.
\en
Using $(A2)_p$, the equation after \eqref{A2}, and $\cf(p-1) \supset \cf(p)$, we see 
\begn
 [e_1(\alpha),e_j^{p+1}(\alpha)]
	&=&
    s_{jk(j)}(\alpha)\cdot \underbrace{ {\rm pr}_{\bldss{\cf}(p)}
    (E_j^p(\alpha)+X_j^p(\alpha))}_{=:\ti  E^{p+1}_j(\alpha)}
    \in \cf (p)\cap \tf(\infty)^\perp\,.\label{eqn:tEjp1}
\enn 
By  $(A3)_p$ we know 
\begn  \label{eqn:Ejinfty}
\lim_{\alpha\to \infty}X_j^p(\alpha)=0
 \textrm{   and   }
 \lim_{\alpha \to \infty}\ti E^{p+1}_j(\alpha)=E_j(\infty) \in (\cf(p)\cap \tf' \cap \tf(\infty)^\perp)\bs \{0\}\,.
\enn
Since $s_{jk(j)}(\alpha)>0$ we must have $r_1^{p+1}(\alpha)\neq
0$ for all sufficiently large $\alpha \in \NN$.
Set 
$$
  e_1^{p+1}(\alpha):=\tfrac{ r_1^{p+1}(\alpha) }{\Vert r_1^{p+1}(\alpha)\Vert}\,.
$$
 Then, by passing to a further subsequence, we may assume ${\dsp \lim_{\alpha \to \infty} 
e_1^{p+1}(\alpha)=e_1^{p+1} \in \tf(\infty)}$, with $e_1^{p+1}\perp
e_1,...,e_1^p$ and $\Vert e_1^{p+1} \Vert =1$. This shows $(A1)_{p+1}$.

 The induction hypothesis $(A4)_p$ guarantees for all $j \in I'$,
\begn \label{eqn:ejcp}
 \lim_{\alpha
\to \infty} e_j^{p+1}(\alpha)=e_j \in \cf(p) \cap \tf' \,.
\enn 

Thus for all sufficiently large  $\alpha \in \NN$,
$$
 V_{p+1}(\alpha):=\langle e_2^{p+1}(\alpha),...,
e_{\vert I\vert}^{p+1}(\alpha)\rangle_{\bldss\RR}
$$
is a  subspace of $\cf(p)$ of dimension $\vert I'\vert$, converging
to $\tf'$ as $\alpha \to \infty$.
Using \eqref{eqn:rpplus1}, the definition of $\cf(p)$, and $e_1^{p+1}(\alpha)$ we deduce from \eqref{eqn:tEjp1} for all $j \in I'$
\begn
 [e_1^{p+1}(\alpha),e_j^{p+1}(\alpha)]=
    \tfrac{s_{jk(j)}(\alpha)}{\Vert
r_1^{p+1}(\alpha) \Vert}\cdot 
     ( E^{p+1}_j(\alpha)+ X^{p+1}_j(\alpha))
   \in \cf (p) \cap \tf (\infty)^\perp \,,\label{Xjinf2}
\enn 
where 
${\dsp
 E^{p+1}_j(\alpha)
	 :=
	{\rm pr}_{V_{p+1}(\alpha) }(\ti E_j^{p+1}(\alpha))\in
  \cf (p)}$ and 
		${\dsp X^{p+1}_j(\alpha)
		:=
		{\rm pr}_{V_{p+1}(\alpha)^\perp }(\ti
  E_j^{p+1}(\alpha))\in \cf (p)}$. \\  
Notice, by Equation \eqref{Xjinf2}, 
$$
 [e_1(\alpha),e_j^{p+1}(\alpha)]=s_{jk(j)}\cdot 
( E^{p+1}_j(\alpha)+ X^{p+1}_j(\alpha)) \in \cf(p) \cap \tf(\infty)^\perp\,.
$$
 Since by \eqref{eqn:Ejinfty}
${\dsp \lim_{\alpha \to \infty}\ti
E^{p+1}_j(\alpha)=E_j(\infty)\in \tf'}$ and $V_{p+1}(\alpha)\to \tf'$
as $\alpha \to \infty$, we get 
\beg 
%\label{eqn:X2tozero}
\lim_{\alpha \to \infty} X^{p+1}_j(\alpha)=0
\quad \textrm{ and }\quad
 \lim_{\alpha \to \infty} E^{p+1}_j(\alpha)=E_j(\infty)\,.
\en 
This shows the first part of $(A3)_{p+1}$.
Next, we claim
${\dsp \lim_{\alpha \to
\infty}\tfrac{s_{jk(j)}(\alpha)}{\Vert
r_1^{p+1}(\alpha) \Vert}  =0}$.
To see this, notice that $\lim_{\alpha \to \infty}e_1^{p+1}(\alpha)=e_1^{p+1}\in \tf(\infty)$,
that $\lim_{\alpha \to \infty} e_j^{p+1}(\alpha) =e_j \in \tf'$ by  \eqref{eqn:ejcp}
and that $\lim_{\alpha \to \infty} (E_j^{p+1}(\alpha)+X_j^{p+1}(\alpha))=E_j(\infty) \in \tf'$
for all $j \in I'$. Now if  ${\dsp \tfrac{s_{jk(j)}(\alpha)}{\Vert
r_1^{p+1}(\alpha) \Vert} }$  had a uniform positive lower bound, then by passing to a subsequence if necessary,
then we deduce $Q([e_1^{p+1},e_j],E_j(\infty))>0$ by \eqref{Xjinf2}, because $E_j(\infty)\neq 0$. 
But since $e_j,E_j(\infty)\in \tf'$ and since $\tf'$ is abelian, the left-hand side equals  zero. 
This shows the above claim.

 By \eqref{Xjinf2} we now deduce  $[e_1^{p+1},e_j]=0$ for all $j \in I'$, that is,
$$
\tf(\infty) \subset
\tf (\infty) \oplus \tf' \subset \cf(e_1^{p+1})\,.
$$ 
Using this and $E_j(\infty)\in \cf(p)\cap \tf'$ we deduce
\beg
  E_j(\infty) \in \cf (p+1)
   \cap \tf (\infty)^\perp \cap \tf'\,,
\en
where $\cf(p+1)=\cf(p)\cap \cf(e_1^{p+1})$.
By \eqref{eqn:ejcp}, this  yields $(A2)_{p+1}$:
\beg
    \tf (\infty)\subsetneq
    \tf (\infty) \oplus \tf' \subset
    \cf(p+1)\,.
\en
Notice that this implies 
$$
  e_2,...,e_{\vert I\vert}, E_2(\infty),...,E_{\vert I\vert}(\infty) \in \cf(p+1)\,,
$$ 
implying both the second part of $(A3)_{p+1}$ and also $(A4)_{p+1}$.

So far we have proved that for a given $b\in Z_{\vert I\vert}$ and
a given sequence $(b_\alpha)_{\alpha\in \NN}$ converging to $b$, such 
a constant $C$ exists for all sufficiently large $\alpha$. 
We next show that $C$ can be chosen independently of both 
$b\in Z_{\vert I\vert}$ and the sequence $(b_\alpha)_{\alpha\in \NN}$.

To this end, suppose that there
exists a sequence $(b_\alpha)_{\alpha\in \NN}$  in $Z_{\vert I\vert}$
and corresponding sequences $(b_{\alpha,\beta})_{\beta \in \NN}$ with
$\lim_{\beta \to \infty}b_{\alpha,\beta}=b_\alpha$ for all $\alpha \in \NN$,
such that for the corresponding optimal constants,  as $\alpha \to \infty$, 
$C_\alpha \to +\infty$. Here by optimal we mean that the constants
$C_\alpha-1$ would not work. Since the constants $C_\alpha$
are chosen to be optimal, for each $\alpha \in \NN$ there exists $N(\alpha) \in \NN$,
such that  we may assume (after passing to a subsequence) that
\begn \label{div2}
 \sum_{j,k \in I'} [1jk]_{\alpha,\beta}  
  &\geq &
	 (C_\alpha-1) \cdot  \sum_{j\in I'\,,k \in I^{\co}} [1jk]_{\alpha,\beta}
\enn
for all $\beta \geq N(\alpha)$. By again passing to a subsequence, we may assume that this holds for all 
$\alpha,\beta \geq 1$.
Since $Z_{\vert I\vert}$ is compact,
we may assume that $\lim_{\alpha \to \infty}b_\alpha=b$. Moreover, it is clear that
we can construct a sequence $(\tilde b_\alpha)_{\alpha\in \N}$ converging
to $b$ and satisfying \eqref{div2}. Since $C_\alpha \to + \infty$
as $\alpha \to \infty$, we obtain a contradiction. This shows that   $C$ can be chosen independently.

To prove the theorem, suppose now that no such constant $C=C(G,Q)$ 
and no such  open neighborhood $U_{\vert I\vert}$ of $Z_{\vert I\vert}$ exist. 
Then there exists a sequence of bases $(b_\alpha)_{\alpha \in \NN}$ converging to $b \in Z_{\vert I\vert}$
for which \eqref{div2} holds. Contradiction.
\end{proof}

%%%%%%%%%%%%%%%%%%%%%%%%%%%%%%%%%%%%%%%%%%%%%%%%%%%%%%

\end{document}